\documentclass[1 [leqno,11pt]{amsart}
\usepackage{amssymb, amsmath}
\usepackage{color}
\usepackage{enumerate}
\def\triplenorm#1{{\vert\kern-0.25ex\vert\kern-0.25ex\vert #1
    \vert\kern-0.25ex\vert\kern-0.25ex\vert}}

\def\longformule#1#2{
\displaylines{ \qquad{#1} \hfill\cr \hfill {#2} \qquad\cr } }
\def\inte#1{
\displaystyle\mathop{#1\kern0pt}^\circ }

\let\pa=\partial
\let\al=\alpha
\let\b=\bar

\let\d=\delta
\let\e=\varepsilon

\let\f=\frac
\let\vf=\varphi
\let\p=\psi

\let\D=\Delta

\let\wt=\widetilde
\let\wh=\widehat

\def\cA{{\mathcal A}}
\def\dH{\dot{H}}

\def\cB{{\mathcal B}}
\def\cC{{\mathcal C}}
\def\cD{{\mathcal D}}
\def\cE{{\mathcal E}}
\def\cF{{\mathcal F}}

\def\cL{{\mathcal L}}

\def\cY{{\mathcal Y}}

\def\pa{\partial}

\def\dB{\dot{B}}

\def\v{{\rm v}}
\def\h{{\rm h}}
\def\Ga{\Gamma}

\def\bY{\bar{Y}}

\def\Djl{\Delta_j\Delta_\ell^{\rm v}}

\def\dD{\dot{\Delta}}
\def\virgp{\raise 2pt\hbox{,}}
\def\cdotpv{\raise 2pt\hbox{;}}

\def\eqdefa{\buildrel\hbox{\footnotesize def}\over =}

\def\C{\mathop{\mathbb C\kern 0pt}\nolimits}
\def\DD{\mathop{\mathbb D\kern 0pt}\nolimits}
\def\EE{\mathop{{\mathbb E \kern 0pt}}\nolimits}
\def\K{\mathop{\mathbb K\kern 0pt}\nolimits}
\def\N{\mathop{\mathbb N\kern 0pt}\nolimits}
\def\Q{\mathop{\mathbb Q\kern 0pt}\nolimits}
\def\R{\mathop{\mathbb R\kern 0pt}\nolimits}
\def\SS{\mathop{\mathbb S\kern 0pt}\nolimits}
\def\ZZ{\mathop{\mathbb Z\kern 0pt}\nolimits}
\def\TT{\mathop{\mathbb T\kern 0pt}\nolimits}
\def\P{\mathop{\mathbb P\kern 0pt}\nolimits}

\newcommand{\la}{\lambda}

\newcommand{\Z}{{\ZZ}}

\newcommand{\vv}[1]{\boldsymbol{#1}}

\def\dv{\mbox{div}}

\def\dive{\mathop{\rm div}\nolimits}

\def\no{\noindent}
\def\na{\nabla}
\def\p{\partial}
\def\th{\theta}

\newcommand{\w}[1]{\langle {#1} \rangle}
\newcommand{\beq}{\begin{equation}}
\newcommand{\eeq}{\end{equation}}
\newcommand{\ben}{\begin{eqnarray}}
\newcommand{\een}{\end{eqnarray}}
\newcommand{\beno}{\begin{eqnarray*}}
\newcommand{\eeno}{\end{eqnarray*}}
\newcommand{\andf}{\quad\hbox{and}\quad}
\newcommand{\with}{\quad\hbox{with}\quad}
\newtheorem{defi}{Definition}[section]
\newtheorem{thm}{Theorem}[section]
\newtheorem{lem}{Lemma}[section]
\newtheorem{rmk}{Remark}[section]
\newtheorem{col}{Corollary}[section]
\newtheorem{prop}{Proposition}[section]

 \numberwithin{equation}{section}

\begin{document}

\title[Decay of solutions to 3-D MHD system]
{Large time behavior  of solutions to 3-D  MHD system with initial data near
equilibrium }
\author[W.  Deng]{Wen Deng}
\address[W. Deng]{Academy of Mathematics $\&$ Systems Science
and  Hua Loo-Keng Key Laboratory of Mathematics, Chinese Academy of
Sciences, Beijing 100190, CHINA} \email{dengwen@amss.ac.cn}
\author[P. Zhang]{Ping Zhang} \address[P. Zhang]{Academy of Mathematics $\&$ Systems Science
and  Hua Loo-Keng Key Laboratory of Mathematics, Chinese Academy of
Sciences, Beijing 100190, CHINA, and School of Mathematical Sciences, University of Chinese Academy of Sciences, Beijing 100049, China.} \email{zp@amss.ac.cn}

\date{\today}

\maketitle \begin{abstract} In \cite{ChCa},  Califano and Chiuderi conjectured that the  energy of incompressible Magnetic hydrodynamical system is
dissipated at a rate that is independent of the ohmic resistivity.   The goal of this paper is to mathematically justify this conjecture in three space dimension provided that the  initial magnetic field and velocity
is a small perturbation of the equilibrium state $(e_3,0).$ In particular, we prove that for such data, 3-D incompressible MHD system without
magnetic diffusion  has a unique global solution. Furthermore, the velocity field and the difference between the magnetic field and $e_3$ decay to zero
in both $L^\infty$ and $L^2$ norms with explicit  rates. We point out that the decay rate in the $L^2$ norm is optimal in sense that this rate coincides with
that of  the linear system. The main idea of the proof is to exploit  H$\ddot{o}$rmander's version of Nash-Moser iteration scheme, which is very much motivated
by the seminar papers \cite{Kl80, Kl82, Kl84} by  Klainerman on the long time behavior to the evolution equations.
\end{abstract}


\vskip 0.2cm





\setcounter{equation}{0}
\section{Introduction}
In this paper, we investigate the large time behavior  of the global smooth
solutions to the following three-dimensional incompressible magnetic
hydrodynamical (or MHD in short)
 system with initial data being sufficiently close to the equilibrium state
 $(e_3,0):$
\beq\label{1.1} \left\{\begin{array}{l}
\displaystyle \pa_t b +u\cdot\nabla b=b\cdot\nabla u,\qquad (t,x)\in\R^+\times\R^3, \\
\displaystyle \pa_t  u+u\cdot\nabla u -\Delta u+\na p=b\cdot\nabla b, \\
\displaystyle \dv\, u =\dv\, b= 0, \\
\displaystyle (b,u)|_{t=0}=(b_0,u_0)\with b_0=e_3+\e \phi,
\end{array}\right.
\eeq where $b=(b^1,b^2,b^3)$ denotes the magnetic field,   $
u=(u^1,u^2,u^3)$ and $ p$ stand for the velocity and scalar pressure of the fluid
respectively.  This MHD system \eqref{1.1} with zero diffusivity in
 the magnetic field equation  can be applied to model plasmas when
the plasmas are strongly collisional, or the resistivity due to
these collisions are extremely small. One may check the references
\cite{Ca,CP, Dav01, LL} for more explanations to this system.

\medbreak

Whether there is dissipation or not for the magnetic field of
\eqref{1.1} is a very important problem  from physics of
plasmas. The heating of high temperature plasmas by MHD waves is one
of the most interesting and challenging problems of plasma physics
especially when the energy is injected into the system at the length
scales which are much larger than the dissipative ones. It has been
conjectured that in the two-dimensional MHD system, energy is
dissipated at a rate that is independent of the ohmic resistivity
\cite{ChCa}. In other words, the diffusivity for the
magnetic field equation can be zero yet the whole system may still
be dissipative.  The goal of this paper is to rigorously justify this conjecture in three space dimension provided that the  initial data of \eqref{1.1}
is a small perturbation of the equilibrium state $(e_3,0).$

\smallskip

Concerning the well-posedness issue of the system \eqref{1.1}, Chemin et al \cite{CMRR} proved
the local well-posedness of \eqref{1.1} with initial data in the
critical Besov spaces.
 Lin and
the second author \cite{LZMHD1} proved the global well-posedness to a
modified three-dimensional MHD system  with initial data
sufficiently close to the equilibrium state (see \cite{lin-zhang}
for a simplified proof).
Lin, Xu and the second author  \cite{LXZ}  established the global well-posedness of \eqref{1.1} in 2-D provided  that the initial data
is near the equilibrium state $(e_d,0)$ and the initial magnetic
field, $b_0,$ satisfies sort of admissible condition, namely
\beq \label{S0eq1} \int_{\R}(b_0-e_3)(Z(t,\al))\,dt=0\quad\mbox{for
all }\quad  \al\in \R^d\times\{0\} \eeq with $Z(t,\al)$ being
determined by \beno \left\{\begin{array}{l}
\displaystyle \f{d Z(t,\al)}{dt}= b_0(Z(t,\al)), \\
\displaystyle Z(t,\al)|_{t=0}=\al.
\end{array}\right. \eeno
Similar result in three space dimension was proved by Xu and the
second author in \cite{XZ1}.

In the 2-D case, the restriction \eqref{S0eq1} was removed by Ren,
Wu, Xiang and Zhang in \cite{RWXZ} by carefully exploiting the
divergence structure of the velocity field. Moreover, the authors
proved that \beq\label{S0eq2}
\|\p_{x_2}^kb(t)\|_{L^2}+\|\p_{x_2}^ku(t)\|_{L^2}\leq
C\w{t}^{-\f{1+2k}4+\epsilon} \quad\mbox{for any } \ \epsilon\in ]0, 1/2[\andf
k=0,1,2, \eeq  where $\w{t}\eqdefa\left(1+t^2\right)^{\f12}. $ A more elementary existence proof was also given by
Zhang in \cite{ZhangTing}. Very recently, Abidi and the second author removed the restriction \eqref{S0eq1} in \cite{AZ5} for the 3-D MHD system.
Moreover, if the initial magnetic field equals to $e_3$ and with other technical assumptions, this solution decays to zero according
to
\beq\label{S0eq3}
\|u(t)\|_{H^2}+\|b(t)-e_3\|_{H^2}\leq C\w{t}^{-\f14}.
\eeq
Note that \eqref{S0eq3} corresponding to the critical case of \eqref{S0eq2}, that is, $\epsilon=0$ in \eqref{S0eq2}.

\medbreak This idea of considering the global well-posedness of MHD system with initial data close
to the equilibrium sate $(e_d,0)$ goes back to the work by Bardos, Sulem and Sulem \cite{BSS88} for
the global well-posedness of ideal incompressible MHD system.
 In general,  it is not known whether or not classical
solutions of \eqref{1.1} can develop finite time singularities even
in  two dimension.  In the case when there is full magnetic diffusion
in \eqref{1.1}, one may check   \cite{DL} for its local
well-posedness  in the classical Sobolev spaces, and
\cite{ST} for the global well-posedness of such a  system  in
two space dimension.
With mixed partial dissipation and additional magnetic diffusion in
the two-dimensional MHD system, Cao and Wu \cite{CW} (see also
\cite{CRW}) proved that such a system is globally well-posed for any
data in $H^2(\R^2).$
Lately He, Xu and Yu \cite{HXY16} (see also  \cite{CL16} and  \cite{WZ16}) justified the vanishing viscosity limit of the full
diffusive MHD system to the solution constructed by Bardos et al in \cite{BSS88} for the ideal MHD system.

 \medbreak

The main result of this paper states as follows:

\begin{thm}\label{thmmain}
{\sl Let $e_3=(0,0,1),$ $b_0=e_3+\e\phi$ with
$\phi=(\phi_1,\phi_2,\phi_3)\in C_c^\infty$ and $\dive \phi=0,$ let
$u_0\in W^{N_0,1}\cap H^{N_0}$ for some integer $N_0$ sufficiently large. Then there exist sufficiently
small positive constants $\e_0, c_0$ such that if \beq\label{S0eq4}
\|u_0\|_{W^{N_0,1}}+\|u_0\|_{H^{N_0}}\leq c_0 \andf \e\leq \e_0, \eeq  \eqref{1.1}
has a unique global solution $(b, u)$ so that for any $T>0,$
$b\in C^2([0,T]\times\R^3),$ $ u\in C^2([0,T]\times\R^3)$. Moreover, for some $\kappa>0,$  there hold
\beq\label{S0eq5}
\begin{split}
&\|u(t)\|_{W^{2,\infty}}\leq C_\kappa\w{t}^{-\frac54+\kappa},\quad\qquad \|b(t)-e_3\|_{W^{2,\infty}}\leq C_\kappa\w{t}^{-\f34+\kappa} \andf\\
&\|u(t)\|_{H^2}+\|b(t)-e_3\|_{H^2}\leq C\langle t\rangle^{-\frac12},\quad \|\nabla u(t)\|_{L^2}\leq C\langle t\rangle^{-1}.
\end{split}
\eeq }
\end{thm}

Let us remark that the above theorem recovers the global well-posedness result of the system \eqref{1.1} in \cite{AZ5}. Moreover, the bigger the integer $N_0,$ the smaller
 the positive constant $\kappa.$ The main idea of
the proof here works in both two space dimension and in three space dimension.
The $L^\infty$ decay rates of the solution in \eqref{S0eq5} are completely new. The $L^2$ decay rates of the solution  are optimal in the sense that
these decay rates coincide with those of the linearized system (see Propositions \ref{S0prop1} and \ref{S5thm2} below), which greatly improves
the rate given by
 \eqref{S0eq3}. We can also work on the decay rates for the higher order derivatives of the solutions. But we choose not to pursue on this direction here.

\setcounter{equation}{0}
\section {Structure and strategies  of the proof}
\label{Sect2}

\subsection{Lagrangian formulation of \eqref{1.1}}\label{subsect2.1}
As observed in the previous references (\cite{LXZ, XZ1}), the linearized system of \eqref{1.1} around the equilibrium state $(e_3,0)$
reads
\beq\label{S1eq1}\left\{\begin{aligned}
&Y_{tt}-\Delta Y_t-\p_3^2 Y=\vv f\quad\text{in}\ \R^+\times\R^3,\\
&Y|_{t=0}=Y^{(0)},\quad Y_t|_{t=0}=Y^{(1)}.
\end{aligned}\right.\eeq
 It is easy to calculate that this system has two different
eigenvalues \beq\label{S2eq1qe}
\la_1(\xi)=-\f{|\xi|^2}2+\sqrt{\f{|\xi|^4}4-\xi_3^2} \andf
\la_2(\xi)=-\f{|\xi|^2}2-\sqrt{\f{|\xi|^4}4-\xi_3^2}. \eeq
 The Fourier modes
corresponding to $\la_2(\xi)$ decays like $e^{-t|\xi|^2}$.  Whereas the
decay property of the Fourier modes corresponding to $\la_-(\xi)$  varies
with directions of $\xi$ as \beno
\la_1(\xi)=-\f{2\xi_3^2}{|\xi|^2\bigl(1+\sqrt{1-\f{4\xi_3^2}{|\xi|^4}}\bigr)}
\to -1\quad \mbox{as}\quad |\xi|\to \infty \eeno only in the $\xi_3$
direction. This simple  analysis shows that the dissipative
properties of the system  \eqref{S1eq1} may be more complicated
than that for the linearized system of isentropic compressible
Navier-Stokes system (see \cite{Da} for instance). Moreover, it is well-known that it is in general impossible to
propagate the anisotropic regularities for the transport equation. This motivates us to use the Lagrangian formulation of
the system \eqref{1.1}.\smallskip

Let us now recall the Lagrangian formulation of \eqref{1.1} from \cite{AZ5}.
Let $(b,u)$ be a smooth enough solution of \eqref{1.1}, we
define  \beq\label{1.1f}
\begin{split}
X(t,y)=&y+\int_0^t u(t',X(t',y))dt'\eqdefa y+Y(t,y), \quad
\vv u(t,y)\eqdefa u(t,X(t,y)),\\
  {\vv b}(t,y)\eqdefa & b(t,X(t,y)),\quad
\vv p(t,y)\eqdefa  p(t,X(t,y)),\quad  \cA\eqdefa\left(Id+\na_y Y\right)^{-1} \andf \na_Y\eqdefa \cA^{t}\na_y.
\end{split} \eeq
Then $(Y, \vv b, \vv u, \vv p)$ solves  \beq\label{1.2} \left\{\begin{array}{l}
\displaystyle {\vv b}(t,y) =\p_{b_0}X(t,y),\quad \na_Y\cdot{\vv b}=0,\\
\displaystyle   Y_{tt} -\Delta_y Y_t-\partial_{b_0}^2Y=\p_{b_0}b_0+g,\\
\displaystyle Y_{|t=0}=Y^{(0)}=0,\qquad {Y_t}_{|t=0}=Y^{(1)}=u_0(y),
\end{array}\right.
\eeq where \beq\label{P}
\begin{split}
&g=\dv_y\bigl[(\mathcal{A}\mathcal{A}^{t}-Id)\na_yY_t\bigr]
-\mathcal{A}^{t}\na_y\vv p,\quad \p_{b_0}\eqdefa b_0\cdot\na_y, \andf\\
& (\Delta_x p)(t, X(t,y))
=\sum_{i,j=1}^3\na_{Y^i}\na_{Y^j}\bigl(\pa_{b_0}X^i\pa_{b_0}X^j-Y^i_tY^j_t\bigr)(t,y).
\end{split}
\eeq

In what follows, we assume that \beq \label{S1eq5}
\mbox{supp}(b_0(x_\h,\cdot)-e_3)\subset [0, K]\andf b_0^3\neq 0.
\eeq

Due to the difficulty of the variable coefficients for the
linearized system of \eqref{1.2}, we shall use Frobenius Theorem
type argument to find a new coordinate system $\{z\}$ so that
$\p_{b_0}=\p_{z_3}.$
Toward this, let us define \beq\label{S1eq13}
\left\{\begin{array}{l}
\displaystyle \frac{d y_1}{d y_3}=\f{b_0^1}{b_0^3}(y_1,y_2,y_3),\quad y_1|_{y_3=0}=w_1, \\
\displaystyle \frac{d y_2}{d
y_3}=\f{b_0^2}{b_0^3}(y_1,y_2,y_3),\quad y_2|_{y_3=0}=w_2,\\
\displaystyle y_3=w_3,
\end{array}\right.  \eeq
and
 \beq \label{CH}
\begin{split}
& z_1=w_1,\quad z_2=w_2,\quad
z_3=w_3+\int_{0}^{w_3}\Bigl(\f1{b_0^3(y(w))}-1\Bigr)\,dw_3'.
\end{split}
\eeq Then we have \beq\label{CHa}
\begin{split}
&\p_{b_0}f(y)=\f{\p
f(y(w(z)))}{\p
z_3},\andf \na_y=\na_Z= {\cB}^{t}(z)\na_z \with {\cB}(z)=\Bigl(\f{\p
y(w(z))}{\p z}\Bigr)^{-1}.
\end{split}
\eeq It is easy to observe that \beno \begin{split}
{\cB}(z)=\Bigl(\f{\p y(w(z))}{\p z}\Bigr)^{-1}=&\Bigl(\f{\p
y(w(z))}{\p w}\times\f{\p w(z)}{\p z}
\Bigr)^{-1}\\
=&\Bigl(\f{\p w(z)}{\p z} \Bigr)^{-1}\Bigl(\f{\p y(w(z))}{\p
w}\Bigr)^{-1}=\Bigl(\f{\p z}{\p w}\Bigr)\Bigl(\f{\p y(w(z))}{\p
w}\Bigr)^{-1}.
\end{split} \eeno
Yet it follows from \eqref{S1eq13} that \beq\label{CHab}
\begin{split}
\Bigl(\f{\p y(w)}{\p w}\Bigr) = & \begin{pmatrix}
1&0&\f{b_0^1}{b_0^3}\\ 0&1&\f{b_0^2}{b_0^3}\\0&0&1\end{pmatrix}+
\begin{pmatrix} \int_0^{w_3}\f{\p}{\p
y_1}\bigl(\f{b_0^1}{b_0^3}\bigr)dy_3'&\int_0^{w_3}\f{\p}{\p
y_2}\bigl(\f{b_0^1}{b_0^3}\bigr)dy_3'&0 \\
\int_0^{w_3}\f{\p}{\p y_1}\bigl(\f{b_0^2}{b_0^3}\bigr)dy_3' & \int_0^{w_3}\f{\p}{\p y_2}\bigl(\f{b_0^2}{b_0^3}\bigr)dy_3' & 0\\
0 & 0 &  0 \end{pmatrix}  \begin{pmatrix} \f{\p y_1}{\p w_1} & \f{\p
y_1}{\p w_2} & \f{\p y_1}{\p w_3}\\  \f{\p y_2}{\p w_1} & \f{\p
y_2}{\p w_2} & \f{\p y_2}{\p w_3}\\  \f{\p y_3}{\p w_1} &
\f{\p y_3}{\p w_2} & \f{\p y_3}{\p w_3}\end{pmatrix}\\
\eqdefa & A_1(y(w))+ A_2(y(w))\Bigl(\f{\p y(w)}{\p w}\Bigr),
\end{split} \eeq which gives \beq \label{1.2fg} \Bigl(\f{\p y(w)}{\p
w}\Bigr) =\bigl(Id-A_2(y(w))\bigr)^{-1}A_1(y(w)). \eeq While it is
easy to observe that \beq \label{1.2fh} \Bigl(\f{\p z(w)}{\p
w}\Bigr) =
\begin{pmatrix}
1&0& 0\\
0&1&0\\
\int_0^{w_3}\f{\p}{\p w_1}\bigl(\f{1}{b_0^3(y(w))}\bigr)dw_3' &
\int_0^{w_3}\f{\p}{\p w_2}\bigl(\f{1}{b_0^3(y(w))}\bigr)dy_3' &
\f{1}{b_0^3}
\end{pmatrix}\eqdefa A_3(w).
\eeq As a consequence, we obtain \beq \label{1.2fk}
\begin{split}
& y(w)=(y_\h(w_\h,w_3),w_3),\quad w(z)=(z_\h,w_3(z)), \andf
y(w(z))=\bigl(y_\h(z_\h,w_3(z)),w_3(z)\bigr),
\\
&{\cB}(z)=A_3(w(z))A_1^{-1}(y(w(z))\bigl(Id-A_2(w(z))\bigr),
\end{split} \eeq with the matrices $A_1,A_2, A_3$ being determined by
\eqref{CHab} and \eqref{1.2fh} respectively.

For simplicity, let us abuse the notation that
$Y(t,z)=Y(t,y(w(z))).$ Then the system \eqref{1.2} becomes
\beq\label{S1eq19} \left\{\begin{array}{l}
\displaystyle  Y_{tt} -\Delta_z Y_t-\partial_{z_3}^2Y=\bigl(\na_Z\cdot \na_Z-\D_z)Y_t+\p_{z_3}b_0(y(w(z)))+g(y(w(z))),\\
\displaystyle Y_{|t=0}=Y_0=0,\qquad
{Y_t}_{|t=0}=Y_1(z)=u_0(y(w(z))),
\end{array}\right.
\eeq for $g$ given by \eqref{1.2}. Since $\p_{z_3}b_0(y(w(z)))$ in
the source term is a time independent function, we now introduce  a smooth cut-off
function $\eta(z_3)$ with $\eta(z_3)= \left\{\begin{array}{l}
\displaystyle 0,\,\,z_3\geq 2+K, \\
\displaystyle 1,\,\,-1\leq z_3\leq 1+K, \\
\displaystyle 0,\,\,z_3\leq -2,
\end{array}\right.$ and  a
correction term $\tilde{Y}$ so that $Y=\tilde{Y}+\bY$ and
 \beq \label{S1eq17}\begin{split}
\wt{Y}(z)
=&\eta(z_3)\Bigl(\int_{-1}^{z_3}\bigl(e_3-b_0(y(w(z_\h,z_3')))\bigr)\,dz_3'
-\int_{-1}^{K+1}\bigl(e_3-b_0(y(w(z_\h,z_3')))\bigr)\,dz_3'\Bigr),
\end{split} \eeq
which satisfies
\beq\label{S1eq18}  \p_{z_3}\tilde{Y}(z)=e_3-b_0(y(w(z))), \andf  \p_{z_3}\bigl(\p_{z_3}\tilde{Y}+b_0(y(w(z)))\bigr)=0. \eeq
Then  in view of (2.23), (2.24) and (2.30) of \cite{AZ5}, $\b Y$ solves
 \beq\label{1.2b}
\left\{\begin{array}{l}
\displaystyle  \bY_{tt} -\Delta_z \bY_t-\partial_{z_3}^2\bY=f,\\
\displaystyle {\bY}_{|t=0}={\bY}^{(0)}=-\tilde{Y},\qquad
{\b{Y_t}}_{|t=0}=Y^{(1)},
\end{array}\right.
\eeq with \beq\label{S2-27}
\begin{split}
 \mathcal{A}=&\left(Id+{\cB}^{t}\na_z\wt
Y+{\cB}^{t}\na_z \bY\right)^{-1},\andf\\
f=&{\cB}^{t}\na_z\cdot\bigl[(\mathcal{A}\mathcal{A}^{t}-Id){\cB}^{t}\na_z\b
Y_t\bigr] +{\cB}^{t}\na_z\cdot({\cB}^{t}\na_z\b Y_t\bigr)
-\Delta_z\b Y_t-({\cB}\mathcal{A})^{t}\na_z\vv p,\\
\na_z \vv
p=&-\na_z\D_z^{-1}\dv_z\bigl(\det(\cB^{-1})(\cB\mathcal{A}\mathcal{A}^{t}\cB^{t}-Id)\na_z
\vv p\bigr)\\
& -\na_z\D_z^{-1}\dv_z\bigl((\det(\cB^{-1})Id-Id)\na_z \vv p\bigr)
\\&
+\na_z\D_z^{-1}\dv_z\Bigl(\cB\mathcal{A}\dv_z
\left(\det(\cB^{-1})\cB\mathcal{A}\bigl(\pa_{3}\b Y\otimes\pa_{3}\b
Y -\b Y_t\otimes\b Y_t\bigr)\right)\Bigr).
\end{split}
\eeq

\subsection{The proof of Theorem \ref{thmmain}}
Before presenting the main result for the system (\ref{1.2b}-\ref{S2-27}), let us first introduce notations of the norms:
For $f:\R_y^3\to \R,$  $u:\R^+\times\R_y^3\to \R$, and $p\in[1,+\infty]$, $N\in\N$, we denote
\begin{align*}
\|f\|_{W^{N,p}}\eqdefa\sum_{|\alpha|\leq N}\|D_y^\alpha f\|_{L^p} \andf \|u\|_{L^p;k,N}\eqdefa\sup_{t>0}(1+t)^k\|u(t)\|_{W^{N,p}}.
\end{align*}
In particular, when $p=1,$ $p=2$ and $p=\infty,$ we simplify the notations as
\beq\label{S1eq3}
\begin{split}
&\triplenorm{f}_N\eqdefa \|f\|_{W^{N,1}},\qquad  \|f\|_{N}\eqdefa \|f\|_{H^N}, \qquad \quad\  \ |f|_N\eqdefa \|f\|_{W^{N,\infty}} \\
&\text{and}\qquad\qquad
\|u\|_{k,N}\eqdefa \|u\|_{L^2;k,N},\qquad |u|_{k,N}\eqdefa \|u\|_{L^\infty;k,N}.
\end{split}
\eeq
\begin{thm}\label{Th1}
{\sl There exist an integer $L_0$  and small constants $\eta, \e_0>0$ such that if
\beq\label{S1eq4}
\triplenorm{(\bar{Y}^{(0)},Y^{(1)})}_{L_0}+\|(\bar{Y}^{(0)},Y^{(1)})\|_{L_0}\leq \eta \andf \e\leq \e_0.\eeq
Then the system \eqref{1.2b} has a unique global solution $\bar{Y}\in C^{2}([0,\infty);C^{N_1-4}(\R^3))$, where $N_1=[(L_0-12)/2]$. Furthermore, for any $\kappa>0,$ there hold
\beq \label{S1eq5}
|\p_3\bar Y|_{\frac34-\kappa,2}+|\bar Y_{t}|_{\frac54-\kappa,2}+|\bar Y|_{\frac14-\kappa,2}\leq C_\kappa\eta,
\end{equation}
 and
\begin{equation}\label{S1eq6}
\begin{split}
&\||D|^{-1}(\p_3\bar Y,\bar Y_{t})\|_{0,N_1+2}+\|\nabla \bar Y\|_{0,N_1+1} +\|(\bar Y_{t},\p_3\bar Y)\|_{\frac12,N_1+1}+\|\nabla \bar Y_{t}\|_{1,N_1-1}
\\
&\qquad\qquad\qquad +\|\bar Y_{t}\|_{L_t^2(H^{N_1+2})}+\bigl\|(\p_3\bar Y,\langle t\rangle^\frac12\nabla \bar Y_{t})\bigr\|_{L_t^2(H^{N_1+1})}+\|\bar Y_{tt}\|_{\frac12,N_1-2}\leq C.
\end{split}
\end{equation}}
\end{thm}

Admitting Theorem \ref{Th1} for the time being, let us now turn to the proof of Theorem \ref{thmmain}.

\begin{proof}[Proof of Theorem \ref{thmmain}]
Indeed,
in view of \eqref{1.1f}, one has \beq\label{S10eq1}
\begin{split}
&Y_1(z)=u_0(y_\h(z_\h,w_3(z)),w_3(z))\ \andf
\vv u(t,y)=Y_t(t,y+Y(t,y)), \\
& \vv b(t,y)=b_0(y)+b_0(y)\cdot\na_yY(t,y) \with
Y(t,(y_\h(z_\h,w_3(z)),w_3(z)))=\wt{Y}(z)+\bar{Y}(t,z),
\end{split} \eeq with $\wt{Y}(z)$ and $\bar{Y}(t,z)$ being
determined by \eqref{S1eq17} and \eqref{1.2b} respectively.

Whereas in view of \eqref{CHab},  \eqref{1.2fh} and \eqref{1.2fk}, we get, by a similar proof of Lemma 4.3 of \cite{AZ5} that for any $N\in\N,$
\beq\label{S10eq2}
|(\cB-Id)|_N\leq C_N\e.
\eeq
So that under the assumptions of \eqref{S0eq4}, there holds \eqref{S1eq4}. Then Theorem \ref{Th1} ensures  that the system (\ref{1.2b}-\ref{S2-27}) has
a unique global classical solution
$\bar{Y}\in C^2([0,\infty);C^{N_1-4}(\R^3)),$  which verifies \eqref{S1eq5} and \eqref{S1eq6}. In particular, it follows from \eqref{S1eq17} and
\eqref{S1eq5} that \beno
|\na_zY|_{0,1}\leq |\p_z\wt{Y}|_1+|\na \bar{Y}|_{0,1}\leq C(\e+\eta),
\eeno
which together with \eqref{S10eq1} ensures that $\vv u\in C^2([0,\infty)\times\R^3)$ and $\vv b\in C^2([0,\infty)\times\R^3).$
Furthermore due to
\beno
\bigl|\f{\p X}{\p y}-Id\bigr|_{0,1}=\bigl|^t\cB\na_zY|_{0,1}\leq C(\e+\eta),
\eeno
we deduce from \eqref{1.1f} that $u\in C^2([0,\infty)\times\R^3)$ and $b\in C^2([0,\infty)\times\R^3)$ which verifies the system \eqref{1.1}
thanks to the derivation at the beginning of Subsection \ref{subsect2.1}.

On the other hand,  by virtue of \eqref{S1eq18}, we have
\beno
\vv b(t,y(w(z)))=b_0(y(w(z)))+\p_3\wt{Y}(z)+\p_3\bar{Y}(t,z)=e_3+\p_3\bar{Y}(t,z),
\eeno
which together with \eqref{S1eq5}, \eqref{S1eq6} and \eqref{S10eq1} implies that
there holds \eqref{S0eq5}. This completes the proof of Theorem \ref{thmmain}.
\end{proof}

\subsection{Strategies of the proof to Theorem \ref{Th1}}

Observing  from the calculations in \cite{AZ5} that under the assumptions of Theorem \ref{thmmain},  the matrix $\cB$ given by \eqref{1.2fk} is sufficiently close to
the identity matrix in the norms of $W^{N_0,1}$ and $H^{N_0}$ as long as $\e$ is sufficiently small. To avoid cumbersome calculation, here we just
prove Theorem \ref{Th1} for
the system \eqref{S1eq1}  with \beq\label{S1eq2}
\begin{split}
& \mathcal{A}=(Id+\na_y Y)^{-1},\quad
f=\na_y\cdot\bigl((\mathcal{A}\mathcal{A}^{t}-Id)\na_y Y_t\bigr)
-\mathcal{A}^{t}\na_y\vv p,\andf\\
& \vv p=-\D_y^{-1}\dv_y\bigl((\mathcal{A}\mathcal{A}^{t}-Id)\na_y
\vv p\bigr) +\D_y^{-1}\dv_y\Bigl(\mathcal{A}\dv_y
\bigl(\mathcal{A}\bigl(\pa_{y_3} Y\otimes\pa_{y_3} Y -Y_t\otimes
Y_t\bigr)\bigr)\Bigr), \end{split} \eeq
which corresponds to $\cB=Id$ in \eqref{1.2b}. The general case follows along the same line.

 Let us remark that
the system \eqref{S1eq1} is not scaling, rotation and Lorentz invariant, so that Klainerman's vector field
method (\cite{Kl85}) can not be applied here. Yet the ideas developed by Klainerman in the seminar papers
\cite{Kl80, Kl82, Kl84} can be well adapted for this system. We now recall the classical result on the global well-posedness to some evolutionary
system from \cite{Kl82}.  Let us consider  the following system
\beq\label{S1eq20} \left\{\begin{array}{l}
\displaystyle u_t-\cL u=F(u,Du) \with Du=(u_t, u_{x_1},\cdots, u_{x_d}),\\
\displaystyle u_{|t=0}=u_,\qquad Pu_0=0,
\end{array}\right.
\eeq
where $\cL\eqdefa \sum_{|\al|\leq \gamma}a_\al D_x^\al$ with
$a_\al$ being $r\times r$ matrices with constant entries.  Under the assumptions that

\begin{itemize}
\item[(1)]  $\cL$ satisfies a dissipative condition of the following type: there exists a positive definite $r\times r$ matrix $A$
such that
\beno
\text{either} \ \int_{\R^d}\Re (A \cL f, f)\,dx\leq 0 \quad\text{or} \quad \int_{\R^d}\Re (A \cL f, f)\,dx\leq-\|\na f\|_{L^2}^2
\eeno for any $f\in C_c^\infty;$

\item[(2)] Let $\Gamma(t)u_0$ be the solution of
\beno
\p_tu-\cL u=0\andf u(0,x)=u_0(x).
\eeno
There is a differential matrix $P$ such that
\beno
|\Gamma(t)u_0|_{0}\leq C\w{t}^{-k_0}\triplenorm{u_0}_d
\eeno
for any $u_0\in W^{d,1}\cap L^\infty$ that satisfies $P u_0=0;$

\item[(3)] $AF_{u_t}, A F_{u_{x_i}},$ $i=1,\cdots, d,$ are symmetric matrices and $F_{u_t}$ is independent of $u_t.$ Moreover
\beq \label{S1eq21a}
|F(u,Du)|\leq C(|u|+|Du|)^{p+1}\quad \text{for}\ |u|+|Du|\ \text{sufficiently small;}
\eeq

\item[(4)] $p$ is an integer and $F$ is a smooth function so that there holds
\beq\label{S1eq21b}
\f1p\left(1+\f1p\right)<k_0; \eeq
\end{itemize}
Klainerman proved in \cite{Kl82} the following celebrated theorem:

\begin{thm}[Theorem 1 of \cite{Kl82}]\label{thm2}
{\sl There exist an integer $N_0>0$ and a small constant $\eta>0$ such that if
\beno
\triplenorm{u_0}_{N_0}+\|u_0\|_{N_0}\leq \eta,
\eeno
\eqref{S1eq20} has a unique solution $u\in C^1([0,T]; C^{\gamma})$ for any $T>0.$ Moreover, the solution behaves, for $t$ large, like
\beq \label{S1eq21}
|u(t,x)|= O\left(t^{-\f{1+\e}p}\right)\quad\text{as}\quad t\to\infty,
\eeq
for some small $\e>0.$ Also
\beq
\label{S1eq22}
\|u(t)\|_{L^2}=O(1)\quad\text{as}\quad t\to\infty. \eeq }\end{thm}

Let us remark that due to the appearance of the double Riesz transform in the expression of $f$ in \eqref{S1eq2}, the source term $f$ in \eqref{S1eq1} can not
satisfy the growth condition \eqref{S1eq21a}; secondly, even if we can assume the source term $f$ is in quadratic growth of $(Y_t, \p_3Y),$ that corresponds to $p=1$ in \eqref{S1eq21a},
the growth rate obtained in \eqref{S2eq3} below does not meet the requirement of \eqref{S1eq21b}. This makes it impossible to apply Theorem  \ref{thm2}
for the system \eqref{S1eq1}. Yet by considering the specific anisotropic structure of the system \eqref{S1eq1}, we can still succeed in applying Nash-Moser
scheme to establish the global existence as well as the large time behavior of solutions to (\ref{S1eq1}-\ref{S1eq2}).

Now we outline the proof of Theorem \ref{Th1}. According to the strategy in \cite{Kl80, Kl82, Kl84}, the first step is to study the decay properties
of the linear system:
\beq\label{S2eq1}
\left\{\begin{array}{l}
\displaystyle  \cY_{tt} -\Delta\cY_t-\partial_{3}^2\cY=0,\\
\displaystyle {\cY}_{|t=0}=\cY_0,\qquad {{\cY_t}}_{|t=0}=\cY_1.
\end{array}\right.
\eeq
\begin{prop}\label{S0prop1}
{\sl Let  $\cY(t)$ be a smooth enough solution of  \eqref{S2eq1}. Given $\delta\in[0,1]$, $N\in\N$, there exist $C_{\d,N}, C_N>0$ such that
\begin{equation}\label{S2eq3as}
\begin{split}
&|\p_3\cY|_{1,N}+|\p_t\cY|_{\frac32-\delta, N}+|\cY|_{\frac12,N}\leq  C_{\d,N}\bigl(\||D|^{2\delta}(\cY_0,\cY_1)\|_{L^1}+\||D|^{N+4}(\Delta\cY_0,\cY_1)\|_{L^1}\bigr);
\end{split}
\end{equation}
\begin{equation}\label{S2eq13as}
\begin{split}
\|(\p_t\cY,\p_3\cY)\|_{L^\infty_t(H^{N+1})}+&\|\Delta\cY\|_{L^\infty_t(H^N)}+\|\na\p_t\cY\|_{L^2_t(H^{N+1})}\\
 +&\|\na\p_3\cY\|_{L^2_t(H^N)}\leq C_N\bigl(\|(\p_3\cY_0,\cY_1)\|_{{N+1}}+\|\D\cY_0\|_N\bigr);
\end{split}
\end{equation}
\begin{equation}\label{S2eq16}
\begin{split}
&\|\langle t\rangle^\frac12(\p_t\cY,\p_3\cY)\|_{L_t^\infty(H^N)}+\|\langle t\rangle^\frac12\nabla \p_t\cY\|_{L_t^2(H^{N})}\leq C_N\bigl(\||D|^{-1}(\p_3\cY_0,\cY_1)\|_{{N+1}}+\|\na\cY_0\|_N\bigr);
\end{split}
\end{equation}
\begin{equation}
\begin{split}
&\|\langle t\rangle\Delta\p_t\cY\|_{L_t^\infty(H^N)}\leq C_N\|(\Delta\cY_0,\cY_1)\|_{N+2}.\label{S2eq15'as}
\end{split}
\end{equation}
}
\end{prop}
We emphasize here the estimates of \eqref{S2eq3as} and \eqref{S2eq13as} are of anisotropic type, which means that the decay rates of the partial derivatives of the solution to \eqref{S2eq1} are different, which is consistent with the heuristic discussions at the beginning of Section \ref{Sect2}. Moreover, the estimate of \eqref{S2eq3as} is valid for $\d=0.$ Similar estimates as
  \eqref{S2eq16} and \eqref{S2eq15'as} were not proved in  \cite{Kl80, Kl82, Kl84}. They are purely due to
the special structure of the linearized system \eqref{S2eq1}.

With the above proposition,  we next turn to the  decay estimates for the solutions of the following inhomogeneous equation of \eqref{S2eq1}
\begin{equation}\label{S2eq18}\begin{cases}
Y_{tt}-\Delta Y_t-\p_3^2Y=g,\\
Y|_{t=0}=Y_t|_{t=0}=0.
\end{cases}\end{equation}

\begin{prop}\label{S2.3prop1}
{\sl Let $\delta\in[0,1/4[$ and $
\theta\in [1,\infty[.$ We assume that $g(t)=0$ if $t\geq\theta$. Then the solution $Y$ to \eqref{S2eq18} verifies for any $N\geq0$,
\begin{equation}\label{S2eq19}
\begin{split}
&|\p_3Y|_{1,N}+|\p_t Y|_{\frac32-\d,N}+|Y|_{\frac12,N}\leq C_{\d,N}R_{N,\theta}(g),
\end{split}
\end{equation}
where
\begin{equation}\label{S2eq19'}
\begin{split}
R_{N,\theta}(g)\eqdefa\triplenorm{g}_{L_t^1(\d,N)}+\theta^{\frac12}\big\|\langle t\rangle^{\frac12} |D|^{-1}g\big\|_{L_t^2(H^{N+3})}+\log\langle\theta\rangle\big\||D|^{-1}g\big\|_{\frac32-\delta,N+3},
\end{split}
\end{equation}
where
\begin{equation}\label{S2eq19''}
\triplenorm{g}_{\delta,N}\eqdefa\||D|^{2\delta}g\|_{L^1}+\||D|^{N+4}g\|_{L^1}\andf \triplenorm{g}_{L_t^p(\d,N)}=\Bigl(\int_0^t\triplenorm{g(t')}_{\d,N}^p\,dt'\Bigr)^{\f1p}.
\end{equation}}
\end{prop}
The proof of the above propositions will be presented in Section \ref{Sect3}.

The goal of Section \ref{Sect4} is to calculate the linearized system of \eqref{S1eq1}, which reads
\begin{align}\label{LE1}
\begin{cases}
X_{tt}-\Delta X_t-\p_{3}^2X=f'(Y;X)+g,\\ X|_{t=0}=X_t|_{t=0}=0,
\end{cases}
\end{align}
where $f'(Y;X)=f_0'(Y;X)+f_1'(Y;X)+f_2'(Y;X)$, and $f_0'(Y;X),$ $f_1'(Y;X)$ and $f_2'(Y;X)$ are determined respectively
by \eqref{f0'} and \eqref{fm'}. Furthermore, the second derivative of $f''(Y;X,W)$ will be presented in Subsection \ref{Sect4.2}.

In Section \ref{Sect5}, we shall  derive    the $\dot W^{2\d,1}\cap \dot W^{N+4,1}$ and $\dot H^{N+1}$ estimates for the source term  $f'(Y;X)$  in the linearized
 system \eqref{LE1},  which will be used to derive the decay estimates for  the solutions of  \eqref{LE1}. The main result reads

\begin{prop}\label{S4.1prop1}
{\sl Let the functionals, $f_0'(Y;X), f_1'(Y;X), f_2'(Y;X),$ be given by \eqref{f0'} and \eqref{fm'} respectively, and the norm $ \triplenorm{\cdot}_{\d,N}$ be given by \eqref{S2eq19''}. Then under the assumptions that $\d>0,$ and 
 \beq \label{S4.1prop1assum}
 \|\nabla Y\|_{\dot B^\frac32_{2,1}}\leq\delta_1 \andf \|\nabla Y\|_{\dot B^\frac52_{2,1}}\leq1,\eeq
 for some $\delta_1>0$ sufficiently small, we have
\begin{equation}\label{S4.1eqf0}
\begin{split}
\triplenorm{f_0'(Y;X)}_{\d,N}
\leq& \|\nabla Y\|_{0}\|\nabla X_t\|_{N+6}+\|\nabla Y\|_{N+6}\|\nabla X_t\|_{0}\\
&+\|\nabla Y_t\|_{0}\|\nabla X\|_{N+6}+\big(\|\nabla Y_t\|_{N+6}+\|\nabla Y\|_{N+6}|\nabla Y_t|_{0}\big)\|\nabla X\|_{0},
\end{split}
\end{equation}
and
\ben
\triplenorm{f_1'(Y;X)}_{\d,N}&\lesssim& \frak{f}_1(\p_3Y,\p_3X) \andf \label{S4.1eqf1}\\
\triplenorm{f_2'(Y;X)}_{\d,N}&\lesssim& \frak{f}_1(Y_t,X_t), \label{S4.1eqf2}
\een
where the function $\frak{f}_1(\frak{x},\frak{y})$ is given by 
\begin{equation*}
\begin{split}
\frak{f}_1(\frak{x},\frak{y})\eqdefa &   \|\frak{x}\|_{0}\bigl(\|\frak{y}\|_{N+6}+|\frak{x}|_{0}\|\nabla X\|_{N+6}\bigr)\\
&+\|\frak{y}\|_{1}\big(\|\frak{x}\|_{N+6}+\|\nabla Y\|_{N+6}|\frak{x}|_{1}\big)+\big(\|\frak{x}\|_{N+6}+\|\nabla Y\|_{N+6}\|\frak{x}\|_{3}\big)|\frak{x}|_{1}\|\nabla X\|_{1}.
\end{split}
\end{equation*}
}
\end{prop}

\begin{prop}\label{S4.2prop1}
{\sl
Under the assumption of  Proposition \ref{S4.1prop1}, we have
\begin{equation}\label{S4.2eqf0}
\begin{split}
\||D|^{-1}f_0'(Y;X)\|_{N+1}\lesssim & |\nabla Y|_{0}\|\nabla X_t\|_{N+1}+|\nabla Y|_{N+1}\|\nabla X_t\|_{0}\\
&+ | Y_t|_{1}\|\nabla X\|_{N+1}+\big(| Y_t|_{N+2}+| Y_t|_{1}|\nabla Y|_{N+1}\big)\|\nabla X\|_{0},
\end{split}
\end{equation}
and
\ben
\||D|^{-1}f_1'(Y;X)\|_{N+1}
&\lesssim & \frak{f}_2(\p_3Y,\p_3X)\andf \label{S4.2eqf1}\\
\||D|^{-1}f_2'(Y;X)\|_{N+1}
&\lesssim& \frak{f}_2(Y_t,X_t),\label{S4.2eqf2}
\een
where the functional $\frak{f}_2(\frak{x},\frak{y})$ is given by
\begin{equation*}
\begin{split}
\frak{f}_2(\frak{x},\frak{y})
\eqdefa & \big(|\frak{x}|_{1}^\frac43\|\frak{x}\|_{0}^\frac23+|\frak{x}|_{1}^2\big)\bigl(\|\na X\|_{N+1}+|\nabla Y|_{N+1}\|\nabla X\|_{1}\bigr)+|\frak{x}|_{0}\|\frak{y}\|_{N+1}\\
&+\big(|\frak{x}|_{N+1}+|\nabla Y|_{N+1}|\frak{x}|_{1}\big)\|\frak{y}\|_{1}+\big(|\frak{x}|_{0}^\frac13\|\frak{x}\|_{0}^\frac23+|\frak{x}|_{0}\big)|\frak{x}|_{N+1}\|\nabla X\|_{1}.
\end{split}
\end{equation*}
}
\end{prop}

 Let us remark that Riesz transform does not map continuously from $L^1$ to $L^1.$ Nevertheless due to \eqref{p1'}
 and \eqref{p2'}, we can not avoid estimates of this type. To overcome this difficulty, a natural replacement of $\dot{W}^{s,1}$ will be
 the Besov space $\dot B^s_{1,1},$ which satisfies
$$\|{\na(-\D)^{-\f12}|D|^{s}(f)}\|_{L^1}\lesssim \|f\|_{\dot B^{s}_{1,1}}\quad \forall \ s\in\R.$$
  We  now recall
the precise definition of the Besov norms from \cite{BCD} for
instance.
\begin{defi}
\label {S0def1} {\sl  Let us consider a smooth function~$\vf $
on~$\R,$ the support of which is included in~$[3/4,8/3]$ such that
$$
\forall
 \tau>0\,,\ \sum_{j\in\Z}\varphi(2^{-j}\tau)=1\andf \chi(\tau)\eqdefa 1 - \sum_{j\geq
0}\varphi(2^{-j}\tau) \in \cD([0,4/3]).
$$
Let us define
$$
\Delta_ja=\cF^{-1}(\varphi(2^{-j}|\xi|)\widehat{a}),
 \andf S_ja=\cF^{-1}(\chi(2^{-j}|\xi|)\widehat{a}).
$$
Let $(p,r)$ be in~$[1,+\infty]^2$ and~$s$ in~$\R$. We define the Besov norm by
$$
\|a\|_{\dB^s_{p,r}}\eqdefa\big\|\big(2^{js}\|\Delta_j
a\|_{L^{p}}\big)_j\bigr\|_{\ell ^{r}(\ZZ)}.
$$
}
\end{defi}

We remark that in the special case when $p=r=2,$  the Besov
spaces~$\dB^s_{p,r}$ coincides with the classical homogeneous
Sobolev spaces $\dH^s$.
Moreover, we have the following product laws (see Corollary 2.54 of \cite{BCD}):
\begin{equation}\label{product1}
\|ab\|_{\dot B^s_{p,r}}\leq C\big(|a|_{L^\infty}\|b\|_{\dot B^s_{p,r}}+\|a\|_{\dot B^s_{p,r}}|b|_{L^\infty}\big),
\end{equation}
for $s>0$, $(p,r)\in[1,+\infty]^2$.
Due to the product law \eqref{product1}, we need the index $\delta$ to be positive in Proposition \ref{S4.1prop1}.

In Section \ref{Sect6}, we investigate energy estimates for the solutions of the linearized equation \eqref{LE1}.

\begin{thm}\label{S5thm1}
{\sl Let $Y$ be a smooth enough vector field and $X$ be a smooth solution to  the linearized equation \eqref{LE1}.
 We assume  that $Y$ satisfies \eqref{S4.1prop1assum} and
\beq \label{S5thmassum} \|Y_t\|_{0,0}\leq1,\andf |Y_t|_{0,1}\leq 1.\eeq
Then  for any $\varepsilon>0$, we have
\begin{equation}\label{S5eqE1}
\begin{split}
&{\mathcal E}_0(t)\leq C_\varepsilon\|\langle t\rangle^{\frac{1+\varepsilon}{2}}|D|^{-1}g\|_{L_t^2(H^1)} E_\e(Y)\andf \quad\mbox{for}\quad N\geq 1\\
&\cE_N(t)\leq C_{\varepsilon,N}\Big(\|\langle t\rangle^{\frac{1+\varepsilon}2}g\|_{L_t^2( H^{N})}
+\gamma_{\varepsilon,N+1}(Y)\|\langle t\rangle^{\frac{1+\varepsilon}{2}}|D|^{-1}g\|_{L_t^2(H^1)}\Big)E_\e(Y),
\end{split}
\end{equation}
where
\begin{equation}\label{S5eqE2}
\begin{split}
{\mathcal E}_{N}(t)\eqdefa &\||D|^{-1}(X_t,\p_3X)\|_{0,N+2}+\|\nabla X\|_{0,N+1}+\|X_t\|_{L_t^2( H^{N+2})}+\|\p_3X\|_{L_t^2( H^{N+1})};\\
E_\e(Y)\eqdefa &\exp \Bigl(C\big(|\p_3Y|_{\frac12+\varepsilon,1}^{\frac43}\|\p_3Y\|_{L_t^2(L^2)}^\frac23+|\p_3Y|_{\frac12+\varepsilon,1}^2+|Y_t|_{1+\varepsilon,2}\big)\Bigr),
\end{split}
\end{equation}
and
\begin{equation}\label{S5eqE3}
\begin{split}
&\gamma_{\varepsilon,N+1}(Y)\eqdefa   1+|\p_3Y|_{\frac12+\varepsilon,N+1}\bigl(1+|\p_3Y|_{\frac12+\varepsilon,1}\bigr) +| Y_t|_{1+\varepsilon,N+2}+|\p_3Y|_{\frac12+\varepsilon,1}^\frac13\|\p_3Y\|_{L_t^2(L^2)}^\frac23\times\\
  &\quad\times\bigl(|\p_3Y|_{\frac12+\varepsilon,N+1}+|\nabla Y|_{0,N+1}\bigr)
+|\nabla Y|_{0,N+1}\big(1+|\p_3Y|_{\frac12+\varepsilon,0}^2+|\p_3Y|_{\frac12+\varepsilon,0}+|Y_t|_{1+\varepsilon,1}\big).
\end{split}
\end{equation}}
\end{thm}

We notice that when we perform the energy estimates for the  derivatives of the solutions to
 \eqref{LE1}, we are not able to treat the term $\nabla\cdot\big(({\mathcal A}{\mathcal A}^t-Id)\nabla X_t\big),
$  which appears in $f_0'(Y;X)$ (see \eqref{f0'}),  as a source term. Instead, we need to rewrite \eqref{LE1} as
\begin{align}\label{LEW1}
X_{tt}-\na\cdot\p_t\bigl(\cA\cA^t\na X\bigr)-\p_3^2X =\wt{f}'(Y;X)+g
\end{align}
where $\wt{f}'(Y;X)=\wt{f}_0'(Y;X)+f_1'(Y;X)+f_2'(Y;X)$ with  $f_m'(Y;X),m=1,2,$ given by \eqref{fm'}, and  $\wt{f}_0'(Y;X)$  by
\begin{align}
\label{wf0'}\widetilde f_0'(Y;X)&=-\nabla \cdot\left({\mathcal A}\bigl(\nabla X{\mathcal A}+{\mathcal A}^t(\nabla X)^t\bigr){\mathcal A}^t
\nabla Y_t\right)-\nabla\cdot\big(\p_t({\mathcal A}{\mathcal A}^t)\nabla X\big).
\end{align}
With the energy estimates obtained in Theorem \ref{S5thm1}, we can work on the time-weighted energy estimate for the solutions of \eqref{LE1}.

\begin{col}\label{S5col1}
{\sl Under the assumptions of Theorem \ref{S5thm1}, we have
\begin{equation}\label{S5eqC1}
\begin{split}
&{\mathcal E}_0+\|(X_t,\p_3X)\|_{\frac12,1}+\|\langle t\rangle^\frac12\nabla X_t\|_{L_t^2(H^1)}
\leq C_\varepsilon\|\langle t\rangle^{\frac{1+\varepsilon}{2}}|D|^{-1}g\|_{L_t^2(H^1)}E_\e(Y),
\end{split}
\end{equation}
and for $N\geq 1$,
\begin{equation}\label{S5eqC2}
\begin{split}
&{\mathcal E}_{N}+\|(X_t,\p_3X)\|_{\frac12,N+1}+\|\langle t\rangle^\frac12\nabla X_t\|_{L_t^2( H^{N+1})}\\
&\qquad\leq  C_{\varepsilon,N}\Big(\|\langle t\rangle^{\frac{1+\varepsilon}2}g\|_{L_t^2( H^{N})}
+\gamma_{\varepsilon,N+1}(Y)\|\langle t\rangle^{\frac{1+\varepsilon}{2}}|D|^{-1}g\|_{L_t^2(H^1)}\Big)E_\e(Y).
\end{split}
\end{equation}}
\end{col}

\begin{prop}\label{S5thm2}
{\sl Under the assumptions of Theorem \ref{S5thm1}, we have
for $N\geq0$,
\begin{equation}\label{S5eqC2b}
\begin{split}
\|\na X_t\|_{1,N} &\leq  C_{\e,N}\big(\||D|^{-1}g\|_{1+\e,N+2}+\|\nabla Y\|_{N+2}\||D|^{-1}g\|_{1+\e,2}\big)\\
&\ +C_{\e,N}\Bigl(\|\w{t}^{\f{1+\e}2}g\|_{L^2_t(H^{N+1})}
+\gamma_{\e,N+2}(Y)\|\langle t\rangle^{\frac{1+\varepsilon}{2}}|D|^{-1}g\|_{L_t^2(H^1)}\\
&\qquad\qquad+ \|\nabla Y\|_{0,N+2} \big(\|\w{t}^{\f{1+\e}2}g\|_{L^2_t(H^1)}
+\gamma_{\e,2}(Y)\|\langle t\rangle^{\frac{1+\varepsilon}{2}}|D|^{-1}g\|_{L_t^2(H^1)}\big)\Bigr)E_\e(Y).
\end{split}
\end{equation}}
\end{prop}

We  emphasize that the decay estimates \eqref{S5eqC2b} can not be obtained by energy estimate.
In fact,  we will have to exploit anisotropic Littlewood-Paley analysis and the dissipative properties of the linear system \eqref{S1eq1}.
The proof of Proposition \ref{S5thm2} will be presented in Section \ref{Sect6.3}, which is of independent interest.

Let us summarize that under the assumptions \eqref{S4.1prop1assum}, \eqref{S5thmassum}
and if we assume  moreover
\begin{equation}\label{S5eqC4}
|\p_3Y|_{\frac12+\varepsilon,1}^{\frac43}\|\p_3Y\|_{L_t^2(L^2)}^\frac23+|\p_3Y|_{\frac12+\varepsilon,1}^2+|Y_t|_{1+\varepsilon,2}\leq1,
\end{equation}
we have the following energy estimates: for $N\geq0$, (we make the convention $\|u\|_{k,-1}=0$)
\begin{equation}\label{S5eqC3}
\begin{split}
&{\mathcal E}_{N}+\|(X_t,\p_3X)\|_{\frac12,N+1}+\|\langle t\rangle^\frac12\nabla X_t\|_{L_t^2( H^{N+1})}+\|\na X_t\|_{1,N-1} \leq C_{\varepsilon,N}\Big(\||D|^{-1}g\|_{1+\e,N+1}\\
&\qquad +\|\langle t\rangle^{\frac{1+\varepsilon}2}g\|_{L_t^2( H^{N})}+\widetilde\gamma_{\varepsilon,N+1}(Y)\big(\||D|^{-1}g\|_{1+\e,2}+\|\langle t\rangle^{\frac{1+\varepsilon}{2}}|D|^{-1}g\|_{L_t^2(H^1)}\big)\Big) \with\\
&\widetilde\gamma_{\varepsilon,N+1}(Y)\leq C\big(1+|\p_3Y|_{\frac12+\varepsilon,N+1}+| Y_t|_{1+\varepsilon,N+2}+|\nabla Y|_{0,N+1}+\|\nabla Y\|_{0,N+1}\big).
\end{split}
\end{equation}

In Section \ref{Sect7}, we shall present the  estimates to the nonlinear source term $f(Y)$ given by \eqref{S1eq2}.
The purpose of  Section \ref{Sect8} is concerned with the related estimates for the second derivatives, $f''(Y;X,W),$ of the nonlinear functional $f(Y)$, computed in Section \ref{Sect4.2}.

With the preparations in the previous sections, we can now exploit Nash-Moser iteration scheme to prove Theorem \ref{Th1}.
In order to do so,
 we first recall some basic properties of the smoothing operator from \cite{Kl80, Kl82}.
Let $\chi(t)\in C^\infty(\R;[0,1])$ be such that
$$\chi(t)=1\text{ for }t\leq\frac12,\quad \chi(t)=0\text{ for }t\geq1.$$
Define for $\theta\geq1$, the (cutoff-in-time) operator
\begin{equation}\label{S8eq1}
S^{(1)}(\theta)Y(t,y)\eqdefa\chi\Big(\frac{t}{\theta}\Big)Y(t,y).
\end{equation}
Then we have
\begin{equation*}
\begin{split}
|S^{(1)}(\theta)Y|_{k,N}&\leq C_{k,s}\theta^{k-s}|Y|_{s,N},\quad \text{if }\ k\geq s\geq0,
\end{split}
\end{equation*}
and
\begin{equation*}
\begin{split}
|\big(1-S^{(1)}(\theta)\big)Y|_{s,N}&\leq C_{k,s}\theta^{-(k-s)}|Y|_{k,N}\quad \text{if }\ k\geq s\geq0.
\end{split}
\end{equation*}

For $\theta'\geq1$, we define the usual mollifying operator $S^{(2)}(\theta')$ in the space variables by
\begin{equation}\label{S8eq2}
S^{(2)}(\theta')Y(t,y)\eqdefa\widehat\varphi\Big(\frac{D_y}{\theta'}\Big)Y(t,y)=(\theta')^3\int_{\R^3} \varphi(\theta'(y-z))Y(t,z)dz,
\end{equation}
where $\varphi\in {\mathcal S}(\R^3)$ satisfies
$$\widehat\varphi(\xi)=1\text{ for }|\xi|\leq\frac12,\quad \widehat\varphi(\xi)=0\text{ for }|\xi|\geq1,$$
so that
$$\int_{\R^3}\varphi(y)dy=1,\quad \int_{\R^3}y^\alpha\varphi(y)dy=0,\ \forall\ |\alpha|>0.$$
 We then have
\begin{equation*}
\begin{split}
|S^{(2)}(\theta')Y|_{k,N}&\leq C_{N,M}(\theta')^{N-M}|Y|_{k,M}\quad \text{if } \ N\geq M\geq0,
\end{split}
\end{equation*}
as well as
\begin{align*}
|\big(1-S^{(2)}(\theta')\big)Y|_{k,M}&\leq C_{N,M}(\theta')^{-(N-M)}|Y|_{k,N}\quad \text{if}\  N\geq M\geq0.
\end{align*}

Define the operator
\begin{equation}\label{S8eq3}
S(\theta,\theta')\eqdefa S^{(1)}(\theta)S^{(2)}(\theta'),\quad\mbox{for}\quad \theta,\ \theta'\geq1.
\end{equation}
Then  it follows that
\begin{equation}\label{S8eq4}
\begin{split}
|S(\theta,\theta')Y|_{k,N}&\leq C\theta^{k-s}(\theta')^{N-M}|Y|_{s,M},\\
\|\langle t\rangle^kS(\theta,\theta')g\|_{L_t^p(H^N)}&\leq C\theta^{k-s}(\theta')^{N-M}\|\langle t\rangle^sg\|_{L_t^p(H^M)}\ \ \ \text{if }\ k\geq s\geq0,\ N\geq M\geq 0.
\end{split}
\end{equation}
Moreover, due to  $$1-S(\theta,\theta')=\big(1-S^{(1)}(\theta)\big)+S^{(1)}(\theta)\big(1-S^{(2)}(\theta')\big),$$
one has
\begin{equation}\label{S8eq5}
\begin{split}
|\big(1-S(\theta,\theta')\big)Y|_{s,M}&\leq C\theta^{-(k-s)}|Y|_{k,M}+C(\theta')^{-(N-M)}|Y|_{s,N},\\
\|\langle t\rangle^s\big(1-S(\theta,\theta')\big)g\|_{L_t^p(H^M)}&\leq C\theta^{-(k-s)}\|\langle t\rangle^kg\|_{L_t^p(H^M)}+C(\theta')^{-(N-M)}\|\langle t\rangle^sg\|_{L_t^p(H^N)}
\end{split}
\end{equation}
provided that $k\geq s\geq0,\ N\geq M\geq0.$

Let us denote
$$\Phi(Y)\eqdefa Y_{tt}-\Delta Y_t-\p_3^2Y-f(Y),$$
for $f$ given by \eqref{S1eq2}. Then we can write \eqref{S1eq1} equivalently as
\begin{equation}\label{S9.1eq1}
\Phi(Y)=0,\quad
Y(0,y)=Y^{(0)},\quad Y_t(0,y)=Y^{(1)}.
\end{equation}
We aim to solve \eqref{S9.1eq1} via Nash-Moser iteration scheme in Section \ref{Sect10}.

Let us define $Y_0$ via
\begin{equation}\label{S9.1eq2}
\begin{cases}
\p_{tt}Y_{0}-\Delta \p_tY_{0}-\p_3^2Y_0=0,\\
Y_0(0,y)=Y^{(0)},\quad \p_tY_{0}(0,y)=Y^{(1)}.
\end{cases}
\end{equation}
Inductively, assume that we already determine $Y_p$.
In order to define $Y_{p+1}$, we  introduce a mollified version of $\Phi'(Y_p)$ as follows
\begin{equation}\label{S9.1eq3}
L_pX\eqdefa \Phi'(S_pY_p)X=X_{tt}-\Delta X_t-\p_3^2X-f'(S_pY_p;X),
\end{equation}
where $S_p$ is the smoothing operator defined by
\begin{equation}\label{S9.1eq4}
\begin{split}
S_p&=S(\theta_p,\theta_p'),\with  \theta_p=2^{p},\quad \theta_p'=\theta_p^{\bar\varepsilon}=2^{\bar\varepsilon p},\andf p\geq0,
\end{split}
\end{equation}
where $S(\theta,\theta')$ is defined in \eqref{S8eq3} and $\bar\varepsilon>0$ is a small constant to be chosen later on.
Then it follows from \eqref{S8eq4} and  \eqref{S8eq5} that
\begin{equation}\tag{S I}\label{S9eqSI}
\begin{split}
|S_pY|_{k,N}&\leq C\theta_p^{k-s}\theta_p^{\bar\varepsilon(N-M)}|Y|_{s,M}\\
\|\langle t\rangle^kS_pg\|_{L_t^2(H^N)}&\leq C\theta_p^{k-s}\theta_p^{\bar\varepsilon(N-M)}\|\langle t\rangle^sg\|_{L_t^2(H^M)}\\
\triplenorm{S_pg}_{L_t^1(\d,N)}&\leq C\theta_p^{\bar\varepsilon(N-M)}\triplenorm{g}_{L_t^1(\d,M)},
\end{split}
\end{equation}
and
\begin{equation}\tag{S II}\label{S9eqSII}
\begin{split}
|(1-S_p)Y|_{0,0}&\leq C_{k,N}\big(\theta_p^{-k}|Y|_{k,0}+\theta_p^{-\bar\varepsilon N}|Y|_{0,N}\big),\\
\|\langle t\rangle^s(1-S_p)g\|_{L_t^2(L^2)}&\leq C_{k,N}\big(\theta_p^{-(k-s)}\|\langle t\rangle^kg\|_{L_t^2(L^2)}+\theta_p^{-\bar\varepsilon N}\|\langle t\rangle^sg\|_{L_t^2(H^N)}\big),
\end{split}
\end{equation}
for $k\geq s\geq0$, $N\geq M\geq0,$ where the norm $\triplenorm{\cdot}_{L_t^1(\d,N)}$ is given by \eqref{S2eq19''}.

\begin{rmk}\label{S9rmk1} According to Remark \ref{S3rmk1} below, we can write
$$f'(S_pY_p;X)=f_0'(S_pY_p;X)+f_1'(S_pY_p;X)+f_2'(S_pY_p;X)$$ where
\begin{equation*}
\begin{split}
&f_0'(S_pY_p;X)=F'_{0,U}(S_p\nabla\p_t Y_{p},S_p\nabla Y_p)\nabla X_t+F'_{0,V}(S_p\nabla\p_t Y_{p},S_p\nabla Y_p)\nabla X,\\
&f_1'(S_pY_p;X)=F'_U(S_p\p_3Y_p,S_p\nabla Y_p)\p_3X+F'_V(S_p\p_3Y_p,S_p\nabla Y_p)\nabla X,\\
&f_2'(S_pY_p;X)=F'_U(S_p\p_tY_{p},S_p\nabla Y_p)X_t+F'_V(S_p\p_tY_{p},S_p\nabla Y_p)\nabla X,
\end{split}
\end{equation*}
where the functionals $F_0'$, $F'$ will be presented  in Remark \ref{S3rmk1}.
\end{rmk}

Following H\"ormander's version of Nash-Moser Scheme (\cite{Hormander}) (see also Klainerman's seminar papers \cite{Kl80, Kl82}), we define
\begin{equation}\label{S9.1eq5}
Y_{p+1}=Y_p+X_p,\with X_p=L_p^{-1}g_p,
\end{equation}
where $L_p^{-1}$ is a right inverse operator of $L_p$ with zero initial data, that is:  $X=L_p^{-1}g_p$ solves
\begin{equation}\label{S9.1eq6}
\begin{cases}
L_pX=g_p\quad \text{with $L_p$ given by \eqref{S9.1eq3}},\\
X(0,y)=0,\quad X_t(0,y)=0.
\end{cases}
\end{equation}
In order to prove the convergence of the scheme, we define
\begin{equation}\label{S9.1eq7}
\begin{split}
e_p'&\eqdefa\big(\Phi'(Y_p)-L_p\big)X_p,\quad
e_p''\eqdefa\Phi(Y_{p+1})-\Phi(Y_p)-\Phi'(Y_p)X_p,\andf
e_p\eqdefa e_p'+e_p'',
\end{split}
\end{equation}
from which, we infer
\begin{align*}
\Phi(Y_{p+1})-\Phi(Y_p)&=\Phi'(Y_p)X_p+e_p''=\Phi'(Y_p)L_p^{-1} g_p+e_p''\\
&=\big(\Phi'(Y_p)-L_p\big)L_p^{-1} g_p+g_p+e_p''=e_p'+e_p''+g_p.
\end{align*}
As a result, it comes out
\begin{equation}\label{S9.1eq11} \Phi(Y_{p+1})-\Phi(Y_p)=e_p+g_p \andf
\Phi(Y_{p+1})-\Phi(Y_0)=\sum_{j=0}^p(e_j+g_j).
\end{equation}
To achieve that the above limit as $p\to\infty$ is equal to $-\Phi(Y_0)$, we set
\begin{equation}\label{S9.1eq8}
\sum_{j=0}^p g_j+S_pE_p=-S_p\Phi(Y_0)\quad \text{with}\quad E_p\eqdefa\sum_{j=0}^{p-1}e_j.
\end{equation}
The last relation defines $g_p$ as follows
\begin{equation}\label{S9.1eq9}
\begin{split}
g_0&=-S_0\Phi(Y_0),\andf
g_p=-(S_p-S_{p-1})E_{p-1}-S_pe_{p-1}-(S_p-S_{p-1})\Phi(Y_0).
\end{split}
\end{equation}

\begin{rmk}\label{S9rmk2}
By virtue of  Remarks \ref{S9rmk1}, \ref{S3rmk1}, \ref{S7rmk1},
using a Taylor formula to \eqref{S9.1eq7}, we have
\begin{equation*}
\begin{split}
e_p'&=-\int_0^1f''\big(sY_p+(1-s)S_pY_p;(1-S_p)Y_p,X_p\big)ds,\andf\\
e_p''&=-\int_0^1(1-s)f''\big(sY_{p+1}+(1-s)Y_p;X_p,X_p\big)ds,
\end{split}
\end{equation*}
where $f''$ should be understood in the way explained in Remark \ref{S7rmk1}.
Then we have
\begin{equation}\label{S9.1eq10}
\begin{split}
&e_p=e_{p,0}+e_{p,1}+e_{p,2},\with e_{p,m}\eqdefa e_{p,m}'+e_{p,m}''\andf\\
&e_{p,m}'\eqdefa-\int_0^1f_m''\big(sY_p+(1-s)S_pY_p;(1-S_p)Y_p,X_p\big)ds,\\
&e_{p,m}''\eqdefa-\int_0^1(1-s)f_m''\big(Y_p+sX_p;X_p,X_p\big)ds,\quad\qquad m=0,1,2.
\end{split}
\end{equation}
\end{rmk}

Let us fix the small constants:  $\varepsilon$, $\bar\varepsilon$, $\delta>0,$  so that
\beq \label{S9eqebd} \bar\varepsilon\leq\frac1{20},\quad \delta+5\bar\varepsilon\leq \frac14,\quad \delta+\varepsilon+4\bar\varepsilon\leq\frac14.\eeq
Let us take
\begin{equation}\label{S9eqP}
\begin{split}
&\gamma=\frac14-\bar\varepsilon,\quad\beta=\frac14+\bar\varepsilon,
\end{split}
\end{equation}
and $N_0\in\N$ is chosen such that
\begin{equation}\label{S9eqP'}
\begin{split}
\bar\varepsilon N_0\geq \frac12=\gamma+\beta.
\end{split}
\end{equation}
In Section \ref{Sect10}, we shall inductively prove the following statements:
\begin{prop} \label{S9prop1}
{\sl Let  $\delta_1>0$ be determined by Propositions \ref{S4.1prop1}, \ref{S4.2prop1}, \ref{S6prop1}, \ref{S6prop2}, \ref{S7.2prop1}, \ref{S7.3prop1} and Theorem \ref{S5thm1}.
Then for the constants $\beta, \gamma, N_0, \e, \bar{\e}$ and $\d$ given by (\ref{S9eqebd}-\ref{S9eqP'}), for any $0\leq N\leq N_0,$ we have
\begin{equation}\tag{P$1,p$}
\begin{split}
&\bigl\||D|^{-1}(\p_3X_p,\p_tX_{p})\bigr\|_{0,N+2}+\|\nabla X_p\|_{0,N+1}+\|(\p_tX_{p},\p_3X_p)\|_{\frac12,N+1} \\
&\quad+\|\p_tX_{p}\|_{L_t^2(H^{N+2})}+\bigl\|(\p_3X_p,\langle t\rangle^\frac12\nabla \p_tX_{p})\bigr\|_{L_t^2(H^{N+1})}+\|\nabla \p_tX_{p}\|_{1,N-1}\leq \eta\theta_p^{-\beta+\bar \varepsilon N};
\end{split}
\end{equation}
and
\begin{equation}\tag{P$2,p$}
\begin{split}
|\p_3X_p|_{k,N}&\leq\eta\theta_p^{k-\frac12-\gamma+\bar\varepsilon N}\qquad \ \text{if}\quad \frac12\leq k\leq1,\\
|\p_tX_{p}|_{k,N}&\leq\eta\theta_p^{k-(1-\delta)-\gamma+\bar\varepsilon N}\quad\text{if}\quad 1-\d\leq k\leq \frac32-\delta,\\
|X_p|_{k,N}&\leq\eta \theta_p^{k-\gamma+\bar\varepsilon N}\qquad\quad\ \ \text{if}\quad 0\leq k\leq\frac12;
\end{split}
\end{equation}
and
\begin{equation}\tag{P$3,p$}
\begin{split}
&\|\nabla Y_p\|_{L^\infty_t(\dot B^\frac32_{2,1})}\leq \delta_1,\quad \|\nabla Y_p\|_{L^\infty_t(\dot B^\frac52_{2,1})}\leq1,\quad \|\p_t Y_p\|_{0,0}\leq1,\quad |\p_tY_p|_{0,1}\leq1,\\
&\qquad\qquad|\p_3Y_p|_{\frac12+\varepsilon,1}^{\frac43}\|\p_3Y_p\|_{L_t^2(L^2)}^\frac23+|\p_3Y_p|_{\frac12+\varepsilon,1}^2+|\p_tY_{p}|_{1+\varepsilon,2}\leq1.
\end{split}
\end{equation}}
\end{prop}

Recall the convention that $\|u\|_{k,-1}=0$.
We shall deduce the following propositions  from Proposition \ref{S9prop1}.

\begin{prop}\label{S9col1}
{\sl Under the assumptions of Proposition \ref{S9prop1}, we have, for $N\geq0$,
\begin{equation}\tag{I)  (i}
\begin{split}
|S_{p+1}\p_3Y_{p+1}|_{k,N}&\leq C_{k,N}\eta \theta_{p+1}^{k-\frac12-\gamma+\bar\varepsilon N}\ \ \qquad
\text{if\ \ $k\geq\frac12$,}\qquad\text{ $k-\frac12-\gamma+\bar\varepsilon N\geq\bar \varepsilon$,}\\
|S_{p+1}\p_tY_{p+1}|_{k,N}&\leq C_{k,N}\eta\theta_{p+1}^{k-(1-\delta)-\gamma+\bar \varepsilon N}\ \quad\text{if\ \ $k\geq1-\delta$,}\ \text{ $k-(1-\delta)-\gamma+\bar\varepsilon N\geq\bar \varepsilon$},\\
| S_{p+1}Y_{p+1}|_{k,N}&\leq C_{k,N}\eta\theta_{p+1}^{k-\gamma+\bar\varepsilon N} \qquad\qquad\text{if\ \ $k\geq0$,}\qquad\text{ $k-\gamma+\bar\varepsilon N\geq\bar \varepsilon$};
\end{split}
\end{equation}

\begin{equation}\tag{I)  (ii}
\begin{split}
&\D_{p+1}\eqdefa \bigl\||D|^{-1}S_{p+1}(\p_3Y_{p+1},\p_tY_{p+1})\bigr\|_{0,N+2}+\|S_{p+1}\nabla Y_{p+1}\|_{0,N+1}\\
&\ + \|S_{p+1}(\p_tY_{p+1},\p_3Y_{p+1})\|_{\frac12,N+1} +\bigl\|(S_{p+1}\p_3Y_{p+1},\langle t\rangle^\frac12S_{p+1} \nabla \p_tY_{p+1})\bigr\|_{L_t^2(H^{N+1})}\\
&\ +\|S_{p+1}\p_tY_{p+1}\|_{L_t^2(H^{N+2})} +\| S_{p+1}\nabla \p_tY_{p+1}\|_{1,N-1}\leq C_N\eta\theta_{p+1}^{-\beta+\bar \varepsilon N}\quad \text{if }-\beta+\bar\varepsilon N\geq\bar \varepsilon;
\end{split}
\end{equation}

\begin{equation}\tag{II)  (i}
\begin{split}
|S_{p+1}\p_3Y_{p+1}|_{k,N}&\leq C_{k,N}\eta\qquad\text{if\ \ $k\geq\frac12$,}\qquad\text{ $k-\frac12-\gamma+\bar\varepsilon N\leq-\bar \varepsilon$},\\
|S_{p+1}\p_tY_{p+1}|_{k,N}&\leq C_{k,N}\eta\qquad\text{if\ \ $k\geq1-\delta$,}\ \text{ $k-(1-\delta)-\gamma+\bar\varepsilon N\leq-\bar \varepsilon$},\\
| S_{p+1}Y_{p+1}|_{k,N}&\leq C_{k,N}\eta\qquad\text{if\ \ $k\geq0$,}\qquad\text{ $k-\gamma+\bar\varepsilon N\leq-\bar \varepsilon$};
\end{split}
\end{equation}

\begin{equation}\tag{II)  (ii}
\begin{split}
\D_{p+1}\leq C_N\eta\qquad
\text{if}\quad -\beta+\bar\varepsilon N\leq-\bar \varepsilon;
\end{split}
\end{equation}

\begin{equation}\tag{III}
\begin{split}
|(1-S_{p+1})\p_3Y_{p+1}|_{k,N}&\leq C_{k,N}\eta \theta_{p+1}^{k-\frac12-\gamma+\bar \varepsilon N} \
\qquad \text{if\ \ $\frac12\leq k\leq 1,\  N\leq N_0$},\\
|(1-S_{p+1})\p_tY_{p+1}|_{k,N}&\leq C_{k,N}\eta \theta_{p+1}^{k-(1-\delta)-\gamma+\bar \varepsilon N}\quad \text{if\ \ $1-\delta\leq k\leq \frac32-\delta,\  N\leq N_0$},\\
|(1- S_{p+1})Y_{p+1}|_{k,N}&\leq C_{k,N}\eta\theta_{p+1}^{k-\gamma+\bar \varepsilon N}\quad\qquad \ \ \text{if\ \ $0\leq k\leq \frac12,\  N\leq N_0$}.
\end{split}
\end{equation}
}\end{prop}

\begin{prop}\label{S9prop2}
{\sl Let $e_p,$ $g_p$ and  $R_{N,\theta}(g)$  be given by  \eqref{S9.1eq7},   \eqref{S9.1eq9} and \eqref{S2eq19'}  respectively. Let $\alpha\eqdefa\frac12-\delta-\bar\varepsilon>0$.
Then there hold

\noindent\underline{{\rm (1)}  Estimates for $e_p$.}
\begin{align}
&\|\langle t\rangle^{k+\frac12}|D|^{-1}e_{p}\|_{L_t^2(H^{N+1})}\lesssim\eta^2\theta_{p}^{k+\delta-\gamma-\beta+\bar\varepsilon(N+3)}\quad
 \text{if}\ \ 0\leq k\leq \alpha,\ 0\leq N\leq N_0-2, \tag{IV)  (i}\\
&\||D|^{-1}e_{p}\|_{1+k,N+1}\lesssim \eta^2\theta_p^{k+\delta-\gamma-\beta+\bar\varepsilon(N+2)}\quad \qquad\quad\ \text{if }\ \ 0\leq k\leq\frac12-\delta,\ N\leq N_0-2,\tag{IV) (ii}\\
&\triplenorm{\langle t\rangle^\frac12 e_{p}}_{L_t^2(\delta,N)}
\lesssim \eta^2\theta_p^{-\gamma+\bar\varepsilon(N+5)}\quad\quad \qquad\quad  \text{if}\ \ 0\leq N\leq N_0-6; \tag{IV)  (iii}
\end{align}

\noindent\underline{{\rm (2)} Estimates for $g_{p+1}$.}
\begin{align}
&\|\langle t\rangle^{k+\frac12}|D|^{-1}g_{p+1}\|_{L_t^2(H^{N+1})}\leq C\eta^2\theta_{p+1}^{k+\delta-\gamma-\beta+\bar\varepsilon(N+3)}\quad \text{if }k\geq0,\ N\geq 0,
\tag{V)  (i}\\
&\||D|^{-1}g_{p+1}\|_{1+k,N+1}\lesssim \eta^2\theta_{p+1}^{k+\delta-\gamma-\beta+\bar\varepsilon(N+2)}\qquad\qquad\quad \text{if }k\geq0,\ N\geq0, \tag{V) (ii}\\
&
\triplenorm{g_{p+1}}_{L_t^1(N)}\leq C\eta^2\theta_{p+1}^{-\gamma+\bar\varepsilon(N+6)}\ \text{if}\ -\gamma+\bar\varepsilon (N+5)\geq
\bar\varepsilon,\tag{V)  (iii}\\
&
\triplenorm{g_{p+1}}_{L_t^1(N)}\leq C\eta^2\theta_{p+1}^{\bar\e}\qquad\qquad\quad\quad\text{if}\ -\gamma+\bar\varepsilon (N+5)\leq -\bar\varepsilon;\tag{V)  (iv}
\end{align}

\noindent\underline{{\rm (3)} Estimates for $R_{N,\theta_{p+1}}(g_{p+1})$.}
\begin{align}
 &R_{N,\theta_{p+1}}(g_{p+1})\leq C\eta^2\theta_{p+1}^{\frac12-\gamma+\bar\varepsilon N}\quad \text{if }-\gamma+\bar\varepsilon (N+5)\geq\bar\varepsilon,\tag{VI)  (i}\\
&
 R_{0,\theta_{p+1}}(g_{p+1})\leq C\eta^2\theta_{p+1}^{\frac12-\gamma} \tag{VI)  (ii}.
\end{align}
}
\end{prop}

The following interpolation lemma will be crucial in the proof of the above propositions, whose proof is exactly the same as that of Lemma 6.1 of \cite{Kl80}, which we omit the details here.

\begin{lem}[Interpolation lemma]
{\sl Let $ p\in [1,+\infty]$, $\theta\geq1$ and $\bar\varepsilon>0$, which satisfy
$$\beta>\bar\varepsilon,\quad k_0-\beta\geq \bar\varepsilon,\quad -\beta+\bar\varepsilon N_0\geq\bar\varepsilon.$$
Assume that $u\in C^\infty([0,+\infty)\times\R^n)$ satisfies
\begin{equation}\label{S8eq6}
\begin{split}
&\|u\|_{L_t^p(L^2)}\leq C\theta^{-\beta},\\
&\|\langle t\rangle^ku\|_{L_t^p(H^N)}\leq C\theta^{k-\beta+\bar \varepsilon N},\ \text{for }0\leq k\leq k_0,\ 0\leq N\leq N_0\text{ s.t }k-\beta+\bar \varepsilon N\geq\bar\varepsilon.
\end{split}
\end{equation}
Then for all $0\leq k\leq k_0$, $0\leq N\leq N_0$,
$$\|\langle t\rangle^ku\|_{L_t^p(H^N)}\leq C_{k_0,N_0} \theta^{k-\beta+\bar\varepsilon N}.$$}
\end{lem}

Finally with the previous propositions, we shall prove the convergence of the approximate solutions constructed by \eqref{S9.1eq5} in Subsection \ref{Subsect10.5},
and this completes the proof of Theorem \ref{Th1}.


\section{Decay estimates of the  linear equation}\label{Sect3}

\subsection{Decay estimates for the solution operator}

Following the strategy in \cite{Kl80, Kl82}, we first investigate the decay properties of the solutions to the linear equation \eqref{S2eq1} with $\cY_0=0$ and
$\cY_1=Y_1.$
 By taking Fourier transform to \eqref{S2eq1} with respect
to $y$ variables and solving the resulting ODE, we write
\beq\label{S2eq2} \cY(t,y)=\Ga(t,D)Y_1 \with
\Ga(t,\xi)=\f1{\la_2(\xi)-\la_1(\xi)}\left(e^{t\la_2(\xi)}-e^{t\la_1(\xi)}\right)
\eeq where $\la_1(\xi)$ and $\la_2(\xi)$ are given by \eqref{S2eq1qe}.

\begin{prop}\label{S1prop1}
{\sl Given $\delta\in [0,1[$ and $N\in\N$, there exists $C_{\d,N}>0$ such that there holds
\beq\label{S2eq3}
|\p_3\Gamma(t)Y_1|_{1,N}+|\p_3^2\Gamma(t)Y_1|_{\f32,N}+|\p_t\Gamma(t)Y_1|_{\frac32-\delta, N}+|\Gamma(t)Y_1|_{\frac12,N} \leq  C_{\d,N}\bigl\|(|D|^{2\delta}Y_1,|D|^{N+4}Y_1)\bigr\|_{L^1}.
\eeq
}
\end{prop}

\begin{proof} The estimate \eqref{S2eq3} for general $N\in\N$ follows from the case when  $N=0.$ Due to the anisotropic properties
of the eigenvalues $\la_1(\xi), \la_2(\xi),$ we shall split the frequency space into two parts:
$\{\xi\in\R^3:\ |\xi|^2\geq 2|\xi_3|\ \}$ and $\{\xi\in\R^3:\ |\xi|^2<2|\xi_3|\ \}.$
 When $|\xi|^2\geq 2|\xi_3|$, let us denote $\al(\xi)\eqdefa \sqrt{\f{|\xi|^4}4-\xi_3^2}.$ Then we have
\beno \la_1(\xi)=-\f{|\xi|^2}2+\al(\xi) \andf
\la_2(\xi)=-\f{|\xi|^2}2-\al(\xi),\eeno
and we write
\beq\label{S2eq4}
\Ga(t,\xi){\bf 1}_{|\xi|^2\geq
2|\xi_3|}=e^{-t\bigl(\f{|\xi|^2}2-\al(\xi)\bigr)}\f{1-e^{-2t\al(\xi)}}{2\al(\xi)}{\bf
1}_{|\xi|^2\geq 2|\xi_3|}. \eeq
When $|\xi|^2<2|\xi_3|$, let us denote $\beta(\xi)\eqdefa \sqrt{\xi_3^2-\f{|\xi|^4}4}.$ Then we have
\beno \la_1(\xi)=-\f{|\xi|^2}2+i\beta(\xi) \andf
\la_2(\xi)=-\f{|\xi|^2}2-i\beta(\xi),\eeno and we write
\beq\label{S2eq5} \Ga(t,\xi){\bf 1}_{|\xi|^2<
2|\xi_3|}=e^{-\f{t}2|\xi|^2}\f{\sin(t\beta(\xi))}{\beta(\xi)}{\bf
1}_{|\xi|^2< 2|\xi_3|}. \eeq
Next we handle the estimate of \eqref{S2eq3} term by term below.

\noindent$\bullet$\underline{\it Estimates of $\|\p_3\cY(t)\|_{L^\infty}$ and $\|\p_3^2\cY(t)\|_{L^\infty}.$}

In view of  \eqref{S2eq2}, we deduce that
\beq\label{S2eq6}
\begin{split}
\|\p_3\cY(t)\|_{L^\infty}\leq &
\|\Ga(t,\cdot)\xi_3\widehat{Y_1}(\cdot)\|_{L^1}\\
=&\int_{|\xi|^2\geq
2|\xi_3|}e^{-t\bigl(\f{|\xi|^2}2-\al(\xi)\bigr)}\f{1-e^{-2t\al(\xi)}}{2\al(\xi)}|\xi_3\widehat{Y_1}(\xi)|\,d\xi\\
&+ \int_{|\xi|^2<
2|\xi_3|}e^{-\f{t}2|\xi|^2}\f{|\sin(t\beta(\xi))|}{\beta(\xi)}|\xi_3\widehat{Y_1}(\xi)|\,d\xi
\eqdefa I_1+I_2.
\end{split}
\eeq It is easy to observe that \beno
\begin{split}
I_1=&\left(\int_{|\xi|\geq 3}+\int_{9>|\xi|^2\geq2|\xi_3|}\right)e^{-t\bigl(\f{|\xi|^2}2-\al(\xi)\bigr)}\f{1-e^{-2t\al(\xi)}}{2\al(\xi)}|\xi_3\widehat{Y_1}(\xi)|\,d\xi,
\end{split}
\eeno and \beno
\begin{split} &\int_{|\xi|\geq
3}e^{-t\bigl(\f{|\xi|^2}2-\al(\xi)\bigr)}\f{1-e^{-2t\al(\xi)}}{2\al(\xi)}|\xi_3\widehat{Y_1}(\xi)|\,d\xi\\
&\leq \||\xi|^3\widehat{Y_1}\|_{L^\infty}\int_{|\xi|\geq
3}e^{-t\f{\xi_3^2}{\f{|\xi|^2}2+\al(\xi)}}\f{|\xi_3|}{2\al(\xi)|\xi|^3}\,d\xi\\
&\leq
2\||\xi|^3\widehat{Y_1}\|_{L^\infty}\int_0^{\f{\pi}2}\int_3^{\infty}e^{-t\cos^2\phi}\f1{r\sqrt{r^2-4\cos^2\phi}}\sin\phi\cos\phi\,d\phi\,dr\\
&\leq C\||\xi|^3\widehat{Y_1}\|_{L^\infty}
\int_0^1e^{-t\tau}\int_3^\infty\f1{r\sqrt{r^2-4\tau^2}}\,dr\,d\tau\\
&\leq C{\w{t}}^{-1}\||D|^3Y_1\|_{L^1}.
\end{split}
\eeno
Exactly along the same line, we have \beno
\begin{split} &
 \int_{9>|\xi|^2\geq
2|\xi_3|}e^{-t\bigl(\f{|\xi|^2}2-\al(\xi)\bigr)}\f{1-e^{-2t\al(\xi)}}{2\al(\xi)}|\xi_3\widehat{Y_1}(\xi)|\,d\xi\\
&\leq
2\||\xi|\widehat{Y_1}\|_{L^\infty}\int_0^{\f{\pi}2}\int_{2\cos\phi}^3e^{-t\cos^2\phi}\f{\sin\phi\cos\phi}{\sqrt{r^2-4\cos^2\phi}}r\,dr\,d\phi\\
&\leq
C\||\xi|\widehat{Y_1}\|_{L^\infty}\int_0^1e^{-t\tau}\int_{2\sqrt{\tau}}^3\f{r}{\sqrt{r^2-4\tau}}\,dr\,d\tau\\
&\leq C{\w{t}}^{-1}\||D|Y_1\|_{L^1}.
\end{split}
\eeno This proves
\beq\label{S2eq7} I_1\leq
C{\w{t}}^{-1}\bigl(\||D|Y_1\|_{L^1}+\||D|^3Y_1\|_{L^1}\bigr).\eeq

The estimate of $I_2$ is much simpler. By virtue of \eqref{S2eq6},
we have \beq \label{S2eq8}
\begin{split}
I_2\leq&
2\||\xi|\widehat{Y_1}\|_{L^\infty}\int_0^{\f{\pi}2}\int_0^{2\cos\phi}e^{-\f{t}2r^2}\f1{\sqrt{4\cos^2\phi-r^2}}\sin\phi\cos\phi
r\,dr\,d\phi\\
\leq&
2\||\xi|\widehat{Y_1}\|_{L^\infty}\int_0^{1}e^{-\f{t}2r^2}r\int_{\f{r^2}4}^{1}\f1{\sqrt{4\tau-r^2}}
\,d\tau\,dr\\
\leq& C{\w{t}}^{-1}\||D|Y_1\|_{L^1}.
\end{split}
\eeq
As a result, we achieve
\beq
\label{S2eq9}
\|\p_3\cY(t)\|_{L^\infty}\leq {C}{\w{t}}^{-1}\bigl(\||D|Y_1\|_{L^1}+\||D|^3Y_1\|_{L^1}\bigr).\eeq

Along the same line to the proof of \eqref{S2eq9}, we infer
\beno
\begin{split}
\|\p_3^2\cY(t)\|_{L^\infty}\leq &2\||\xi|^4\widehat{Y_1}\|_{L^\infty}\int_0^{\f{\pi}2}\int_3^{\infty}e^{-t\cos^2\phi}\f1{r\sqrt{r^2-4\cos^2\phi}}\sin\phi\cos^2\phi\,d\phi\,dr\\
&+2\||\xi|^2\widehat{Y_1}\|_{L^\infty}\int_0^{\f{\pi}2}\int_{2\cos\phi}^3e^{-t\cos^2\phi}\f{\sin\phi\cos^2\phi}{\sqrt{r^2-4\cos^2\phi}}r\,dr\,d\phi\\
&+
2\||\xi|\widehat{Y_1}\|_{L^\infty}\int_0^{\f{\pi}2}\int_0^{2\cos\phi}e^{-\f{t}2r^2}\f1{\sqrt{4\cos^2\phi-r^2}}\sin\phi\cos^2\phi
r^2\,dr\,d\phi,
\end{split}
\eeno
so that for $t$ large enough, there holds
\beno
\begin{split}
\|\p_3^2\cY(t)\|_{L^\infty}\leq &Ct^{-\f12}\Bigl(\||\xi|^4\widehat{Y_1}\|_{L^\infty}\int_0^{\f{\pi}2}\int_3^{\infty}e^{-\f{t\cos^2\phi}2}\f1{r\sqrt{r^2-4\cos^2\phi}}\sin\phi\cos\phi\,d\phi\,dr\\
&+\||\xi|^2\widehat{Y_1}\|_{L^\infty}\int_0^{\f{\pi}2}\int_{2\cos\phi}^3e^{-\f{t\cos^2\phi}2}\f{\sin\phi\cos\phi}{\sqrt{r^2-4\cos^2\phi}}r\,dr\,d\phi\Bigr)\\
&+\||\xi|\widehat{Y_1}\|_{L^\infty}\int_0^{1}e^{-\f{t}2r^2}r^2\int_{\f{r^2}4}^{1}\f1{\sqrt{4\tau-r^2}}
\,d\tau\,dr.
\end{split}
\eeno
This gives rise to
\beq
\label{S2eq9a}
\|\p_3^2\cY(t)\|_{L^\infty}\leq {C}{\w{t}}^{-\f32}\bigl(\||D|Y_1\|_{L^1}+\||D|^4Y_1\|_{L^1}\bigr).\eeq

\noindent$\bullet$\underline{\it Estimate of $|\p_t\cY(t)|_{L^\infty}.$}

It follows from \eqref{S2eq2} that
\beno
\p_t\Ga(t,\xi)=\f1{\la_2(\xi)-\la_1(\xi)}\left(\la_2(\xi)e^{t\la_2(\xi)}-\la_1(\xi)e^{t\la_1(\xi)}\right).
\eeno So that one has \beq \label{S2eq10} \begin{split} \p_t\Ga(t,\xi){\bf
1}_{|\xi|^2\geq
2|\xi_3|}=&e^{-t\bigl(\f{|\xi|^2}2+\al(\xi)\bigr)}-e^{-t\bigl(\f{|\xi|^2}2-\al(\xi)\bigr)}\bigl(\f{|\xi|^2}2-\al(\xi)\bigr)
\f{1-e^{-2t\al(\xi)}}{2\al(\xi)},\\
 \p_t\Ga(t,\xi){\bf
1}_{|\xi|^2<
2|\xi_3|}=&e^{-\f{t}2|\xi|^2}\left(-\f{|\xi|^2}2\f{\sin(t\beta(\xi))}{\beta(\xi)}+\cos(t\beta(\xi))\right).
\end{split}
\eeq It is easy to observe that  for any $\d\in [0,1[,$
\beno
\begin{split}
\int_{\R^3}e^{-\f{t}2|\xi|^2}|\widehat{Y_1}(\xi)|\,d\xi\leq &\||\xi|^{2\d}\widehat{Y_1}\|_{L^\infty}\int_{\R^3}|\xi|^{-2\d}e^{-\f{t}2|\xi|^2}\,d\xi\\
\leq &Ct^{-\left(\f32-\d\right)}\||D|^{2\d}Y_1\|_{L^1},
\end{split}
\eeno
and
\beno
\begin{split}
\int_{\R^3}e^{-\f{t}2|\xi|^2}|\widehat{Y_1}(\xi)|\,d\xi\leq &\||\xi|^{2\d}\widehat{Y_1}\|_{L^\infty}\int_{|\xi|\leq 1}|\xi|^{-2\d}\,d\xi+
\||\xi|^{4}\widehat{Y_1}\|_{L^\infty}\int_{|\xi|>1}|\xi|^{-4}\,d\xi\\
\leq &C\bigl(\||D|^{2\d}Y_1\|_{L^1}+\||D|^{4}Y_1\|_{L^1}\bigr).
\end{split}
\eeno
This leads to
\beno
\int_{\R^3}e^{-\f{t}2|\xi|^2}|\widehat{Y_1}(\xi)|\,d\xi\leq C\w{t}^{-\left(\f32-\d\right)}\bigl(\||D|^{2\d}Y_1\|_{L^1}+\||D|^{4}Y_1\|_{L^1}\bigr).
\eeno
While similar to estimate of \eqref{S2eq7} and \eqref{S2eq8}, we infer
\beno
\begin{split}
&\int_{|\xi|^2\geq
2|\xi_3|}e^{-t\f{\xi_3^2}{\f{|\xi|^2}2+\al(\xi)}}\f{\xi_3^2}{\f{|\xi|^2}2+\al(\xi)}\f{1-e^{-2t\al(\xi)}}{2\al(\xi)}|\widehat{Y_1}(\xi)|\,d\xi\\
&\leq2\int_0^{\f{\pi}2}\int_0^{2\pi}\int_{2\cos\phi}^\infty
e^{-t\cos^2\phi}2\cos^2\phi\f{|\widehat{Y_1}(\xi(r,\th,\phi))|}{\sqrt{r^2-4\cos^2\phi}}\sin\phi
r\,dr\,d\th\,d\phi\\
&\leq
2\||\xi|^{2\d}\widehat{Y_1}\|_{L^\infty}\int_0^1e^{-t\tau}\sqrt{\tau}\int_{2\sqrt{\tau}}^3\f{r^{1-2\d}}{\sqrt{r^2-4\tau}}\,dr\,d\tau\\
&\quad+2\||\xi|^2\widehat{Y_1}\|_{L^\infty}\int_0^1e^{-t\tau}\sqrt{\tau}\int_3^\infty\f1{r\sqrt{r^2-4\tau}}\,dr\,d\tau\\
&\leq{C}{\w{t}^{-\f32}}\bigl(\||D|^{2\d}Y_1\|_{L^1}+\||D|^2Y_1\|_{L^1}\bigr),
\end{split}
\eeno
and  \beno
\begin{split}
\int_{|\xi|^2\leq
2|\xi_3|}e^{-\f{t}4|\xi|^2}\f{|\xi_3|}{\beta(\xi)}|\widehat{Y_1}(\xi)|\,d\xi
&\leq 2\||\xi|^{2\d}\widehat{Y_1}\|_{L^\infty}\int_0^{\f{\pi}2}\f{2\sin\phi\cos\phi}{\sqrt{4\cos^2\phi-r^2}}\int_0^{2\cos\phi}e^{-\f{t}4r^2}r^{2(1-\d)}\,dr\,d\phi\\
&\leq C\w{t}^{-\left(\f32-\d\right)}\||\xi|^{2\d}\widehat{Y_1}\|_{L^\infty}\leq
C\w{t}^{-\left(\f32-\d\right)}\||D|^{2\d}Y_1\|_{L^1}.
\end{split}
\eeno Hence by virtue of \eqref{S2eq10}, we obtain \beq\label{S2eq11}
\|\p_t\cY(t)\|_{L^\infty}\leq {C}{\w{t}^{-\left(\f32-\d\right)}}\bigl(\||D|^{2\d}Y_1\|_{L^1}+\|D^4 Y_1\|_{L^1}\bigr). \eeq

\noindent$\bullet$\underline{\it Estimate of $\|\cY(t)\|_{L^\infty}.$}

Note that
\beno
\begin{split}
\int_0^1\int_{2\tau}^{3}e^{-t\tau^2}\frac{r^{\f12}}{\sqrt{r^2-4\tau^2}}dr\,d\tau\leq \int_0^1e^{-t\tau^2}\int_{2\tau}^{3}(r-2\tau)^{-\f12}dr\,d\tau\leq C\w{t}^{-\f12}.
\end{split}
\eeno
we find
\begin{align*}
&\int_{|\xi|^2\geq2|\xi_3|}e^{-t\big(\frac{|\xi|^2}{2}-\alpha(\xi)\big)}\frac{1-e^{-2t\alpha(\xi)}}{2\alpha(\xi)}|\widehat{Y_1}(\xi)|d\xi\\
&\leq\int_0^{\frac{\pi}{2}}\int_0^{2\pi}\int_{2\cos\phi}^\infty e^{-t\cos^2\phi}\frac{|\widehat{Y_1}(\xi(r,\theta,\phi))|}{\frac12r\sqrt{r^2-4\cos^2\phi}}r^2\sin\phi\, dr\,d\theta\,d\phi\\
&=C\||\xi|^{\f12}\widehat{Y_1}\|_{L^\infty}\int_0^1\int_{2\tau}^{3}e^{-t\tau^2}\frac{r^{\f12}}{\sqrt{r^2-4\tau^2}}dr\,d\tau+C\||\xi|^2\widehat{Y_1}\|_{L^\infty}\int_0^1e^{-t\tau^2}\int_3^\infty\frac{1}{r\sqrt{r^2-4\tau^2}}dr\,d\tau\\
&\leq {C}{\w{t}}^{-\frac12}\bigl(\||D|^{\f12}Y_1\|_{L^1}+\||D|^2Y_1\|_{L^1}\bigr).
\end{align*}
Similarly, we have
\begin{align*}
\int_{|\xi|^2<2|\xi_3|}e^{-t\frac{|\xi|^2}{2}}\frac{\sin(t\beta(\xi))}{2\beta(\xi)}|\widehat{Y_1}(\xi)|d\xi
&\leq \int_0^\frac{\pi}{2}\int_0^{2\cos\phi} e^{-\frac{t}{2}r^2}\frac{|\widehat{Y_1}(\xi(r,\theta,\phi))|}{r\sqrt{4\cos^2\phi-r^2}}r^2\sin\phi dr\,d\phi\\
&\leq\||\xi|^{\f12}\widehat{Y_1}\|_{L^\infty}\int_0^2e^{-\frac{t}{2}r^2}\int_{r/2}^1\frac{r^{\f12}}{\sqrt{4\tau^2-r^2}}d\tau\,dr\\
&\leq {C}{\w{t}}^{-\frac12}\||D|^{\f12}Y_1\|_{L^1}.
\end{align*}
As a result, by virtue of \eqref{S2eq4} and \eqref{S2eq5}, it comes out
\begin{equation}\label{S2eq12}
\|\cY(t)\|_{L^\infty}\leq {C}{\w{t}}^{-\frac12}\bigl(\||D|^{\f12}Y_1\|_{L^1}+\|D^2Y_1\|_{L^1}\bigr).
\end{equation}
\eqref{S2eq9} together with \eqref{S2eq9a}, \eqref{S2eq11} and \eqref{S2eq12} imply the estimate \eqref{S2eq3} for $N=0.$
\end{proof}

\begin{lem}\label{S1prop2a}
{\sl For $N\in\N$, there exists $C_N>0$ such that for $t>0$,
\beq\begin{split}
&\|t\Delta\p_t\Gamma(t)Y_1\|_{L_t^\infty(H^N)}\leq C_N\|Y_1\|_{N} \andf \|t\na\p_3^2\Gamma(t)Y_1\|_{L_t^\infty(H^N)}\leq C_N\|Y_1\|_{N+1}. \end{split} \label{S2eq15'} \eeq
}
\end{lem}

\begin{proof} The two inequalities of \eqref{S2eq15'} follows from the claim that
\beq \label{S2eq15'a}
 t|\xi|^2\p_t\Gamma(t,\xi)\in L^\infty_{t}(L^\infty_{\xi}),\andf \f{t|\xi|}{1+|\xi|}|\xi_3|^{2}\Gamma(t,\xi)\in L_t^\infty(L_\xi^\infty).
\eeq

\no (1)\ When $|\xi|^2\geq 2|\xi_3|$, we separate the proof of \eqref{S2eq15'a} into the following two cases:
\begin{itemize}
\item
If $\frac{\sqrt{3}}{4}|\xi|^2\leq |\xi_3|\leq\frac12|\xi|^2$, we deduce from \eqref{S2eq10} that
\beno
\begin{split}
 &|\p_t\Gamma(t,\xi){\bf 1}_{\frac{\sqrt{3}}{4}|\xi|^2\leq |\xi_3|\leq\frac12|\xi|^2}|\leq e^{-t\frac{|\xi|^2}{4}}(1+|\xi|^2t),\\
 &|\xi_3^2\Gamma(t,\xi){\bf 1}_{\frac{\sqrt{3}}{4}|\xi|^2\leq |\xi_3|\leq\frac12|\xi|^2}\leq t\xi_3^2e^{-t\f{\xi_3^2}{\f{|\xi|^2}2+\al(\xi)}}{\bf 1}_{\frac{\sqrt{3}}{4}|\xi|^2\leq |\xi_3|\leq\frac12|\xi|^2}|\leq Ct\xi_3^2e^{-ct|\xi_3|}.
\end{split}\eeno
As a result, it comes out
\begin{align*}
t|\xi|^{2}|\p_t\Gamma(t,\xi)|{\bf 1}_{\frac{\sqrt{3}}{4}|\xi|^2\leq |\xi_3|\leq\frac12|\xi|^2}\leq C \andf \f{t|\xi|}{1+|\xi|}|\xi_3|^{2}\Gamma(t,\xi){\bf 1}_{\frac{\sqrt{3}}{4}|\xi|^2\leq |\xi_3|\leq\frac12|\xi|^2}\leq C.
\end{align*}

\item
If $|\xi_3|\leq\frac{\sqrt{3}}{4}|\xi|^2$, then $\f{|\xi|^2}4\leq\alpha(\xi)\leq \frac{|\xi|^2}2,$ we  deduce from \eqref{S2eq10} that
\begin{align*}
& |\p_t\Gamma(t,\xi)|{\bf 1}_{|\xi_3|\leq\frac{\sqrt{3}}{4}|\xi|^2}\leq e^{-t\frac{|\xi|^2}2}+e^{-t\frac{\xi_3^2}{|\xi|^2}}\frac{\xi_3^2}{|\xi|^4},\\
&\xi_3^2|\Gamma(t,\xi)|{\bf 1}_{|\xi_3|\leq\frac{\sqrt{3}}{4}|\xi|^2}\leq \f{\xi_3^2}{\alpha(\xi)}e^{-t\f{\xi_3^2}{\f{|\xi|^2}2+\al(\xi)}}\leq C\f{\xi_3^2}{\f{|\xi|^2}2+\alpha(\xi)}e^{-t\f{\xi_3^2}{\f{|\xi|^2}2+\al(\xi)}},
\end{align*}
so that there holds
\begin{align*}
&t|\xi|^{2}|\p_t\Gamma(t,\xi)|\xi_3|{\bf 1}_{|\xi_3|\leq\frac{\sqrt{3}}{4}|\xi|^2}\leq t|\xi|^{2}e^{-t\frac{|\xi|^2}2}+te^{-t\frac{\xi_3^2}{|\xi|^2}}\frac{\xi_3^2}{|\xi|^2}\leq C,\\
&\f{t|\xi|}{1+|\xi|}|\xi_3|^{2}\Gamma(t,\xi){\bf 1}_{|\xi_3|\leq\frac{\sqrt{3}}{4}|\xi|^2}\leq C.
\end{align*}

\end{itemize}

\no (2)\ When $|\xi|^2>2|\xi_3|$, we infer from \eqref{S2eq10} that
\begin{align*}
 |\p_t\Gamma(t,\xi)|{\bf 1}_{|\xi|^2>2|\xi_3|}\leq e^{-t\frac{|\xi|^2}{2}}(|\xi|^2t+1),
\end{align*}
which implies
\begin{align*}
 t|\xi|^{2}|\p_t\Gamma(t,\xi){\bf 1}_{|\xi|^2>2|\xi_3|}|\leq C.
\end{align*}
To prove the second estimate of \eqref{S2eq15'a}, we divide further the region, $\{  |\xi|^2>2|\xi_3|\},$ into two parts.

\begin{itemize}
\item
If $|\xi|^2\leq \sqrt{3}|\xi_3|$, then we have $\f{|\xi_3|}2\leq \beta(\xi)\leq |\xi_3|,$ and it follows from \eqref{S2eq5} that
\beno
\xi_3^2|\Gamma(t,\xi)|{\bf 1}_{|\xi|^2\leq\sqrt{3}|\xi_3|}\leq C|\xi_3|e^{-t\f{|\xi|^2}2}\leq \f{C}{t|\xi|}.
\eeno

\item  When $ \sqrt{3}|\xi_3|<|\xi|^2\leq 2|\xi_3|$, we have
\beno
\xi_3^2|\Gamma(t,\xi)|{\bf 1}_{\sqrt{3}|\xi_3|<|\xi|^2\leq 2|\xi_3|}\leq Ct|\xi_3|^2e^{-ct|\xi_3|}\leq \f{C}{t}.
\eeno

\end{itemize}
By summarizing the above estimates, we obtain
the second estimate of \eqref{S2eq15'a}. This completes the proof of Lemma \ref{S1prop2a}.
\end{proof}

\subsection{Energy estimates for the linear equation}

\begin{lem}\label{S1lem2}
{\sl Let $\cY(t)$ be a smooth enough solution of the linear equation \eqref{S2eq1} with initial data $(\cY_0,\cY_1)$. Then for any $N\in\N$, there exists $C_N>0$ such that there hold \eqref{S2eq13as} and \eqref{S2eq16}.}
\end{lem}

\begin{proof}
Taking the $L^2$-inner product of the equation \eqref{S2eq1} with $\cY_t$ and $\cY_t-\frac14\Delta\cY-\Delta \cY_t$, respectively, we get
$$\frac12\frac{d}{dt}(\|\cY_t\|_{0}^2+\|\p_3\cY\|_{0}^2)+\|\nabla \cY_t\|_{0}^2=0$$
and
\begin{align*}
\frac{d}{dt}\Big(\frac12\|\cY_t\|_{1}^2+\|\p_3\cY\|_{1}^2+&\frac14\|\Delta \cY\|_{0}^2-\frac14(\cY_t|\Delta \cY)_{L^2}\Big)+\frac34\|\nabla\cY_t\|_{0}^2+\|\Delta\cY_t\|_{0}^2+\frac14\|\nabla\p_3\cY\|_{0}^2=0.
\end{align*}
Integrating the above equalities with respect to $t$ gives rise to
$$\|(\cY_t,\p_{3}\cY)\|_{L^\infty_t(L^2)}+\|\nabla\p_t \cY\|_{L_t^2(L^2)}\leq \|(\p_3\cY_0,\cY_1)\|_{0} \andf $$
$$
\|(\cY_t,\p_3\cY)\|_{L^\infty_t(H^1)}+\|\Delta\cY\|_{L^\infty_t(L^2)}
+\|\na\cY_t\|_{L^2_t(H^1)}+\|\na\p_3\cY\|_{L^2_t(L^2)}\leq C\bigl(\|(\p_3\cY_0,\cY_1)\|_{1}+\|\D \cY_0\|_0\bigr). $$
This proves \eqref{S2eq13as} for $N=0$. The general case with $N>0$ follows similarly.

To show \eqref{S2eq16}, we first get, by
taking the $H^N$-inner product of the equation \eqref{S2eq1} with $\cY_t$, that
$$\frac{1}{2}\frac{d}{dt}\big(\|\cY_t\|_{N}^2+\|\p_3\cY\|_{N}^2\big)+\|\nabla\cY_t\|_{N}^2=0.$$
So that for any nonegative $f(t)\in C^1([0,\infty[)$, we have
\begin{align*}
\frac{d}{dt}\left(f(t)\big(\|\cY_t\|_{N}^2+\|\p_3\cY\|_{N}^2\big)\right)+2f(t)\|\nabla \cY_t\|_{{N}}^2=f'(t)\big(\|\cY_t\|_{N}^2+\|\p_3\cY\|_{N}^2\big).
\end{align*}
Taking $f(t)=\langle t\rangle$ and integrating the resulting equality over $[0,t]$, we find
\begin{equation}\label{S2eq17}\begin{split}
\langle t\rangle\big(\|\cY_t(t)\|_{N}^2+\|\p_3\cY(t)\|_{N}^2\big)+&2\int_0^t\langle s\rangle\|\nabla \cY_t(s)\|_{{N}}^2ds\leq\|(\p_3\cY_0,\cY_1)\|_{N}^2+\int_0^t\big(\|\cY_t\|_{N}^2+\|\p_3\cY\|_{N}^2\big)ds.
\end{split}
\end{equation}
Yet it from \eqref{S2eq13as} that
\beno \|\cY_t\|_{L^2_t(H^{N+1})}+\|\p_3\cY\|_{L^2_t(H^N)}\leq C_N\bigl(\||D|^{-1}(\p_3\cY_0,\cY_1)\|_{{N+1}}+\|\na\cY_0\|_{N}\bigr),
\eeno
which together with \eqref{S2eq17} ensures \eqref{S2eq16}.
\end{proof}

Recall that $\cY(t)=\Gamma(t)Y_1$ is the solution to \eqref{S2eq1} with initial data $(\cY_0,\cY_1)=(0,Y_1)$,
so that one can deduce estimates for the operator $\Gamma$ from the energy estimates \eqref{S2eq13as} and \eqref{S2eq16}:
Indeed
combining \eqref{S2eq15'} with \eqref{S2eq13as} gives
\begin{equation}\label{S2eq15'c}
\begin{split}
&\|\langle t\rangle\D\p_t\Gamma(t)Y_1\|_{L_t^\infty(H^N)}+
\|\langle t\rangle\p_3^2\Gamma(t)Y_1\|_{L_t^\infty(H^N)}\leq C_N\|Y_1\|_{N+2}.
\end{split}
\end{equation}
Let us remark that
\begin{equation}\label{S2eq14}
\begin{split}
\|Y\|_{L^\infty}\leq \int|\widehat Y(\xi)|d\xi&\leq \int_{|\xi|\leq1}{|\xi|}^{-1}\cdot|\xi||\widehat Y(\xi)|d\xi+\int_{|\xi|>1}|\xi|^{-2}\cdot|\xi|^2|\widehat Y(\xi)|d\xi\\
&\leq C(\||D|Y\|_{L^2}+\||D|^2Y\|_{L^2})\leq C\||D|Y\|_{1}.
\end{split}
\end{equation}
 summarizing \eqref{S2eq13as}, \eqref{S2eq15'c} and \eqref{S2eq14} then leads to
%
\begin{col}
{\sl For $N\geq0$, there exists $C_N>0$ such that
\beq \label{S2eq15}
\begin{split}
&\|\Gamma(t)Y_1\|_{L^\infty_t(W^{N,\infty})}\ \leq C_N\||D|^{-1}Y_1\|_{N+2},\\
&\|\p_3\Gamma(t)Y_1\|_{L_t^2(W^{N,\infty})}\leq C_N\|Y_1\|_{N+2},\\
&\|\langle t\rangle\p_t\Gamma(t)Y_1\|_{L_t^\infty(W^{N,\infty})}\leq C_N\||D|^{-1}Y_1\|_{N+3},
\end{split} \eeq
where $\Gamma(t)$ is the solution operator given by \eqref{S2eq2}.}
\end{col}

Now we are in a position to complete the proof of Proposition \ref{S0prop1}.

\begin{proof}[Proof of Proposition \ref{S0prop1}]
\eqref{S2eq13as} and \eqref{S2eq16} are already proved by Lemma \ref{S1lem2}. So it remains to
deal with the estimates of \eqref{S2eq3as} and \eqref{S2eq15'as}. As a matter of fact,
according to the definition of the solution operator $\Gamma(t)$ given by \eqref{S2eq2},
we have
\beq\label{S2eq15'b}
\cY(t)=\p_t\Gamma(t)\cY_{0}+\Gamma(t)(\cY_{1}-\D \cY_{0}),
\eeq
from which and \eqref{S2eq3}, we infer
that for any $\d\in ]0,1[$ and for $ N\in\N$,
\beq\label{S2eq15'b}
\begin{split}
|\p_3\cY|_{1,N}+& |\p_t\cY|_{\frac32-\delta,N}+ |\cY|_{\frac12,N}\leq |\p_3\p_t\Gamma(t)\cY_{0}|_{1,N}+|\p_t^2\Gamma(t)\cY_{0}|_{\frac32-\delta,N}\\
&+|\p_t\Gamma(t)\cY_{0}|_{\frac12,N}
+
C_{N}\bigl(\||D|^{2\d}(\D \cY_{0},\cY_{1})\|_{L^1}+\||D|^{N+4}(\D \cY_{0},\cY_{1})\|_{L^1}\bigr).\end{split} \eeq
While notice that $\p_t^2\Gamma(t)\cY_{0}=\D\p_t\Gamma(t)\cY_{0}+\p_3^2\Gamma(t)\cY_{0}, $ we get, by
applying \eqref{S2eq3} once again, that
\beno
\begin{split}
|\p_3\p_t\Gamma(t)\cY_{0}|_{1,N}=&|\p_t\Gamma(t)\p_3\cY_{0}|_{1,N}\leq C_{N}\bigl(\||D|^{2\d}\p_3 \cY_{0}\|_{L^1}
+\||D|^{N+4}\p_3\cY_{0}\|_{L^1}\bigr),\\
|\p_t^2\Gamma(t)\cY_{0}|_{\frac32-\delta,N}\leq& |\D\p_t\Gamma(t)\cY_{0}|_{\frac32-\delta,N}+|\p_3^2\Gamma(t)\cY_{0}|_{\frac32,N}
\leq  C_{N}\bigl(\||D|^{2\d} \cY_{0}\|_{L^1}
+\||D|^{N+6}\cY_{0}\|_{L^1}\bigr),\\
|\p_t\Gamma(t)\cY_{0}|_{\frac12,N}\leq &C_{N}\bigl(\||D|^{2\d} \cY_{0}\|_{L^1}
+\||D|^{N+4}\cY_{0}\|_{L^1}\bigr).
\end{split}
\eeno
Inserting the above estimates into \eqref{S2eq15'b} leads to \eqref{S2eq3as}.

Finally notice that $\D\p_t^2\Gamma(t)\cY_0=\D^2\p_t\Gamma(t)\cY_0 +\D\p_3^2\Gamma(t)\cY_0.$
Then by virtue of \eqref{S2eq15'c}, we deduce
\beno
\begin{split}
\|\langle t\rangle\D\p_t\cY(t)\|_{L_t^\infty(H^N)} &\leq  \|\langle t\rangle \D\p_t^2\Gamma(t)\cY_0\|_{L_t^\infty(H^N)}+\|\langle t\rangle\D\p_t\Gamma(t)(\cY_1-\D\cY_0)\|_{L_t^\infty(H^N)}\\
&\leq C_N\bigl(\|\D\cY_0\|_{N+2}+\|(\D \cY_0,\cY_1)\|_{N+2}\bigr).
\end{split}
\eeno
This proves \eqref{S2eq15'as}, and thus we complete the proof of Proposition \ref{S0prop1}.
\end{proof}


\subsection{Decay estimates for the inhomogeneous equation}\label{Subsect3.3}

\begin{proof}[Proof of Proposition \ref{S2.3prop1}]  In view of \eqref{S2eq2}, we get, by applying Duhamel's principle to \eqref{S2eq18}, that
\begin{equation}\label{S2eq20}
Y(t)=\int_0^t\Gamma(t-s)g(s)ds.
\end{equation}
In what follows, we shall present the proof of \eqref{S2eq19} term by term.

\noindent$\bullet$\underline{\it Decay estimate of $\p_3Y.$}

We first separate the integral in \eqref{S2eq20} as
\begin{align*}
\p_3Y(t,y)&=\int_0^t\p_3\Gamma(t-s)g(s)ds=\int_0^{t/2}\p_3\Gamma(t-s)g(s)ds+\int_{t/2}^t\p_3\Gamma(t-s)g(s)ds.
\end{align*}
We deduce from  \eqref{S2eq9} that
\begin{align*}
\langle t\rangle\int_0^{t/2}\big|\p_3\Gamma(t-s)g(s)\big|_{N}\,ds\leq& C_N\langle t\rangle\int_0^{t/2}\langle t-s\rangle^{-1}\triplenorm{|D| g(s)}_{N+2}\,ds\\
\leq &C_N\int_0^{t/2}\triplenorm{|D| g(s)}_{N+2}ds
\leq C\||D|g\|_{L_t^1(W^{N+2,1})}.
\end{align*}
While it follows from the second inequality in \eqref{S2eq15} that
\begin{align*}
&\langle t\rangle\int_{t/2}^t\big|\p_3\Gamma(t-s)g(s)\big|_{N}ds\leq\langle t\rangle\Big(\int_{t/2}^t\| g(s)\|_{N+2}^2ds\Big)^\frac12\leq C\theta^\frac12\|\langle t\rangle^\frac12g\|_{L_t^2(H^{N+2})}.
\end{align*}
Hence we achieve
\begin{equation}\label{S2eq21}
\begin{split}
&|\p_3Y|_{1,N}\leq C_N\left(\||D| g\|_{L_t^1(W^{N+2,1})}+\theta^\frac12\|\langle t\rangle^\frac12g\|_{L_t^2(H^{N+2})}\right).
\end{split}
\end{equation}

\noindent$\bullet$\underline{\it Decay estimate of $Y_t.$}

Noticing that $\Gamma(0)=0$, we have
$$Y_t(t)=\int_0^t\p_t\Gamma(t-s)g(s)ds=\int_0^{t/2}\p_t\Gamma(t-s)g(s)ds+\int_{t/2}^t\p_t\Gamma(t-s)g(s)ds.$$
It follows from \eqref{S2eq11} that
\begin{align*}
\langle t\rangle^{\frac32-\delta}\int_0^{t/2}\big|\p_t\Gamma(t-s)g(s)\big|_{N}ds
&\leq C_N\langle t\rangle^{\frac32-\delta}\int_0^{t/2}\langle t-s\rangle^{-\left(\frac32-\delta\right)}\bigl(\triplenorm{|D|^{2\delta}g(s)}_{N}+\triplenorm{D^4g(s)}_{N}\bigr)ds\\
&\leq C_N\bigl(\||D|^{2\delta}g\|_{L_t^1(W^{N,1})}+\|D^4g\|_{L_t^2(W^{N,1})}\bigr).
\end{align*}
Whereas it follows from the third inequality in \eqref{S2eq15} that
\begin{align*}
\langle t\rangle^{\frac32-\delta}\int_{t/2}^t\big|\p_t\Gamma(t-s)g(s)\big|_{N}ds\leq& \int_{t/2}^t\langle t-s\rangle^{-1}\||D|^{-1}g(s)\|_{N+3}\langle s\rangle^{\frac32-\delta}ds\\
\leq& C_N\log\langle\theta\rangle\|\langle t\rangle^{\frac32-\delta}|D|^{-1}g\|_{L_t^\infty(H^{N+3})}\leq C_N\log\langle\theta\rangle\big\||D|^{-1}g\big\|_{\frac32-\delta,N+3}.
\end{align*}
As a result, it comes out
\begin{equation}\label{S2eq22}
\begin{split}
&|Y_t|_{\frac32-\d,N}\leq C_N\left(\bigl(\||D|^{2\d}g\|_{L_t^1(W^{N,1})}+\||D|^{4}g\|_{L_t^1(W^{N,1})}\big)+\log\langle\theta\rangle\big\||D|^{-1}g\big\|_{\frac32-\delta,N+3}\right).
\end{split}
\end{equation}

\noindent$\bullet$\underline{\it Decay estimate of $Y.$}

As in the previous steps,
we first split the integral \eqref{S2eq20} into two parts. For the integral from $0$ to $t/2$, we use \eqref{S2eq12} to deduce that
\begin{align*}
\langle t\rangle^\frac12\int_0^{t/2}\big|\Gamma(t-s)g(s)\big|_{N}ds
&\leq \langle t\rangle^\frac12\int_0^{t/2}C_N\langle t-s\rangle^{-\frac12}\bigl(\triplenorm{|D|^{\f12}g(s)}_{N}+\triplenorm{|D|^{2}g(s)}_{N}\bigr)ds\\
&\leq C_N\left(\||D|^{\f12}g\|_{L_t^1(W^{N,1})}+\||D|^{2}g\|_{L_t^1(W^{N,1})}\right).
\end{align*}
For the integral from $t/2$ to $t$, we apply the first inequality of \eqref{S2eq15} to get
\begin{align*}
\langle t\rangle^\frac12\int_{t/2}^t\big|\Gamma(t-s)g(s)\big|_{N}ds
&\leq C_N\langle t\rangle^\frac12\int_{t/2}^t\||D|^{-1}g(s)\|_{N+2}ds\\
&\leq C_N\langle t\rangle\Big(\int_{t/2}^t\||D|^{-1}g(s)\|_{N+2}^2ds\Big)^\frac12\leq C_N\theta^\frac12\|\langle t\rangle^\frac12|D|^{-1}g\|_{L_t^2(H^{N+2})}.
\end{align*}
Hence we obtain
\begin{equation}\label{S2eq23}
\begin{split}
&|Y|_{\frac12,N}\leq C_N\left(\big(\||D|^{\f12}g\|_{L_t^1(W^{N,1})}+\||D|^{2}g\|_{L_t^1(W^{N,1})}\big)+\theta^{\frac12}\|\langle t\rangle^\frac12|D|^{-1}g\|_{L_t^2(H^{N+2})}\right).
\end{split}
\end{equation}
By summarizing the estimates \eqref{S2eq21}, \eqref{S2eq22} and \eqref{S2eq23}, we complete the proof of  \eqref{S2eq19}.
\end{proof}


\section{The derivatives of $f$ given by \eqref{S1eq2}}\label{Sect4}

\subsection{Computation of $f'(Y;X)$} The goal of this subsection is to derive the linearized equations of the system (\ref{S1eq1}-\ref{S1eq2}).
We first decompose the pressure function $\vv p$ given by \eqref{S1eq2} as $\vv p=\vv p_{1}+\vv p_2$ with
\begin{align}
\label{p1}&\vv p_1\eqdefa-\Delta^{-1}\dv\big(({\mathcal A}{\mathcal A}^t-Id)\nabla \vv  p_1\big)+\Delta^{-1}\dv\Big({\mathcal A}\dv\big({\mathcal A}(\p_3Y\otimes\p_3Y)\big)\Big)\\
\label{p2}&\vv p_2\eqdefa-\Delta^{-1}\dv\big(({\mathcal A}{\mathcal A}^t-Id)\nabla \vv p_2\big)+\Delta^{-1}\dv\Big({\mathcal A}\dv\big({\mathcal A}(Y_t\otimes Y_t)\big)\Big).
\end{align}
Let us denote
\begin{equation}\label{fm}
\begin{split}
 f_0\eqdefa\na_y\cdot\bigl((\mathcal{A}\mathcal{A}^{t}-Id)\na_y Y_t\bigr),\quad
 f_1\eqdefa\mathcal{A}^{t}\na_y\vv p_1 \andf f_2\eqdefa\mathcal{A}^{t}\na_y\vv p_2.
 \end{split}
 \end{equation}
Then the functional $f$   given by \eqref{S1eq2} can be decomposed as $f_0-f_1+f_2$.

Before proceeding, let us recall that for a map $f:{\mathcal U}\to {\mathcal Y} $, where  ${\mathcal U}$ is an open set of $ {\mathcal X}$  and ${\mathcal X}\eqdefa C^\infty([0,\infty[\times\R^3;\R^3)$,
the differentiation  of $f$ at $Y\in {\mathcal U}$ along the direction $X\in{\mathcal X}$ is defined as
$$f'(Y;X)\eqdefa\lim_{s\to0}\frac{f(Y+sX)-f(Y)}{s}=\frac{d}{ds}f(Y+sX)|_{s=0}.$$
For $f\in C^\infty([0,+\infty)\times \R^3;M_{3\times3}(\R))$, $g\in C^\infty([0,+\infty)\times \R^3;\R^3)$, we have
\begin{align*}
&(fg)'(Y;X)=f'(Y;X)g(Y)+f(Y)g'(Y;X).
\end{align*}
Then for ${\mathcal A}(Y)=(Id+\nabla Y)^{-1}$, we have
\beq\label{S7.1eq1}
\begin{split}
{\mathcal A}'(Y;X)={\mathcal A}(-\nabla X){\mathcal A},\andf
({\mathcal A}^t)'(Y;X)={\mathcal A}^t(-\nabla X)^t{\mathcal A}^t
\end{split}\eeq
and thus
\begin{align} \label{S7.1eq3}
({\mathcal A}{\mathcal A}^t-Id)'(Y;X)&={\mathcal A}(-\nabla X){\mathcal A}{\mathcal A}^t+{\mathcal A}{\mathcal A}^t(-\nabla X)^t{\mathcal A}^t.
\end{align}
As a result, we deduce that
\beq\label{f0'}
\begin{split}
f_0'(Y;X)&=\nabla\cdot\Big(({\mathcal A}{\mathcal A}^t-Id)'(Y;X)\nabla Y_t\Big)+\nabla\cdot\big(({\mathcal A}{\mathcal A}^t-Id)\nabla X_t\big)\\
&=\nabla \cdot\Big(\big({\mathcal A}(-\nabla X){\mathcal A}{\mathcal A}^t+{\mathcal A}{\mathcal A}^t(-\nabla X)^t{\mathcal A}^t\big)\nabla Y_t\Big)+\nabla\cdot\big(({\mathcal A}{\mathcal A}^t-Id)\nabla X_t\big).
\end{split} \eeq
For $m=1,2$, we have
\beq \label{fm'}
\begin{split}
f_m'(Y;X)&=({\mathcal A}^t)'(Y;X)\nabla \vv p_m(Y)+{\mathcal A}^t\nabla \vv p_m'(Y;X)\\
&=-\cA^t(\na X)^t\cA^t(\na\vv p_m)(Y)+{\mathcal A}^t\nabla \vv p_m'(Y;X).
\end{split} \eeq  Moreover, it follows from \eqref{p1} that
\begin{equation}\label{p1'}
\begin{split}
 \vv p_1'(Y;X)
=&\Delta^{-1}\dv\Big(-({\mathcal A}{\mathcal A}^t-Id)\nabla \vv p_1'(Y;X)-\big({\mathcal A}(-\nabla X){\mathcal A}{\mathcal A}^t+{\mathcal A}{\mathcal A}^t(-\nabla X){\mathcal A}^t\big)\nabla \vv p_1(Y)\\
&\ \ +{\mathcal A}\dv\big(({\mathcal A}(-\nabla X){\mathcal A})(\p_{3}Y\otimes \p_{3}Y)\big)+({\mathcal A}(-\nabla X){\mathcal A})\dv\big({\mathcal A}(\p_{3}Y\otimes \p_{3}Y)\big)\\
&\ \ +{\mathcal A}\dv\big({\mathcal A}(\p_{3}Y\otimes\p_{3}X+\p_{3}X\otimes \p_{3}Y)\big)\Big).
\end{split}
\end{equation}
And similarly, it follows from \eqref{p2} that
\begin{equation}\label{p2'}
\begin{split}
 \vv p_2'(Y;X)
=\Delta^{-1}\dv\Big(&-({\mathcal A}{\mathcal A}^t-Id)\nabla \vv p_2'(Y;X)-\big({\mathcal A}(-\nabla X){\mathcal A}{\mathcal A}^t+{\mathcal A}{\mathcal A}^t(-\nabla X){\mathcal A}^t\big)\nabla \vv p_2(Y)\\
&+{\mathcal A}\dv\big(({\mathcal A}(-\nabla X){\mathcal A})(Y_t\otimes Y_t)\big)+({\mathcal A}(-\nabla X){\mathcal A})\dv\big({\mathcal A}(Y_t\otimes Y_t)\big)\\
&+
{\mathcal A}\dv\big({\mathcal A}(Y_t\otimes X_t+X_t\otimes Y_t)\big)\Big).
\end{split}
\end{equation}
The linearized equation  of (\ref{S1eq1}-\ref{S1eq2}) then reads as \eqref{LE1}.

\begin{rmk}\label{S3rmk1}
Let $V\in C^\infty([0,+\infty)\times\R^3;M_{3\times3}(\R))$ and  $U\in C^\infty([0,+\infty)\times\R^3;\R^3)$, we denote
 $h(V)\eqdefa(Id+V)^{-1}$,  and
 \beno
 \begin{split}
 F_0(U,&V)\eqdefa\nabla\cdot \Big(\big(h(V)h(V)^t-Id\big)U\Big),\qquad F(U,V)\eqdefa h(V)^t\vv q(U,V) \with\\
\vv q\eqdefa &-\Delta^{-1}\dv\big((h(V)h(V)^t-Id)\nabla \vv  q\big)+\Delta^{-1}\dv\Big(h(V)\dv\big(h(V)(U\otimes U)\big)\Big).\end{split}
\eeno
Then $f_0,f_1,f_2$ defined by \eqref{fm} can be written as
\begin{equation*}
\begin{split}
f_0=F_0(\nabla Y_t,\nabla Y),\quad
f_1=F(\p_3Y,\nabla Y) \andf  f_2=F(Y_t,\nabla Y),
\end{split}
 \end{equation*}
and hence $f_0',f_1'$ and $f_2'$ read
\begin{equation*}
\begin{split}
&f_0'(Y;X)=F'_{0,U}(\nabla Y_t,\nabla Y)\nabla X_t+F'_{0,V}(\nabla Y_t,\nabla Y)\nabla X,\\
&f_1'(Y;X)=F'_U(\p_3Y,\nabla Y)\p_3X+F'_V(\p_3Y,\nabla Y)\nabla X,\\
&f_2'(Y;X)=F'_U(Y_t,\nabla Y)X_t+F'_V(Y_t,\nabla Y)\nabla X,
\end{split}
\end{equation*}
where the functionals $F'_{0,U}(U,V),$ $
F'_{0,V}(U,V),$
$F'_U(U,V)$ and $
F'_V(U,V)$ are given by
\begin{equation*}
\begin{split}
F'_{0,U}(U,V)\dot U&=\nabla\cdot \Big(\big(h(V)h(V)^t-Id\big)\dot U\Big),\\
F'_{0,V}(U,V)\dot V&=\nabla\cdot \left(\big((h'(V)\dot V)h(V)^t+h(V)(h'(V)\dot V)^t\big) U\right),\\
F'_U(U,V)\dot U&=h(V)^t\vv q'_U(U,V)\dot U\andf F'_V(U,V)\dot V=\big(h'(V)\dot V\big)^t\vv q(U,V)+h(V)^t\vv q'_V(U,V)\dot V,
\end{split}
\end{equation*}
and
\begin{equation*}
\begin{split}
h'(V)\dot V&=(Id+V)^{-1}(-\dot V)(Id+V)^{-1},\\
\vv q'_U(U,V)\dot U&=-\Delta^{-1}\dv\Big((h(V)h(V)^t-Id)\nabla \vv q'_U(U,V)\dot U-h(V)\dv\big(h(V)(U\otimes\dot U+\dot U\otimes U)\big)\Big)\\
\vv q'_V(U,V)\dot V&=-\Delta^{-1}\dv\Big((h(V)h(V)^t-Id)\nabla \vv q'_V(U,V)\dot V-(h'(V)\dot V)\dv\big(h(V)(U\otimes U)\big)
\\
&\qquad-h(V)\dv\big((h'(V)\dot V)(U\otimes U)\big)+\big((h'(V)\dot V) h(V)^t+h(V)(h'(V)\dot V)^t\big)\nabla \vv q\Big).
\end{split}
\end{equation*}
\end{rmk}

\subsection{Computation of $f''(Y;X,W)$}\label{Sect4.2}
 In order to estimate the error arisen in the Nash-Moser iteration scheme, we need the second derivatives of $f.$ Toward this, let us recall the product rule
 \beq\label{S7.1eq0}
\begin{split}
(fg)''(Y;X,W)&=f''(Y;X,W)g(Y)+f(Y)g''(Y;X,W)\\
&\quad +f'(Y;X)g'(Y;W)+f'(Y;W)g'(Y;X).
\end{split} \eeq
It is easy to observe from \eqref{S7.1eq1} that
\begin{equation}\label{S7.1eq2}
{\mathcal A}''(Y;X,W)
={\mathcal A}(\nabla X){\mathcal A}(\nabla W){\mathcal A}+{\mathcal A}(\nabla W){\mathcal A}(\nabla X){\mathcal A}.
\end{equation}
Then applying the product rule \eqref{S7.1eq0} and \eqref{S7.1eq1} gives
\begin{equation}\label{S7.1eq4}
\begin{split}
({\mathcal A}{\mathcal A}^t-Id)''(Y;X,W)&={\mathcal A}(\nabla X){\mathcal A}(\nabla W){\mathcal A}{\mathcal A}^t+{\mathcal A}(\nabla W){\mathcal A}(\nabla X){\mathcal A}{\mathcal A}^t\\
&\quad+{\mathcal A}{\mathcal A}^t(\nabla X)^t{\mathcal A}^t(\nabla W)^t{\mathcal A}^t+{\mathcal A}{\mathcal A}^t(\nabla W)^t{\mathcal A}^t(\nabla X)^t{\mathcal A}^t\\
&\quad+{\mathcal A}(\nabla X){\mathcal A}{\mathcal A}^t(\nabla W)^t{\mathcal A}^t+{\mathcal A}(\nabla W){\mathcal A}{\mathcal A}^t(\nabla X)^t{\mathcal A}^t.
\end{split}
\end{equation}
Recall that $f_0$ is given by \eqref{fm}, we deduce from \eqref{S7.1eq0} that
\begin{equation}\label{S7.1eq5}
\begin{split}
f_0''(Y;X,W)=&\nabla\cdot\Big(({\mathcal A}{\mathcal A}^t-Id)''(Y;X,W)\nabla Y_t\Big)\\
&+\nabla\cdot\big(({\mathcal A}{\mathcal A}^t-Id)'(Y;X)\nabla W_t\big)+\nabla\cdot\big(({\mathcal A}{\mathcal A}^t-Id)'(Y;W)\nabla X_t\big).
\end{split}
\end{equation}
Similarly for $f_m(Y)={\mathcal A}^t\nabla\vv p_m$, $m=1,2$, we have
\begin{equation}\label{S7.1eq6}
\begin{split}
f_m''(Y;X,W)&=({\mathcal A}^t)''(Y;X,W)\nabla\vv p_m(Y)+{\mathcal A}^t\nabla \big(\vv p_m''(Y;X,W)\big)\\
&\quad+({\mathcal A}^t)'(Y;X)\nabla\big(\vv p_m'(Y;W)\big)+({\mathcal A}^t)'(Y;W)\nabla\big(\vv p_m'(Y;X)\big).
\end{split}
\end{equation}
Then in view of \eqref{p1'}, \eqref{p2'}, to obtain the expression of $f_m''(Y;X,W), m=1,2,$ it remains to calculate  $\vv p_m''(Y;X,W),$ $m=1,2.$  Indeed, it follows from \eqref{p1}, \eqref{p2} and \eqref{S7.1eq0} that
\begin{equation}\label{p1''}
\begin{split}
&\vv p_1''(Y;X,W)=-\Delta^{-1}\dv\Big(({\mathcal A}{\mathcal A}^t-Id)\nabla \vv p_1''(Y;X,W)+({\mathcal A}{\mathcal A}^t-Id)''(Y;X,W)\nabla \vv p_1(Y)\big)\\
&\quad+({\mathcal A}{\mathcal A}^t-Id)'(Y;X)\nabla \vv p_1'(Y;W)+({\mathcal A}{\mathcal A}^t-Id)'(Y;W)\nabla \vv p_1'(Y;X)\\
&\quad-{\mathcal A}\dv\big[{\mathcal A}(\p_{3}X\otimes\p_{3}W+\p_3W\otimes\p_3X) +{\mathcal A}''(Y;X,W)(\p_{3}Y\otimes \p_{3}Y)\\
&\quad+{\mathcal A}'(Y;X)(\p_{3}Y\otimes \p_{3}W+\p_{3}W\otimes\p_{3}Y)\big)+{\mathcal A}'(Y;W)(\p_{3}Y\otimes \p_{3}X+\p_{3}X\otimes\p_{3}Y)\bigr]\\
&\quad -{\mathcal A}'(Y;X)\dv\bigl[{\mathcal A}'(Y;W)(\p_{3}Y\otimes \p_{3}Y)+{\mathcal A}(\p_{3}Y\otimes \p_{3}W+\p_{3}W\otimes\p_{3}Y)\bigr]\\
& \quad -{\mathcal A}'(Y;W)\dv\bigl[{\mathcal A}'(Y;X)(\p_{3}Y\otimes \p_{3}Y)+{\mathcal A}(\p_{3}Y\otimes \p_{3}X+\p_{3}X\otimes\p_{3}Y)\bigr]\\
&\qquad\qquad\qquad\qquad\qquad\qquad\qquad\qquad\qquad  -{\mathcal A}''(Y;X,W)\dv\big({\mathcal A}(\p_{3}Y\otimes \p_{3}Y)\big)\Big).
\end{split}
\end{equation}
and
\begin{equation}\label{p2''}
\begin{split}
\vv p_2''(Y;X,W)=&-\Delta^{-1}\dv\Bigl(({\mathcal A}{\mathcal A}^t-Id)\nabla \vv p_2''(Y;X,W)+({\mathcal A}{\mathcal A}^t-Id)''(Y;X,W)\nabla \vv p_2(Y)\\
&+({\mathcal A}{\mathcal A}^t-Id)'(Y;X)\nabla \vv p_2'(Y;W)+({\mathcal A}{\mathcal A}^t-Id)'(Y;W)\nabla \vv p_2'(Y;X)\\
& -{\mathcal A}\dv\big[{\mathcal A}(X_t\otimes W_t+W_t\otimes X_t)+{\mathcal A}''(Y;X,W)(Y_t\otimes Y_t)\\
& +{\mathcal A}'(Y;X)(Y_t\otimes W_t+W_t\otimes Y_t)+{\mathcal A}'(Y;W)(Y_t\otimes X_t+X_t\otimes Y_t)\big]\\
& -{\mathcal A}'(Y;X)\dv\big[{\mathcal A}(Y_t\otimes W_t+W_t\otimes Y_t)+{\mathcal A}'(Y;W)(Y_t\otimes Y_t)\big]\\
& -{\mathcal A}'(Y;W)\dv\big[{\mathcal A}(Y_t\otimes X_t+X_t\otimes Y_t)+{\mathcal A}'(Y;X)(Y_t\otimes Y_t)\big]\\
&\qquad\qquad\qquad\qquad\qquad\qquad\quad - {\mathcal A}''(Y;X,W)\dv\big({\mathcal A}(Y_t\otimes Y_t)\Bigr).
\end{split}
\end{equation}

\begin{rmk}\label{S7rmk1}
In view of Remark \ref{S3rmk1}, $f_m''$ can be written as follows
\begin{equation*}
\begin{split}
f_0''(Y;X,W)&=F''_{0,UU}(\nabla Y_t,\nabla Y)\nabla X_t\cdot\nabla W_t+F''_{0,UV}(\nabla Y_t,\nabla Y)\nabla X_t\cdot \nabla W\\
&\qquad+F''_{0,VU}(\nabla Y_t,\nabla Y)\nabla W_t\cdot \nabla X+F''_{0,VV}(\nabla Y_t,\nabla Y)\nabla X\cdot \nabla W,\\
f_1''(Y;X,W)&=F''_{UU}(\p_3Y,\nabla Y)\p_3X\cdot\p_3W+F''_{UV}(\p_3Y,\nabla Y)\p_3 X\cdot \nabla W\\
&\qquad\qquad+F''_{VU}(\p_3Y,\nabla Y)\nabla X\cdot\p_3W+F''_{VV}(\p_3Y,\nabla Y)\nabla X\cdot\nabla W,\\
f_2''(Y;X,W)&=F''_{UU}(Y_t,\nabla Y)X_t\cdot W_t+F''_{UV}(Y_t,\nabla Y) X_t\cdot\nabla W\\
&\qquad\qquad+F''_{VU}(Y_t,\nabla Y)\nabla X\cdot W_t+F''_{VV}(Y_t,\nabla Y)\nabla X\cdot\nabla W,
\end{split}
\end{equation*}
where $F''_{0,UU}(U,V)\dot U_1\cdot \dot U_2=0$,
\begin{equation*}
\begin{split}
F''_{0,UV}(U,V)\dot U\cdot \dot V&=\nabla\cdot \Big(\big((h'(V)\dot V)h(V)^t+h(V)(h'(V)\dot V)^t\big) \dot U\Big)=F''_{0,VU}(U,V)\dot V\cdot \dot U,\\
F''_{0,VV}(U,V)\dot V_1\cdot \dot V_2&=\nabla\cdot \Big(\big((h''(V)\dot V_1\cdot \dot V_2)h(V)^t+h(V)(h''(V)\dot V_1\cdot \dot V_2)^t+\\
&\quad+\big((h'(V)\dot V_1)(h'(V)\dot V_2)^t+(h'(V)\dot V_2)(h'(V)\dot V_1)^t\big) U\Big) \with\\
 h''(V)\dot V_1\cdot \dot V_2&=(Id+V)^{-1}(-\dot V_1)(Id+V)^{-1}(-\dot V_2)(Id+V)^{-1}\\
&\quad+(Id+V)^{-1}(-\dot V_2)(Id+V)^{-1}(-\dot V_1)(Id+V)^{-1};
\end{split}
\end{equation*}
and
\begin{equation*}
\begin{split}
F''_{UU}(U,V)\dot U_1\cdot\dot U_2&=h(V)^t\vv q''_{UU}(U,V)\dot U_1\cdot\dot U_2,\\
F''_{UV}(U,V)\dot U\cdot\dot V&=h(V)^t\vv q''_{UV}(U,V)\dot U\cdot\dot V+(h'(V)\dot V)^t\vv q'_U(U,V)\dot U=F''_{VU}(U,V)\dot V\cdot\dot U,\\
F''_{VV}(U,V)\dot V_1\cdot\dot V_2&=\big(h''(V)\dot V_1\cdot\dot V_2\big)^t\vv q(U,V)+h(V)^t\vv q''_{VV}(U,V)\dot V_1\cdot \dot V_2\\
&\qquad+(h'(V)\dot V_1)^t\vv q'_V(U,V)\dot V_2+(h'(V)\dot V_2)^t\vv q_V'(U,V)\dot V_1,
\end{split}
\end{equation*}
where
\begin{equation*}
\begin{split}
 \vv q''_{UU}(U,V)\dot U_1\cdot\dot U_2&=-\Delta^{-1}\dv\Big((h(V)h(V)^t-Id)\nabla \vv q''_{UU}(U,V)\dot U_1\cdot \dot U_2\\
 &\qquad\qquad\quad-
 h(V)\dv\big(h(V)(\dot U_1\otimes\dot U_2+\dot U_2\otimes \dot U_1)\big)\Big),
 \end{split}
\end{equation*}
\begin{equation*}
\begin{split}
&\vv q''_{UV}(U,V)\dot U\cdot\dot V=-\Delta^{-1}\dv\Big((h(V)h(V)^t-Id)\nabla \vv q''_{UV}(U,V)\dot U\cdot\dot V\\
&\quad+\big((h'(V)\dot V) h(V)^t+h(V)(h'(V)\dot V)^t\big)\nabla \vv q'_U(U,V)\dot U-(h'(V)\dot V)\dv\big(h(V)(\dot U\otimes U+U\otimes\dot U)\big)\\
&\quad-h(V)\dv\big((h'(V)\dot V)(\dot U\otimes U+U\otimes\dot U)\big)\Big)=\vv q''_{VU}(U,V)\dot V\cdot\dot U,
\end{split}
\end{equation*}
and
\begin{equation*}
\begin{split}
&\vv q''_{VV}(U,V)\dot V_1\cdot\dot V_2
=-\Delta^{-1}\dv\Big(\big((h'(V)\dot V_2)h(V)^t+h(V)(h'(V)\dot V_2)^t\big)\nabla \vv q'_V(U,V)\dot V_1\\
&\quad+(h(V)h(V)^t-Id)\nabla \vv q''_{VV}(U,V)\dot V_1\cdot\dot V_2-(h''(V)\dot V_1\cdot \dot V_2)\dv\big(h(V)(U\otimes U)\big)\\
&\quad+\big((h'(V)\dot V_1) h(V)^t+h(V)(h'(V)\dot V_1)^t\big)\nabla \vv q_V'(U,V)\dot V_2-(h'(V)\dot V_1)\dv\big((h'(V)\dot V_2)(U\otimes U)\big)\\
&\quad+\big((h''(V)\dot V_1\cdot\dot V_2) h(V)^t+h(V)(h''(V)\dot V_1\cdot\dot V_2)^t-(h'(V)\dot V_2)\dv\big((h'(V)\dot V_1)(U\otimes U)\big)\\
&\quad+(h'(V)\dot V_1) (h'(V)\dot V_2)^t+(h'(V)\dot V_2)(h'(V)\dot V_1)^t\big)\nabla \vv q-h(V)\dv\big((h''(V)\dot V_1\cdot\dot V_2)(U\otimes U)\big)\Big).
\end{split}
\end{equation*}
\end{rmk}

\section{The estimates of  $f'(Y;X)$ } \label{Sect5}

\subsection{The estimate of $\triplenorm{f'(Y;X)}_{\d,N}$}\label{Sect5.1} The main result of this subsection is listed in Proposition \ref{S4.1prop1}.
As we explained in the Section \ref{Sect2}, the main idea is to use the norm of the homogeneous Besov spaces $\dot B^{s}_{1,1}$ to replace
the norm of the classical Sobolev spaces $\dot W^{s,1}.$ In order to do so, we  need not only the product law \eqref{product1} but also the following one.

\begin{lem}
For any $s>0,$ there holds
\beq \label{product3}
\|ab\|_{\dB^s_{1,1}}\leq C\left(\min\bigl(|a|_0\|b\|_{\dB^s_{1,1}}, \|a\|_0\|b\|_{\dB^s_{2,1}}\bigr)+\|a\|_{\dB^s_{2,1}}\|b\|_0\right).
\eeq
\end{lem}

\begin{proof} We first get, by applying Bony's decomposition \cite{Bo} that
\beno
\begin{split}
ab&=T_ab+R'(a,b)\with\\
 T_ab=\sum_{j\in\Z}S_{j-1}a\D_jb&\andf R'(a,b)=\sum_{j\in\Z}\D_{j}a S_{j+2}b.
 \end{split}
\eeno
Due to the support properties to the Fourier transform of the terms in $T_ab,$ we have
\beno
\begin{split}
\|\dD_j(T_ab)\|_{L^1}\leq &\sum_{|j'-j|\leq 4}|S_{j'-1}a|_0\|\dD_{j'}b\|_{L^1}
\lesssim d_j2^{-js}|a|_0\|b\|_{\dB^s_{1,1}},
\end{split}
\eeno
where $(d_j)_{j\in\Z}$ is a non-negative generic element of $\ell^1(\Z)$ so that $\sum_{j\in \Z}d_j=1.$

Along the same line, we also have
\beno
\begin{split}
\|\dD_j(T_ab)\|_{L^1}\leq &\sum_{|j'-j|\leq 4}\|S_{j'-1}a\|_0\|\dD_{j'}b\|_0
\lesssim d_j2^{-js}\|a\|_0\|b\|_{\dB^s_{2,1}},
\end{split}
\eeno
and
\beno
\begin{split}
\|\dD_j(R'(a,b))\|_{L^1}\leq &\sum_{j'\geq j-N_0}\|\dD_{j'}a\|_0\|S_{j'+2}b\|_0\\
\leq &\sum_{j'\geq j-N_0}d_{j'}2^{-j's}\|a\|_{\dB^s_{2,1}}\|b\|_0
\lesssim d_j2^{-js}\|a\|_{\dB^s_{2,1}}\|b\|_0,
\end{split}
\eeno
where in the last step, we used the fact that $s>0.$
By summing up the above inequalities, we arrive at \eqref{product3}.
\end{proof}

Notice that  ${\mathcal A}(\na Y)=(Id+\nabla Y)^{-1}$, we write
$${\mathcal A}{\mathcal A}^t-Id=({\mathcal A}-Id)({\mathcal A}-Id)^t+({\mathcal A}-Id)+({\mathcal A}-Id)^t,\qquad
{\mathcal A}-Id=\sum_{n=1}^\infty(-1)^{n}(\nabla Y)^n.$$
So that under the assumption of \eqref{S4.1prop1assum},  for $s>0$,  we get, by applying
 \eqref{product1}, that
 \beq\label{product4}
 \begin{split}
\|{\mathcal A}f\|_{\dot B^s_{p,r}}\lesssim &(1+|{\mathcal A}-Id|_0)\|f\|_{\dot B^s_{p,r}}+\|{\mathcal A}-Id\|_{\dot B^s_{p,r}}|f|_0\\
\lesssim &\|f\|_{\dot B^s_{p,r}}+\|\nabla Y\|_{\dot B^s_{p,r}}|f|_0,
\end{split}\eeq
Along the same line, we get, by applying
 \eqref{product3}, that
\begin{equation}\label{product4a}
\|{\mathcal A}f\|_{\dot B^s_{1,1}}\lesssim \|f\|_{\dot B^s_{1,1}}+\|\nabla Y\|_{\dot B^s_{2,1}}\|f\|_0,
\end{equation}

\subsubsection{Estimate of $\|f_0'(Y;X)\|_{\dot B^s_{1,1}}$}
In view of \eqref{f0'}, we have
\begin{align*}
\|f_0'(Y;X)\|_{\dot B^{s}_{1,1}}
&\leq\bigl\|{\mathcal A}\bigl(\nabla X{\mathcal A}+{\mathcal A}^t(\nabla X)^t\bigr){\mathcal A}^t\nabla Y_t\bigr\|_{\dot B^{s+1}_{1,1}}+\|({\mathcal A}{\mathcal A}^t-Id)\nabla X_t\|_{\dot B^{s+1}_{1,1}}.
\end{align*}
It follows from \eqref{S4.1prop1assum} and \eqref{product3} that
\begin{align*}
\|({\mathcal A}{\mathcal A}^t-Id)\nabla X_t\|_{\dot B^{s+1}_{1,1}}
&\lesssim \|\nabla Y\|_0\|\nabla X_t\|_{\dot B^{s+1}_{2,1}}+\|\nabla Y\|_{\dot B^{s+1}_{2,1}}\|\nabla X_t\|_0.
\end{align*}
While applying \eqref{product4a} gives
\beno
\|{\mathcal A}\nabla X{\mathcal A}{\mathcal A}^t\nabla Y_t\|_{\dot B^{s+1}_{1,1}}
\lesssim \|\nabla X{\mathcal A}{\mathcal A}^t\nabla Y_t\|_{\dot B^{s+1}_{1,1}}+\|\nabla Y\|_{\dot B^{s+1}_{2,1}}\|\nabla X{\mathcal A}{\mathcal A}^t\nabla Y_t\|_0, \eeno
yet it follows from  \eqref{product1} and  \eqref{product3}  that
\begin{align*}
\|\nabla X{\mathcal A}{\mathcal A}^t\nabla Y_t\|_{\dot B^{s+1}_{1,1}}
\lesssim &\|\nabla X\|_0\|{\mathcal A}{\mathcal A}^t\nabla Y_t\|_{\dot B^{s+1}_{2,1}}+\|\nabla X\|_{\dot B^{s+1}_{2,1}}\|{\mathcal A}{\mathcal A}^t\nabla Y_t\|_0
\\
\lesssim& \|\nabla X\|_0\big(\|\nabla Y_t\|_{\dot B^{s+1}_{2,1}}+\|\nabla Y\|_{\dot B^{s+1}_{2,1}}|\nabla Y_t|_0\big)
+\|\nabla X\|_{\dot B^{s+1}_{2,1}}\|\nabla Y_t\|_0,
\end{align*}
so that there holds
\beno  \|{\mathcal A}\nabla X{\mathcal A}{\mathcal A}^t\nabla Y_t\|_{\dot B^{s+1}_{1,1}} \lesssim \|\nabla Y_t\|_0\|\nabla X\|_{\dot B^{s+1}_{2,1}}+\big(\|\nabla Y_t\|_{\dot B^{s+1}_{2,1}}+|\nabla Y_t|_0\|\nabla Y\|_{\dot B^{s+1}_{2,1}}\big)\|\nabla X\|_0.
\eeno
The same estimate holds for $\|{\mathcal A}{\mathcal A}^t(-\nabla X)^t{\mathcal A}^t\nabla Y_t\|_{\dot B^{s+1}_{1,1}}$.
As a result, we obtain
\begin{equation}\label{S4.1eq1}
\begin{split}
\|f_0'(Y;X)\|_{\dot B^s_{1,1}}
\leq& \|\nabla Y\|_0\|\nabla X_t\|_{\dot B^{s+1}_{2,1}}+\|\nabla Y\|_{\dot B^{s+1}_{2,1}}\|\nabla X_t\|_0\\
&+\|\nabla Y_t\|_0\|\nabla X\|_{\dot B^{s+1}_{2,1}}+\big(\|\nabla Y_t\|_{\dot B^{s+1}_{2,1}}+\|\nabla Y\|_{\dot B^{s+1}_{2,1}}|\nabla Y_t|_0\big)\|\nabla X\|_0.
\end{split}
\end{equation}

\subsubsection{Estimate of $\|f_m'(Y;X)\|_{\dot B^s_{1,1}}$, $m=1,2$}
In view of \eqref{fm'},
we have
\begin{align*}
&\|f_m'(Y;X)\|_{\dot B^s_{1,1}}
\leq \|{\mathcal A}^t(\nabla X)^t{\mathcal A}^t\nabla \vv p_m\|_{\dot B^{s}_{1,1}}+\|{\mathcal A}^t\nabla(\vv p_m'(Y;X))\|_{\dot B^s_{1,1}}.
\end{align*}
Applying \eqref{product4a} gives
\beno
&&\|{\mathcal A}^t(\nabla X)^t{\mathcal A}^t\nabla \vv p_m\|_{\dot B^s_{1,1}}
\lesssim \|(\nabla X)^t{\mathcal A}^t\nabla \vv p_m\|_{\dot B^s_{1,1}}+\|\nabla Y\|_{\dot B^s_{2,1}}\|(\nabla X)^t{\mathcal A}^t\nabla \vv p_m\|_0,\\
&&\|{\mathcal A}^t\nabla(\vv p_m'(Y;X))\|_{\dot B^s_{1,1}}\lesssim \|\nabla\big (\vv p_m'(Y;X)\big)\|_{\dot B^s_{1,1}}+\|\nabla Y\|_{\dot B^s_{2,1}}\|\nabla \big(\vv p_m'(Y;X)\big)\|_0.
\eeno
While applying  \eqref{product1} and \eqref{product3} leads to
\begin{align*}
\|(\nabla X)^t{\mathcal A}^t\nabla \vv p_m\|_{\dot B^s_{1,1}}
&\lesssim \|\nabla X\|_0(\|\nabla \vv p_m\|_{\dot B^s_{2,1}}+\|\nabla Y\|_{\dot B^s_{2,1}}|\nabla\vv p_m|_0)+\|\nabla X\|_{\dot B^s_{2,1}}\|\nabla \vv p_m\|_0,
\end{align*}
which yields
\begin{align*}
\|{\mathcal A}^t(\nabla X)^t{\mathcal A}^t\nabla \vv p_m\|_{\dot B^s_{1,1}}
\lesssim\|\nabla\vv p_m\|_0\|\nabla X\|_{\dot B^s_{2,1}}+ \big(\|\nabla \vv p_m\|_{\dot B^{s}_{2,1}}+\|\nabla Y\|_{\dot B^s_{2,1}}|\nabla\vv p_m|_0\big)\|\nabla X\|_0.
\end{align*}
Hence it comes out
\begin{align}
\notag\|f_m'(Y;X)\|_{\dot B^s_{1,1}}&\lesssim\|\nabla\vv p_m\|_0\|\nabla X\|_{\dot B^s_{2,1}}+\big(\|\nabla \vv p_m\|_{\dot B^{s}_{2,1}}+\|\nabla Y\|_{\dot B^s_{2,1}}|\nabla\vv p_m|_0\big)\|\nabla X\|_0\\
\label{S4.1eq2}&\qquad+\|\nabla\big (\vv p_m'(Y;X)\big)\|_{\dot B^s_{1,1}}+\|\nabla Y\|_{\dot B^s_{2,1}}\|\nabla \big(\vv p_m'(Y;X)\big)\|_0.
\end{align}
It remains to handle the  estimates of
$$\|\nabla\vv p_m\|_0,\quad \|\nabla \vv p_m\|_{\dot B^s_{2,1}},\quad \|\nabla\big(\vv p_m'(Y;X)\big)\|_{\dot B^s_{1,1}}\quad\text{and}\quad \|\nabla\big(\vv p_m'(Y;X)\big)\|_0.$$

\noindent$\bullet${\underline{Estimate of $\|\nabla\vv p_m\|_0.$}}

We first deduce from \eqref{p1} that
\begin{align*}
\|\nabla \vv p_1\|_0
&\leq |{\mathcal A}{\mathcal A}^t-Id|_0\|\nabla \vv p_1\|_0+|{\mathcal A}|_0\|{\mathcal A}(\p_3Y\otimes\p_3Y)\|_{\dot H^1}\\
&\leq |{\mathcal A}{\mathcal A}^t-Id|_0\|\nabla \vv p_1\|_0+|{\mathcal A}|_0\bigl(1+\|\cA-Id\|_{\dB^{\f32}_{2,1}}\bigr)\|\p_3Y\otimes\p_3Y\|_{\dot H^1}.
\end{align*}
Due to the assumption \eqref{S4.1prop1assum}, one has
$$|{\mathcal A}{\mathcal A}^t-Id|_0\lesssim |\nabla Y|_0\lesssim \|\nabla Y\|_{\dot B^\frac32_{2,1}}\leq \delta_1,$$
so that we infer
\begin{align}
\label{S4.1eq3}\|\nabla \vv p_1\|_0\lesssim \|\p_3Y\otimes\p_3Y\|_{\dot H^1}\lesssim |\p_3Y|_0\|\p_3Y\|_{1}.
\end{align}
Similarly, we have
\begin{align}
\label{S4.1eq4}\|\nabla \vv p_2\|_0&\lesssim |Y_t|_0\|Y_t\|_{1}.
\end{align}

\noindent$\bullet${\underline{Estimates of $\|\nabla \vv p_m\|_{\dot B^s_{2,1}}$ for $s>0.$}}

We start with the estimate of  $\|\nabla \vv p_m\|_{\dot B^\frac32_{2,1}}.$ Indeed by \eqref{p1}, one has
\begin{align*}
\|\nabla \vv p_1\|_{\dot B^\frac32_{2,1}}
&\lesssim \|{\mathcal A}{\mathcal A}^t-Id\|_{\dot B^\frac32_{2,1}}\|\nabla\vv p_1\|_{\dot B^\frac32_{2,1}}+\|{\mathcal A}\dv\big({\mathcal A}(\p_3Y\otimes\p_3Y)\big)\|_{\dot B^\frac32_{2,1}},
\end{align*}
from which, \eqref{S4.1prop1assum} and the product law \eqref{product1}, we infer
\begin{align*}
\|\nabla \vv p_1\|_{\dot B^\frac32_{2,1}}
&\lesssim \bigl(1+\|{\mathcal A}-Id\|_{\dot B^\frac32_{2,1}}\bigr)\|{\mathcal A}(\p_3Y\otimes\p_3Y)\|_{\dot B^\frac52_{2,1}}\\
&\lesssim (1+|{\mathcal A}-Id|_0)\|\p_3Y\otimes\p_3Y\|_{\dot B^\frac52_{2,1}}+\|{\mathcal A}-Id\|_{\dot B^\frac52_{2,1}}|\p_3Y\otimes\p_3Y|_0.
\end{align*}
As a result, by virtue of \eqref{S4.1prop1assum}, it comes out
\begin{align}
 \|\nabla \vv p_1\|_{\dot B^\frac32_{2,1}}
\label{S4.1eq5}&\lesssim |\p_3Y|_0\big(\|\p_3Y\|_{\dot B^\frac52_{2,1}}+\|\nabla Y\|_{\dot B^\frac52_{2,1}}|\p_3Y|_0\big)\lesssim |\p_3Y|_0\|\p_3Y\|_{{3}}.
\end{align}

In general, for $s>0$, we deduce from \eqref{p1} that
\begin{align*}
\|\nabla \vv p_1\|_{\dot B^s_{2,1}}
&\lesssim |{\mathcal A}{\mathcal A}^t-Id|_0\|\nabla \vv p_1\|_{\dot B^s_{2,1}}+\|{\mathcal A}{\mathcal A}^t-Id\|_{\dot B^s_{2,1}}|\nabla \vv p_1|_0+\|{\mathcal A}\dv\big({\mathcal A}(\p_3Y\otimes\p_3Y)\big)\|_{\dot B^s_{2,1}},
\end{align*}
from which and \eqref{S4.1prop1assum}, we infer
\begin{align*}
\|\nabla \vv p_1\|_{\dot B^s_{2,1}}
&\lesssim \|{\mathcal A}{\mathcal A}^t-Id\|_{\dot B^s_{2,1}}|\nabla \vv p_1|_0+\|{\mathcal A}\dv\big({\mathcal A}(\p_3Y\otimes\p_3Y)\big)\|_{\dot B^s_{2,1}}.
\end{align*}
Yet it follows from the product law \eqref{product4} that
\begin{align*}
\|{\mathcal A}\dv\big({\mathcal A}(\p_3Y\otimes\p_3Y)\big)\|_{\dot B^s_{2,1}}
&\lesssim\|\p_3Y\otimes\p_3Y\|_{\dot B^{s+1}_{2,1}}+\|\nabla Y\|_{\dot B^{s+1}_{2,1}}|\p_3Y\otimes\p_3Y|_0\\
&\qquad +\|\nabla Y\|_{\dot B^s_{2,1}}\big(|\p_3Y|_{1}|\p_3Y|_0+|\nabla Y|_{1}|\p_3Y|_0^2\big),
\end{align*}
which together with  \eqref{S4.1prop1assum} and \eqref{S4.1eq5} ensures that
\begin{equation}\label{S4.1eq6}
\|\nabla \vv p_1\|_{\dot B^s_{2,1}}
\lesssim  |\p_3Y|_0\left(\|\p_3Y\|_{\dot B^{s+1}_{2,1}}+\bigl(\|\nabla Y\|_{\dot B^{s}_{2,1}}+\|\nabla Y\|_{\dot B^{s+1}_{2,1}}\bigr)\|\p_3Y\|_{3}\right).
\end{equation}

Exactly along the same line, we have
\begin{align}
& \|\nabla \vv p_2\|_{\dot B^\frac32_{2,1}}
\label{S4.1eq7}\lesssim  |Y_t|_0\|Y_t\|_{3} \andf\\
&\label{S4.1eq8}
\|\nabla \vv p_2\|_{\dot B^s_{2,1}}
\lesssim |Y_t|_0\left(\|Y_t\|_{\dot B^{s+1}_{2,1}}+\bigl(\|\nabla Y\|_{\dot B^{s}_{2,1}}+\|\nabla Y\|_{\dot B^{s+1}_{2,1}}\bigr)\|Y_t\|_{3}\right).
\end{align}


\noindent$\bullet${\underline{Estimate of $\|\nabla\vv p_m'(Y;X)\|_0.
$}}

We first deduce from \eqref{p1'} that
\beq\label{S4.1eq8a}
\begin{split}
\|\na\vv p_1'(Y;X)\|_0\lesssim &\d_1\|\na p_1'(Y;X)\|_0+\bigl\|{\mathcal A}\left(\nabla X{\mathcal A}
+{\mathcal A}^t\nabla X\right){\mathcal A}^t\nabla \vv p_1\bigr\|_0\\
&+\|{\mathcal A}\dv\big({\mathcal A}\nabla X{\mathcal A}(\p_{3}Y\otimes \p_{3}Y)\big)\|_0
+\|{\mathcal A}\nabla X{\mathcal A}\dv\big({\mathcal A}(\p_{3}Y\otimes \p_{3}Y)\big)\|_0\\
&+\|{\mathcal A}\dv\big({\mathcal A}(\p_{3}Y\otimes\p_{3}X+\p_{3}X\otimes \p_{3}Y)\big)\|_0.
\end{split}
\eeq
We observe that
\beno
\|{\mathcal A}\nabla X{\mathcal A}{\mathcal A}^t\nabla \vv p_1\|_0\lesssim \|\na X\|_{L^6}\|\na \vv p_1\|_{L^3}.
\eeno
Yet it follows by a similar derivation of \eqref{S4.1eq3} that
\beq \label{S4.1eq9}
\begin{split}
\|\na\vv p_1\|_{L^3}\lesssim &\|\cA(\p_3Y\otimes\p_3Y)\|_{W^{1,3}}
\lesssim |\p_3Y|_{1}\|\p_3Y\|_{L^3}\lesssim |\p_3Y|_{1}^{\f43}\|\p_3Y\|_0^{\f23},
\end{split}
\eeq
so that
\beno
\|{\mathcal A}\nabla X{\mathcal A}{\mathcal A}^t\nabla \vv p_1\|_0\leq \|\na X\|_{L^6}\|\na \vv p_1\|_{L^3}\lesssim |\p_3Y|_{1}^{\f43}\|\p_3Y\|_0^{\f23}\|\na X\|_{1}.
\eeno
Let us handle the remaining terms in \eqref{S4.1eq8a}.  Indeed with the assumption \eqref{S4.1prop1assum},
 a direct calculation shows that
\beno \begin{split}
&\|{\mathcal A}\dv\big({\mathcal A}\nabla X{\mathcal A}(\p_{3}Y\otimes \p_{3}Y)\big)\|_0
\lesssim\|\cA\na X\|_{1}|\cA(\p_3Y\otimes\p_3Y)|_{1}
\lesssim |\p_3Y|_{1}^2\|\na X\|_{1},\\
&
\|{\mathcal A}\nabla X{\mathcal A}\dv\big({\mathcal A}(\p_{3}Y\otimes \p_{3}Y)\big)\|_0\lesssim  \|\na X\|_0|{\mathcal A}(\p_{3}Y\otimes \p_{3}Y)|_{1}
\lesssim |\p_3Y|_{1}^2\|\na X\|_0,\\
&
\|{\mathcal A}\dv\big({\mathcal A}(\p_{3}Y\otimes\p_{3}X+\p_{3}X\otimes \p_{3}Y)\big)\|_0
\lesssim  \|\p_3Y\otimes\p_3X\|_{1}\lesssim |\p_3Y|_{1}\|\p_3X\|_{1}.
\end{split}
\eeno
Substituting the above estimates into \eqref{S4.1eq8a} leads to
\beq
\|\na\vv p_1'(Y;X)\|_0\lesssim \big(|\p_3Y|_{1}^{\f43}\|\p_3Y\|_0^\frac23+|\p_3Y|_{1}^2\big)\|\na X\|_{1}+|\p_3Y|_{1}\|\p_3X\|_{1}. \label{S4.1eq10}
\eeq
The same procedure gives rise to
\begin{align}
& \|\na\vv p_2\|_{L^3}\lesssim |Y_t|_{1}^{\f43}\|Y_t\|_0^\frac23;\andf \label{S4.1eq11}\\
&\|\na\vv p_2'(Y;X)\|_0\lesssim \big( |Y_t|_{1}^{\f43}\|Y_t\|_0^\frac23+|Y_t|_{1}^2\big)\|\na X\|_{1}+|Y_t|_{1}\|X_t\|_{1}. \label{S4.1eq12}
\end{align}

\noindent$\bullet${\underline{Estimate of $\|\nabla\vv p_m'(Y;X)\|_{\dot B^s_{1,1}}$ with $s>0$}}

For any $s>0,$ we deduce from \eqref{p1'} that
\beq\label{S4.1eq12a}
\begin{split}
\|\na\vv p_1'(Y;X)\|_{\dB^s_{1,1}}&\lesssim \big\|({\mathcal A}{\mathcal A}^t-Id)\nabla \vv p_1'(Y;X)\bigr\|_{\dB^s_{1,1}}
+ \bigl\|{\mathcal A}\big(\nabla X{\mathcal A}+{\mathcal A}^t\nabla X\big){\mathcal A}^t\nabla \vv p_1\bigr\|_{\dB^s_{1,1}}\\
&+\bigl\|{\mathcal A}\dv\big({\mathcal A}\nabla X{\mathcal A}(\p_{3}Y\otimes \p_{3}Y)\big)\bigr\|_{\dB^s_{1,1}} +\bigl\|{\mathcal A}\nabla X{\mathcal A}\dv\big({\mathcal A}(\p_{3}Y\otimes \p_{3}Y)\big)\bigr\|_{\dB^s_{1,1}}\\
 &+\bigl\|{\mathcal A}\dv\big({\mathcal A}(\p_{3}Y\otimes\p_{3}X+\p_{3}X\otimes \p_{3}Y)\big)\bigr\|_{\dB^s_{1,1}}.
\end{split}
\eeq
It follows from \eqref{product3} that
\beno
\begin{split}
\big\|({\mathcal A}{\mathcal A}^t-Id)\nabla \vv p_1'(Y;X)\bigr\|_{\dB^s_{1,1}}
\lesssim& \d_1 \|\na\vv p_1'(Y;X)\|_{\dB^s_{1,1}}+\|\na Y\|_{\dB^s_{2,1}}\|\na\vv p_1'(Y;X)\|_0.
\end{split}
\eeno
And applying \eqref{product4} and  \eqref{product3} gives
\beno
\begin{split}
\bigl\|{\mathcal A}\bigl(\nabla X{\mathcal A}+{\mathcal A}^t(\nabla X)^t\bigr){\mathcal A}^t\nabla \vv p_1\bigr\|_{\dB^s_{1,1}}
\lesssim &\|\nabla \vv p_1\|_0\|\nabla X\|_{\dot B^s_{2,1}}+\big(\|\nabla \vv p_1\|_{\dot B^s_{2,1}}+\|\nabla Y\|_{\dot B^s_{2,1}}|\nabla\vv p_1|_0\big)\|\nabla X\|_0,
\end{split}
\eeno
and
\begin{align*}
\bigl\|{\mathcal A}&\dv\big({\mathcal A}\nabla X{\mathcal A}(\p_{3}Y\otimes \p_{3}Y)\big)\bigr\|_{\dB^s_{1,1}}\\
&\lesssim \|\nabla X\|_{\dot B^{s+1}_{2,1}}\|{\mathcal A}(\p_{3}Y\otimes \p_{3}Y)\|_0+
\|\nabla X\|_0\|{\mathcal A}(\p_{3}Y\otimes \p_{3}Y)\|_{\dot B^{s+1}_{2,1}}\\
&\qquad+\|\nabla Y\|_{\dot B^{s+1}_{2,1}}\|\nabla X\|_0|\p_{3}Y|_0^2+\|\nabla Y\|_{\dot B^s_{2,1}}\|\nabla X\|_{1}|\p_{3}Y|_{1}|\p_{3}Y|_0\\
&\lesssim |\p_{3}Y|_0\| \p_{3}Y\|_0 \|\nabla X\|_{\dot B^{s+1}_{2,1}}+|\p_3Y|_0\Big(\| \p_{3}Y\|_{\dot B^{s+1}_{2,1}}+(\|\nabla Y\|_{\dot B^{s+1}_{2,1}}+\|\nabla Y\|_{\dot B^s_{2,1}})|\p_{3}Y|_{1}\Big)\|\nabla X\|_{1}.
\end{align*}
Exactly along the same line, we find
$$\longformule{
\bigl\|{\mathcal A}\nabla X{\mathcal A}\dv\big({\mathcal A}(\p_{3}Y\otimes \p_{3}Y)\big)\bigr\|_{\dB^s_{1,1}}
\lesssim |\p_3Y|_0\|\p_3Y\|_{\dot H^1}\|\nabla X\|_{\dot B^s_{2,1}}}{{}
+|\p_3Y|_0\Big(\| \p_{3}Y\|_{\dot B^{s+1}_{2,1}}+(\|\nabla Y\|_{\dot B^{s+1}_{2,1}}+\|\nabla Y\|_{\dot B^s_{2,1}})|\p_{3}Y|_{1}\Big)\|\nabla X\|_0,
} $$
and
\begin{align*}
\bigl\|{\mathcal A}\dv\big({\mathcal A}(\p_{3}Y\otimes\p_{3}X+\p_{3}X\otimes \p_{3}Y)\big)\bigr\|_{\dB^s_{1,1}}
&\lesssim \|\p_3Y\|_0\|\p_3X\|_{\dot B^{s+1}_{2,1}}+\|\p_3Y\|_{\dot B^{s+1}_{2,1}}\|\p_3X\|_0\\
&+\|\nabla Y\|_{\dot B^{s+1}_{2,1}}|\p_3Y|_0\|\p_3X\|_0+\|\nabla Y\|_{\dot B^{s}_{2,1}}|\p_3Y|_{1}\|\p_3X\|_{1}.
\end{align*}
Substituting the above estimates into \eqref{S4.1eq12a} and using the estimates \eqref{S4.1eq3}, \eqref{S4.1eq5}, \eqref{S4.1eq6} and \eqref{S4.1eq10},
we obtain
\begin{equation}\label{S4.1eq13}
\begin{split}
\|\nabla\vv p_1'(Y;X)\|_{\dot B^s_{1,1}}
\lesssim & \frak{g}_1(\p_3Y,\p_3X) \with\\
\frak{g}_1(\frak{x},\frak{y})\eqdefa & \|\frak{x}\|_0\|\frak{y}\|_{\dot B^{s+1}_{2,1}}+|\frak{x}|_0\big(\|\frak{x}\|_0\|\nabla X\|_{\dot B^{s+1}_{2,1}}
+\|\frak{x}\|_1\|\nabla X\|_{\dot B^s_{2,1}}\big)\\
& +\big(\|\frak{x}\|_{\dot B^{s+1}_{2,1}}+(\|\nabla Y\|_{\dot B^{s+1}_{2,1}}+\|\nabla Y\|_{\dot B^s_{2,1}})|\frak{x}|_{1}\big)\|\frak{y}\|_{1}\\
&+|\frak{x}|_{1}\big(\|\frak{x}\|_{\dot B^{s+1}_{2,1}}+(\|\nabla Y\|_{\dot B^{s+1}_{2,1}}+\|\nabla Y\|_{\dot B^s_{2,1}})\|\frak{x}\|_{3}\big)\|\nabla X\|_{1}.
\end{split}
\end{equation}
The same procedure gives rise to
\begin{equation}\label{S4.1eq14}
\begin{split}
\|\nabla\vv p_2'(Y;X)\|_{\dot B^s_{1,1}}
\lesssim & \frak{g}_1(Y_t,X_t). 
\end{split}
\end{equation}

Inserting the estimates \eqref{S4.1eq3}, \eqref{S4.1eq5}, \eqref{S4.1eq6}, \eqref{S4.1eq10} and \eqref{S4.1eq13} into \eqref{S4.1eq2} for $m=1$ yields
\begin{equation}\label{S4.1eq15}
\begin{split}
\|f_1'(Y;X)\big)\|_{\dot B^s_{1,1}}
\lesssim & \frak{g}_1(\p_3Y,\p_3X).
\end{split}
\end{equation}
While by inserting the estimates \eqref{S4.1eq4}, \eqref{S4.1eq7}, \eqref{S4.1eq8}, \eqref{S4.1eq11}, \eqref{S4.1eq12} and \eqref{S4.1eq14} into \eqref{S4.1eq2} for $m=2,$
we obtain
\begin{equation}\label{S4.1eq16}
\begin{split}
\|f_2'(Y;X)\big)\|_{\dot B^s_{1,1}}
\lesssim & \frak{g}_1(Y_t,X_t).
\end{split}
\end{equation}

Let us now complete the proof of Proposition \ref{S4.1prop1}.

\begin{proof}[Proof of Proposition \ref{S4.1prop1}]
Note that for $s_1<s<s_2$ and $\alpha=\frac{s_2-s}{s_2-s_1},$  one has
$$\|f\|_{\dot B^s_{2,1}}\leq C\left(\frac1{s-s_1}+\frac1{s_2-s}\right)\|f\|_{\dot H^{s_1}}^\alpha \|f\|_{\dot H^{s_2}}^{1-\alpha}.$$
In particular, for $s>0$, this yields
\begin{equation}\label{S4.1eq17}
\|f\|_{\dot B^{s}_{2,1}}\leq C\big(\|f\|_0+\|f\|_{\dot H^{[s]+1}}\big)\leq C\|f\|_{{[s]+1}}.
\end{equation}
On the other hand, recall \eqref{S2eq19''}, we deduce from  \eqref{S4.1eq1} that
\beno
\begin{split}
&\triplenorm{f_0'(Y;X)}_{\d,N}\lesssim \|\nabla Y\|_0\bigl(\|\nabla X_t\|_{\dot B^{2\d+1}_{2,1}}+\|\nabla X_t\|_{\dot B^{N+5}_{2,1}}\bigr)\\
&\qquad+\bigl(\|\nabla Y\|_{\dot B^{2\d+1}_{2,1}}+|\nabla Y\|_{\dot B^{N+5}_{2,1}}\bigr)\|\nabla X_t\|_0 +\|\nabla Y_t\|_0\bigl(\|\nabla X\|_{\dot B^{2\d+1}_{2,1}}+\|\nabla X\|_{\dot B^{N+5}_{2,1}}\bigr)\\
&\qquad+\big(\|\nabla Y_t\|_{\dot B^{2\d+1}_{2,1}}+\|\nabla Y_t\|_{\dot B^{N+5}_{2,1}}+(\|\nabla Y\|_{\dot B^{2\d+1}_{2,1}}
+\|\nabla Y\|_{\dot B^{N+5}_{2,1}})|\nabla Y_t|_0\big)\|\nabla X\|_0,
\end{split}
\eeno
which together with \eqref{S4.1eq17} ensures
\eqref{S4.1eqf0}.
Along the same line, we deduce \eqref{S4.1eqf1} and \eqref{S4.1eqf2} from  \eqref{S4.1eq15} and \eqref{S4.1eq16} respectively.
This completes the proof of Proposition \ref{S4.1prop1}. \end{proof}

\subsection{The  estimate of $\||D|^{-1}f'(Y;X)\|_{N}$}
 The purpose of this subsection is to prove Proposition  \ref{S4.2prop1}. We split its proof into the following steps:

\subsubsection{The  estimate of $\||D|^{-1}f_0'(Y;X)\|_{N}$}
We first deduce from \eqref{f0'} that
\begin{align}\label{S4.2eqf2a}
 \||D|^{-1}f_0'(Y;X)\|_{N}
\lesssim \|({\mathcal A}{\mathcal A}^t-Id)\nabla X_t\|_{N}+\|{\mathcal A}\bigl(\nabla X{\mathcal A}+{\mathcal A}^t(\nabla X)^t\bigr){\mathcal A}^t\nabla Y_t\|_{N}.
\end{align}
Applying Moser type inequality  and using \eqref{S4.1prop1assum} gives
\begin{align*}
\|({\mathcal A}{\mathcal A}^t-Id)\nabla X_t\|_{N}
&\lesssim |\nabla Y|_0\|\nabla X_t\|_{N}+|\nabla Y|_{N}\|\nabla X_t\|_0,\\
\|{\mathcal A}\nabla X{\mathcal A}{\mathcal A}^t\nabla Y_t\|_{N}
&\lesssim |\nabla Y_t|_0\|\nabla X\|_{N}+\big(|\nabla Y_t|_{N}+|\nabla Y_t|_0|\nabla Y|_{N}\big)\|\nabla X\|_0.
\end{align*}
Substituting the above estimates into \eqref{S4.2eqf2a} leads to \eqref{S4.2eqf0}.

\subsubsection{$L^2$-estimates for $f_m'(Y;X)$} We shall divide the proof of \eqref{S4.2eqf1} and \eqref{S4.2eqf2} into the following steps:

\no (i) \underline{ Estimates of $\||D|^{-1}f_m'(Y;X)\|_0.$}

By virtue of \eqref{fm'}, we have
\beq\label{S4.2eqf2b}
\||D|^{-1}f_m'(Y;X)\|_0\leq \bigl\||D|^{-1}\cA^t(\na X)^t\cA^t(\na\vv p_m)(Y)\bigr\|_0+\bigl\|{\mathcal A}^t\nabla \vv p_m'(Y;X)\bigr\|_0. \eeq
It follows from the law of product in Besov spaces and the imbedding: $L^{\f65}(\R^3)\hookrightarrow \dH^{-1}(\R^3),$ that
\beq\label{S4.2eqf2c}
\begin{split}
\||D|^{-1}\cA^t(\na X)^t\cA^t\na\vv p_m\|_0\leq &\bigl(1+\|\cA-Id\|_{\dB^{\f32}_{2,1}}\bigr)\|(\na X)^t\cA^t\na\vv p_m\|_{\dH^{-1}}
\leq C\|\na X\|_0\|\na\vv p_m\|_{L^3},
\end{split}
\eeq
 from which \eqref{S4.1eq9} and \eqref{S4.1eq11}, we infer
 \beno
\begin{split}
&\||D|^{-1}\cA^t(\na X)^t\cA^t\na\vv p_1\|_0\leq C|\p_3Y|_{1}^{\f43}\|\p_3Y\|_0^\frac23\|\na X\|_0,\\
&\||D|^{-1}\cA^t(\na X)^t\cA^t\na\vv p_2\|_0\leq C|Y_t|_{1}^{\f43}\|Y_t\|_0^\frac23\|\na X\|_0.
\end{split}
\eeno
Similarly, we get, by applying the law of product in Besov spaces, that
\beno
\||D|^{-1}\cA^t\na\vv p_m'(Y;X)\|_0\lesssim \bigl(1+\|\cA-Id\|_{\cB^{\f32}_{2,1}}\bigr)\|\na\vv p_m'(Y;X)\|_{\dH^{-1}}.
\eeno
To deal with the estimate of $\|\na\vv p_m'(Y;X)\|_{\dH^{-1}},$  we deduce from \eqref{p1'} and a similar derivation of \eqref{S4.2eqf2c} that
\beno
\begin{split}
\|&\vv p_1'(Y;X)\|_0
 \lesssim \|\cA\cA^t-Id\|_{\dot B^{\f32}_{2,1}}\|\na\vv p_1'(Y;X)\|_{\dH^{-1}}+\|\na X\|_0\|\na\vv p_1\|_{L^3} +\bigl(1+\|\cA-Id\|_{\dot B^{\f32}_{2,1}}\bigr)\times\\
 &\quad\times\Bigl(\|\cA\na X\cA(\p_3Y\otimes\p_3Y)\|_0 +\|\na X\|_0\|\cA\dive(\cA(\p_3Y\otimes\p_3Y))\|_{L^3}+\|\cA\p_{3}Y\otimes\p_{3}X\|_0\Big)\\
 &\lesssim \d_1\| \vv p_1'(Y;X)\|_0+\bigl(|\p_3Y|_{1}^{\f43}\|\p_3Y\|_0^\frac23+|\p_3Y|_0^2\bigr)\|\na X\|_0
 +|\p_3Y|_0\|\p_3X\|_0,
 \end{split}
 \eeno
 which together with \eqref{S4.1prop1assum} ensures that
 \beq \label{S4.2eq3}
 \| \vv p_1'(Y;X)\|_0\lesssim\big( |\p_3Y|_{1}^{\f43}\|\p_3Y\|_0^\frac23+|\p_3Y|_0^2\big)\|\na X\|_0+|\p_3Y|_0\|\p_3X\|_0. \eeq
 Exactly along the same line,
 we deduce from \eqref{p2'} that
\beq \label{S4.2eq4}
\|\vv p_2'(Y;X)\|_0\lesssim \big(|Y_t|_{1}^{\f43}\|Y_t\|_0^\frac23+|Y_t|_0^2\big)\|\na X\|_0+|Y_t|_0\|X_t\|_0.
\eeq
Inserting the above estimates into \eqref{S4.2eqf2b} leads to
\begin{align}
\label{S4.2eq5}&\||D|^{-1}f_1'(Y;X)\|_0\lesssim \big( |\p_3Y|_{1}^{\f43}\|\p_3Y\|_0^\frac23+|\p_3Y|_0^2\big)\|\na X\|_0 +|\p_3Y|_0\|\p_3X\|_0,\\
\label{S4.2eq6}&\||D|^{-1}f_2'(Y;X)\|_0\lesssim\big(|Y_t|_{1}^{\f43}\|Y_t\|_0^\frac23+|Y_t|_0^2\big)\|\na X\|_0+|Y_t|_0\|X_t\|_0.
\end{align}

\no (ii) \underline{Estimates of $\|f_m'(Y;X)\|_{\dot H^k}$ for $k\geq0$}
By \eqref{fm'} we have
\begin{align} \label{S4.2eq6a}
\|f_m'(Y;X)\|_{\dot H^k}&\leq\|{\mathcal A}^t(\nabla X)^t{\mathcal A}^t\nabla\vv p_m\|_{\dot H^k}+\|{\mathcal A}^t\nabla\big(\vv p_m'(Y;X)\big)\|_{\dot H^k}.
\end{align}

\noindent$\bullet$\underline{ Estimates for $\|{\mathcal A}^t(\nabla X)^t{\mathcal A}^t\nabla\vv p_m\|_{\dot H^k}.$}

We get, by applying Moser type inequalities, that
\begin{align*}
\|{\mathcal A}^t(\nabla X)^t{\mathcal A}^t\nabla\vv p_m\|_{\dot H^k}
&\lesssim \|\cA^t(\na X)^t\|_{L^6}\|D^k(\cA^t\na \vv p_m)\|_{L^3}+\|D^k(\cA^t(\na X)^t)\|_{L^6}\|\cA^t\na\vv p_m\|_{L^3}\\
&\lesssim \|D^{k}\nabla X\|_{L^6}\|\nabla\vv p_m\|_{L^3}+\|\nabla X\|_{L^6}\big(\|D^{k}\nabla\vv p_m\|_{L^3}+|D^k{\mathcal A}|_0\|\nabla \vv p_m\|_{L^3}\big).
\end{align*} Here and in all that follows, we always denote $D^k=\sum_{|\al|=k}\p^\al.$

In view of \eqref{p1}, applying Moser type inequalities yields
\begin{align*}
\|D^k\nabla \vv p_1\|_{L^3}
\leq &|{\mathcal A}{\mathcal A}^t-Id|_0\|D^k\nabla\vv p_1 \|_{L^3}+|{\mathcal A}{\mathcal A}^t-Id|_{k}\|\nabla \vv p_1\|_{L^3}+\|D^{k+1}\big({\mathcal A}(\p_3Y\otimes\p_3Y)\big)\|_{L^3},
\end{align*}
from which and \eqref{S4.1prop1assum}, we infer
\begin{align*}
\|D^k\nabla \vv p_1\|_{L^3}&\lesssim|\nabla Y|_{k}\|\nabla \vv p_1\|_{L^3}+\|D^{k+1}\big({\mathcal A}(\p_3Y\otimes\p_3Y)\big)\|_{L^3}.
\end{align*}
While it is easy to observe that
\begin{align*}
\|D^{k+1}\big({\mathcal A}(\p_3Y\otimes\p_3Y)\big)\|_{L^3}&\lesssim \|D^{k+1}(\p_3Y\otimes\p_3Y)\|_{L^3}+|D^{k+1}{\mathcal A}|_0\|\p_3Y\otimes\p_3Y\|_{L^3}\\
&\lesssim |\p_3Y|_{k+1}\|\p_3Y\|_{L^3}+|\nabla Y|_{k+1}|\p_3Y|_0\|\p_3Y\|_{L^3},
\end{align*}
which together with \eqref{S4.1eq9} ensures that
\begin{align*}
\|D^k\nabla \vv p_1\|_{L^3}&\leq \big(|\p_3Y|_{k+1}|\p_3Y|_0^\frac13+|\nabla Y|_{k+1}|\p_3Y|_0^\frac43\big)\|\p_3Y\|_0^\frac23,
\end{align*}
and hence, we obtain
\beq \label{S4.2eq6b}
\begin{split}
\|{\mathcal A}^t(\nabla X)^t{\mathcal A}^t\nabla\vv p_1\|_{\dot H^k}
\leq&  |\p_3Y|_{1}^\frac43\|\p_3Y\|_0^\frac23\|\nabla X\|_{\dot H^{k+1}}\\
&+\big( |\p_3Y|_{k+1}|\p_3Y|_0^\frac13+|\nabla Y|_{k+1}|\p_3Y|_0^\frac43\big)\|\p_3Y\|_0^\frac23\|\nabla X\|_{\dot H^1}.
\end{split} \eeq
By the same procedure, we can show that
\begin{align*}
\|D^k\nabla \vv p_2\|_{L^3}&\leq \big(|Y_t|_{k+1}|Y_t|_0^\frac13+|\nabla Y|_{k+1}|Y_t|_0^\frac43\big)\|Y_t\|_0^\frac23
\end{align*}
and
\beq\label{S4.2eq6c}
\begin{split}
\|{\mathcal A}^t(\nabla X)^t{\mathcal A}^t\nabla\vv p_2\|_{\dot H^k}
\leq&  |Y_t|_{1}^\frac43\|Y_t\|_0^\frac23\|\nabla X\|_{\dot H^{k+1}}\\
&+\big( |Y_t|_{k+1}|Y_t|_0^\frac13+|\nabla Y|_{k+1}|Y_t|_0^\frac43\big)\|Y_t\|_0^\frac23\|\nabla X\|_{\dot H^1}.
\end{split} \eeq
Furthermore, there hold
\begin{align}
\label{S4.2eq7}&\|\nabla\vv p_1\|_{W^{N,3}}\leq  \big(|\p_3Y|_{N+1}|\p_3Y|_0^\frac13+|\nabla Y|_{N+1}|\p_3Y|_0^\frac43\big)\|\p_3Y\|_0^\frac23,\\
\label{S4.2eq8}&\|\nabla \vv p_2\|_{W^{N,3}}\leq \big(|Y_t|_{N+1}|Y_t|_0^\frac13+|\nabla Y|_{N+1}|Y_t|_0^\frac43\big)\|Y_t\|_0^\frac23.
\end{align}

\noindent$\bullet$\underline{ Estimates of $\|{\mathcal A}^t\nabla\left(\vv p_m'(Y;X)\right)\|_{\dot H^k}.$}

Applying Moser type inequality gives
\begin{align}\label{S4.2eq8a}
\|{\mathcal A}^t\nabla\left(\vv p_m'(Y;X)\right)\|_{\dot H^k}\leq \|\nabla\big(\vv p_m'(Y;X)\big)\|_{\dot H^k}+|{\mathcal A}^t-Id|_{k}\|\nabla\left(\vv p_m'(Y;X)\right)\|_0.
\end{align}
Yet in view of \eqref{p1'}, we have
\beno
\begin{split}
\|\na\vv p_1'(Y;X)\|_{\dot H^k}\lesssim &\|({\mathcal A}{\mathcal A}^t-Id)\na\vv p_1'(Y;X)\|_{\dot H^k}
+\|{\mathcal A}\bigl(\nabla X{\mathcal A}+{\mathcal A}^t\nabla X\bigr){\mathcal A}^t\nabla \vv p_1\|_{\dot H^k}\\
&+\|{\mathcal A}\dv\big({\mathcal A}\nabla X{\mathcal A}(\p_{3}Y\otimes \p_{3}Y)\big)\|_{\dot H^k} +\|{\mathcal A}\nabla X{\mathcal A}\dv\big({\mathcal A}(\p_{3}Y\otimes \p_{3}Y)\big)\|_{\dot H^k}\\
&+\|{\mathcal A}\dv\big({\mathcal A}(\p_{3}Y\otimes\p_{3}X+\p_{3}X\otimes \p_{3}Y)\big)\|_{\dot H^k}.
\end{split}
\eeno
It follows from a similar derivation of \eqref{S4.2eq6b} that
\begin{align*}
\|{\mathcal A}\nabla X{\mathcal A}{\mathcal A}^t\nabla \vv p_1\|_{\dot H^k}
\leq & |\p_3Y|_{1}^\frac43\|\p_3Y\|_0^\frac23\|\nabla X\|_{\dot H^{k+1}}\\
&+\big( |\p_3Y|_{k+1}|\p_3Y|_0^\frac13+|\nabla Y|_{k+1}|\p_3Y|_0^\frac43\big)\|\p_3Y\|_0^\frac23\|\nabla X\|_{\dot H^1}
\end{align*}
And we get, by applying Moser type inequality, that
\begin{align*}
\|({\mathcal A}{\mathcal A}^t-Id)\na\vv p_1'(Y;X)\|_{\dot H^k}
&\leq C\delta_1\|\nabla\vv p_1'(Y;X)\|_{\dot H^k}+|\nabla Y|_{k}\|\nabla\vv p_1'(Y;X)\|_0,
\end{align*}
and
\beno \begin{split}
\|{\mathcal A}&\dv\big({\mathcal A}\nabla X{\mathcal A}(\p_{3}Y\otimes \p_{3}Y)\big)\|_{\dot H^k}
\lesssim|\p_3Y|_0^2\|\nabla X\|_{\dot H^{k+1}}+\big(|\p_3Y|_{k+1}|\p_3Y|_0+|\nabla Y|_{k+1}|\p_3Y|_0^2\big)\|\nabla X\|_0,
\end{split}
\eeno
and
\beno \begin{split}
\|{\mathcal A}\nabla X{\mathcal A}\dv\big({\mathcal A}(\p_{3}Y\otimes \p_{3}Y)\big)\|_{\dot H^k}
\lesssim &|\p_3Y|_{1}|\p_3Y|_0\|\na X\|_{\dot H^k}\\
&+\big(|\p_3Y|_{k+1}|\p_3Y|_0+|\nabla Y|_{k+1}|\p_3Y|_0^2\big)\|\nabla X\|_0,
\end{split}
\eeno
and finally
\beno
\begin{split}
\|{\mathcal A}\dv\big({\mathcal A}(\p_{3}Y\otimes\p_{3}X+\p_{3}X\otimes \p_{3}Y)\big)\|_{\dot H^k}
\lesssim & |\p_3Y|_0\|\p_3X\|_{\dot H^{k+1}}\\
&+\big(|\p_3Y|_{k+1}+|\nabla Y|_{k+1}|\p_3Y|_0\big)\|\p_3X\|_0.
\end{split}
\eeno
As a result, by virtue of  \eqref{S4.1eq10}, it comes out
\begin{equation}\label{S4.2eq11}
\begin{split}
\|\na\vv p_1'(Y;X)\|_{\dot H^k}\lesssim&\frak{g}_2(\p_3Y,\p_3X) \with\\
\frak{g}_2(\frak{x},\frak{y})\eqdefa&\big(|\frak{x}|_{1}^\frac43\|\frak{x}\|_0^\frac23+|\frak{x}|_1^2\big)\bigl(\|\na X\|_{\dot H^{k+1}}+|\nabla Y|_{k+1}\|\nabla X\|_{1}\bigr)+ |\frak{x}|_0\|\frak{y}\|_{\dot H^{k+1}}\\
&+\big(|\frak{x}|_{k+1}+|\nabla Y|_{k+1}|\frak{x}|_{1}\big)\|\frak{y}\|_{1}+|\frak{x}|_{k+1}\big(|\frak{x}|_0^\frac13\|\frak{x}\|_0^\frac23+|\frak{x}|_0\big)\|\nabla X\|_{1}.
\end{split}
\end{equation}
Substituting the above estimate and \eqref{S4.1eq10} into \eqref{S4.2eq8a} for $m=1$ shows that $\|{\mathcal A}^t\nabla\left(\vv p_1'(Y;X)\right)\|_{\dot H^k}$
shares the same estimate as above.

Similarly, we can show that
\begin{equation}\label{S4.2eq11a}
\begin{split}
\|\na\vv p_2'(Y;X)\|_{\dot H^k}
&\lesssim  \frak{g}_2(Y_t,X_t).
\end{split} \eeq
Substituting the above estimate and \eqref{S4.1eq12} into \eqref{S4.2eq8a} for $m=2$ shows that $\|{\mathcal A}^t\nabla\left(\vv p_2'(Y;X)\right)\|_{\dot H^k}$
shares the same estimate as above.

Let us now turn to the estimates of $\|f_1'(Y;X)\|_{\dot H^k}$ and $\|f_2'(Y;X)\|_{\dot H^k}.$
As a matter of fact, by inserting \eqref{S4.2eq6b} and \eqref{S4.2eq11} into \eqref{S4.2eq6a} for $m=1,$  we achieve
\begin{equation}\label{S4.2eq9}
\begin{split}
\|f_1'(Y;X)\|_{\dot H^k}
&\lesssim \frak{g}_2(\p_3Y,\p_3X).
\end{split}
\end{equation}
Similarly  by inserting \eqref{S4.2eq6c} and \eqref{S4.2eq11a} into \eqref{S4.2eq6a} for $m=2,$  we obtain
\begin{equation}\label{S4.2eq10}
\begin{split}
\|f_2'(Y;X)\|_{\dot H^k}
&\lesssim  \frak{g}_2(Y_t,X_t).
\end{split}
\end{equation}

Now we are in a position to complete the proof of Proposition \ref{S4.2prop1}.

\begin{proof}[Proof of Proposition \ref{S4.2prop1}] It remains to prove \eqref{S4.2eqf1} and  \eqref{S4.2eqf2}. Indeed, combining \eqref{S4.2eq5} with \eqref{S4.2eq9}, we obtain \eqref{S4.2eqf1}. While combining \eqref{S4.2eq6} with \eqref{S4.2eq10} leads to \eqref{S4.2eqf2}. This completes
the proof of Proposition \ref{S4.2prop1}.
\end{proof}



\section{Energy estimates for the linearized equation}\label{Sect6}

The goal of this section is to present the proof of Theorem \ref{S5thm1}.

\subsection{First-order energy estimates}
Let us first carry out the estimate of $\cE_0(t)$ \eqref{S5eqE1}.

\noindent$\bullet${\underline{The estimate of $\|\na X\|_0.$}}

We first get, by taking $L^2$ inner product of \eqref{LE1} with $X$, that
\beq \label{S5eq1}
\f{d}{dt}\left(\f12\|\na X\|_0^2+(X_t | X)_{L^2}\right)+\|\p_3X\|_0^2-\|X_t\|_0^2=(f'(Y;X)+g | X)_{L^2}.
\eeq
And it follows by taking $L^2$ inner product of \eqref{LE1} with $(-\D)^{-1}X_t$ that
\beno
\f12\f{d}{dt}\left(\||D|^{-1}X_t\|_0^2+\||D|^{-1}\p_3X\|_0^2\right)+\|X_t\|_0^2=\big((-\D)^{-1}(f'(Y;X)+g) | X_t\big)_{L^2}.
\eeno
Summing up the above equality with $\f14\times$\eqref{S5eq1} yields
\beq \label{S5eq2}
\begin{split}
\f{d}{dt}&\left(\f12\bigl(\||D|^{-1}X_t\|_0^2+\||D|^{-1}\p_3X\|_0^2+\f14\|\na X\|_0^2\bigr)+\f14(X_t | X)_{L^2}\right)
\\
&\qquad+\f34\|X_t\|_0^2+\f14\|\p_3X\|_0^2=\big(|D|^{-1}(f'(Y;X)+g) | \f14|D|X+|D|^{-1} X_t\big)_{L^2}.
\end{split}
\eeq
It is easy to observe that
$$\longformule{
\big|\bigl(|D|^{-1}\na\cdot\bigl({\mathcal A}(\nabla X{\mathcal A}+{\mathcal A}^t(\nabla X)^t){\mathcal A}^t\bigr)\nabla Y_t\bigr) |
\f14|D|X+|D|^{-1} X_t\bigr)_{L^2}\big|}{{} \leq C|\na Y_t|_0\|\na X\|_0\bigl(\|\na X\|_0+\||D|^{-1}X_t\|_0\bigr),} $$
and
$$\longformule{
\bigl(|D|^{-1}\na\cdot\bigl((\cA\cA^t-Id)\na X_t\bigr) | |D|X \bigr)_{L^2}=-\bigl((\cA\cA^t-Id)\na X_t | \na X \bigr)_{L^2}}
{{}
=-\f12\f{d}{dt}\bigl((\cA\cA^t-Id)\na X | \na X \bigr)_{L^2}+\int_{\R^3}\p_t(\cA\cA^t)|\na X|^2\,dx,} $$
and
\beno
\begin{split}
\bigl|\bigl(|D|^{-1}\na\cdot\bigl((\cA\cA^t-Id)\na X_t\bigr) \big| |D|^{-1}X_t \bigr)_{L^2}\bigr|
\lesssim &\|\cA\cA^t-Id\|_{\dot B^{\f32}_{2,1}}\|\na X_t\|_{\dH^{-1}}\|X_t\|_0\leq C\d_1\|X_t\|_0^2.
\end{split}
\eeno
Hence in view of \eqref{f0'}, under the assumption of \eqref{S4.1prop1assum},  by taking $\d_1$ so small that $C\d_1\leq\f14,$ we obtain
\beq \label{S5eq3}
\begin{split}
\bigl|\bigl(|D|^{-1}&f_0'(Y;X) | \f14|D|X+|D|^{-1} X_t\bigr)_{L^2}+\f18\f{d}{dt}\bigl((\cA\cA^t-Id)\na X | \na X \bigr)_{L^2}\bigr|\\
&\qquad\qquad\qquad\qquad\leq C|\na Y_t|_0\|\na X\|_0\bigl(\|\na X\|_0+\||D|^{-1}X_t\|_0\bigr)+\f14 \|X_t\|_0^2.
\end{split}
\eeq
While by virtue of \eqref{S4.2eq5} and \eqref{S4.2eq6}, we have
\beq \label{S5eq4}
\begin{split}
&\bigl|\bigl(|D|^{-1}(f_1'(Y;X)+f_2'(Y;X)) | \f14|D|X+|D|^{-1}X_t\bigr)_{L^2}\bigr|\\
&\leq \f18\bigl(\|X_t\|_0^2+\|\p_3X\|_0^2\bigr)
+C\Bigl(|\p_3Y|_1^{\f43}\|\p_3Y\|_0^\frac23+|\p_3Y|_0^2\\
&\qquad\qquad\qquad\qquad\qquad\qquad\ +|Y_t|_1^{\f43}\|Y_t\|_0^\frac23+|Y_t|_0^2\Bigr)\big(\|\na X\|_0^2+\||D|^{-1}X_t\|_0^2\big).
\end{split}
\eeq
Inserting \eqref{S5eq3} and \eqref{S5eq4} into \eqref{S5eq2} gives rise to
\beq \label{S5eq5}
\begin{split}
&\f{d}{dt}\left(\f12\bigl(\||D|^{-1}X_t\|_0^2+\||D|^{-1}\p_3X\|_0^2+\f14\bigl(\cA\cA^t\na X | \na X \bigr)_{L^2}\bigr)+\f14(X_t | X)_{L^2}\right)
\\
&+\f18\bigl(\|X_t\|_0^2+\|\p_3X\|_0^2\bigr)\leq \||D|^{-1}g\|_0\big(\|\na X\|_0+\||D|^{-1}X_t\|_0\big)\\
&\quad\qquad +C\Bigl(|\p_3Y|_1^{\f43}\|\p_3Y\|_0^\frac23+|\p_3Y|_0^2
+|Y_t|_1\Bigr)\big(\|\na X\|_0^2+\||D|^{-1}X_t\|_0^2\big),
\end{split}
\eeq
by applying the assumption \eqref{S5thmassum}.

On the other hand,
since
$\cA\cA^t$ is a positive definite matrix $(|\cA\cA^t-Id|_0\leq C\d_1\leq\frac14)$, it holds that
\beno
\bigl(\cA\cA^t\na X | \na X \bigr)_{L^2}\geq (1-C\delta_1) \|\na X\|_0^2\geq\frac34\|\nabla X\|_0^2, \eeno
so that one has
\beq \label{S5eq5a} \begin{split}
\f12\bigl(\||D|^{-1}X_t\|_0^2+\||D|^{-1}\p_3X\|_0^2+&\f14\bigl(\cA\cA^t\na X | \na X \bigr)_{L^2}\bigr)+\f14(X_t | X)_{L^2}\\
\geq& \f14\||D|^{-1}X_t\|_0^2+\f12\||D|^{-1}\p_3X\|_0^2+\f1{32}\|\na X\|_0^2. \end{split}
\eeq

\noindent$\bullet${\underline{The estimate of $\|X_t\|_0.$}}

Multiplying \eqref{LE1} by $X_t$ and integrating the resulting equality over $\R^3$, we get
\begin{align*}
\frac12\frac{d}{dt}\big(\|X_t\|_0^2+\|\p_3X\|_0^2\big)+\|\nabla X_t\|_0^2=\big(f'(Y;X)+g\big|X_t\big)_{L^2}.
\end{align*}
In view of \eqref{f0'}, we infer
\beno
\begin{split}
\bigl|\bigl(f_0'(Y;X) | X_t\bigr)_{L^2}\bigr|
\leq &C|\na Y_t|_0^2\|\na X\|_0^2+\f14\|\na X_t\|_0^2.
\end{split}
\eeno
While it follows from \eqref{S4.2eq5} to \eqref{S4.2eq6} that
\beno
\begin{split}
\bigl|\bigl(f_1'(Y;X)+f_2'(Y;X) | X_t\bigr)_{L^2}\bigr|
&\leq  C\Bigl(\big(|\p_3Y|_0\|\p_3X\|_0+|Y_t|_0\|X_t\|_0\big)+\bigl(|\p_3Y|_1^{\f43}\|\p_3Y\|_0^\frac23
\\
&\qquad\quad +|\p_3Y|_1^2+|Y_t|_1^{\f43}\|Y_t\|_0^\frac23+|Y_t|_1^2\bigr)\|\na X\|_0\Bigr)\|\nabla X_t\|_0.
\end{split}
\eeno
As a result, thanks to the assumption \eqref{S5thmassum}, it comes out
\beq \label{S5eq6}
\begin{split}
\frac{d}{dt}\big(\|X_t\|_0^2+\|\p_3X\|_0^2\big)+\|\nabla X_t\|_0^2
\leq & C\bigl(|\p_3Y|_1^{\f83}\|\p_3Y\|_0^\frac43+|\p_3Y|_1^4+|Y_t|_1^{2}\bigr)\|\na X\|_0^2\\
&+C\big(|\p_3Y|_0^2\|\p_3X\|_0^2+|Y_t|_0^2\|X_t\|_0^2\big)+4\||D|^{-1}g\|_0^2.
\end{split}
\eeq

\noindent$\bullet${\underline{The estimate of $\|\na X_t\|_0.$}}

By taking $L^2$ inner product of \eqref{LE1} with $-\Delta X_t$ gives
\beq\label{S5eq7}
\begin{split}
&\frac12\frac{d}{dt}\big(\|\na X_t\|_{ L^2}^2+\|\na\p_{3}X\|_0^2\bigr)
+\|\Delta X_t\|_0^2
=-\big(f'(Y;X)+g | \Delta X_t\big)_{L^2}.
\end{split} \eeq
It is easy to observe from  \eqref{S4.1prop1assum} and \eqref{f0'}  that
\beq \label{S5eq8}
\|f_0'(Y;X)\|_0\leq \frac14\|\D X_t\|_0+|\na Y_t|_1\|\na X\|_1. \eeq
Then by substituting the estimates \eqref{S5eq8}, \eqref{S4.2eq9} and \eqref{S4.2eq10} into \eqref{S5eq7} and using the assumptions \eqref{S4.1prop1assum} and \eqref{S5thmassum},
we obtain
\beq\notag
\begin{split}
\frac12\frac{d}{dt}\big(\|\na X_t\|_{ L^2}^2+\|\na\p_{3}X\|_0^2\bigr)
+\|\Delta X_t\|_0^2\leq& C\Bigl(\bigl(|Y_t|_2+|\p_3Y|_1^\frac43\|\p_3Y\|_0^\frac23+|\p_3Y|_1^2\bigr)\|\na X\|_1\\
&+|\p_3Y|_1\|\p_3X\|_1+|Y_t|_1\|X_t\|_1+\|g\|_0\Bigr)\|\D X_t\|_0,
\end{split} \eeq
which implies
\beq\label{S5eq10}
\begin{split}
\frac{d}{dt}\big(\|\na X_t\|_{ L^2}^2+\|\na\p_{3}X\|_0^2\bigr)
+\|\Delta X_t\|_0^2\leq & C\bigl(|Y_t|_2^2+|\p_3Y|_1^\frac83\|\p_3Y\|_0^\frac43+|\p_3Y|_1^4\bigr)\|\na X\|_1^2\\
&+C\bigl(|\p_3Y|_1^2\|\p_3X\|_1^2+|Y_t|_1^2\|X_t\|_1^2\bigr)+\|g\|_0^2.
\end{split} \eeq

\noindent$\bullet${ \underline{The estimate of $\|\na X\|_{\dH^1}.$}}

In this step, we shall  use the equivalent formulation, \eqref{LEW1},  of \eqref{LE1}. We first get, by
taking $L^2$ inner product of \eqref{LEW1} with  $-\na\cdot\bigl(\cA\cA^t\na X\bigr),$ that
\beno
\begin{split}
\f12\f{d}{dt}\|\na\cdot\bigl(\cA\cA^t\na X\bigr)\|_0^2+\bigl(\p_3^2X | \na\cdot\bigl(\cA\cA^t\na X\bigr)&\bigr)_{L^2}-\bigl(X_{tt} | \na\cdot\bigl(\cA\cA^t\na X\bigr)\bigr)_{L^2}\\
=&-\bigl(\wt{f}'(Y;X)+g | \na\cdot\bigl(\cA\cA^t\na X\bigr)\bigr)_{L^2}.
\end{split}
\eeno
By using integration by parts, one has
\beno
\begin{split}
&\bigl(X_{tt} | \na\cdot\bigl(\cA\cA^t\na X\bigr)\bigr)_{L^2}=-\f{d}{dt}\bigl(\na X_{t} | \cA\cA^t\na X\bigr)_{L^2}+\bigl(\na X_t | \p_t(\cA\cA^t\na X)\bigr)_{L^2},
\\
&\bigl(\p_3^2X | \na\cdot\bigl(\cA\cA^t\na X\bigr)\bigr)_{L^2}
=\bigl(\na\p_3X | \cA\cA^t\na \p_3X\bigr)\bigr)_{L^2}+\bigl(\na\p_3X | \p_3\bigl(\cA\cA^t\bigr)\na X\bigr)_{L^2}.
\end{split}
\eeno
Since $\cA\cA^t$ is a positive definitive matrix, we infer
\beq
\label{S5eq11}
\begin{split}
\f{d}{dt}\Bigl(&\f12\|\na\cdot\bigl(\cA\cA^t\na X\bigr)\|_0^2+\bigl(\na X_t | \cA\cA^t\na X\bigr)_{L^2}\Bigr)+\f12\|\nabla \p_3X \|_0^2\\
\leq &2\|\na X_t\|_0^2+\f14\|\na\p_3X\|_0^2+C\bigl(|\na Y_t|_0^2+|\p_3\na Y|_0^2\bigr)\|\na X\|_0^2\\
&\qquad\qquad\qquad\qquad\qquad\qquad\quad-\bigl(\wt{f}'(Y;X)+g | \na\cdot\bigl(\cA\cA^t\na X\bigr)\bigr)_{L^2}.
\end{split}
\eeq
Yet under the assumption of \eqref{S4.1prop1assum}, it is easy to observe from \eqref{wf0'} that
\beno
\|\wt{f}_0'(Y;X)\|_0\leq C|\na Y_t|_1\|\na X\|_1.
\eeno
Whereas it follows from \eqref{S4.2eq9}, \eqref{S4.2eq10} that
$$\longformule{
\bigl|\bigl(f_1'(Y;X)+f_2'(Y;X) | \na\cdot(\cA\cA^t\na X)\bigr)_{L^2}\bigr|\lesssim  \Bigl(|\p_3Y|_1\|\p_3X\|_1+|Y_t|_1\|X_t\|_1}{{}
+\bigl(|\p_3Y|_1^{\f43}\|\p_3Y\|_0^\frac23+|\p_3Y|_1^2+|Y_t|_1^{\f43}\|Y_t\|_0^\frac23+|Y_t|_1^2\bigr)\|\na X\|_1\Bigr)\|\na X\|_1.
}
$$
Inserting the above estimates into \eqref{S5eq11} yields
\beq
\label{S5eq12}
\begin{split}
\f{d}{dt}\Bigl(\f12\|\na\cdot\bigl(&\cA\cA^t\na X\bigr)\|_0^2+\bigl(\na X_t | \cA\cA^t\na X\bigr)_{L^2}\Bigr)+\f18\|\nabla \p_3X \|_0^2
\leq 3\| X_t\|_1^2+\frac1{20}\|\p_3X\|_0^2\\
&+\|g\|_0\|\nabla X\|_1 +C\bigl(|\p_3Y|_1^2+|Y_t|_1^2+|\p_3Y|_1^{\f43}\|\p_3Y\|_0^\frac23+|Y_t|_1^{\f43}\|Y_t\|_0^\frac23\bigr)\|\na X\|_1^2.
\end{split}
\eeq
Let us denote
 \begin{equation}\label{S5eq13}
 \begin{split}
 E_0(t) &\eqdefa\f12\Bigl(\||D|^{-1} X_t\|_{ H^2}^2+\||D|^{-1}\p_{3}X\|_2^2+\f14\bigl(\cA\cA^t\na X | \na X \bigr)_{L^2}\Bigr)\\
 &\qquad+\f14(X_t | X)_{L^2}+\frac1{48}\Big(\f12\|\na\cdot\bigl(\cA\cA^t\na X\bigr)\|_0^2+\bigl(\na X_t | \cA\cA^t\na X\bigr)_{L^2}\Big).
\end{split}
 \end{equation}
 Then by summing up the inequalities  \eqref{S5eq5}, \eqref{S5eq6}, \eqref{S5eq10} and $\frac1{48}\times$\eqref{S5eq12}, we obtain,
 \beq \label{S5eq14}
\begin{split}
&\f{d}{dt}E_0(t)+\frac1{16}\|X_t\|_2^2+\frac1{384}\|\p_3X\|_1^2
\leq +\||D|^{-1}g\|_1^2+ C\bigl(|\p_3Y|_1^{\f43}\|\p_3Y\|_0^\frac23+|\p_3Y|_1^2+|Y_t|_2\bigr)\times\\
&\quad\times\bigl(\|\na X\|_1^2+\||D|^{-1}X_t\|_0^2+\|\p_3X\|_1^2+\|X_t\|_1^2\big)+\||D|^{-1}g\|_1\big(\|\na X\|_1+\||D|^{-1}X_t\|_0\big).
\end{split}
\eeq
 Notice that
 \begin{align*}
 &\bigl(\na X_t | \cA\cA^t\na X\bigr)_{L^2}\geq -\|X_t\|_0^2-\frac14\|\nabla\cdot\big({\mathcal A}{\mathcal A}^t\nabla X\big)\|_0^2,
 \end{align*}
and
 \begin{align*}
 &\|\na\cdot\bigl(\cA\cA^t\na X\bigr)\|_0\geq\|\nabla X\|_{\dot H^1}-\|({\mathcal A}{\mathcal A}^t-Id)\nabla X\|_{\dot H^1}\geq(1-C\delta_1)\|\nabla X\|_{\dot H^1},
 \end{align*}
so that we deduce from from \eqref{S5eq5a} and \eqref{S5eq13} that
\begin{equation}\label{S5eq15}
E_0(t)\geq \frac1{16^2}\Big(\||D|^{-1}X_t\|_2^2+\||D|^{-1}\p_3X\|_2^2+\|\nabla X\|_1^2\Big).
\end{equation}
Hence for any $\varepsilon>0$, we deduce from \eqref{S5eq14} that
 \begin{equation}\label{S5eq16}
\begin{split}
\f{d}{dt}&E_0(t)+\frac1{16}\|X_t\|_2^2+\frac1{384}\|\p_3X\|_1^2\leq \langle t\rangle^{1+\varepsilon}\||D|^{-1}g\|_1^2\\
&\qquad\qquad+ C_\varepsilon\Bigl(|\p_3Y|_1^{\f43}\|\p_3Y\|_0^\frac23+|\p_3Y|_1^2+|Y_t|_2+\langle t\rangle^{-(1+\varepsilon)}\Bigr)E_0(t).
\end{split}
\end{equation}
Applying Gronwall's inequality yields for any  $\varepsilon>0$ that
\begin{align*}
E_0(t)+\frac1{16}\|X_t\|_{L^2_t(H^2)}^2&+\frac1{384}\|\p_3X\|_{L^2_t(H^1)}^2\leq C_\varepsilon\Big( \int_0^t\langle s\rangle^{1+\varepsilon}\||D|^{-1}g(s)\|_1^2ds\Big)\times\\
&\times \exp C\big(|\p_3Y|_{\frac12+\varepsilon,1}^{\frac43}\|\p_3Y\|_{L_t^2(L^2)}^\frac23+|\p_3Y|_{\frac12+\varepsilon,1}^2+|Y_t|_{1+\varepsilon,2}\big),
\end{align*}
which together with \eqref{S5eq15} ensures the first inequality of \eqref{S5eqE1}.

\subsection{Higher-order energy estimates.}

In this subsection, we shall derive the estimates for
\begin{equation}\label{S5H0}
\dot E_{k+1}(t)\eqdefa \|\p_3X\|_{\dot H^{k+1}}^2+\|X_t\|_{\dot H^{k+1}}^2+\|\nabla X\|_{\dot H^{k+1}}^2\quad\mbox{for}\  k\geq0.
\end{equation}
We first get, by taking the $\dot H^{k+1}$-inner product of \eqref{LE1} with $X_t,$  that
\begin{align*}
&\frac12\frac{d}{dt}\big(\| X_t\|_{\dot H^{k+1}}^2+\|\p_3X\|_{\dot H^{k+1}}^2\big)+\| X_t\|_{\dot H^{k+2}}^2=\big(f'(Y;X)+g\big|X_t\big)_{\dot H^{k+1}},
\end{align*}
which implies
\begin{align}\label{S5H1}
&\frac{d}{dt}\big(\| X_t\|_{\dot H^{k+1}}^2+\|\p_3X\|_{\dot H^{k+1}}^2\big)+\| X_t\|_{\dot H^{k+2}}^2\leq\|f'(Y;X)\|_{\dot H^k}^2+\|g\|_{\dot H^k}^2.
\end{align}
Yet in view of \eqref{f0'}, it follows from Moser type inequality that
\begin{equation}\label{S5H6}
\begin{split}
\|f_0'(Y;X)\|_{\dot H^k}
&\lesssim |\nabla Y_t|_0\|\nabla X\|_{\dot H^{k+1}}+\big(|\nabla Y_t|_{k+1}+|\nabla Y_t|_0|\nabla Y|_{k+1}\big)\|\nabla X\|_0\\
&\qquad+|\nabla Y|_0\|\nabla X_t\|_{\dot H^{k+1}}+|\nabla Y|_{k+1}\|\nabla X_t\|_{0},
\end{split}
\end{equation}
from which, \eqref{S4.2eq9}, \eqref{S4.2eq10}
and the assumption \eqref{S5thmassum}, we infer
\begin{equation}\label{S5H8}
\begin{split}
\|f'(Y;X)\|_{\dot H^k}
&\lesssim |\p_3Y|_0\|\p_3X\|_{\dot H^{k+1}}+ |Y_t|_0\|X_t\|_{\dot H^{k+1}}+|\nabla Y|_0\| X_t\|_{\dot H^{k+2}}\\
&+\big(|\p_3Y|_1^\frac43\|\p_3Y\|_{0}^\frac23+|\p_3Y|_1^2+| Y_t|_1\big)\bigl(\|\na X\|_{\dot H^{k+1}}+|\nabla Y|_{k+1}\|\nabla X\|_{1}\bigr)\\
&+\big(|\p_3Y|_{k+1}+|\nabla Y|_{k+1}|\p_3Y|_1\big)\|\p_3X\|_{1}+\big(|Y_t|_{k+1}+|\nabla Y|_{k+1}|Y_t|_1\big)\|X_t\|_{1}\\
&+|\nabla Y|_{k+1}\|\nabla X_t\|_{0}+\bigl(|\p_3Y|_{k+1}\big(|\p_3Y|_0^\frac13\|\p_3Y\|_{0}^\frac23+|\p_3Y|_1\big)+|Y_t|_{k+2}\bigr)\|\nabla X\|_{1}.
\end{split}
\end{equation}
Inserting \eqref{S5H8} into \eqref{S5H1},  and using the assumption \eqref{S4.1prop1assum} so that $|\nabla Y|_0\leq\delta_1$,  we deduce that
\begin{equation}\label{S5H10}
\begin{split}
&\frac{d}{dt}\big(\| X_t\|_{\dot H^{k+1}}^2+\|\p_3X\|_{\dot H^{k+1}}^2\big)+\frac34\| X_t\|_{\dot H^{k+2}}^2\lesssim |\p_3Y|_0^2\|\p_3X\|_{\dot H^{k+1}}^2+|Y_t|_0^2\|X_t\|_{\dot H^{k+1}}^2\\
&\quad+\|g\|_{\dot H^k}^2+\big(|\p_3Y|_1^{\frac83}\|\p_3Y\|_0^\frac43+|\p_3Y|_1^4+|Y_t|_1^2\big)\bigl(\|\nabla X\|_{\dot H^{k+1}}^2+|\nabla Y|_{k+1}^2\|\nabla X\|_1^2\bigr)\\
&\quad+\big(|\p_3Y|_{k+1}^2+|\nabla Y|_{k+1}^2|\p_3Y|_1^2\big)\|\p_3Y\|_1^2+\big(|Y_t|_{k+1}^2+|\nabla Y|_{k+1}^2|Y_t|_1^2\big)\|X_t\|_1^2\\
&\quad+|\nabla Y|_{k+1}^2\|\nabla X_t\|_0^2+\big(|\p_3Y|_{k+1}^2\big(|\p_3Y|_0^{\frac23}\|\p_3Y\|_0^\frac43+|\p_3Y|_1^2\big)+|Y_t|_{k+2}^2\bigr)\|\nabla X\|_1^2.
\end{split}
\end{equation}

Secondly, by taking the $\dot H^k$-inner product of \eqref{LEW1} with  $-\na\cdot\bigl(\cA\cA^t\na X\bigr),$ we obtain
\begin{equation}\label{S5H2}
\begin{split}
\f12\f{d}{dt}\|\na\cdot\bigl(&\cA\cA^t\na X\bigr)\|_{\dot H^k}^2+\bigl(\p_3^2X | \na\cdot\bigl(\cA\cA^t\na X\bigr)\bigr)_{\dot H^k}\\
&-\bigl(X_{tt} | \na\cdot\bigl(\cA\cA^t\na X\bigr)\bigr)_{\dot H^k}
=-\bigl(\wt{f}'(Y;X)+g | \na\cdot\bigl(\cA\cA^t\na X\bigr)\bigr)_{\dot H^k}.
\end{split}
\end{equation}
By using integration by parts, one has
\beno
-\bigl(X_{tt} | \na\cdot\bigl(\cA\cA^t\na X\bigr)\bigr)_{\dot H^k}=-\f{d}{dt}\bigl( X_{t} | \nabla \cdot(\cA\cA^t\na X)\bigr)_{\dot H^k}-\bigl(\nabla  X_t | \p_t(\cA\cA^t\na X)\bigr)_{\dot H^k},
\eeno
and
\begin{align*}
\big|\bigl(\nabla  X_t | \p_t(\cA\cA^t\na X)\bigr)_{\dot H^k}\big|
\leq  \| X_t\|_{\dot H^{k+1}}\Big(&\frac32\|\nabla X_t\|_{\dot H^k}+|\nabla Y|_{k}\|\nabla X_t\|_0\\
&+|\nabla Y_t|_0\|\nabla X\|_{\dot H^k}+|\nabla Y_t|_{k}\|\nabla X\|_0\Bigr),
\end{align*}
so that it comes out
$$\longformule{
\bigl|\bigl(X_{tt} | \na\cdot\bigl(\cA\cA^t\na X\bigr)\bigr)_{\dot H^k}-\f{d}{dt}\bigl( X_{t} | \nabla \cdot(\cA\cA^t\na X)\bigr)_{\dot H^k}\bigr|}{{}
\leq 2\| X_t\|_{\dot H^{k+1}}^2+C_k\big(|\nabla Y_t|_0^2\|\nabla X\|_{\dot H^k}^2+|\nabla Y|_{k}^2\|\nabla X_t\|_0^2+|\nabla Y_t|_{k}^2\|\nabla X\|_0^2\big).}$$
Similarly, again by using integration by parts, one has
\beno
\begin{split}
\bigl(\p_3^2X | \na\cdot\bigl(\cA\cA^t\na X\bigr)\bigr)_{\dot H^k}
=&\bigl(\na\p_3X | \cA\cA^t\na \p_3X\bigr)_{\dot H^k}+\bigl(\na\p_3X | \p_3\bigl(\cA\cA^t\bigr)\na X\bigr)_{\dot H^k}.
\end{split}
\eeno
Since $|{\mathcal A}{\mathcal A}^t-Id|_0\leq C\delta_1\leq\frac14$ due to \eqref{S4.1prop1assum}, applying Moser type inequality gives
\begin{align*}
\bigl(\na\p_3X | \cA\cA^t\na \p_3X\bigr)_{\dot H^k}
&\geq\frac12\|\nabla\p_3X\|_{\dot H^k}^2-C_k|\nabla Y|_{k}^2\|\nabla \p_3X\|_0^2,
\end{align*}
and
\begin{align*}
\big|\bigl(\na\p_3X | \p_3\bigl(\cA\cA^t\bigr)\na X\bigr)_{\dot H^k}\big|
&\leq\frac14\|\nabla\p_3X\|_{\dot H^k}^2+C_k\big(|\p_3 Y|_1^2\|\nabla X\|_{\dot H^k}^2+|\p_3Y|_{k+1}^2\|\nabla X\|_0^2\big),
\end{align*}
so that there holds
\beno
\begin{split}
\bigl(\na\p_3X | \cA\cA^t\na \p_3X\bigr)_{\dot H^k}
\geq &\frac14\|\nabla\p_3X\|_{\dot H^k}^2-C_k\big(|\nabla Y|_{k}^2\|\nabla \p_3X\|_0^2+
|\p_3 Y|_1^2\|\nabla X\|_{\dot H^k}^2+|\p_3Y|_{k+1}^2\|\nabla X\|_0^2\big).
\end{split}
\eeno
Inserting the above estimates into \eqref{S5H2} gives rise to
\beq\label{S5H3}
\begin{split}
&\f{d}{dt}\Bigl(\f12\|\na\cdot\bigl(\cA\cA^t\na X\bigr)\|_{\dot H^k}^2-\bigl(
 X_t |\nabla\cdot( \cA\cA^t\na X)\bigr)_{\dot H^k}\Bigr)+\f14\| \p_3X \|_{\dot H^{k+1}}^2\\
&\leq 2\| X_t\|_{\dot H^{k+1}}^2+\big(|\p_3 Y|_0^2+| Y_t|_0^2\big)\|\nabla X\|_{\dot H^{k+1}}^2+C_k|\nabla Y|_{k}^2\big(\|\nabla\p_3X\|_0^2+\|\nabla X_t\|_0^2\big)\\
&\quad+C_k(|\p_3Y|_{k+1}^2+|Y_t|_{k+2}^2)\|\nabla X\|_0^2 +\big(\|\wt{f}'(Y;X)\|_{\dot H^k}+\|g \|_{\dot H^k}\big)\| \na\cdot\bigl(\cA\cA^t\na X\bigr)\|_{\dot H^k}.
\end{split}
\eeq
We remark  that
\beq\label{S5H4}
\begin{split}
\| \na\cdot\bigl(\cA\cA^t\na X\bigr)-\D X\|_{\dot H^k}
\leq &|{\mathcal A}{\mathcal A}^t-Id|_0\|\nabla X\|_{\dot H^{k+1}}+C_k|{\mathcal A}{\mathcal A}^t-Id|_{k+1}\|\nabla X\|_0\\
\leq &\frac12\|\nabla X\|_{\dot H^{k+1}}+C_k|\nabla Y|_{k+1}\|\nabla X\|_0.
\end{split}\eeq
Moreover, in view of  \eqref{wf0'}, we have
\begin{equation}\label{S5H7}
\begin{split}
\|\widetilde f_0'(Y;X)\|_{\dot H^k}
&\lesssim | Y_t|_1\|\nabla X\|_{\dot H^{k+1}}+\big(| Y_t|_{k+2}+| Y_t|_1|\nabla Y|_{k+1}\big)\|\nabla X\|_0,
\end{split}
\end{equation}
which together with \eqref{S4.2eq9} and \eqref{S4.2eq10} ensures that
\begin{equation}\label{S5H9}
\begin{split}
&\|\widetilde f'(Y;X)\|_{\dot H^k}
\lesssim |\p_3Y|_0\|\p_3X\|_{\dot H^{k+1}}+|Y_t|_0\|X_t\|_{\dot H^{k+1}}+\big(|\p_3Y|_{k+1}+|\nabla Y|_{k+1}|\p_3Y|_1\big)\|\p_3X\|_1\\
&\quad+\big(|\p_3Y|_1^\frac43\|\p_3Y\|_0^\frac23+|\p_3Y|_1^2+| Y_t|_1\big)\bigl(\|\na X\|_{\dot H^{k+1}}+|\nabla Y|_{k+1}\|\nabla X\|_1\bigr)+\|X_t\|_1\big(|Y_t|_{k+1}\\
&\quad+|\nabla Y|_{k+1}|Y_t|_1\big)+\big(|\p_3Y|_{k+1}\big(|\p_3Y|_0^\frac13\|\p_3Y\|_0^\frac23+|\p_3Y|_1\big)+| Y_t|_{k+2}\bigr)\|\nabla X\|_1.
\end{split}
\end{equation}
Inserting the above inequalities to \eqref{S5H3} yields
\beq\label{S5H11}
\begin{split}
\f{d}{dt}&\Bigl(\f12\|\na\cdot\bigl(\cA\cA^t\na X\bigr)\|_{\dot H^k}^2-\bigl(
 X_t |\nabla\cdot( \cA\cA^t\na X)\bigr)_{\dot H^k}\Bigr)+\f18\| \p_3X \|_{\dot H^{k+1}}^2\\
 \leq &3\| X_t\|_{\dot H^{k+1}}^2+\langle t\rangle^{1+\varepsilon}\|g \|_{\dot H^k} ^2+C_k\big(|\p_3Y|_1^\frac43\|\p_3Y\|_0^\frac23+|\p_3Y|_1^2+| Y_t|_1+\langle t\rangle^{-(1+\varepsilon)}\big)\|\na X\|_{\dot H^{k+1}}^2\\
&+C_k\big(\bigl(|\p_3Y|_{k+1}^2+|\nabla Y|_{k+1}^2|\p_3Y|_1^2\bigr)\langle t\rangle^{1+\varepsilon}+|\nabla Y|_{k}^2\big)\|\p_3X\|_{1}^2\\
&+C_k\big(\bigl(|Y_t|_{k+1}^2+|\nabla Y|_{k+1}^2|Y_t|_1^2\bigr)\langle t\rangle^{1+\varepsilon}+|\nabla Y|_{k}^2\big)\|X_t\|_{1}^2\\
&+C_k\Big\{\bigl(|\p_3Y|_{k+1}^2\big(|\p_3Y|_0^\frac23\|\p_3Y\|_0^\frac43+|\p_3Y|_1^2\big)+| Y_t|_{k+2}^2\bigr)\langle t\rangle^{1+\varepsilon}+|\p_3Y|_{k+1}^2
\\
&\qquad+|\nabla Y|_{k+1}^2\big(\big(|\p_3Y|_1^\frac83\|\p_3Y\|_0^\frac43+|\p_3Y|_1^4
+|Y_t|_1^2\big)\langle t\rangle^{1+\varepsilon}+\langle t\rangle^{-(1+\varepsilon)}\big)\big)\Big\}\|\nabla X\|_1^2
\end{split}
\eeq
Let us introduce
\begin{align}\label{S512}
\dot D_{k+1}(t)\eqdefa\|X_t\|_{\dot H^{k+1}}^2+\|\p_3X\|_{\dot H^{k+1}}^2+\frac12\|\nabla\cdot ({\mathcal A}{\mathcal A}^t\nabla X)\|_{\dot H^k}^2-\big(X_t | \nabla\cdot({\mathcal A}{\mathcal A}^t\nabla X)\big)_{\dot H^k}.
\end{align}
Then it follows from \eqref{S5H4} that
\begin{equation}\label{S5H13}
\dot D_{k+1}(t)\geq\frac1{8} \dot E_{k+1}(t)-C_{k+1}\|X_t\|_0^2-C_{k+1}|\nabla Y|_{k+1}^2\|\nabla X\|_0^2,
\end{equation}
with $\dot E_{k+1}(t)$ being given by \eqref{S5H0}.

Hence by summing up \eqref{S5H10} and \eqref{S5H11}, and then
integrating the resulting  inequality over $[0,t]$  and using \eqref{S5H13}, we achieve
\begin{equation}\label{S5H14}
\begin{split}
&\dot E_{k+1}(t)+\int_0^t\big(\frac12\| X_t\|_{\dot H^{k+2}}^2+\frac18\| \p_3X \|_{\dot H^{k+1}}^2\big)\,ds\\
&\leq 8\dot D_{k+1}(t)+\|X_t\|_{0,0}^2+|\nabla Y|_{0,k+1}^2\|\nabla X\|_{0,0}^2+\int_0^t\big(\frac12\| X_t\|_{\dot H^{k+2}}^2+\frac18\| \p_3X \|_{\dot H^{k+1}}^2\big)\,ds\\
&\lesssim  \int_0^t\big(|\p_3Y|_1^\frac43\|\p_3Y\|_0^\frac23+|\p_3Y|_0^2+| Y_t|_1+\langle s\rangle^{-(1+\varepsilon)}\big)\dot E_{k+1}(s)\,ds\\
&\qquad\qquad\qquad\qquad\qquad\qquad\qquad\qquad\qquad\qquad+ \|\langle t\rangle^{\frac{1+\varepsilon}2}g \|_{L_t^2(\dot H^k)}^2 +\gamma_{\varepsilon, k+1}(Y)^2{\mathcal E}_0^2(t),
\end{split}
\end{equation}
where ${\mathcal E}_0(t)$ is given by \eqref{S5eqE1} and $\gamma_{\varepsilon,k+1}(Y)$ by \eqref{S5eqE3}.
Applying Gronwall's inequality to \eqref{S5H14} and using \eqref{S5eqE1}, we obtain
\begin{equation*}
\begin{split}
\dot E_{k+1}(t)
&\leq C_{\varepsilon,k}\big(\|\langle t\rangle^{\frac{1+\varepsilon}2}g\|_{L_t^2(\dot H^{k})}^2
+\gamma_{\varepsilon,k+1}(Y)^2\|\langle t\rangle^{\frac{1+\varepsilon}{2}}|D|^{-1}g\|_{L_t^2(H^1)}^2\big)E_\e(Y),
\end{split}
\end{equation*}
from which and \eqref{S5H14},  we infer
\begin{equation}\label{S5H16}
\begin{split}
&\|(X_t,\p_3X,\nabla X)\|_{L_t^\infty(\dot H^{k+1})}+\|X_t\|_{L_t^2(\dot H^{k+2})}+\|\p_3X\|_{L_t^2(\dot H^{k+1})}\\
&\quad\leq C_{\varepsilon,k}\big(\|\langle t\rangle^{\frac{1+\varepsilon}2}g\|_{L_t^2(\dot H^{k})}
+\gamma_{\varepsilon,k+1}(Y)\|\langle t\rangle^{\frac{1+\varepsilon}{2}}|D|^{-1}g\|_{L_t^2(H^1)}\big)E_\e(Y).
\end{split}
\end{equation}
Summing up the above inequality with respect to $k$ leads to \eqref{S5eqE2}. This completes the proof of  Theorem \ref{S5thm1}.

Now let us turn to the proof of Corollary \ref{S5col1}.

\begin{proof}[Proof of Corollary \ref{S5col1}]
By summing up \eqref{S5eq6} and  \eqref{S5eq10}, and then multiplying the resulting inequality by $\w{t}$ and
integrating the above inequality over $[0,t]$, we find
\begin{align*}
\langle t\rangle&\big(\|X_t\|_1^2+\|\p_3X\|_1^2\big)+\int_0^t\langle s\rangle\|\nabla X_t\|_1^2ds\leq \|X_t\|_{L_t^2(H^1)}^2+\big(1+|\p_3Y|_{\frac12,1}^2\big)\|\p_3X\|_{L_t^2(H^1)}^2\\
&+C\|\langle t\rangle^\frac12|D|^{-1}g\|_{L_t^2(H^1)}^2+C\big(|\p_3Y|_{\frac12+\varepsilon,1}^\frac83\|\p_3Y\|_{L_t^2(L^2)}^\frac43+|\p_3Y|_{\frac12+\varepsilon,1}^4+|Y_t|_{1+\varepsilon,2}^2\big){\mathcal E}_0^2(t).
\end{align*}
\eqref{S5eqC1} then follows from \eqref{S5eqE1}.

 Similarly, we get, by multiplying
 \eqref{S5H10} by $\w{t}$, then integrating the inequality over $[0,t]$ and taking the square root of the resulting inequality, that
\begin{align*}
&\langle t\rangle^\frac12\big(\|(X_t,\p_3X)\|_{\dot H^{k+1}}\big)+\Big(\frac34\int_0^t\langle s\rangle\| X_t\|_{\dot H^{k+2}}^2ds\Big)^\frac12\lesssim \|\langle t\rangle^\frac12g\|_{L_t^2(\dot H^{k})}\\
&+\bigl(1+|Y_t|_{\frac12,0}\bigr)\|X_t\|_{L_t^2(\dot H^{k+1})}+\bigl(1+|\p_3Y|_{\frac12,0}\bigr)\|\p_3X\|_{L_t^2(\dot H^{k+1})}+|\nabla Y|_{0,k+1}\|\langle t\rangle^\frac12\nabla X_t\|_{L_t^2(L^2)}\\
&+\big(|\p_3Y|_{\frac12+\varepsilon,1}^\frac43\|\p_3Y\|_{L_t^2(L^2)}^\frac23+|\p_3Y|_{\frac12+\varepsilon,1}^2+|Y_t|_{1+\varepsilon,1}\big)\bigl(\|\nabla X\|_{L_t^\infty(\dot H^{k+1})}+|\nabla Y|_{0,k+1}\|\nabla X\|_{L_t^\infty(H^1)}\bigr)\\
&+\|\p_3Y\|_{L_t^2(L^2)}\big(|\nabla Y|_{0,k+1}|\p_3Y|_{\frac12,1}+|\p_3Y|_{\frac12,k+1}\big)+\big(|Y_t|_{\frac12,k+1}+|\nabla Y|_{0,k+1}|Y_t|_{\frac12,1}\big)\|X_t\|_{L_t^2(L^2)}\\
&+\big(|\p_3Y|_{\frac12+\varepsilon,k+1}\big(|\p_3Y|_{\frac12+\varepsilon,0}^{\frac13}\|\p_3Y\|_{L_t^2(L^2)}^\frac23+|\p_3Y|_{\frac12+\varepsilon,1}\big)
+|Y_t|_{1+\varepsilon,k+2}\bigr)\|\nabla X\|_{L_t^\infty(H^1)}.
\end{align*}
\eqref{S5eqC2} then follows from \eqref{S5eqE2} and \eqref{S5eqC1}, and this completes the proof of Corollary \ref{S5col1}.
\end{proof}

\section{Energy decay for $\nabla X_t$}\label{Sect6.3}

The main idea to prove Proposition \ref{S5thm2} is to use the following proposition:

\begin{prop}\label{S3-Lem1}
{\sl Let $X$ be a smooth enough solution of
\beq \label{appeq2}
\left\{\begin{array}{l}
\displaystyle
X_{tt}-\D X_t-\p_3^2X=\na\cdot\bigl((\cA\cA^t-Id)\na X_t\bigr)+\hbar\eqdefa f,\\
\displaystyle X(0)=0 \andf X_t(0)=0,
\end{array}\right.
\eeq on $[0,T].$ Then
 under the assumption that
\begin{equation}\label{appeq3}
\|\na Y\|_{{L}^\infty_t(\dot B^{\f32}_{2,1})}<\delta_1,
\end{equation}
 we have, for any $t\in[0, T]$ and any $\e>0,$ that
 \beq\label{appeq20}
 t\|\na X_t(t)\|_{L^2}\leq C_\e\bigl(\sup_{s\in [0,t]}\|s^{1+\e}|D|^{-1}\hbar\|_{L^2}+\sup_{s\in [0,t]}\|s^{1+\e}|D|\hbar\|_{L^2}\bigr)\leq C_\e\||D|^{-1}\hbar\|_{1+\e,2}.
\eeq
Moreover, we have for $k\in\N$,
\begin{equation}\label{appeq20'}
 t\|\na X_t(t)\|_{\dot H^k}\leq C_{\e,k}\Big(\bigl(\d_1+\| D^{k}\nabla Y\|_{L^\infty_t(\dot B^{\frac32}_{2,1})}\bigr)\||D|^{-1}\hbar\|_{1+\e,2}+\||D|^{-1}\hbar\|_{1+\e,k+2}\Big).
\end{equation}
 }
\end{prop}

Admitting this proposition for the time being, we present the proof of Proposition \ref{S5thm2}.

\begin{proof}[Proof of Proposition \ref{S5thm2}]
In our situation \eqref{LE1},
$$\hbar=\nabla\cdot\Big({\mathcal A}\big((-\nabla X){\mathcal A}+{\mathcal A}^t(-\nabla X)^t\big){\mathcal A}^t\nabla Y_t\Big)-f_1'(Y;X)+f_2'(Y;X)+g.$$
 We infer from \eqref{S4.2eqf2a}, \eqref{S4.2eqf1}, \eqref{S4.2eqf2} that for $k\geq0$,
\begin{equation}\label{S5H17}
\begin{split}
\||D|^{-1}\hbar\|_{1+\e,k+2}
\lesssim &|\p_3Y|_{\frac12+\e,0}\|\p_3X\|_{\frac12,k+2}+|Y_t|_{\frac12+\e,0}\|X_t\|_{{\frac12,k+2}}+\||D|^{-1}g\|_{1+\varepsilon,k+2}\\
&+ \big(|\p_3Y|_{\frac12+\e,1}^\frac43\|\p_3Y\|_{\frac12,0}^\frac23+|\p_3Y|_{\frac12+\e,1}^2+|Y_t|_{1+\e,1}\big)\|\na X\|_{{0,k+2}}\\
& +\gamma_{\e,k+2}(Y)(\|\p_3X\|_{\frac12,1}+\|X_t\|_{\frac12,1}+\|\nabla X\|_{0,1}),
\end{split}
\end{equation}
where $\gamma_{\e,k+2}(Y)$ is given in \eqref{S5eqE3}.
Proposition \ref{S5thm2} then follows from Proposition \ref{S3-Lem1}, \eqref{S5H17}, Corollary \ref{S5col1} and the fact that $\|D^k\nabla Y\|_{L_t^\infty(\dot B^\frac32_{2,1})}\leq \|\nabla Y\|_{0,k+2}$.
\end{proof}

In order to prove Proposition \ref{S3-Lem1}, we need to exploit the tool of anisotropic Littlewood-Paley analysis.
Similar to the dyadic operators $\D_j,$ and $S_j$ given by Definition \ref{S0def1}, let us recall the dyadic operators in the $x_3$ variable
\beq\label{appeq21}
\Delta_\ell^\v a
\eqdefa\cF^{-1}(\varphi(2^{-\ell}|\xi_3|)\widehat{a}),\andf
S^\v_\ell a \eqdefa \cF^{-1}(\chi(2^{-\ell}|\xi_3|)\widehat{a}).
\eeq
Let us also recall the following anisotropic type Besov norm
from \cite{LZMHD1, LXZ}:

\begin{defi}\label{def2}
{\sl  Let  $s_1,s_2\in\R,$ $r\in [1,\infty]$ and
$a\in{\mathcal S}_h'(\R^3),$ we define the norm
$$
\|a\|_{\cB^{s_1,s_2}_{r}}\eqdefa \Bigl(2^{js_1} 2^{\ell
s_2}\|\D_{j}\D^\v_{\ell}a\|_{L^2}\Bigr)_{\ell^{r}(\Z^2)}.
$$
In particular, when $r=2,$ we denote $
\|a\|_{\dH^{s_1,s_2}}\eqdefa \|a\|_{\cB^{s_1,s_2}_{2}}=\bigl\||D|^{s_1}|D_{x_3}|^{s_2}a\bigr\|_{L^2}.$ }
\end{defi}

In  order to obtain a better description of the regularizing effect
for the transport-diffusion equation, we will use anisotropic version of Chemin-Lerner type
norm (see \cite{BCD} for
instance).

\begin{defi}\label{def3}
{\sl Let  $(r,q)\in[1,+\infty]^2$ and $T\in(0,+\infty]$.
 We define the norm  $\wt{L}^q_T(\cB^{s_1,s_2}_r(\R^3))$ by
\beno
\|u\|_{\wt{L}^q_T(\cB^{s_1,s_2}_r)}\eqdefa\Bigl(\sum_{(j,\ell)\in\Z^2}\bigl(2^{js_1}2^{\ell
s_2} \|\D_j\D_\ell^\v u\|_{L^q_T(L^2)}\bigr)^r\Bigr)^{\f1r}, \eeno with the usual change
if $r=\infty$.}
\end{defi}

 For the convenience of the readers, we  recall the
following Bernstein type lemma from \cite{BCD, CZ, Pa02}:

\begin{lem}\label{L2} {\sl Let $\frak{B}_{\h}$ (resp.~$\frak{B}_{\v}$) be a ball
of~$\R^2$ (resp.~$\R$), and~$\cC_{\h}$ (resp.~$\cC_{\v}$) a ring
of~$\R^2$ (resp.~$\R$); let~$1\leq p_2\leq p_1\leq \infty$ and
~$1\leq q_2\leq q_1\leq \infty.$ Then there holds:
\smallbreak\noindent If the support of~$\wh a$ is included
in~$2^k\frak{B}_{\h}$, then
\[
\|\partial_\h^\alpha a\|_{L^{p_1}_\h(L^{q_1}_\v)} \lesssim
2^{k\left(|\al|+2\left(\frac1{p_2}-\frac1{p_1}\right)\right)}
\|a\|_{L^{p_2}_\h(L^{q_1}_\v)}.
\]
If the support of~$\wh a$ is included in~$2^\ell\frak{B}_{\v}$, then
\[
\|\partial_3^\beta a\|_{L^{p_1}_\h(L^{q_1}_\v)} \lesssim
2^{\ell\left(\beta+\left(\frac1{q_2}-\frac1{q_1}\right)\right)} \|
a\|_{L^{p_1}_\h(L^{q_2}_\v)}.
\]
If the support of~$\wh a$ is included in~$2^k\cC_{\h}$, then
\[
\|a\|_{L^{p_1}_\h(L^{q_1}_\v)} \lesssim 2^{-kN}
\max_{|\al|=N}\|\partial_\h^\al a\|_{L^{p_1}_\h(L^{q_1}_\v)}.
\]
If the support of~$\wh a$ is included in~$2^\ell\cC_{\v}$, then
\[
\|a\|_{L^{p_1}_\h(L^{q_1}_\v)} \lesssim 2^{-\ell N} \|\partial_3^N
a\|_{L^{p_1}_\h(L^{q_1}_\v)}.
\]}
\end{lem}

Let us now turn to the proof of Proposition \ref{S3-Lem1}.

\begin{proof}[Proof of Proposition \ref{S3-Lem1}] The proof of this lemma is motivated by the proof of
Proposition 4.1 of \cite{LXZ, XZ1}. By applying the operator $\D_j\D_\ell^\v$
to \eqref{appeq2} and then taking the $L^2$ inner product of the
resulting equation with $\D_j\D_\ell^\v X_t,$ we write
\beq\label{appeq1}
\begin{aligned}
\f12\f{d}{dt}\bigl(\|\D_j\D_\ell^\v
X_t\|_{L^2}^2+\|\D_j\D_\ell^\v\pa_{3}X\|_{L^2}^2\bigr) +\|\na
\D_j\D_\ell^\v{X}_t\|_{L^2}^2 =\bigl(\D_j\D_\ell^\v f |
\D_j\D_\ell^\v X_t\bigr)_{L^2}.
\end{aligned}
\eeq Along the same line, one has
 \beno (\D_j\D_\ell^\v X_{tt} |
\D\D_j\D_\ell^\v X)-\f12\f{d}{dt}\|\D\D_j\D_\ell^\v
X\|_{L^2}^2-\|\pa_3\na\D_j\D_\ell^\v X\|_{L^2}^2=(\D_j\D_\ell^\v f |
\D\D_j\D_\ell^\v X). \eeno Notice that \beno (\D_j\D_\ell^\v X_{tt}
| \D\D_j\D_\ell^\v X)=\f{d}{dt}(\D_j\D_\ell^\v X_{t}|
\D\D_j\D_\ell^\v X) +\|\na\D_j\D_\ell^\v X_{t}\|_{L^2}^2, \eeno so
that there holds \beq\label{appeq4}
\begin{split}
\f{d}{dt}\Bigl(\f12\|\D\D_j\D_\ell^\v X\|_{L^2}^2-&(\D_j\D_\ell^\v
X_{t} |
\D\D_j\D_\ell^\v X)\Bigr)\\
-&\|\na\D_j\D_\ell^\v X_t\|_{L^2}^2+\|\pa_3\na\D_j\D_\ell^\v
X\|_{L^2}^2 =-(\D_j\D_\ell^\v f |\D\D_j\D_\ell^\v X).
\end{split}
\eeq By summing up \eqref{appeq1} with $\f14$ of \eqref{appeq4}, we
obtain \beq\label{appeq5}
\begin{split}
\f{d}{dt}g_{j,\ell}^2(t) +\f34\|\na\D_j\D_\ell^\v X_t\|_{L^2}^2&
+\f14\|\pa_3\na\D_j\D_\ell^\v X\|_{L^2}^2\\
&=\bigl(\D_j\D_\ell^\v  f\ |\ \D_j\D_\ell^\v
X_t-\f14\D\D_j\D_\ell^\v X\bigr),
\end{split}
\eeq where
 \beno\begin{split}
g_{j,\ell}^2(t) \eqdefa \f12\Bigl(\|\D_j\D_\ell^\v X_t(t)\|_{L^2}^2&
+\|\D_j\D_\ell^\v \pa_3X(t)\|_{L^2}^2 +\f14\|\D_j\D_\ell^\v \D
X(t)\|_{L^2}^2\Bigr)
\\&\qquad\qquad\qquad
-\f14\bigl(\D_j\D_\ell^\v X_t(t) | \D_j\D_\ell^\v \D X(t)\bigr).
\end{split}
\eeno It is easy to observe that
 \beq \label{appeq6}
g_{j,\ell}^2(t) \sim \|\D_j\D_\ell^\v X_t(t)\|_{L^2}^2
+\|\D_j\D_\ell^\v\pa_3 X(t)\|_{L^2}^2 +\|\D_j\D_\ell^\v \D
X(t)\|_{L^2}^2. \eeq Now according to the heuristic analysis
presented at the beginning of Section \ref{Sect2}, we split the
frequency analysis into the following two cases:

\no$\bullet$ \underline{When $j\leq \f{\ell+1}2$}

 In this case,
one has \beno g^2_{j,\ell}(t)\sim  \|\D_j\D_\ell^\v X_t(t)\|_{L^2}^2
+\|\D_j\D_\ell^\v \pa_3X(t)\|_{L^2}^2, \eeno and Lemma \ref{L2}
implies that \beno \f34\|\na \D_j\D_\ell^\v X_t\|_{L^2}^2
+\f14\|\pa_3\na\Djl X\|_{L^2}^2 \geq  c2^{2j}\bigl(\|\D_j\D_\ell^\v
X_t\|_{L^2}^2 +\|\Djl\p_3 X\|_{L^2}^2\bigr). \eeno Hence it follows
from \eqref{appeq5} that
\beno
\f{d}{dt}g_{j,\ell}(t)+c2^{2j}g_{j,\ell}(t)\leq \|\D_j\D_\ell^\v f(t)\|_{L^2},
\eeno
which in particular implies that
\beq \label{appeq7}
g_{j,\ell}(t)\leq \int_0^te^{-c(t-s)2^{2j}}\|\D_j\D_\ell^\v f(s)\|_{L^2}\,ds,
\eeq
and
\beq \label{appeq8}
2^{j}\|\Djl X_t\|_{L^1_t(L^2)}
\lesssim 2^{-j}\|\Djl
f\|_{L^1_t(L^2)}.
\eeq

Now let us turn to the estimate of $\|\Djl f\|_{L^1_t(L^2)}.$ Indeed it follows by the law of product in
the anisotropic Besov spaces (see Lemma 3.3 of \cite{XZ1}) that
\beq \label{appeq15}
\begin{split}
\bigl\|(\cA\cA^t-Id)\na X_t\bigr\|_{L^1_t(\dH^{0,0})}\lesssim& \|(\cA\cA^t-Id)\|_{L^\infty_t(\cB^{1,\f12}_{1})}\|\na X_t\|_{\wt{L}^1_t(\dH^{0,0})}\\
\lesssim &\|(\cA\cA^t-Id)\|_{L^\infty_t(\dB^{\f32}_{2,1})}\|\na X_t\|_{\wt{L}^1_t(\dH^{0,0})}\\
\lesssim & \d_1\|\na X_t\|_{\wt{L}^1_t(\dH^{0,0})},
\end{split}
\eeq
where we used the fact that $\dB^{\f32}_{2,1}(\R^3)\hookrightarrow\cB^{1,\f12}_{1}$ (one may check Lemma 3.2 of \cite{LZMHD1,XZ1} for details).
Hence we obtain
\beq\label{appeq16}
2^{-j}\|\Djl f\|_{L^1_t(L^2)}\lesssim c_{j,\ell}\d_1\|\na X_t\|_{\wt{L}^1_t(\dH^{0,0})}+\|\Djl |D|^{-1}\hbar\|_{L^1_t(L^2)}.
\eeq
where $(c_{j,\ell})_{j,\ell\in\Z^2}$ is a generic element of $\ell^2(\Z^2)$ so that $\sum_{j,\ell\in\Z^2}c_{j,\ell}^2=1.$

Whereas it follows from Lemma \ref{L2} and \eqref{appeq7} that
\beq\label{appeq17}
\begin{split}
2^j t\|\Djl X_t(t)\|_{L^2}\lesssim &\int_0^t2^{2j}(t-s)e^{-c(t-s)2^{2j}}\bigl\|\D_j\D_\ell^\v\bigl((\cA\cA^t-Id)\na X_t\bigr)(s)\bigr\|_{L^2}\,ds\\
&+\int_0^t2^{2j}e^{-c(t-s)2^{2j}}\bigl\|\D_j\D_\ell^\v\bigl((\cA\cA^t-Id)s\na X_t\bigr)(s)\bigr\|_{L^2}\,ds\\
&+t\Bigl(\int_0^{\f{t}2}+\int_{\f{t}2}^t\Bigr)2^{j}e^{-c(t-s)2^{2j}}\|\Djl\hbar(s)\|_{L^2}\,ds.
\end{split}
\eeq
By virtue of \eqref{appeq15}, we have
$$\longformule{
\int_0^t2^{2j}(t-s)e^{-c(t-s)2^{2j}}\bigl\|\D_j\D_\ell^\v\bigl((\cA\cA^t-Id)\na X_t\bigr)(s)\bigr\|_{L^2}\,ds}{{}
\lesssim \bigl\|\D_j\D_\ell^\v\bigl((\cA\cA^t-Id)\na X_t\bigr)\bigr\|_{L^1_t(L^2)}
\lesssim c_{j,\ell}\d_1\|\na X_t\|_{\wt{L}^1_t(\dH^{0,0})}.}
$$
Along the same line, we have
\begin{equation}\label{appeq15'}
\begin{split}
\int_0^t2^{2j}&e^{-c(t-s)2^{2j}}\bigl\|\D_j\D_\ell^\v\bigl((\cA\cA^t-Id)s\na X_t\bigr)(s)\bigr\|_{L^2}\,ds\\
\lesssim  &\bigl\|\D_j\D_\ell^\v\bigl((\cA\cA^t-Id)s\na X_t\bigr)\bigr\|_{L^\infty_t(L^2)}\\
\lesssim &c_{j,\ell}\|(\cA\cA^t-Id)\|_{L^\infty_t(\cB^{1,\f12}_{1})}\|t\na X_t\|_{\wt{L}^\infty_t(\dH^{0,0})}
\lesssim c_{j,\ell}\d_1\|t\na X_t\|_{\wt{L}^\infty_t(\dH^{0,0})}.
\end{split}
\end{equation}
While it is easy to observe from Lemma \ref{L2} that
\beno
\begin{split}
t\int_0^{\f{t}2}2^{j}e^{-c(t-s)2^{2j}}\|\Djl\hbar(s)\|_{L^2}\,ds\lesssim &\int_0^{\f{t}2}(t-s)2^{2j}e^{-c(t-s)2^{2j}}\|\Djl|D|^{-1}\hbar(s)\|_{L^2}\,ds\\
\lesssim &\|\Djl|D|^{-1}\hbar\|_{L^1_t(L^2)},
\end{split}
\eeno
and
\beno
\begin{split}
t\int_{\f{t}2}^t2^{j}e^{-c(t-s)2^{2j}}&\|\Djl\hbar(s)\|_{L^2}\,ds\lesssim \int_{\f{t}2}^t\w{t-s}^{-1}\bigl(2^{-j}+2^j\bigr)\|s\Djl\hbar(s)\|_{L^2}\,ds\\
&\qquad\lesssim \int_{\f{t}2}^t\w{t-s}^{-1}\bigl(\|s\Djl|D|^{-1}\hbar(s)\|_{L^2}+\|s\Djl|D|\hbar(s)\|_{L^2}\bigr)\,ds.
\end{split}
\eeno
Substituting the above estimates into \eqref{appeq17} leads to
\beq
\label{appeq18}
\begin{split}
2^j t\|\Djl X_t(t)\|_{L^2}\lesssim & c_{j,\ell}\d_1\bigl(\|\na X_t\|_{\wt{L}^1_t(\dH^{0,0})}+\|t\na X_t\|_{\wt{L}^\infty_t(\dH^{0,0})}\bigr)+\|\Djl|D|^{-1}\hbar\|_{L^1_t(L^2)}\\
&+\int_{\f{t}2}^t\w{t-s}^{-1}\bigl(\|s\Djl|D|^{-1}\hbar(s)\|_{L^2}+\|s\Djl|D|\hbar(s)\|_{L^2}\bigr)\,ds
\end{split}
\eeq
for all $(j,\ell)$ satisfying   $j\leq \f{\ell+1}2.$

\no$\bullet$\underline{When $j> \f{\ell+1}2$}

In this case, we have \beno g_{j,\ell}^2(t) \sim \|\D_j\D_\ell^\v
X_t(t)\|_{L^2}^2 +\|\Djl\D X(t)\|_{L^2}^2 \eeno and Lemma \ref{L2}
implies that \beno
\begin{split}
\f34\|\na \D_j\D_\ell^\v X_t\|_{L^2}^2 +\f14\|\pa_3\na\Djl
X\|_{L^2}^2 \geq & c\bigl(2^{2j}\|\D_j\D_\ell^\v X_t\|_{L^2}^2
+2^{2j}2^{2\ell}\|\Djl X\|_{L^2}^2\bigr)\\
\geq & c\f{2^{2\ell}}{2^{2j}}\bigl(\|\D_j\D_\ell^\v X_t\|_{L^2}^2
+\|\Djl\D X\|_{L^2}^2\bigr).
\end{split}
\eeno Then we deduce from \eqref{appeq5} that
\beno
\f{d}{dt}g_{j,\ell}(t)+c2^{2(\ell-j)}g_{j,\ell}(t)\leq \|\D_j\D_\ell^\v f(t)\|_{L^2},
\eeno
which implies that
\beq\label{appeq9}
g_{j,\ell}(t)\leq \int_0^te^{-c(t-s)2^{2(\ell-j)}}\|\D_j\D_\ell^\v f(s)\|_{L^2}\,ds,
\eeq
and
\beq\label{appeq10}
2^{2\ell}\|\Djl X\|_{L^1_t(L^2)}
\lesssim \|\Djl
f\|_{L^1_t(L^2)}.
\eeq On the other hand,  we get, by taking $L^2$ inner product of \eqref{appeq2} with $\Djl X_t,$ that
$$
\f12\f{d}{dt}\|\D_j\D_\ell^\v X_t\|_{L^2}^2+\|\na \D_j\D_\ell^\v
X_t\|_{L^2}^2 =\bigl(\pa_3^2\D_j\D_\ell^\v X+\D_j\D_\ell^\v f\ |\
\D_j\D_\ell^\v X_t\bigr)_{L^2},
$$
from which, Lemma \ref{L2}, we infer
\beno
\f{d}{dt}\|\Djl X_t(t)\|_{L^2}+c2^{2j}\|\Djl X_t(t)\|_{L^2}\lesssim 2^{2\ell}\|\Djl X(t)\|_{L^2}+\|\Djl f(t)\|_{L^2},
\eeno
so that there hold
\beq \label{appeq11}
\begin{split}
2^j\|\Djl X_t(t)\|_{L^2}\lesssim &2^{2\ell+j}\int_0^te^{-c(t-s)2^{2j}}\|\Djl X(s)\|_{L^2}\,ds\\
&+2^{j}\int_0^te^{-c(t-s)2^{2j}}\|\Djl f(s)\|_{L^2}\,ds. \end{split}
\eeq
And then we deduce from  \eqref{appeq10} that for $
j>\f{\ell+1}2$ \beq\label{appeq12}
\begin{split}
2^{j}\|\Djl X_t\|_{L^1_t(L^2)}
&\lesssim
2^{2\ell-j}\|\Djl X\|_{L^1_t(L^2)}+2^{-j}\|\Djl f\|_{L^1_t(L^2)}\\
&\lesssim 2^{-j}\|\Djl
f\|_{L^1_t(L^2)}.
\end{split}
\eeq
Moreover, in this case, it follows from Lemma \ref{L2} and \eqref{appeq9} that
\beno
\begin{split}
2^{2\ell+j}t\int_0^te^{-c2^{2j}(t-s)}\|\Djl X(s)\|_{L^2}\,ds\lesssim &2^{2\ell-j}t\|\Djl X\|_{L^\infty_t(L^2)}\\
\lesssim &2^{2\ell-3j}t\|\Djl \D X\|_{L^\infty_t(L^2)}\lesssim 2^{2\ell-3j}t\|g_{j,\ell}\|_{L^\infty_t}\\
\lesssim&2^{2\ell-3j}t \int_0^te^{-c(t-s)2^{2(\ell-j)}}\|\D_j\D_\ell^\v f(s)\|_{L^2}\,ds,
\end{split}
\eeno
from which and a similar proof of \eqref{appeq18}, we infer
\beq\label{appeq19}
\begin{split}
&2^{2\ell+j}t\int_0^te^{-c2^{2j}(t-s)}\|\Djl X(s)\|_{L^2}\,ds
\lesssim
 c_{j,\ell}\d_1\bigl(\|\na X_t\|_{\wt{L}^1_t(\dH^{0,0})}+\|t\na X_t\|_{\wt{L}^\infty_t(\dH^{0,0})}\bigr)\\
&\quad+\|\Djl|D|^{-1}\hbar\|_{L^1_t(L^2)}+\int_{\f{t}2}^t\w{t-s}^{-1}\big(\|s\Djl|D|^{-1}\hbar(s)\|_{L^2}+\|s\Djl|D|\hbar(s)\|_{L^2}\big)\,ds.
\end{split}
\eeq
Here we used the fact $j\geq \ell-N_0$ for some fixed integer $N_0$ in the operator $\Djl.$

By virtue of \eqref{appeq11} and \eqref{appeq19}, we get, by a similar derivation of \eqref{appeq18} that
\eqref{appeq18} holds for all $(j,\ell)\in \Z^2.$ Furthermore, in view of \eqref{appeq8}-\eqref{appeq12}, we obtain for all
$(j,\ell)\in\Z^2,$  that \beq\label{appeq13}
\begin{split}
2^{j}\|\Djl X_t\|_{L^1_t(L^2)}
\lesssim & 2^{-j}\|\Djl f\|_{L^1_t(L^2)}.
\end{split}
\eeq
Inserting \eqref{appeq16} into
\eqref{appeq13} gives rise to
\beno
\begin{split}
\|\na X_t\|_{\wt{L}^1_t(\dH^{0,0})}=&\Bigl(\sum_{j,\ell\in\Z^2}2^{2j}\|\Djl X_t\|_{L^1_t(L^2)}^2\Bigr)^{\f12}\\
\leq &C\d_1\|\na X_t\|_{\wt{L}^1_t(\dH^{0,0})}+C\Bigl(\sum_{j,\ell\in\Z^2}\|\Djl |D|^{-1}\hbar\|_{L^1_t(L^2)}^2\Bigr)^{\f12}\\
\leq &C\d_1\|\na X_t\|_{\wt{L}^1_t(\dH^{0,0})}+C\int_0^t\Bigl(\sum_{j,\ell\in\Z^2}\|\Djl |D|^{-1}\hbar(s)\|_{L^2}^2\Bigr)^{\f12}\,ds\\
\leq& C\bigl(\d_1\|\na X_t\|_{\wt{L}^1_t(\dH^{0,0})}+\||D|^{-1}\hbar\|_{L^1_t(L^2)}\bigr).
\end{split}
\eeno
In particular, by taking $\d_1$ to be sufficiently small in \eqref{appeq3}, we conclude that
\beq \label{appeq14}
\|\na X_t\|_{\wt{L}^1_t(\dH^{0,0})}\leq C\||D|^{-1}\hbar\|_{L^1_t(L^2)}. \eeq
Along the same line, we deduce from \eqref{appeq18} that
\beq\label{appeq26}
\begin{split}
\|t\na X_t\|_{\wt{L}^\infty_t(\dH^{0,0})}=&\Bigl(\sum_{j,\ell\in\Z^2}2^{2j}\|t\Djl X_t\|_{L^\infty_t(L^2)}^2\Bigr)^{\f12}\\
\leq & C\Bigl(\d_1\bigl(\|\na X_t\|_{\wt{L}^1_t(\dH^{0,0})}+\|t\na X_t\|_{\wt{L}^\infty_t(\dH^{0,0})}\bigr)+\||D|^{-1}\hbar\|_{L^1_t(L^2)}\\
&\quad+\int_{\f{t}2}^t\w{t-s}^{-1}\bigl(\|s|D|^{-1}\hbar(s)\|_{L^2}+\|s|D|\hbar(s)\|_{L^2}\bigr)\,ds\Bigr).
\end{split}
\eeq
So that by taking $\d_1$ is small enough in \eqref{appeq3}, we obtain
\beq\label{appeq27}
\begin{split}
t\|\na X_t(t)\|_{L^2}\leq &\|t\na X_t\|_{\wt{L}^\infty_t(\dH^{0,0})}\\
\leq &C\Bigl(\||D|^{-1}\hbar\|_{L^1_t(L^2)}+\int_{\f{t}2}^t\w{t-s}^{-1}\bigl(\|s|D|^{-1}\hbar(s)\|_{L^2}+\|s|D|\hbar(s)\|_{L^2}\bigr)\,ds\Bigr)\\
\leq &C_\e\bigl(\sup_{s\in [0,t]}\|s^{1+\e}|D|^{-1}\hbar\|_{L^2}+\sup_{s\in [0,t]}\|s^{1+\e}|D|\hbar\|_{L^2}\bigr),
\end{split}
\eeq
which leads to \eqref{appeq20}.

The proof of the general estimates \eqref{appeq20'} follows along the same line. Indeed  for any $k\geq 1,$ we have
\begin{equation*}
\begin{split}
\bigl\|D^k\big((\cA\cA^t-Id)\na X_t\big)\bigr\|_{\wt{L}^1_t(\dH^{0,0})}\lesssim& C_k\sum_{k_1+k_2=k}\|D^{k_1}(\cA\cA^t-Id)\|_{L^\infty_t(\dot B^\frac32_{2,1})}
\|D^{k_2}\na X_t\|_{\wt{L}^1_t(\dH^{0,0})}\\
\lesssim& C_k\sum_{k_1+k_2=k}\|D^{k_1}\na Y\|_{L^\infty_t(\dot B^\frac32_{2,1})}
\|D^{k_2}\na X_t\|_{\wt{L}^1_t(\dH^{0,0})},
\end{split}
\end{equation*}
from which and a similar derivation of \eqref{appeq14}, we  inductively infer that
$$\longformule{
\|D^{k}\na X_t\|_{\wt{L}^1_t(\dH^{0,0})}\leq C\||D|^{k-1}\hbar\|_{L^1_t(L^2)}}{{} +C_k\sum_{k_1+\cdots+k_\ell=k}\|D^{k_1}\na Y\|_{L^\infty_t(\dot B^\frac32_{2,1})}\cdots \|D^{k_\ell}\na Y\|_{L^\infty_t(\dot B^\frac32_{2,1})}\||D|^{-1}\hbar\|_{L^1_t(L^2)}.}
$$
Hence by applying the interpolation inequality that $$ \|D^{k_i}\na Y\|_{L^\infty_t(\dot B^\frac32_{2,1})}\lesssim \|\na Y\|_{L^\infty_t(\dot B^\frac32_{2,1})}^{1-k_i/k}
\|D^{k}\na Y\|_{L^\infty_t(\dot B^\frac32_{2,1})}^{k_i/k}\quad\mbox{for} \ k_i>1,$$
and the assumption \eqref{appeq3}, we obtain\beq\label{appeq25}
\begin{split}
\|D^{k}\na X_t\|_{\wt{L}^1_t(\dH^{0,0})}\leq C_k\Bigl(\bigl(\d_1+\|D^{k}\na Y\|_{L^\infty_t(\dot B^\frac32_{2,1})}\bigr)\||D|^{-1}\hbar\|_{L^1_t(L^2)}+\||D|^{k-1}\hbar\|_{L^1_t(L^2)}\Bigr).
\end{split}
\eeq
While it follows from a similar derivation of  \eqref{appeq26} that
\beno
\begin{split}
\|tD^k\na X_t\|_{\wt{L}^\infty_t(\dH^{0,0})}
\leq  C_k\Bigl(&\||D|^{k-1}\hbar\|_{L^1_t(L^2)}+\d_1\bigl(\|D^k\na X_t\|_{\wt{L}^1_t(\dH^{0,0})}+\|tD^k\na X_t\|_{\wt{L}^\infty_t(\dH^{0,0})}\bigr)\\
&+\bigl(\d_1+\|D^{k}\na Y\|_{L^\infty_t(\dot B^\frac32_{2,1})}\bigr)\bigl(\|\na X_t\|_{\wt{L}^1_t(\dH^{0,0})}+\|t\na X_t\|_{\wt{L}^\infty_t(\dH^{0,0})}\bigr)\\
&\qquad\quad+\int_{\f{t}2}^t\w{t-s}^{-1}\bigl(\|s|D|^{k-1}\hbar(s)\|_{L^2}+\|s|D|^{k+1}\hbar(s)\|_{L^2}\bigr)\,ds\Bigr).
\end{split}
\eeno
Thus \eqref{appeq20'} follows \eqref{appeq25} and the argument in \eqref{appeq27}.
This completes the proof of Proposition \ref{S3-Lem1}.
\end{proof}



\section{Estimates of the source term $f(Y)$}\label{Sect7}

In this section, we shall present the  estimates to the nonlinear source term $f(Y)$ determined by \eqref{S1eq2}.

\noindent$\bullet$\underline{The estimate of $\triplenorm{f(Y)}_{\d,N}$}
\begin{prop}\label{S6prop1}
{\sl Let the functionals $f_0,f_1,f_2$ be given in \eqref{fm} and the norm $\||\cdot\||_{\d,N}$ by \eqref{S2eq19''}. Then under the
assumption of \eqref{S4.1prop1assum}, we have
\begin{align}
\label{S6eq1}&\triplenorm{f_0(Y)}_{\d,N}\lesssim \|\nabla Y\|_0\|\nabla Y_t\|_{N+6}+\|\nabla Y\|_{N+6}\|\nabla Y_t\|_{0};\\
\label{S6eq2}&\triplenorm{f_1(Y)}_{\d,N}\lesssim\|\p_3Y\|_{0}\|\p_3Y\|_{N+6}+\|\nabla Y\|_{N+6}|\p_3Y|_0\|\p_3Y\|_{1};\\
\label{S6eq3}&\triplenorm{f_2(Y)}_{\d,N}\lesssim\|Y_t\|_{0}\|Y_t\|_{N+6}+\|\nabla Y\|_{N+6}|Y_t|_{0}\|Y_t\|_{1}.
\end{align}
}
\end{prop}

\begin{proof}
As in Section 4, we shall deal with the estimate of $f(Y)$ by the norm of  the homogeneous
Besov space $\dot B^{s}_{1,1}$ instead of the one  in the homogeneous Sobolev space $\dot{W}^{s,1}.$
 Indeed in view  of \eqref{fm}, we get, by applying the law of product, \eqref{product3}, that
 for $s>0$,
\begin{align*}
\|f_0(Y)\|_{\dot B^s_{1,1}}&\lesssim \|({\mathcal A}^t{\mathcal A}-Id)\nabla Y_t\|_{\dot B^{s+1}_{1,1}}
\lesssim\|\nabla Y\|_{0}\|\nabla Y_t\|_{\dot B^{s+1}_{2,1}}+\|\nabla Y\|_{\dot B^{s+1}_{2,1}}\|\nabla Y_t\|_{0}.
\end{align*}
\eqref{S6eq1} then follows from the above inequality and the interpolation inequality \eqref{S4.1eq17}.
Along the same line, we deduce from  \eqref{fm} that
\begin{align*}
\|f_m(Y)\|_{\dot B^s_{1,1}}&\lesssim (1+|{\mathcal A}^t-Id|_{0})\|\nabla\vv p_m\|_{\dot B^s_{1,1}}+\|{\mathcal A}-Id\|_{\dot B^s_{2,1}}\|\nabla\vv p_m\|_{0}\\
&\lesssim \|\nabla \vv p_m\|_{\dot B^s_{1,1}}+\|\nabla Y\|_{\dot B^s_{2,1}}\|\nabla\vv p_m\|_{0}.
\end{align*}
Yet it follows from  \eqref{p1} that
\begin{align*}
\|\nabla\vv p_1\|_{\dot B^s_{1,1}}
\lesssim\delta_1\|\nabla\vv p_1\|_{\dot B^s_{1,1}}+\|\nabla Y\|_{\dot B^s_{2,1}}\|\nabla \vv p_1\|_{0}&+\|{\mathcal A}(\p_3Y\otimes\p_3Y)\|_{\dot B^{s+1}_{1,1}}\\
&+\|\nabla Y\|_{\dot B^s_{2,1}}\|{\mathcal A}(\p_3Y\otimes\p_3Y)\|_{\dot{H}^1},
\end{align*}
which together with \eqref{S4.1eq3} implies
\begin{align*}
\|\nabla\vv p_1\|_{\dot B^s_{1,1}}
&\lesssim\|\p_3Y\|_{0}\|\p_3Y\|_{\dot B^{s+1}_{2,1}}+\|\nabla Y\|_{\dot B^s_{2,1}\cap \dot B^{s+1}_{2,1}}|\p_3Y|_{0}\|\p_3Y\|_{1}.
\end{align*}
As a result, it comes out
$$\|f_1(Y)\|_{\dot B^s_{1,1}}\lesssim\|\p_3Y\|_{0}\|\p_3Y\|_{\dot B^{s+1}_{2,1}}+\|\nabla Y\|_{\dot B^s_{2,1}\cap \dot B^{s+1}_{2,1}}|\p_3Y|_{0}\|\p_3Y\|_{1}.$$
Similarly, we have
$$\|f_2(Y)\|_{\dot B^s_{1,1}}\lesssim\|Y_t\|_{0}\|Y_t\|_{\dot B^{s+1}_{2,1}}+\|\nabla Y\|_{\dot B^s_{2,1}\cap \dot B^{s+1}_{2,1}}|Y_t|_{0}\|Y_t\|_{1}.$$
\eqref{S6eq2} and \eqref{S6eq3} then follow from the above estimates and the interpolation inequality \eqref{S4.1eq17}. This completes the
 proof of Proposition \ref{S6prop1}.
\end{proof}

\noindent$\bullet$\underline{The estimate of $\||D|^{-1}f(Y)\|_{N+1}$}

\begin{prop}\label{S6prop2}
{\sl Under the asme assumptions of Proposition \ref{S6prop1}, we have
\begin{align}
\label{S6eq4}&\||D|^{-1}f_0(Y)\|_{N+1}\lesssim |\nabla Y|_{0}\|\nabla  Y_t\|_{N+1}+|\nabla Y|_{N+1}\|\nabla Y_t\|_{0};\\
\label{S6eq5}&\||D|^{-1}f_1(Y)\|_{N+1}\lesssim|\p_3Y|_{0}\|\p_3Y\|_{N+1}+|\nabla Y|_{N+1}|\p_3Y|_{0}\|\p_3Y\|_{1};\\
\label{S6eq6}&\||D|^{-1}f_2(Y)\|_{N+1}\lesssim|Y_t|_{0}\|Y_t\|_{N+1}+|\nabla Y|_{N+1}|Y_t|_{0}\|Y_t\|_{1}.
\end{align}
}
\end{prop}

\begin{proof}
In view of \eqref{fm}, we get, by applying Moser type inequality, that
\begin{align*}
\||D|^{-1}f_0(Y)\|_{N}&\leq \|({\mathcal A}^t{\mathcal A}-Id)\nabla Y_t\|_{N}
\lesssim |\nabla Y|_{0}\|\nabla  Y_t\|_{N}+|\nabla Y|_{N}\|\nabla Y_t\|_{0},
\end{align*} which gives \eqref{S6eq4}.
While again by \eqref{fm} and the law of product in Besov spaces, one has
\begin{align*}
\||D|^{-1}f_m(Y)\|_{0}&\lesssim \big(1+\|{\mathcal A}^t-Id\|_{\dot B^\frac32_{2,1}}\big)\|\nabla\vv p_m\|_{\dot H^{-1}},
\end{align*}
yet it follows from \eqref{p1} that
\begin{align*}
\|\nabla \vv p_1\|_{\dot H^{-1}}
&\lesssim \|\nabla Y\|_{\dot B^\frac32_{2,1}}\|\nabla\vv p_1\|_{\dot H^{-1}}+(1+\|{\mathcal A}-Id\|_{\dot B^\frac32_{2,1}})\|{\mathcal A}(\p_3Y\otimes\p_3Y)\|_{0},
\end{align*}
from which and the assumption \eqref{S4.1prop1assum}, we infer
\begin{align}\label{S6eq7}
\||D|^{-1}f_1(Y)\|_{0}\lesssim \|\nabla \vv p_1\|_{\dot H^{-1}}&\lesssim |\p_3Y|_{0}\|\p_3Y\|_{0}.
\end{align}
Similarly, we have
\begin{align}\label{S6eq8}
\||D|^{-1}f_2(Y)\|_{0}\lesssim \|\nabla \vv p_2\|_{\dot H^{-1}}&\lesssim |Y_t|_{0}\|Y_t\|_{0}.
\end{align}
For $N\geq0$, we deduce from \eqref{fm} that
\begin{align*}
\|f_1(Y)\|_{N}\lesssim \|\nabla \vv p_1\|_{N}+|\nabla Y|_{N}\|\nabla \vv p_1\|_{0}.
\end{align*}
And it follows from \eqref{p1} that
\begin{align*}
&\|\nabla \vv p_1\|_{N}\lesssim |\nabla Y|_{0}\|\nabla\vv p_1\|_{N}+|\nabla Y|_{N}\|\nabla \vv p_1\|_{0}+\|{\mathcal A}\dv\big({\mathcal A}(\p_3Y\otimes\p_3Y)\big)\|_{N},
\end{align*}
which together with \eqref{S4.1prop1assum} and \eqref{S4.1eq3} ensures that
\begin{align*}
\|\nabla \vv p_1\|_{N}
&\lesssim |\p_3Y|_{0}\|\p_3Y\|_{N+1}+|\nabla Y|_{N+1}|\p_3Y|_{0}\|\p_3Y\|_{1}.
\end{align*}
As a result, it comes out
\begin{align}\label{S6eq9}
\|f_1(Y)\|_{N}\lesssim |\p_3Y|_{0}\|\p_3Y\|_{N+1}+|\nabla Y|_{N+1}|\p_3Y|_{0}\|\p_3Y\|_{1}.
\end{align}
The same procedure for $f_2(Y)$ yields
\begin{align}\label{S6eq10}
\|f_2(Y)\|_{N}\lesssim |Y_t|_{0}\|Y_t\|_{N+1}+|\nabla Y|_{N+1}|Y_t|_{0}\|Y_t\|_{1}.
\end{align}
\eqref{S6eq5} and \eqref{S6eq6} follow from \eqref{S6eq7}-\eqref{S6eq10}. This completes the proof of Proposition \ref{S6prop2}.
\end{proof}



\section{Estimates of $f''(Y;X,W)$}\label{Sect8}

The purpose of  this section is  to present the related estimates to the second derivatives, $f''(Y;X,W),$ of the nonlinear functional $f(Y)$ given by \eqref{S1eq2}.

\subsection{The estimate of $\||D|^{-1}f''(Y;X,W)\|_{N}$}

\begin{prop}\label{S7.2prop1}
{\sl Let $f_0'', f_1'',f_2''$ be given by \eqref{S7.1eq5} and \eqref{S7.1eq6} respectively. Then under the assumption of \eqref{S4.1prop1assum}, we have
\begin{equation}\label{S7.2eqE1}
\begin{split}
\||D|^{-1}&f_0''(Y;X,W)\|_{N}
\lesssim   | Y_t|_1\big(|\nabla X|_{N}\|\nabla W\|_0+|\nabla X|_0\|\nabla W\|_{N}\big)\\
&\quad+\bigl((| Y_t|_{N+1}+|\nabla Y|_{N}| Y_t|_1)|\nabla X|_0+|X_t|_{N+1}\bigr)\|\nabla W\|_0+ |\nabla X|_{N}\|\nabla W_t\|_{0}\\
&\quad+|\nabla X|_0\|\nabla W_t\|_{N}+|\nabla Y|_{N}\bigl(|\nabla X|_0\|\nabla W_t\|_0+|X_t|_1\|\nabla W\|_0\bigr)+|X_t|_1\|\nabla W\|_{N},
\end{split}
\end{equation}
and
\ben
\||D|^{-1}f_1''(Y;X,W)\|_{N}&\lesssim&\frak{f}_3(\p_3Y,\p_3X, \p_3W) \andf \label{S7.2eqE2}\\
\||D|^{-1}f_2''(Y;X,W)\|_{N}&\lesssim&\frak{f}_3(Y_t, X_t, W_t), \label{S7.2eqE3}
\een
where the functional $\frak{f}_3(\frak{x},\frak{y},\frak{z})$ is given by
\begin{equation*}
\begin{split}
\frak{f}_3(\frak{x},\frak{y},\frak{z})\eqdefa & \bigl(|\frak{y}|_{N}+|\nabla Y|_{N}|\frak{y}|_0\bigr)\|\frak{z}\|_0+|\frak{y}|_{0}\|\frak{z}\|_{N}\\
&\ +  \big(|\frak{x}|_0+|\frak{x}|_1^\frac13\|\frak{x}\|_{1}^\frac23\big)\bigl(|\nabla X|_{N}\|\frak{z}\|_0+|\nabla X|_0\|\frak{z}\|_{N} +|\frak{y}|_1\|\nabla W\|_{N}
 \\
&\ +(|\frak{y}|_{N}+|\frak{x}|_{N}|\nabla X|_1)\|\nabla W\|_1\bigr)+\big(|\frak{x}|_{N}+|\nabla Y|_{N}|\frak{x}|_1\big)|\nabla X|_1\|\frak{z}\|_1\\
&\ +\big(|\frak{x}|_1^\frac43\|\frak{x}\|_0^\frac23+|\frak{x}|_1^2\big)\big((|\nabla X|_{N}+|\nabla Y|_{N}|\nabla X|_1)\|\nabla W\|_1+|\nabla X|_1\|\nabla W\|_{N}\big)\\
&\ +\bigl(|\frak{x}|_{N}+|\frak{x}|_{N}^\frac13\|\frak{x}\|_{N}^\frac23+|\nabla Y|_{N}(|\frak{x}|_1+|\frak{x}|_0^\frac13\|\frak{x}\|_0^\frac23)\bigr)|\frak{y}|_1\|\nabla W\|_1.
\end{split}
\end{equation*}
}
\end{prop}

\begin{rmk}
We mention that in the above inequalities, it is crucial to  estimate the vector, $X,$ by $L^\infty$-norm. In Section \ref{Sect10},  we shall deal with the estimate of
the error term
$$e_p'=-\int_0^1f''\big(Y_p+s(1-S_p)Y_p;(1-S_p)Y_p,X_p\big)ds,$$  where the variable,
$(1-S_p)Y_p,$ is ``small" in the $L^\infty$-norm, but only ``bounded" in $L^2$-norm.
\end{rmk}

Let us start the proof of Proposition \ref{S7.2prop1} by the following lemma:

\begin{lem}\label{S7.1lem1}
{\sl Under the assumption of \eqref{S4.1prop1assum}, one has
\beno
\begin{split}
&
 \|{\mathcal A}'(Y;X)\|_0\lesssim \|\nabla X\|_0\qquad\quad\quad \  \andf |{\mathcal A}'(Y;X)|_N\lesssim |\nabla X|_N+|\nabla Y|_N|\nabla X|_0;\\
&\|({\mathcal A}{\mathcal A}^t-Id)'(Y;X)\|_0\lesssim \|\nabla X\|_{0} \andf |({\mathcal A}{\mathcal A}^t-Id)'(Y;X)|_N
\lesssim |\nabla X|_N+|\nabla Y|_N|\nabla X|_0;
\end{split}
\eeno
and \beno
\begin{split}
\|{\mathcal A}''(Y;X,W)\|_N+&\|({\mathcal A}{\mathcal A}^t-Id)''(Y;X,W)\|_N\\
&\leq  |\nabla X|_N\|\nabla W\|_0+|\nabla X|_0\big(\|\nabla W\|_N+|\nabla Y|_N\|\nabla W\|_0\big).\end{split} \eeno
}
\end{lem}

\begin{proof} Indeed the estimates for ${\mathcal A}'(Y;X),$ $({\mathcal A}{\mathcal A}^t-Id)'(Y;X)$ and $
{\mathcal A}''(Y;X,W)$ can be deduced by applying Moser type inequality to \eqref{S7.1eq1}, \eqref{S7.1eq3} and \eqref{S7.1eq2}.
Whereas in view of  \eqref{S7.1eq0}, we have
\beq\label{S7.2lem1eq1}
\begin{split}
({\mathcal A}{\mathcal A}^t-Id)''(Y;X,W)=&{\mathcal A}''(Y;X,W){\mathcal A}^t+{\mathcal A}({\mathcal A}^t)''(Y;X,W)\\
&+{\mathcal A}'(Y;X)({\mathcal A}^t)'(Y;W)+{\mathcal A}'(Y;W)({\mathcal A}^t)'(Y;X).
\end{split} \eeq
Thus the estimate for $({\mathcal A}{\mathcal A}^t-Id)''(Y;X,W)$ follows.
 \end{proof}

\begin{proof}[Proof of Proposition \ref{S7.2prop1}] We divide the proof of this proposition into the following steps.

\noindent$\bullet$\underline{The estimate of $\||D|^{-1}f_0''(Y;X,W)\|_{N}$}

We first deduce from Moser type inequality and Lemma \ref{S7.1lem1} that
\begin{align*}
\|({\mathcal A}{\mathcal A}^t-Id)''(&Y;X,W)\nabla Y_t\|_{N}
\lesssim |\nabla Y_t|_{N}|\nabla X|_0\|\nabla W\|_0\\
& +|\nabla Y_t|_0\big(|\nabla X|_{N}\|\nabla W\|_0+|\nabla X|_0\big(\|\nabla W\|_{N}+|\nabla Y|_{N}\|\nabla W\|_0\big)\big),
\end{align*}
and
\begin{align*}
\|({\mathcal A}{\mathcal A}^t-Id)'(Y;X)\nabla W_t\|_{N}
&\lesssim |\nabla X|_{N}\|\nabla W_t\|_{0}+|\nabla X|_0\big(\|\nabla W_t\|_{N}+|\nabla Y|_{N}\|\nabla W_t\|_0\big),
\end{align*}
and
\begin{align*}
\|({\mathcal A}{\mathcal A}^t-Id)'(Y;W)\nabla X_t\|_{N}
&\lesssim |X_t|_{N+1}\|\nabla W\|_0+|X_t|_1\big(\|\nabla W\|_{N}+|\nabla Y|_{N}\|\nabla W\|_0\big).
\end{align*}
Hence thanks to \eqref{S7.1eq5}, we obtain \eqref{S7.2eqE1}.

Next,
we shall only present the estimates for $f_1''$, the one for $f_2''$ follows along the same line. According to \eqref{S7.1eq6}, we write
\begin{equation}\label{S7.2eqf1}
\begin{split}
f_1''(Y;X,W)=&I_1+I_2+I_3+I_4 \with \\
I_1\eqdefa & ({\mathcal A}^t)''(Y;X,W)\nabla\vv p_1(Y),\qquad\qquad I_2\eqdefa ({\mathcal A}^t)'(Y;X)\nabla\big(\vv p_1'(Y;W)\big),\\
I_3\eqdefa & ({\mathcal A}^t)'(Y;W)\nabla\big(\vv p_1'(Y;X)\big),\andf I_4\eqdefa {\mathcal A}^t\nabla \big(\vv p_1''(Y;X,W)\big).
\end{split}
\end{equation}

\noindent$\bullet$\underline{The estimate of $\|f_1''(Y;X,W)\|_{\dot H^{-1}}$}

\noindent{(i) $\dot H^{-1}$ estimate of $I_1$.} By virtue of  Sobolev embedding: $L^\frac65(\R^3)\hookrightarrow \dot H^{-1}(\R^3),$ and \eqref{S4.1eq9},
 we infer
\begin{align*}
\|I_1\|_{\dot H^{-1}}
\lesssim \|({\mathcal A}^t)''(Y;X,W)\|_{0}\|\nabla\vv p_1\|_{L^3}\lesssim |\p_3Y|_1^\frac43\|\p_3Y\|_0^\frac23|\nabla X|_0\|\nabla W\|_0.
\end{align*}

\noindent{(ii) $\dot H^{-1}$ estimate of  $I_2.$}
Similar to the estimate of $I_1,$ we have
\begin{align*}
\|I_2\|_{\dot H^{-1}}\lesssim |({\mathcal A}^t)'(Y;X)|_{0}\|\nabla\vv p_1'(Y;W)\|_{L^\frac65}.
\end{align*}
Yet it follows from \eqref{p1'} that
\begin{align*}
\|\nabla \vv p_1'(Y;W)\|_{L^\frac65}
&\lesssim\delta_1\|\nabla \vv p_1'(Y;W)\|_{L^\frac65}
+\|({\mathcal A}{\mathcal A}^t-Id)'(Y;W)\|_0\|\nabla \vv p_1\|_{L^3}\\
&\quad +\|{\mathcal A}'(Y;W)\|_{1}\|\p_{3}Y\otimes \p_{3}Y\|_{W^{1,3}}
 +\|{\mathcal A}'(Y;W)\|_0\|{\mathcal A}(\p_{3}Y\otimes \p_{3}Y)\|_{W^{1,3}}\\
&\quad +\|{\mathcal A}(\p_{3}Y\otimes\p_{3}W+\p_{3}W\otimes \p_{3}Y)\|_{W^{1,\frac65}},
\end{align*}
which together with \eqref{S4.1eq9} and Lemma \ref{S7.1lem1} implies that
\begin{equation}\label{S7.2eq1}
\|\nabla \vv p_1'(Y;W)\|_{L^\frac65}\lesssim|\p_3Y|_1^\frac43\|\p_3Y\|_0^\frac23\|\nabla W\|_{1}+|\p_3Y|_1^\frac13\|\p_3Y\|_1^\frac23\|\p_3W\|_1.
\end{equation}
As a result, it comes out
\begin{align*}
\|I_2\|_{\dot H^{-1}}
\lesssim \bigl(|\p_3Y|_1^\frac43\|\p_3Y\|_0^\frac23\|\nabla W\|_{1}+|\p_3Y|_1^\frac13\|\p_3Y\|_1^\frac23\|\p_3W\|_1\bigr)|\nabla X|_0.
\end{align*}

\noindent{(iii) $\dot H^{-1}$ estimate of  $I_3.$ } Note that
\begin{align*}
\|I_3\|_{\dot H^{-1}}\lesssim \|({\mathcal A}^t)'(Y;W)\|_0\|\nabla\big(\vv p_1'(Y;X)\big)\|_{L^3}.
\end{align*}
While in view of  \eqref{p1'}, one has
\begin{align*}
\|\nabla \vv p_1'(Y;X)\|_{L^3}
&\lesssim\delta_1\|\nabla \vv p_1'(Y;X)\|_{L^3}
+|({\mathcal A}{\mathcal A}^t-Id)'(Y;X)|_0\|\nabla \vv p_1\|_{L^3}\\
&\quad +\|{\mathcal A}'(Y;X)(\p_{3}Y\otimes \p_{3}Y)\|_{W^{1,3}}
 +|{\mathcal A}'(Y;X)|_0\|{\mathcal A}(\p_{3}Y\otimes \p_{3}Y)\|_{W^{1,3}}\\
&\quad +\|{\mathcal A}(\p_{3}Y\otimes\p_{3}X+\p_{3}X\otimes \p_{3}Y)\|_{W^{1,3}},
\end{align*}
which together with \eqref{S4.1eq9} and Lemma \ref{S7.1lem1} ensures that
\begin{equation}\label{S7.2eq2}
\|\nabla \vv p_1'(Y;X)\|_{L^3}
\lesssim |\p_3Y|_1^\frac43\|\p_3Y\|_0^\frac23|\nabla X|_1+|\p_3Y|_1^\frac13\|\p_3Y\|_{1}^\frac23|\p_3X|_1.
\end{equation}
And we thus obtain
\begin{align*}
\|I_3\|_{\dot H^{-1}}
&\lesssim \bigl(|\p_3Y|_1^\frac43\|\p_3Y\|_0^\frac23|\nabla X|_{1} +|\p_3Y|_1^\frac13\|\p_3Y\|_{1}^\frac23|\p_3X|_1\bigr)\|\nabla W\|_0.
\end{align*}

\noindent{(iv)  $\dot H^{-1}$ estimate of  $I_4.$ } We  first get, by applying the law of product in Besov spaces, that
\begin{align*}
\|I_4\|_{\dot H^{-1}}\lesssim \big(1+\|{\mathcal A}^t-Id\|_{\dot B^\frac32_{2,1}}\big)\|\vv p_m''(Y;X,W)\|_{0}.
\end{align*}
Thanks to \eqref{p1''}, we get, by applying  Sobolev embedding: $L^\frac65(\R^3)\hookrightarrow \dot H^{-1}(\R^3),$ that
\begin{align*}
\|\vv p_1''&(Y;X,W)\|_{0}\lesssim   \|\nabla Y\|_{\dot B^{\frac32}_{2,1}}\|\vv p_1''(Y;X,W)\|_{0}
+\|({\mathcal A}{\mathcal A}^t-Id)''(Y;X,W)\|_{0}\|\nabla \vv p_1\|_{L^3}\\
&+|({\mathcal A}{\mathcal A}^t-Id)'(Y;X)|_{0}\|\nabla \vv p_1'(Y;W)\|_{L^\frac65}+\|({\mathcal A}{\mathcal A}^t-Id)'(Y;W)\|_{0}\|\nabla \vv p_1'(Y;X)\|_{L^3}\\
&+\|{\mathcal A}(\p_{3}X\otimes\p_{3}W+\p_3W\otimes\p_3X)\|_0+\|{\mathcal A}''(Y;X,W)\|_{0}\bigl(|\p_3Y|_0^2+
\|{\mathcal A}(\p_{3}Y\otimes \p_{3}Y)\|_{W^{1,3}}\bigr)\\
&+|{\mathcal A}'(Y;X)|_0\bigl(\|(\p_{3}Y\otimes \p_{3}W+\p_{3}W\otimes\p_{3}Y)\|_{0}+\|{\mathcal A}'(Y;W)(\p_{3}Y\otimes \p_{3}Y)\|_{W^{1,\frac65}}\\
& +\|{\mathcal A}(\p_{3}Y\otimes \p_{3}W+\p_{3}W\otimes\p_{3}Y)\|_{W^{1,\frac65}}\bigr)+|{\mathcal A}'(Y;W)|_0\|(\p_{3}Y\otimes \p_{3}X+\p_{3}X\otimes\p_{3}Y)\|_{0}\\
&+\|{\mathcal A}'(Y;W)\|_0\bigl(\|{\mathcal A}(\p_{3}Y\otimes \p_{3}X+\p_{3}X\otimes\p_{3}Y)\|_{W^{1,3}} +\|{\mathcal A}'(Y;X)(\p_{3}Y\otimes \p_{3}Y)\|_{W^{1,3}}\bigr),
\end{align*}
from which, \eqref{S4.1prop1assum}, \eqref{S4.1eq9}, Lemma \ref{S7.1lem1}, \eqref{S7.2eq1} and \eqref{S7.2eq2}, we infer
\begin{equation}\label{S7.2eq3}
\begin{split}
\|\vv p_1''(Y;X,W)\|_{0}\lesssim
&|\p_3X|_0\|\p_3W\|_0+|\p_3Y|_1^\frac43\|\p_3Y\|_{0}^\frac23|\nabla X|_1\|\nabla W\|_1\\
& +|\p_3Y|_1^\frac13\|\p_3Y\|_{1}^\frac23\big(|\nabla X|_0\|\p_3W\|_1 +|\p_3X|_1\|\nabla W\|_1\big).
\end{split}
\end{equation}
Therefore, we obtain
\begin{align*}
\|I_4\|_{\dot H^{-1}}\lesssim
&|\p_3X|_0\|\p_3W\|_0+|\p_3Y|_1^\frac43\|\p_3Y\|_{0}^\frac23|\nabla X|_1\|\nabla W\|_1\\
& +|\p_3Y|_1^\frac13\|\p_3Y\|_{1}^\frac23\big(|\nabla X|_0\|\p_3W\|_1 +|\p_3X|_1\|\nabla W\|_1\big).
\end{align*}
By summarizing the estimates of $I_1,I_2,I_3,I_4$, we achieve
\begin{equation}\label{S7.2eq4}
\begin{split}
\|f_1''(Y;X,W)\|_{\dot H^{-1}}\lesssim
&|\p_3X|_0\|\p_3W\|_0+|\p_3Y|_1^\frac43\|\p_3Y\|_{0}^\frac23|\nabla X|_1\|\nabla W\|_1\\
& +|\p_3Y|_1^\frac13\|\p_3Y\|_{1}^\frac23\big(|\nabla X|_0\|\p_3W\|_1 +|\p_3X|_1\|\nabla W\|_1\big) .
\end{split}
\end{equation}

\noindent$\bullet$\underline{The estimate of $\|f_m''(Y;X,W)\|_{N}$}

\noindent{(i) $H^N$ estimate of  $I_1.$} In view of \eqref{S7.2eqf1},  we deduce from \eqref{S4.1eq9}, \eqref{S4.2eq7} and Lemma
\ref{S7.1lem1} that
\begin{align*}
\|I_1\|_N
&\lesssim \|({\mathcal A}^t)''(Y;X,W)\|_{N+1}\|\nabla \vv p_m\|_{L^3}+\|({\mathcal A}^t)''(Y;X,W)\|_{1}\|\nabla \vv p_m\|_{W^{N,3}}\\
&\lesssim |\p_3Y|_1^\frac43\|\p_3Y\|_0^\frac23\big(|\nabla X|_{N+1}\|\nabla W\|_0+|\nabla X|_{0}\|\nabla W\|_{N+1}\big)\\
&\qquad+\big(|\p_3Y|_{N+1}|\p_3Y|_0^\frac13+|\nabla Y|_{N+1}|\p_3Y|_{1}^\frac43\big)\|\p_3Y\|_0^\frac23|\nabla X|_1\|\nabla W\|_1.
\end{align*}

\noindent{(ii) $H^N$ estimate of  $I_2.$} By virtue of \eqref{S4.2eq11}, we get, by applying Moser type inequality, that
\begin{align*}
\|I_2\|_{N}
\lesssim& |\p_3Y|_0\big(|\nabla X|_{N+1}\|\p_3W\|_0+ |\nabla X|_0\|\p_3W\|_{N+1}\big)\\
&+\big(|\nabla X|_{N+1}\|\nabla W\|_0 +|\nabla X|_0\|\na W\|_{N+1}+|\nabla Y|_{N+1}|\nabla X|_0\|\nabla W\|_{1}\big)\times\\
&\times\big(|\p_3Y|_{1}^\frac43\|\p_3Y\|_{0}^\frac23+|\p_3Y|_{1}^2\big)+\big(|\p_3Y|_{N+1}+|\nabla Y|_{N+1}|\p_3Y|_{1}\big)|\nabla X|_0\|\p_3W\|_{1}\\
&+|\nabla X|_0\|\nabla W\|_{1} |\p_3Y|_{N+1}\big(|\p_3Y|_{0}^\frac13\|\p_3Y\|_{0}^\frac23+|\p_3Y|_{0}\big).
\end{align*}

\noindent{(iii) $H^N$ estimate of  $I_3.$}  Applying Moser type inequality gives
\begin{align*}
\|I_3\|_{N}\lesssim \|({\mathcal A}^t)'(Y;W)\|_1\|\nabla \vv p_1'(Y;X)\|_{W^{N,3}}+\|({\mathcal A}^t)'(Y;W)\|_{N+1}\|\nabla \vv p_1'(Y;X)\|_{L^3}.
\end{align*}
Yet it follows from \eqref{p1'} that
\begin{align*}
\|\nabla \vv p_1'(Y;X)\|_{W^{N,3}}
\lesssim & \delta_1\|\nabla \vv p_1'(Y;X)\|_{W^{N,3}}+|\nabla Y|_N\|\nabla\vv p_1'(Y;X)\|_{L^3}\\
&+|({\mathcal A}{\mathcal A}^t-Id)'(Y;X)|_{N}\|\nabla \vv p_1\|_{L^3}+|({\mathcal A}{\mathcal A}^t-Id)'(Y;X)|_{0}\|\nabla \vv p_1\|_{W^{N,3}}\\
&+\|{\mathcal A}'(Y;X)(\p_{3}Y\otimes \p_{3}Y)\|_{W^{N+1,3}}+|\nabla Y|_N\|{\mathcal A}'(Y;X)(\p_3Y\otimes\p_3Y)\|_{W^{1,3}}\\
&+|{\mathcal A}'(Y;X)|_{N}\|{\mathcal A}(\p_{3}Y\otimes \p_{3}Y)\|_{W^{1,3}}+|{\mathcal A}'(Y;X)|_0\|{\mathcal A}(\p_3Y\otimes\p_3Y)\|_{W^{N+1,3}}\\
&+\|{\mathcal A}(\p_{3}Y\otimes\p_{3}X)\|_{W^{N+1,3}}+|\nabla Y|_N\|{\mathcal A}(\p_{3}Y\otimes\p_{3}X)\|_{W^{1,3}},
\end{align*}
from which, \eqref{S4.2eq7}, \eqref{S7.2eq2},
 we infer
\begin{equation}\label{S7.2eq5}
\begin{split}
\|\nabla \vv p_1'(Y;X)\|_{W^{N,3}}
&\lesssim |\p_3Y|_{0}^\frac13\|\p_3Y\|_0^\frac23|\p_3X|_{N+1}+ |\p_3Y|_0^\frac43\|\p_3Y\|_0^\frac23|\nabla X|_{N+1}\\
&\quad+\big(|\p_3Y|_{N+1}^\frac13\|\p_3Y\|_{N+1}^\frac23+|\nabla Y|_{N+1}|\p_3Y|_1^\frac13\|\p_3Y\|_1^\frac23\big)|\p_3X|_1\\
&\quad+\big(|\p_3Y|_{N+1}|\p_3Y|_0^\frac13+|\nabla Y|_{N+1}|\p_3Y|_1^\frac43\big)\|\p_3Y\|_0^\frac23|\nabla X|_1.
\end{split}
\end{equation}
Together with \eqref{S7.2eq2}, we deduce that
\begin{align*}
\|I_3\|_N
&\lesssim |\p_3Y|_{1}^\frac13\|\p_3Y\|_1^\frac23\big(|\p_3X|_{N+1}\|\nabla W\|_1+|\p_3X|_1\|\nabla W\|_{N+1}+|\nabla Y|_{N+1}|\p_3X|_1\|\nabla W\|_1\big)\\
&\quad+|\p_3Y|_1^\frac43\|\p_3Y\|_0^\frac23\big( |\nabla X|_{N+1}\|\nabla W\|_1+|\nabla X|_1\|\nabla W\|_{N+1}+|\nabla Y|_{N+1}|\nabla X|_1\|\nabla W\|_1\big)\\
&\quad+|\p_3Y|_{N+1}^\frac13\|\p_3Y\|_{N+1}^\frac23 |\p_3X|_1\|\nabla W\|_1+|\p_3Y|_{N+1}|\p_3Y|_0^\frac13 \|\p_3Y\|_0^\frac23|\nabla X|_1\|\nabla W\|_1.
\end{align*}

\noindent{(iv) $H^N$ estimate of  $I_4.$} We first get, by applying Moser type inequality, that
\begin{align*}
\|I_4\|_N\lesssim \|\nabla \vv p_1''(Y;X,W)\|_{N}+|\nabla Y|_{N}\|\nabla \vv p_1''(Y;X,W)\|_{0}.
\end{align*}
We first deal with the  estimate of $\|\nabla \vv p_1''(Y;X,W)\|_{0}.$ Indeed by \eqref{p1''}, we have
\beno
\begin{split}
\|\nabla\vv p_1''(Y;X,W)\|_{0}\lesssim&  |\na Y|_0\|\nabla\vv p_1''(Y;X,W)\|_{0}+\|({\mathcal A}{\mathcal A}^t-Id)''(Y;X,W)\|_{\dot H^1}\|\nabla \vv p_1\|_{L^3}\\
&+|({\mathcal A}{\mathcal A}^t-Id)'(Y;X)|_0\|\nabla \vv p_1'(Y;W)\|_{0}+\big(1+\|{\mathcal A}\|_{B^{\f32}_{2,1}}\bigr)\|\p_{3}X\otimes\p_{3}W\|_1
\\
&+ \|({\mathcal A}{\mathcal A}^t-Id)'(Y;W)\|_1\|\nabla \vv p_1'(Y;X)\|_{L^3}+|\p_{3}Y\otimes \p_{3}Y|_1\|{\mathcal A}''(Y;X,W)\|_1\\
& +|{\mathcal A}'(Y;X)|_1\|\p_{3}Y\otimes \p_{3}W\|_{1}+\|{\mathcal A}'(Y;W)\|_1|\p_{3}Y\otimes \p_{3}X|_1\\
&+ |{\mathcal A}'(Y;X)|_{0}\bigr(\|\p_{3}Y\otimes \p_{3}W\|_{1}+\|{\mathcal A}'(Y;W)(\p_{3}Y\otimes \p_{3}Y)\|_1\bigr)\\
&+\|{\mathcal A}'(Y;W)\|_0|{\mathcal A}(\p_{3}Y\otimes \p_{3}X)|_1+\|{\mathcal A}'(Y;X)(\p_{3}Y\otimes \p_{3}Y)\|_{1}\\
&\qquad\qquad\qquad\qquad\qquad\qquad\qquad+\|{\mathcal A}''(Y;X,W)\|_{0}|\p_{3}Y\otimes \p_{3}Y|_{1},
\end{split}\eeno
from which, \eqref{S4.2eq11}
and \eqref{S7.2eq2}, we deduce that
\begin{equation}\label{S7.2eq6}
\begin{split}
\|\nabla\vv p_1''(Y;X,W)\|_{0}
\lesssim & |\p_3X|_1\bigl(\|\p_3W\|_1+(|\p_3Y|_1+ |\p_3Y|_1^\frac13\|\p_3Y\|_{1}^\frac23)\|\nabla W\|_1\bigr)\\
 &+|\nabla X|_1\bigl( |\p_3Y|_1\|\p_3W\|_1+(|\p_3Y|_{1}^\frac43\|\p_3Y\|_{0}^\frac23+|\p_3Y|_{1}^2)\|\na W\|_{1}\bigr).
\end{split}
\end{equation}
In general,  along the same line to the proof of \eqref{S7.2eq6}. we get, by using
 the estimates \eqref{S4.2eq7}, \eqref{S4.2eq11}, \eqref{S7.2eq2}, \eqref{S7.2eq3}, \eqref{S7.2eq5} and \eqref{S7.2eq6},
  that
\begin{equation}\label{S7.2eq7}
\begin{split}
&\|\nabla\vv p_1''(Y;X,W)\|_{N}\lesssim  \bigl(|\p_3X|_{N+1}+|\p_3X|_0|\nabla Y|_{N+1}\bigr)\|\p_3W\|_0+|\p_3X|_{0}\|\p_3W\|_{N+1}\\
&\ + |\p_3Y|_0\big(|\nabla X|_{N+1}\|\p_3W\|_0+|\nabla X|_0\|\p_3W\|_{N+1}\big)\\
&\ +\big(|\p_3Y|_{N+1}+|\p_3Y|_1|\nabla Y|_{N+1}\big)|\nabla X|_1\|\p_3W\|_1+\big(|\p_3Y|_0+|\p_3Y|_{1}^\frac13\|\p_3Y\|_1^\frac23\big)\\
&\ \times \big((|\p_3X|_{N+1}+|\nabla X|_1|\p_3Y|_{N+1})\|\nabla W\|_{1}+|\p_3X|_1\|\nabla W\|_{N+1}\big)\\
&\ +\big(|\p_3Y|_1^\frac43\|\p_3Y\|_0^\frac23+|\p_3Y|_1^2\big)\big((|\nabla X|_{N+1}+
|\nabla X|_1|\nabla Y|_{N+1})\|\nabla W\|_1+|\nabla X|_1\|\nabla W\|_{N+1}\big)\\
&\ +\bigl(|\p_3Y|_{N+1}+|\p_3Y|_{N+1}^\frac13\|\p_3Y\|_{N+1}^\frac23 +|\nabla Y|_{N+1}(|\p_3Y|_1+|\p_3Y|_0^\frac13\|\p_3Y\|_0^\frac23)\bigr)|\p_3X|_1
\|\nabla W\|_{1}.
\end{split}
\end{equation}
The same estimate holds for $\|I_4\|_N.$
By summing up the estimates of$I_1,I_2,I_3$  and $I_4$, we achieve
\begin{equation}\label{S7.2eq8}
\begin{split}
&\|f_1''(Y;X,W)\|_{N}
\lesssim \bigl(|\p_3X|_{N+1}+|\p_3X|_0|\nabla Y|_{N+1}\bigr)\|\p_3W\|_0\\
&\ + |\p_3Y|_0\big(|\nabla X|_{N+1}\|\p_3W\|_0+|\nabla X|_0\|\p_3W\|_{N+1}\big)+|\p_3X|_{0}\|\p_3W\|_{N+1}\\
&\ +\big(|\p_3Y|_{N+1}+|\p_3Y|_1|\nabla Y|_{N+1}\big)|\nabla X|_1\|\p_3W\|_1+\big(|\p_3Y|_0+|\p_3Y|_{1}^\frac13\|\p_3Y\|_1^\frac23\big)\\
&\  \times \big((|\p_3X|_{N+1}+|\nabla X|_1|\p_3Y|_{N+1})\|\nabla W\|_{1}+|\p_3X|_1\|\nabla W\|_{N+1}\big)\\
&\ +\big(|\p_3Y|_1^\frac43\|\p_3Y\|_0^\frac23+|\p_3Y|_1^2\big)\big((|\nabla X|_{N+1}+|\nabla X|_1|\nabla Y|_{N+1})\|\nabla W\|_1+|\nabla X|_1\|\nabla W\|_{N+1}\big)\\
&\ +\bigl(|\p_3Y|_{N+1}+|\p_3Y|_{N+1}^\frac13\|\p_3Y\|_{N+1}^\frac23 +(|\p_3Y|_1+|\p_3Y|_0^\frac13\|\p_3Y\|_0^\frac23)|\nabla Y|_{N+1}\big)|\p_3X|_1\|\nabla W\|_{1}.
\end{split}
\end{equation}
Then \eqref{S7.2eqE2} follows from \eqref{S7.2eq4} and \eqref{S7.2eq8}.
Exactly along the same line, we can prove \eqref{S7.2eqE3}, and we omit the details here. This complete the proof of Proposition \ref{S7.2prop1}.
\end{proof}

\subsection{The estimate of $\triplenorm{f''(Y;X,W)}_{\d,N} $ }

\begin{prop}\label{S7.3prop1}
{\sl Let $f_m''$, $m=0,1,2$ be given in \eqref{S7.1eq5} and \eqref{S7.1eq6} , the norm $ \triplenorm{\cdot}_{\d,N}$ be given by \eqref{S2eq19''}. Then under the assumption of \eqref{S4.1prop1assum},
  we have
\begin{equation}\label{S7.3eqE1}
\begin{split}
\triplenorm{f_0''(Y;X,W)}_{\d,N} \lesssim & | Y_t|_1\big(\|\nabla X\|_{N+6}\|\nabla W\|_0+\|\nabla X\|_0\|\nabla W\|_{N+6}\big)\\
&+\big(\|\nabla Y_t\|_{N+6}+| Y_t|_1\|\nabla Y\|_{N+6}\big)\bigl(|\nabla X|_0\|\nabla W\|_0+\|\nabla X\|_0|\nabla W|_0\bigr)\\
&+ \|\nabla X\|_0\|\nabla W_t\|_{N+6}+\bigl(\|\nabla X\|_{N+6}+|\nabla X|_0\|\nabla Y\|_{N+6}\bigr)\|\nabla W_t\|_0\\
&+\|\nabla W\|_0\|\nabla X_t\|_{N+6}+\bigl(\|\nabla W\|_{N+6}+|\nabla W|_0\|\nabla Y\|_{N+6}\bigr)\|\nabla X_t\|_0,
\end{split}
\end{equation}
and
\ben
\triplenorm{f_1''(Y;X,W)}_{\d,N}&\lesssim&\frak{f}_4(\p_3Y,\p_3X,\p_3W) \andf \label{S7.3eqE2}\\
\triplenorm{f_2''(Y;X,W)}_{\d,N}&\lesssim&\frak{f}_4(Y_t,X_t,W_t), \label{S7.3eqE3}
\een
where the functional $\frak{f}_4(\frak{x},\frak{y},\frak{z})$ is given by
\begin{equation*}
\begin{split}
\frak{f}_4(\frak{x},\frak{y},\frak{z})\eqdefa 
& \bigl(\|\frak{z}\|_{0}+|\frak{x}|_0\|\nabla W\|_{0}+\|\frak{x}\|_{0}|\nabla W|_0\bigr)\|\frak{y}\|_{N+6}+\bigl(\|\frak{x}\|_{0}|\nabla X|_0+\|\frak{y}\|_0\bigr)\|\frak{z}\|_{N+6}\\
&+|\frak{y}|_0\|\frak{z}\|_{0}\|\nabla Y\|_{N+6}
+ |\frak{x}|_0\big(\|\frak{z}\|_{0}\|\nabla X\|_{N+6}+\|\nabla X\|_{0}\|\frak{z}\|_{N+6}+\|\frak{y}\|_{0}\|\nabla W\|_{N+6}\big)\\
&+\big(\|\frak{x}\|_{N+6}+|\frak{x}|_0\|\nabla Y\|_{N+6}\big)\big(\|\nabla X\|_{0}|\frak{z}|_0 +|\nabla X|_1\|\frak{z}\|_{1}\big)\\
&+\big(\|\frak{x}\|_{N+6}+\|\frak{x}\|_{3}\|\nabla Y\|_{N+6}\big)\big(|\frak{y}|_1\|\nabla W\|_{1}+\|\frak{y}\|_{1}|\nabla W|_0\big)\\
&+ |\frak{x}|_{1}\|\frak{x}\|_{3}\big(\|\nabla X\|_{N+6}(\|\nabla W\|_0+|\nabla W|_0)+ (\|\nabla X\|_{0}+|\nabla X|_0)\|\nabla W\|_{N+6}\big)\\
&+\big(|\frak{x}|_{1}\|\frak{x}\|_{N+6}+|\frak{x}|_{1}\|\frak{x}\|_{3}\|\nabla Y\|_{N+6}\big)\big(|\nabla X|_0\|\nabla W\|_{1}+\|\nabla X\|_{1}|\nabla W|_0\big).
\end{split}
\end{equation*}
}
\end{prop}

\begin{lem}\label{S7.3lem1}
{\sl Let $s>0.$ Then under the assumption of \eqref{S4.1prop1assum}, one has
\begin{align}
&\|{\mathcal A}'(Y;X)\|_{\dot B^{s}_{2,1}}+\|({\mathcal A}{\mathcal A}^t-Id)'(Y;X)\|_{\dot B^s_{2,1}}
\lesssim \|\nabla X\|_{\dot B^s_{2,1}}+|\nabla X|_0\|\nabla Y\|_{\dot B^s_{2,1}}, \label{S7.3eq-1}\\
&\|{\mathcal A}''(Y;X,W)\|_{\dot B^s_{1,1}}\lesssim \bigl(\|\nabla X\|_{\dot B^s_{2,1}}+|\nabla X|_0\|\nabla Y\|_{\dot B^s_{2,1}}\bigr)\|\nabla W\|_0+\|\nabla X\|_{0}\|\nabla W\|_{\dot B^s_{2,1}},
\label{S7.3eq1}
\end{align}
and
\begin{equation}\label{S7.3eq2}
\begin{split}
\|({\mathcal A}{\mathcal A}^t-Id)''(Y;X,W)\|_{\dot B^s_{1,1}}\lesssim &\bigl(\|\nabla X\|_{\dot B^s_{2,1}}+|\nabla X|_0\|\nabla Y\|_{\dot B^s_{2,1}}\bigr)\|\nabla W\|_0\\
&+\bigl(\|\nabla W\|_{\dot B^s_{2,1}}+|\nabla W|_0\|\nabla Y\|_{\dot B^s_{2,1}}\bigr)\|\nabla X\|_0.
\end{split}
\end{equation}
}
\end{lem}
\begin{proof}
Indeed by virtue of \eqref{S7.1eq1}, we get, by applying the law of product \eqref{product1}, that
\begin{align*}
\|{\mathcal A}'(Y;X)\|_{\dot B^{s}_{2,1}}
&\leq\big(1+ |{\mathcal A}-Id|_0\big)\|\nabla X{\mathcal A}\|_{\dot B^s_{2,1}}+\|{\mathcal A}-Id\|_{\dot B^s_{2,1}}|\nabla X{\mathcal A}|_0.
\end{align*} Similarly, we deduce estimates for $(\cA\cA^t-Id)'(Y;X)$ from \eqref{S7.1eq3} and complete the proof of \eqref{S7.3eq-1}.
While it follows from \eqref{S7.1eq2} and the law of product \eqref{product3} that
\begin{equation*}
\begin{split}
\|{\mathcal A}''(Y;X,W)\|_{\dot B^s_{1,1}}
& \lesssim \|\nabla X\|_{\dot B^s_{2,1}}\|{\mathcal A}\nabla W\|_0+\|\nabla X\|_{0}\|{\mathcal A}\nabla W\|_{\dot B^s_{2,1}}+\|\nabla Y\|_{\dot B^s_{2,1}}|\nabla X|_0\|\nabla W\|_0,
\end{split}
\end{equation*} which yields \eqref{S7.3eq1}.

Finally it is easy to observe from the law of product and the previous estimates  that
\begin{align*}
\|{\mathcal A}''(Y;X,W){\mathcal A}^t\|_{\dot B^s_{1,1}}
&\lesssim \bigl(\|\nabla X\|_{\dot B^s_{2,1}}+|\nabla X|_0\|\nabla Y\|_{\dot B^s_{2,1}}\bigr)\|\nabla W\|_0+\|\nabla X\|_{0}\|\nabla W\|_{\dot B^s_{2,1}},
\end{align*}
and
\begin{align*}
\|{\mathcal A}'(Y;X)({\mathcal A}^t)'(Y;W)\|_{\dot B^s_{1,1}}
\lesssim& \big(\|\nabla X\|_{\dot B^s_{2,1}}+|\nabla X|_0\|\nabla Y\|_{\dot B^s_{2,1}}\big)\|\nabla W\|_0\\
&+\big(\|\nabla W\|_{\dot B^s_{2,1}}+|\nabla W|_0\|\nabla Y\|_{\dot B^s_{2,1}}\big)\|\nabla X\|_0,
\end{align*}
so that we can deduce \eqref{S7.3eq2} from \eqref{S7.2lem1eq1}.
\end{proof}

\begin{proof}[Proof of Proposition \ref{S7.3prop1}] Again we divide the proof of this proposition into the following steps:

\no\underline{{\bf Step 1}. Estimate  of $f_0''(Y;X,W)$.}
We first deduce from the law of product \eqref{product3} and  Lemmas \ref{S7.1lem1}, \ref{S7.3lem1} that
\begin{align*}
\|({\mathcal A}{\mathcal A}^t-Id)''(Y;&X,W)\nabla Y_t\|_{\dot B^{s+1}_{1,1}}
\lesssim | Y_t|_1\big(\|\nabla W\|_0\|\nabla X\|_{\dot B^{s+1}_{2,1}}+\|\nabla X\|_0\|\nabla W\|_{\dot B^{s+1}_{2,1}}\big)\\
&\qquad+\big(\|\nabla Y_t\|_{\dot B^{s+1}_{2,1}}+| Y_t|_1\|\nabla Y\|_{\dot B^{s+1}_{2,1}}\big)\bigl(|\nabla X|_0\|\nabla W\|_0+\|\nabla X\|_0|\nabla W|_0\bigr),
\end{align*}
and
\begin{align*}
\|({\mathcal A}{\mathcal A}^t&-Id)'(Y;X)\nabla W_t\|_{\dot B^{s+1}_{1,1}}\lesssim \|\nabla X\|_0\|\nabla W_t\|_{\dot B^{s+1}_{2,1}}+\bigl(\|\nabla X\|_{\dot B^{s+1}_{2,1}}+|\nabla X|_0\|\nabla Y\|_{\dot B^{s+1}_{2,1}}\bigr)\|\nabla W_t\|_0,
\end{align*}
and
\begin{align*}
\|({\mathcal A}{\mathcal A}^t-Id)'(Y;W)\nabla X_t\|_{\dot B^{s+1}_{1,1}}
\lesssim & \|\nabla W\|_0\|\nabla X_t\|_{\dot B^{s+1}_{2,1}}+\bigl(\|\nabla W\|_{\dot B^{s+1}_{2,1}}+|\nabla W|_0\|\nabla Y\|_{\dot B^{s+1}_{2,1}}\bigr)\|\nabla X_t\|_0.
\end{align*}
Hence by virtue of \eqref{S7.1eq5}, we conclude that
\begin{equation}\label{S7.3eq3}
\begin{split}
\|f_0''(Y;X,W)\|_{\dot B^s_{1,1}}
 \lesssim& | Y_t|_1\big(\|\nabla W\|_0\|\nabla X\|_{\dot B^{s+1}_{2,1}}+\|\nabla X\|_0\|\nabla W\|_{\dot B^{s+1}_{2,1}}\big)\\
&+\big(\|\nabla Y_t\|_{\dot B^{s+1}_{2,1}}+| Y_t|_1\|\nabla Y\|_{\dot B^{s+1}_{2,1}}\big)(|\nabla X|_0\|\nabla W\|_0+\|\nabla X\|_0|\nabla W|_0)\\
&+ \|\nabla X\|_0\|\nabla W_t\|_{\dot B^{s+1}_{2,1}}+\bigl(\|\nabla X\|_{\dot B^{s+1}_{2,1}}+|\nabla X|_0\|\nabla Y\|_{\dot B^{s+1}_{2,1}}\bigr)\|\nabla W_t\|_0\\
&+\|\nabla W\|_0\|\nabla X_t\|_{\dot B^{s+1}_{2,1}}+\bigl(\|\nabla W\|_{\dot B^{s+1}_{2,1}}+|\nabla W|_0\|\nabla Y\|_{\dot B^{s+1}_{2,1}}\bigr)\|\nabla X_t\|_0.
\end{split}
\end{equation}
Then \eqref{S7.3eqE1} follows from \eqref{S7.3eq3} and the interpolation inequality \eqref{S4.1eq17}.

\medskip

\no\underline{{\bf Step 2.} Estimate  of $f_m''(Y;X,W)$.}
Again we only present the  estimates of $f_1''(Y;X,W)$. Recall \eqref{S7.2eqf1}, we shall split the estimate of $f_1''(Y;X,W)$ into the following 4 parts:

\noindent{(i) Estimate of $I_1.$} It follows from \eqref{S4.1eq6} and \eqref{S7.3eq1} that
\begin{align*}
\|I_1\|_{\dot B^s_{1,1}}
&\lesssim \|({\mathcal A}^t)''(Y;X,W)\|_{\dot B^s_{1,1}}|\nabla\vv p_1|_0+\|({\mathcal A}^t)''(Y;X,W)\|_0\|\nabla \vv p_1\|_{\dot B^s_{2,1}}\\
&\lesssim \big(\|\nabla X\|_{\dot B^s_{2,1}}\|\nabla W\|_0+\|\nabla X\|_0\|\nabla W\|_{\dot B^s_{2,1}}\big)|\p_3Y|_0\|\p_3Y\|_{3}\\
&\qquad +\bigl(\|\p_3Y\|_{\dot B^{s+1}_{2,1}}+\|\p_3Y\|_{3}\|\nabla Y\|_{\dot B^{s+1}_{2,1}}\bigr)|\p_3Y|_0|\nabla X|_0\|\nabla W\|_0.
\end{align*}

\noindent{(ii) Estimate of $I_2$.}
We deduce from  \eqref{S4.1eq13} and \eqref{S4.2eq11} that
\begin{align*}
\|I_2\|_{\dot B^s_{1,1}}
&\lesssim |({\mathcal A}^t)'(Y;X)|_0\|\nabla\big(\vv p_1'(Y;W)\big)\|_{\dot B^s_{1,1}}+\|({\mathcal A}^t)'(Y;X)\|_{\dot B^s_{2,1}}\|\nabla\big(\vv p_1'(Y;W)\big)\|_0\\
&\lesssim  \bigl(|\p_3Y|_{1}\|\p_3W\|_{1}+\big(|\p_3Y|_{1}^\frac43\|\p_3Y\|_{0}^\frac23+|\p_3Y|_{1}^2\big)\|\na W\|_{1}\bigr)\|\nabla X\|_{\dot B^s_{2,1}}\\
&\quad+\Bigl(\bigl(\|\p_3W\|_{1}+|\p_3Y|_{1}\|\na W\|_{1}\bigr)\|\p_3Y\|_{\dot B^{s+1}_{2,1}}+\bigl(\|\p_3W\|_{1}+\|\p_3Y\|_{3}\|\na W\|_{1}\bigr)\times
\\
&\qquad\quad\times |\p_3Y|_{1}\|\nabla Y\|_{\dot B^{s+1}_{2,1}}
 + \bigl(\|\p_3W\|_{\dot B^{s+1}_{2,1}}
+|\p_3Y|_{0}\|\nabla W\|_{\dot B^{s+1}_{2,1}}\bigr)\|\p_3Y\|_{0}\Bigr)|\nabla X|_0.
\end{align*}

\noindent{(iii) Estimate of $I_3.$} Similar to the estimate of $I_2,$ we have
\begin{align*}
\|I_3\|_{\dot B^s_{1,1}}
&\lesssim  \bigl(|\p_3Y|_{1}\|\p_3X\|_{1}+(|\p_3Y|_{1}^\frac43\|\p_3Y\|_{0}^\frac23+|\p_3Y|_{1}^2)\|\na X\|_{1}\bigr)\|\nabla W\|_{\dot B^s_{2,1}}\\
&+\Bigl(\big(\|\p_3X\|_{1}+|\p_3Y|_{1}
\|\nabla X\|_{1}\bigr)\|\p_3Y\|_{\dot B^{s+1}_{2,1}}+\bigl(\|\p_3X\|_{1}+\|\p_3Y\|_{3}\|\nabla X\|_{1}\big)\times\\
& \qquad\times|\p_3Y|_{1}\|\nabla Y\|_{\dot B^{s+1}_{2,1}}
+\bigl(|\p_3Y|_{0}\|\nabla X\|_{\dot B^{s+1}_{2,1}}+ |\nabla W|_0\|\p_3X\|_{\dot B^{s+1}_{2,1}}\bigr)\|\p_3Y\|_{0}\Bigr)|\nabla W|_0.
\end{align*}

\noindent{(iv) Estimate of $I_4.$} We first deduce from the law of product \eqref{product3} that
\begin{align*}
\|I_4\|_{\dot B^s_{1,1}}
&\lesssim\|\nabla \big(\vv p_m''(Y;X,W)\big)\|_{\dot B^s_{1,1}}+\|\nabla Y\|_{\dot B^s_{2,1}}\|\nabla\big(\vv p_m''(Y;X,W)\big)\|_0.
\end{align*}
Thanks to \eqref{S7.2eq6}, it remains to handle the
estimate of $\|\nabla \big(\vv p_m''(Y;X,W)\big)\|_{\dot B^s_{1,1}}.$
As a matter of fact, thanks to  \eqref{p1''}, by  applying the laws of product, \eqref{product1} and \eqref{product3}, and using the estimates \eqref{S4.1eq6},
 \eqref{S4.1eq13}, \eqref{S4.2eq11}, \eqref{S7.2eq6} and \eqref{S7.3eq1},
  we obtain
\begin{equation}\label{S7.3eq4}
\begin{split}
\|&\nabla\vv p_1''(Y;X,W)\|_{\dot B^s_{1,1}}
\lesssim \bigl(\|\p_3W\|_0+|\p_3Y|_0\|\nabla W\|_{0}+\|\p_3Y\|_{0}|\nabla W|_0\bigr)\|\p_3X\|_{\dot B^{s+1}_{2,1}}\\
&+\bigl(\|\nabla X\|_{0}|\p_3W|_0 +|\nabla X|_1\|\p_3W\|_{1}+|\p_3X|_1\|\nabla W\|_{1}+\|\p_3X\|_{1}|\nabla W|_0\big)\|\p_3Y\|_{\dot B^{s+1}_{2,1}}\\
&+ \bigl(\|\p_3X\|_0+\|\nabla X\|_{0}|\p_{3}Y|_0+\|\p_3Y\|_{0}|\nabla X|_0\bigr)\|\p_3W\|_{\dot B^{s+1}_{2,1}}+\Bigl(|\p_3X|_0\|\p_3W\|_{0}\\
&+|\p_3Y|_0(\|\nabla X\|_{0}|\p_3W|_0 +|\nabla X|_1\|\p_3W\|_{1})+|\p_3Y|_{1}\|\p_3X\|_{1}|\nabla W|_0\\
&+\bigl(|\p_3Y|_0+ |\p_3Y|_1^\frac13\|\p_3Y\|_{1}^\frac23\big)|\p_3X|_1\|\nabla W\|_{1}\Bigr)\|\nabla Y\|_{\dot B^{s+1}_{2,1}}+ |\p_3Y|_{0}\Bigl(\bigl(\| \p_{3}W\|_{0}\\
&+\|\p_3Y\|_{0}|\nabla W|_0\bigr)\|\nabla X\|_{\dot B^{s+1}_{2,1}}+\bigl(\| \p_{3}X\|_{0}+\|\p_3Y\|_{0}|\nabla X|_0\bigr)\|\nabla W\|_{\dot B^{s+1}_{2,1}}\Bigr)\\
&+ \big(|\p_3Y|_{1}^\frac43\|\p_3Y\|_{0}^\frac23+|\p_3Y|_{1}\|\p_3Y\|_{3}\big)\big(\|\nabla X\|_{\dot B^{s+1}_{2,1}}\|\nabla W\|_0+ \|\nabla X\|_{0}\|\nabla W\|_{\dot B^{s+1}_{2,1}}\big)\\
&+\big(|\p_3Y|_{1}\|\p_3Y\|_{\dot B^{s+1}_{2,1}}+|\p_3Y|_{1}\|\p_3Y\|_{3}\|\nabla Y\|_{\dot B^{s+1}_{2,1}}\big)\big(|\nabla X|_0\|\nabla W\|_{1}+\|\nabla X\|_{1}|\nabla W|_0\big),
\end{split}
\end{equation}
The same  estimate  holds for $\|I_4\|_{\dot B^s_{1,1}}$.

 By  summing up the above estimates of $I_1,I_2,I_3,I_4$, we arrive at
\begin{equation}\label{S7.3eq5}
\begin{split}
\|&f_1''(Y;X,W)\|_{\dot B^s_{1,1}}\lesssim \bigl(\|\p_3W\|_0+|\p_3Y|_0\|\nabla W\|_{0}+\|\p_3Y\|_{0}|\nabla W|_0\bigr)\|\p_3X\|_{\dot B^{s+1}_{2,1}}\\
&\qquad+\bigl(\|\p_3X\|_0+|\p_{3}Y|_0\|\nabla X\|_{0}+\|\p_3Y\|_{0}|\nabla X|_0\bigr)\|\p_3W\|_{\dot B^{s+1}_{2,1}}\\
&\qquad+ \bigl(|\p_{3}Y|_0\| \p_{3}W\|_{0}+(|\p_3Y|_{1}\|\nabla W\|_0+|\p_3Y|_{0}|\nabla W|_0)\|\p_3Y\|_{3}\bigr)\|\nabla X\|_{\dot B^{s+1}_{2,1}}\\
&\qquad+\bigl(|\p_3Y|_0\| \p_{3}X\|_{0}+(|\p_3Y|_{1}\|\nabla X\|_{0}+|\p_3Y|_{0}|\nabla X|_0)\|\p_3Y\|_{3}\big)\|\nabla W\|_{\dot B^{s+1}_{2,1}}\\
&\qquad+|\p_3X|_0\|\p_3W\|_{0}\|\nabla Y\|_{\dot B^{s+1}_{2,1}}+\big(\|\p_3Y\|_{\dot B^{s+1}_{2,1}}+|\p_3Y|_{1}\|\nabla Y\|_{\dot B^{s+1}_{2,1}}\big)\|\p_3X\|_{1}|\nabla W|_0
\\
&\qquad+\big(\|\p_3Y\|_{\dot B^{s+1}_{2,1}}+|\p_3Y|_0\|\nabla Y\|_{\dot B^{s+1}_{2,1}}\big)\big(\|\nabla X\|_{0}|\p_3W|_0 +|\nabla X|_1\|\p_3W\|_{1}\big)\\
&\qquad+\big(\|\p_3Y\|_{\dot B^{s+1}_{2,1}}+(|\p_3Y|_0+ |\p_3Y|_1^\frac13\|\p_3Y\|_{1}^\frac23)\|\nabla Y\|_{\dot B^{s+1}_{2,1}}\big)|\p_3X|_1\|\nabla W\|_{1}\\
&\qquad+\big(|\p_3Y|_{1}\|\p_3Y\|_{\dot B^{s+1}_{2,1}}+|\p_3Y|_{1}\|\p_3Y\|_{3}\|\nabla Y\|_{\dot B^{s+1}_{2,1}}\big)\big(|\nabla X|_0\|\nabla W\|_{1}+\|\nabla X\|_{1}|\nabla W|_0\big).
\end{split}
\end{equation}
Then \eqref{S7.3eqE2} follows from \eqref{S7.3eq5} and interpolation inequality \eqref{S4.1eq17}.
Finally, the proof of \eqref{S7.3eqE3} follows along the same line to that of \eqref{S7.3eqE2}. We omit the details here.
This completes the proof of Proposition \ref{S7.3prop1}. \end{proof}




\section{The proof of Theorem \ref{Th1}}\label{Sect10}
The goal of this section is to prove Theorem  \ref{Th1} by using Nash-Moser scheme.
The key ingredients  are the uniform estimates of the approximate solutions obtained in  Propositions \ref{S9prop1}, \ref{S9col1} and \ref{S9prop2}, which we will prove by induction in what follows.

\subsection{The estimates of $Y_0$}

Recall that $Y_0$ solves the linear equation \eqref{S9.1eq2}.
Let $\bar{N}_0=N_0+6$, for $\eta\in]0,1[$,  we choose the initial data $(Y^{(0)},Y^{(1)})$  such that \eqref{S1eq4} holds for $L_0= N_0+12$.
Then we get, by applying \eqref{S2eq3as} of Proposition \ref{S0prop1}, that
\begin{equation}\label{S9.2eq1}
\begin{split}
|\p_3Y_0|_{1,\bar N_0}+ &|\p_tY_{0}|_{\frac32-\delta,\bar N_0}+ |Y_0|_{\frac12,\bar N_0}\\
\leq& C_{N_0}\bigl(\||D|^{2\delta}(Y^{(0)},Y^{(1)})\|_{L^1}+\||D|^{\bar{N}_0+4}(|D|^2Y^{(0)},Y^{(1)})\|_{L^1}\bigr)\leq \eta.
\end{split}
\end{equation}
Note that \beno
\||D|^{-1}h\|_0\leq \Bigl(\int_{|\xi|\leq 1}\f1{|\xi|^2}|\hat{h}(\xi)|^2\,d\xi\bigr)^{\f12}+\|h\|_0\leq |\hat{h}|_0+\|h\|_0\leq
\|h\|_{L^1}+\|h\|_0,\eeno
so that we get,  by applying  \eqref{S2eq13as}, \eqref{S2eq16} and \eqref{S2eq15'as} of Proposition \ref{S0prop1}, that
\beq \label{S9.2eq2}
\begin{split}
&\||D|^{-1}(\p_3Y_0,\p_tY_{0})\|_{0,\bar N_0+2}+\|\nabla Y_0\|_{0,\bar N_0+1}+\|\nabla \p_tY_{0}\|_{1,\bar N_0-1}\\
&+\|(\p_tY_{0},\p_3Y_{0})\|_{\frac12,\bar N_0+1}+\|\p_tY_{0}\|_{L_t^2(H^{\bar N_0+2})}+\bigl\|(\p_3Y_0,\langle t\rangle^\frac12\nabla \p_tY_{0})\bigr\|_{L_t^2(H^{\bar N_0+1})}\\
&\leq C_{\bar N_0}\||D|^{-1}(\p_3Y^{(0)},Y^{(1)},\Delta Y^{(0)})\|_{\bar N_0+2}\\
&
\leq C_{\bar N_0}\bigl(\|(\p_3Y^{(0)},Y^{(1)},\Delta Y^{(0)})\|_{L^1}+\|(\p_3Y^{(0)},Y^{(1)},\Delta Y^{(0)})\|_{\bar N_0+1} \bigr)
\leq \eta.
\end{split}
\eeq
By virtue of \eqref{S9.2eq1} and \eqref{S9.2eq2}, we deduce from Proposition \ref{S6prop2} that
 \beq \label{S9.2eq3}
\begin{split}
\|\langle t\rangle|D|^{-1}f&(Y_0)\|_{L_t^2(H^{N_0+1})}
\lesssim |\p_3Y_0|_{1,0}\|\p_3Y_0\|_{L_t^2(H^{N_0+1})}
+ |\p_tY_{0}|_{1,0}\|\p_tY_{0}\|_{L_t^2(H^{N_0+1})}\\
&
+|\nabla Y_{0}|_{\frac12,0}\|\langle t\rangle^\frac12\nabla \p_tY_{0}\|_{L_t^2(H^{N_0+1})}+|\nabla Y_{0}|_{\frac12,N_0+1}\|\langle t\rangle^\frac12\nabla \p_tY_{0}\|_{L_t^2(L^2)}\\
&+|\nabla Y_0|_{0,N_0+1}\bigl(|\p_3Y_0|_{1,0}\|\p_3Y_0\|_{L_t^2(H^{1})}
+ |\p_tY_{0}|_{1,0}\|\p_tY_{0}\|_{L_t^2(H^{1})}\bigr)
\lesssim C_{N_0}\eta^2,
\end{split}\eeq
and
\beq\label{S9.2eq3'}
\begin{split}
\||D|^{-1}f(Y_0)&\|_{\frac32,N_0+1}
\lesssim |\p_3Y_0|_{1,0}\|\p_3Y_0\|_{\frac12,N_0+1}+ |\p_tY_{0}|_{1,0}\|\p_tY_{0}\|_{\frac12,N_0+1}\\
&+|\nabla Y_{0}|_{\frac12,0}\|\nabla \p_tY_{0}\|_{1,N_0+1}+|\nabla Y_{0}|_{\frac12,N_0+1}\|\nabla \p_tY_{0}\|_{1,0}\\
& +|\nabla Y_0|_{0,N_0+1}\big(|\p_3Y_0|_{1,0}\|\p_3Y_0\|_{\frac12,1}+|\p_tY_{0}|_{1,0}\|\p_tY_{0}\|_{\frac12,1}\big)
\lesssim C_{N_0}\eta^2.
\end{split} \eeq
Similarly, we deduce from Proposition \ref{S6prop1} and \eqref{S9.2eq1}, \eqref{S9.2eq2} that
\begin{equation}\label{S9.2eq4}
\begin{split}
&\triplenorm{\langle t\rangle^\frac12 f(Y_0)}_{L_t^2(\d,N_0)}
\lesssim \|\nabla Y_0\|_{0,0} \| \langle t\rangle^\frac12\nabla \p_tY_0\|_{L_t^2(H^{N_0+6})}+\|\nabla Y_0\|_{0,N_0+6}\| \langle t\rangle^\frac12\nabla \p_tY_0\|_{L_t^2(L^2)}\\
&\quad\qquad+\|\p_3Y_0\|_{\frac12,0}\|\p_3Y_0\|_{L_t^2(H^{N_0+6})}+\|\p_tY_{0}\|_{\frac12,0}\|\p_tY_{0}\|_{L_t^2(H^{N_0+6})}\\
&\qquad\quad+\|\nabla Y_0\|_{0,N_0+6}\big(|\p_3Y_0|_{\frac12,0}\|\p_3Y_0\|_{L_t^2(H^1)}+|\p_tY_{0}|_{\frac12,0}\|\p_tY_{0}\|_{L_t^2(H^1)}\big)
\lesssim C_{N_0}\eta^2.
\end{split}
\end{equation}

\subsection{The proof of Proposition \ref{S9col1} and Proposition \ref{S9prop2} from Proposition \ref{S9prop1}}

Let us assume that
\beq \label{assumS91}
 ({\rm P}1, j), \ ({\rm P}2, j),\  ({\rm P}3, j)\ \text{of Proposition \ref{S9prop1} hold for} \ j\leq p,
\eeq
 we are going to prove Proposition \ref{S9col1} and Proposition \ref{S9prop2}.

\begin{proof}[Proof of Proposition \ref{S9col1}]
Notice from \eqref{S9.1eq5} that
$$|\p_3Y_{p+1}|_{k,N}\leq |\p_3Y_0|_{k,N}+\sum_{j=0}^p|\p_3X_{j}|_{k,N},
$$
which together with \eqref{S9.2eq1} and (P2, $j$) with $j\leq p$ ensures that for $\frac12\leq k\leq1$, $0\leq N\leq N_0$,
\begin{equation}\label{Ieq1}
\begin{split}
& |\p_3Y_{p+1}|_{k,N}\leq C \eta \theta_{p+1}^{k-\frac12-\gamma+\bar\varepsilon N},\quad  \text{if $k-\frac12-\gamma+\varepsilon N\geq\bar\varepsilon$,}\\
& |\p_3Y_{p+1}|_{k,N}\leq C \eta ,\qquad\qquad \qquad \text{if $k-\frac12-\gamma+\varepsilon N\leq-\bar\varepsilon$}.
\end{split}
\end{equation}
While for $\hat k\eqdefa\min(k,1)$, $\hat N\eqdefa\min(N,N_0),$ we observe from  the property \eqref{S9eqSI} of smoothing operator $S_{p+1}$ that
\beno
\begin{split}
&|S_{p+1}\p_3Y_{p+1}|_{k,N}\leq C|\p_3Y_{p+1}|_{k,N} \quad \mbox{for}\quad \frac12\leq k\leq 1, 0\leq N\leq N_0,\\
&|S_{p+1}\p_3Y_{p+1}|_{k,N}\leq C_{k,N}\theta_{p+1}^{\max (0,k-\hat k)}\theta_{p+1}^{\bar\varepsilon\max(0,N-\hat N)}|\p_3Y_{p+1}|_{\hat k,\hat N}
\quad \mbox{for}\quad k\geq 1\ \mbox{ or }\ N\geq N_0, \end{split}
\eeno
the first inequalities of (I)(i) and (II)(i) of Proposition \ref{S9col1} then follow from \eqref{Ieq1}.

Along the same line to proof of \eqref{Ieq1}, we have
\begin{itemize}
\item
for $1-\delta\leq k\leq\frac32-\delta$, $0\leq N\leq N_0$,
\begin{equation}\label{Ieq2}
\begin{split}
& |\p_tY_{p+1}|_{k,N}\leq C \eta \theta_{p+1}^{k-(1-\delta)-\gamma+\bar\varepsilon N},\quad  \text{if $k-(1-\delta)-\gamma+\varepsilon N\geq\bar\varepsilon$,}\\
& |\p_tY_{p+1}|_{k,N}\leq C \eta ,\qquad\qquad\quad \qquad \ \text{if $k-(1-\delta)-\gamma+\varepsilon N\leq-\bar\varepsilon$};
\end{split}\end{equation}
\item
for $0\leq k\leq\frac12$, $0\leq N\leq N_0$,
\begin{equation}\label{Ieq3}
\begin{split}
& |Y_{p+1}|_{k,N}\leq C \eta \theta_{p+1}^{k-\gamma+\bar\varepsilon N},\quad  \text{if $k-\gamma+\varepsilon N\geq\bar\varepsilon$,}\\
& |Y_{p+1}|_{k,N}\leq C \eta ,\qquad\qquad  \ \ \text{if $k-\gamma+\varepsilon N\leq-\bar\varepsilon$}.
\end{split}\end{equation}
\end{itemize}
Then  other inequalities in (I)(i) and (II)(i) of Proposition \ref{S9col1} follows.

(I)(ii) and (II)(ii) of Proposition \ref{S9col1} follow from property \eqref{S9eqSI} of the mollifying operator and the following fact
\begin{equation}\label{Ieq4}
\begin{split}
&\bigl\||D|^{-1}(\p_3Y_{p+1},\p_tY_{p+1})\bigr\|_{0,N+2}+\|\nabla Y_{p+1}\|_{0,N+1}+\|\p_tY_{p+1}\|_{L_t^2(H^{N+2})}\\
&\quad+ \|(\p_tY_{p+1},\p_3Y_{p+1})\|_{\frac12,N+1}+\|\nabla \p_tY_{p+1}\|_{1,N-1}
+\bigl\|(\p_3Y_{p+1},\langle t\rangle^\frac12\nabla  \p_tY_{p+1})\bigr\|_{L_t^2(H^{N+1})}\\
&\quad \leq
\begin{cases}
C\eta\theta_{p+1}^{-\beta+\bar \varepsilon N},\quad\text{for }-\beta+\bar\varepsilon N\geq\bar \varepsilon,\ N\leq N_0,\\
C\eta,\qquad\qquad\ \text{for }-\beta+\bar\varepsilon N\leq-\bar \varepsilon, \ N\leq N_0,
\end{cases}
\end{split}
\end{equation}
which is a direct consequence of (P1,$j$) of  Proposition \ref{S9prop1} for $j\leq p$ and \eqref{S9.2eq2}.

Finally let us prove (III) of Proposition \ref{S9col1}. Indeed it follows from
property \eqref{S9eqSII} of $S_{p+1}$ that
\begin{align*}
|(1-S_{p+1})\p_3Y_{p+1}|_{\frac12,0}\leq C\big(\theta_{p+1}^{-\frac12}|\p_3Y_{p+1}|_{1,0}+\theta_{p+1}^{-\bar\varepsilon N_0}|\p_3Y_{p+1}|_{\frac12,N_0}\big)
\end{align*}
While due to  \eqref{S9eqP} and \eqref{S9eqP'}, there hold $\frac12-\gamma\geq\bar\varepsilon$ and $-\gamma+\bar\varepsilon N_0\geq\bar\varepsilon$, so that
 we can apply \eqref{Ieq1} to deduce that
\begin{align}\label{IIIeq1}
|(1-S_{p+1})\p_3Y_{p+1}|_{\frac12,0}\leq C\eta\big(\theta_{p+1}^{-\frac12}\theta_{p+1}^{\frac12-\gamma}+\theta_{p+1}^{-\bar\varepsilon N_0}\theta_{p+1}^{-\gamma+\bar\varepsilon N_0}\big)\leq C\eta\theta_{p+1}^{-\gamma}.
\end{align}
Using \eqref{Ieq1} once  again gives rise to
\begin{align}
&|(1-S_{p+1})\p_3Y_{p+1}|_{1,N}\leq C|\p_3Y_{p+1}|_{1,N}\leq \eta\theta_{p+1}^{\frac12-\gamma+\bar \varepsilon N} \qquad\quad\mbox{for}\quad 0\leq N\leq N_0, \label{IIIeq2}\\
&|(1-S_{p+1})\p_3Y_{p+1}|_{k,N_0}\leq C|\p_3Y_{p+1}|_{k,N_0}\leq \eta\theta_{p+1}^{k-\frac12-\gamma+\bar \varepsilon N_0} \quad\mbox{for}\quad \frac12\leq k\leq 1. \label{IIIeq3}
\end{align}
Interpolating between \eqref{IIIeq1}, \eqref{IIIeq2} and \eqref{IIIeq3} leads to
\begin{align*}
|(1-S_{p+1})\p_3Y_{p+1}|_{k,N}\leq C\eta\theta_{p+1}^{k-\frac12-\gamma+\bar \varepsilon N},\quad \text{for all } \frac12\leq k\leq 1,\ 0\leq N\leq N_0.
\end{align*}
The other two inequalities in (III) of Proposition \ref{S9col1} can be proved by the same procedure. This completes the proof of Proposition
\ref{S9col1}.  \end{proof}

Let us now turn to the proof of Proposition \ref{S9prop2}.

\begin{proof}[Proof of Proposition \ref{S9prop2}] We shall divide the proof of this proposition by the following steps:

\noindent\underline{{\bf Step 1.} The Proof of (IV) of Proposition \ref{S9prop2}.}
%
The  proof of (IV) will be based on the following lemmas:
\begin{lem}\label{S9lem1}
{\sl Let  $e_{p,j}', e_{p,j}'',$ for $j=0,1,2,$ be given by \eqref{S9.1eq10}. Then under the assumption of \eqref{assumS91}, one has
\begin{align}
&\|\langle t\rangle^{\frac12+k}|D|^{-1}(e_{p,1}''+e_{p,2}'')\|_{L_t^2(H^{N+1})}
\lesssim \eta^2\theta_{p}^{k-\gamma-\beta+\bar\varepsilon(N+1)}\quad \text{if}\  0\leq k\leq\frac12, 0\leq N\leq N_0-1;\label{IVieq0}\\
&\|\langle t\rangle^{\frac12+k}|D|^{-1}(e_{p,1}'+e_{p,2}')\|_{L_t^2(H^{N+1})}\lesssim\eta^2\theta_p^{k-\gamma-\beta+\bar\varepsilon(N+1)} \quad \text{if}\ 0\leq k\leq \frac12,
0\leq N\leq N_0-1; \label{IVieq-1}\\
&\|\langle t\rangle^{\frac12+k}|D|^{-1}e_{p,0}''\|_{L_t^2(H^{N+1})}
\lesssim \eta^2\theta_p^{k+\delta-\gamma-\beta+\bar\varepsilon(N+3)}\quad \text{if}\  0\leq k\leq \alpha, 0\leq N\leq N_0-2;
 \label{IVieq-2}\\
&\|\langle t\rangle^{k+\frac12}|D|^{-1}e_{p,0}'\|_{L_t^2(H^{N+1})}\lesssim \eta^2\theta_p^{k+\delta-\gamma-\beta+\bar\varepsilon(N+3)}\quad \text{if}\
0\leq k\leq \alpha, 0\leq N\leq N_0-2. \label{IVieq-3}
\end{align}
}
\end{lem}

\begin{lem}\label{S9lem2}
{\sl Under the assumption of Lemma \ref{S9lem1}, one has
\begin{align}
&
\||D|^{-1}(e_{p,1}''+e_{p,2}'')\|_{1+k,N+1}\lesssim \eta^2\theta_p^{k-\gamma-\beta+\bar\varepsilon(N+1)}\quad \text{if}\ \ 0\leq k\leq\frac12, 0\leq N\leq N_0-1;
\label{IVieq-4}\\
&\||D|^{-1}(e_{p,1}'+e_{p,2}'')\|_{1+k,N+1}\lesssim \eta^2\theta_p^{k-\gamma-\beta+\bar\varepsilon(N+1)}\quad \text{if}\  \  0\leq k\leq\frac12,  0\leq N\leq N_0-1; \label{IVieq-5}\\
& \||D|^{-1}e_{p,0}''\|_{1+k,N+1}\lesssim \eta^2\theta_{p}^{k+\delta-\gamma-\beta+\bar\varepsilon(N+2)}\quad \text{if}\  \  0\leq k\leq\frac12-\delta, N\leq N_0-2;
 \label{IVieq-6}\\
& \||D|^{-1}e_{p,0}'\|_{1+k,N+1}\lesssim \eta^2\theta_p^{k+\delta-\gamma-\beta+\bar\varepsilon(N+2)}\quad \text{if}\  \ 0\leq k\leq\frac12-\delta, N\leq N_0-2.
 \label{IVieq-7}
\end{align}
}
\end{lem}


\begin{lem}\label{S9lem3}
{\sl Under the assumption of Lemma \ref{S9lem1}, for  $0\leq N\leq N_0-6,$
there hold
\begin{align}
& \triplenorm{(e_{p,1}''+e_{p,2}'')}_{L_t^1(\d,N)}
\lesssim \eta^2\theta_p^{-\beta-\gamma+\bar\varepsilon(N+5)};\label{IVieq-a}\\
& \triplenorm{(e_{p,1}'+e_{p,2}')}_{L_t^1(\d,N)}
\lesssim \eta^2\theta_p^{-\gamma+\bar\varepsilon(N+5)}; \label{IVieq-b}\\
& \triplenorm{\langle t\rangle^\frac12 e_{p,0}''}_{L_t^2(\d,N)} \lesssim  \eta^2\theta_p^{-\beta-\gamma+\bar\varepsilon(N+5)};\label{IVieq-c}\\
& \triplenorm{\langle t\rangle^\frac12 e_{p,0}'}_{L_t^2(\d,N)}
\lesssim \eta^2\theta_p^{-\gamma+\bar\varepsilon(N+5)}.
 \label{IVieq-d}
\end{align}
}
\end{lem}

We shall postpone the proof of the above lemmas in the Appendix \ref{appendixa}.
It is easy to observe that  (IV) (i) follows from Lemma \ref{S9lem1}, (IV) (ii) from Lemma \ref{S9lem2}, and (IV) (iii) from Lemma \ref{S9lem3}.

\medskip

\noindent\underline{{\bf Step 2.}
The proof of (V) of Proposition \ref{S9prop2}.}
Recall \eqref{S9.1eq9}  that
\begin{equation*}
\begin{split}
g_{p+1}&=-(S_{p+1}-S_{p})E_{p}-S_{p+1}e_{p}+(S_{p+1}-S_{p})f(Y_0).
\end{split}
\end{equation*}
In the sequel, we shall handle term by term above.

\noindent$\bullet$\underline{\bf Estimates of $S_{p+1}e_{p}$}

It follows from  (IV) of Proposition \ref{S9prop2} and property \eqref{S9eqSI} that for $k\geq0$ and $ N\geq0$
\begin{equation}\label{Veqa}
\begin{split}
&\|\langle t\rangle^{k+\frac12}|D|^{-1}S_{p+1}e_{p}\|_{L_t^2(H^{N+1})}
\lesssim \eta^2\theta_{p+1}^{k+\delta-\gamma-\beta+\bar\varepsilon(N+3)};\\
&
\||D|^{-1}S_{p+1}e_{p}\|_{1+k,N+1}
\lesssim \eta^2\theta_{p+1}^{k+\delta-\gamma-\beta+\bar\varepsilon(N+2)};\\
&\triplenorm{\langle t\rangle^\frac12 S_{p+1}e_{p}}_{L_t^2(\d,N)}
\lesssim \eta^2\theta_{p+1}^{-\gamma+\bar\varepsilon(N+5)}.
\end{split}
\end{equation}
Notice that the operator $S_{p+1}$ contains a cutoff in the variable $t$ of size $\theta_{p+1}$ so that
\begin{equation}\label{Veqa1}
\begin{split}
&\triplenorm{S_{p+1}e_p}_{L_t^1(\delta,N)}\lesssim (\log\theta_{p+1})^\frac12\triplenorm{\langle t\rangle^\frac12 S_{p+1}e_{p}}_{L_t^2(\d,N)}
\lesssim \eta^2\theta_{p+1}^{-\gamma+\bar\varepsilon(N+6)}.
\end{split}
\end{equation}

\noindent$\bullet$\underline{\bf Estimates for $(S_{p+1}-S_p)E_{p}$}

We first deduce  from (IV) (i) of Proposition \ref{S9prop2} that for $0\leq k\leq\alpha$ and $0\leq N\leq N_0-2$
\begin{equation}\label{Veq1}
\begin{split}
\|\langle t\rangle^{k+\frac12}&|D|^{-1}E_{p}\|_{L_t^2(H^{N+1})}\leq \sum_{j=0}^{p-1}\|\langle t\rangle^{k+\frac12}D|^{-1}e_{j}\|_{L_t^2(H^{N+1})}\\
&\lesssim \begin{cases}
C\eta^2\theta_{p}^{k+\delta-\gamma-\beta+\bar\varepsilon(N+3)} \quad\text{ if}\  k+\delta-\gamma-\beta+\bar\varepsilon(N+3)\geq\bar\varepsilon;
\\
C\eta^2,\qquad\qquad\qquad\qquad\text{ if}\ k+\delta-\gamma-\beta+\bar\varepsilon(N+3)\leq-\bar\varepsilon.
\end{cases}.
\end{split}
\end{equation}
In particular, due to the choice of parameters \eqref{S9eqP}, \eqref{S9eqP'}, there hold
\begin{equation}\label{Veq1'}
\frac12-\gamma-\beta+2\bar\varepsilon\geq\bar\varepsilon,\quad -\gamma-\beta+\bar\varepsilon (N_0+1)\geq\bar\varepsilon,
\end{equation}
we deduce from \eqref{Veq1} and the property \eqref{S9eqSII} of $1-S_p$ that
\begin{equation}\label{Veq2}
\begin{split}
&\|\langle t\rangle^\frac12|D|^{-1}(S_{p+1}-S_{p})E_{p}\|_{L_t^2(H^1)}\\
&\lesssim\theta_{p}^{-\alpha} \|\langle t\rangle^{\frac12+\alpha}|D|^{-1}E_{p}\|_{L_t^2(H^1)}+\theta_{p}^{-\bar\varepsilon (N_0-1)}\|\langle t\rangle^\frac12|D|^{-1}E_{p}\|_{L_t^2(H^{N_0-1})}\\
&\lesssim\eta^2\bigl(\theta_p^{-\alpha}\theta_p^{\alpha+\delta-\gamma-\beta+3\bar\varepsilon}
+\theta_p^{-\bar\varepsilon (N_0-1)}\theta_p^{\delta-\gamma-\beta+\bar\varepsilon(N_0+1)}\bigr)
\lesssim \eta^2\theta_{p+1}^{\delta-\gamma-\beta+3\bar\varepsilon}.
\end{split}
\end{equation}
On the other hand, for $k\leq\alpha$, $N\leq N_0-2$ with $k+\delta-\gamma-\beta+\bar\varepsilon (N+3)\geq\bar\varepsilon$, we have
\begin{equation}\label{Veq3}
\begin{split}
\|\langle t\rangle^{k+\frac12}|D|^{-1}(S_{p+1}-S_p)E_{p}\|_{L_t^2(H^{N+1})}
&\lesssim \|\langle t\rangle^{k+\frac12}|D|^{-1}E_{p}\|_{L_t^2(H^{N+1})}\\
&\lesssim \eta^2\theta_{p+1}^{k+\delta-\gamma-\beta+\bar\varepsilon(N+3)}.
\end{split}
\end{equation}
Interpolating between \eqref{Veq2} and \eqref{Veq3}, we conclude that
\begin{equation}\label{Veqb1}
\|\langle t\rangle^{k+\frac12}|D|^{-1}(S_{p+1}-S_p)E_{p}\|_{L_t^2(H^{N+1})}\lesssim  \eta^2\theta_{p+1}^{k+\delta-\gamma-\beta+\bar\varepsilon(N+3)}
\end{equation}  for $ 0\leq k\leq \alpha$ and $ 0\leq N\leq N_0-2.$ This together with
property \eqref{S9eqSI} of $S_{p}$ ensures that \eqref{Veqb1} holds for any $k\geq0$, $N\geq0$.

Similarly we infer from (IV) (ii) of Proposition \ref{S9prop2} that for $0\leq k\leq\frac12-\delta$, $0\leq N\leq N_0-2$,
\begin{equation}\label{Veq4}
\begin{split}
\||D|^{-1}E_{p}\|_{1+k,N+1}&\leq \sum_{j=0}^{p-1}\||D|^{-1}e_{j}\|_{1+k,N+1}\\
&\lesssim  \begin{cases}
 C\eta^2\theta_p^{k+\delta-\gamma-\beta+\bar\varepsilon(N+2)},\text{ if}\  k+\delta-\gamma-\beta+\bar\varepsilon(N+2)\geq\bar\varepsilon;
\\
C\eta^2,\qquad\qquad\qquad\quad\ \text{ if}\ k+\delta-\gamma-\beta+\bar\varepsilon(N+2)\leq-\bar\varepsilon.
\end{cases}
\end{split}
\end{equation}
Then due to \eqref{Veq1'},
we deduce from \eqref{Veq4} and the property \eqref{S9eqSII} of $1-S_p$ that
\begin{equation}\label{Veq5}
\begin{split}
\||D|^{-1}(S_{p+1}-S_p)E_{p}\|_{1,1}
&\lesssim\theta_p^{-\frac12+\delta} \||D|^{-1}E_{p}\|_{\frac32-\delta,1}+\theta_p^{-\bar\varepsilon (N_0-1)}\||D|^{-1}E_{p}\|_{1,N_0-1}\\
&\lesssim\eta^2\bigl(\theta_p^{-\frac12+\delta}\theta_p^{\frac12-\gamma-\beta+2\bar\varepsilon}+\theta_p^{-\bar\varepsilon (N_0-1)}\theta_p^{\delta-\gamma-\beta+\bar\varepsilon(N_0+1)}\bigr)\\
&\lesssim \eta^2\theta_{p+1}^{\delta-\gamma-\beta+2\bar\varepsilon},
\end{split}
\end{equation}
On the other hand, for $k\leq\frac12-\delta$, $N\leq N_0-2$ such that $k+\delta-\gamma-\beta+\bar\varepsilon (N+2)\geq\bar\varepsilon$, we get
\begin{equation}\label{Veq6}
\begin{split}
&\||D|^{-1}(S_{p+1}-S_p)E_{p}\|_{1+k,N+1}\lesssim\||D|^{-1}E_{p}\|_{1+k,N+1}\lesssim \eta^2\theta_{p+1}^{k+\delta-\gamma-\beta+\bar\varepsilon(N+2)}.
\end{split}
\end{equation}
Interpolating between the inequalities \eqref{Veq5} and \eqref{Veq6}, we achieve \eqref{Veq6} for any $0\leq k\leq \frac12-\delta$, $0\leq N\leq N_0-2$.
This together with the
 property \eqref{S9eqSI} of $S_p$ ensures that \eqref{Veq6} holds for any $k\geq0$ and $N\geq0$.

It follows  from (IV) (iii) of Proposition \ref{S9prop2} that for $N\leq N_0-6$,
\begin{equation*}\label{}
\begin{split}
\triplenorm{\langle t\rangle^\frac12E_{p}}_{L_t^2(\d,N)} &\leq\sum_{j=0}^{p-1}\triplenorm{\langle t\rangle^\frac12e_{j}}_{L_t^2(\delta,N)}
\lesssim \eta^2\sum_{j=0}^{p-1}\theta_j^{-\gamma+\bar\varepsilon(N+5)}\\
&\lesssim  \begin{cases}
\eta^2\theta_{p+1}^{-\gamma+\bar\varepsilon(N+5)}\quad \text{if}\ -\gamma+\bar\varepsilon(N+5)\geq\bar\varepsilon;
\\
\eta^2,\quad\qquad\qquad\qquad\quad\  \text{if} \ -\gamma+\bar\varepsilon(N+5)\leq-\bar\varepsilon,
\end{cases}
\end{split}
\end{equation*}
which together with the property \eqref{S9eqSI} and compact support of mollifying operator ensures that for any $N\geq0$
\begin{equation}\label{Veqb2}
\triplenorm{(S_{p+1}-S_p)E_{p}}_{L_t^1(\d,N)} \lesssim
\begin{cases}
\eta^2\theta_{p+1}^{-\gamma+\bar\varepsilon(N+6)},\ \text{if}\ -\gamma+\bar\varepsilon(N+5)\geq\bar\varepsilon;
\\
\eta^2\theta_{p+1}^{\bar\e},\quad\qquad\qquad\ \ \  \text{if} -\gamma+\bar\varepsilon(N+5)\leq-\bar\varepsilon.
\end{cases}
\end{equation}

\noindent$\bullet$\underline{\bf Estimates for $(S_{p+1}-S_p)f(Y_0)$}

Recall \eqref{Veq1'},
we get, by applying \eqref{S9eqSII} and \eqref{S9.2eq3}, that
\begin{align*}
\|\langle t\rangle^\frac12|D|^{-1}(S_{p+1}-S_p)f(Y_0)\|_{L_t^2(H^{1})}
\lesssim& \theta_{p+1}^{-\frac12}\|\langle t\rangle|D|^{-1}(S_{p+1}-S_p)f(Y_0)\|_{L_t^2(H^{1})}\\
&+\theta_{p+1}^{-\bar\varepsilon N_0}\|\langle t\rangle^\frac12|D|^{-1}(S_{p+1}-S_p)f(Y_0)\|_{L_t^2(H^{N_0+1})}\\
\lesssim& \eta^2\big(\theta_{p+1}^{-\frac12}+\theta_{p+1}^{-\bar\varepsilon N_0}\big)\lesssim \eta^2\theta_{p+1}^{-\gamma-\beta+\bar\varepsilon}.
\end{align*}
Whereas for $k\leq\frac12$ and $N\leq N_0$ with  $k-\gamma-\beta+\bar\varepsilon(N+3)\geq\bar\varepsilon$, we deduce from \eqref{S9.2eq3} that
\begin{align*}
\|\langle t\rangle^{k+\frac12}|D|^{-1}(S_{p+1}-S_p)f(Y_0)\|_{L_t^2(H^{N+1})}
\lesssim& \|\langle t\rangle|D|^{-1}f(Y_0)\|_{L_t^2(H^{N_0+1})}\\
\lesssim&\eta^2\leq \eta^2\theta_{p+1}^{k-\gamma-\beta+\bar\varepsilon (N+3)}.
\end{align*}
Interpolating the above two inequalities gives rise to
\begin{align*}
&\|\langle t\rangle^{k+\frac12}|D|^{-1}(S_{p+1}-S_p)f(Y_0)\|_{L_t^2(H^{N+1})}\leq \eta^2\theta_{p+1}^{k-\gamma-\beta+\bar\varepsilon (N+3)}
\end{align*} for all $0\leq k\leq\frac12,0\leq N\leq N_0$.
This together with the  property \eqref{S9eqSI} of $S_{p+1}$ ensures that
\begin{align}\label{Veqc1}
&\|\langle t\rangle^{k+\frac12}|D|^{-1}(S_{p+1}-S_p)f(Y_0)\|_{L_t^2(H^{N+1})}\leq \eta^2\theta_{p+1}^{k-\gamma-\beta+\bar\varepsilon (N+3)}
\end{align}
 for all $k\geq0$ and $N\geq0$.

Along the same line, it follows from  \eqref{S9.2eq3'} that  for $k\geq0,\ N\geq 0$,
\begin{align}\label{Veqc2}
&\||D|^{-1}(S_{p+1}-S_p)f(Y_0)\|_{1+k,N+1}\leq \eta^2\theta_{p+1}^{k-\gamma-\beta+\bar\varepsilon (N+3)}.
\end{align}
And it follows from \eqref{S9.2eq4} that if $-\gamma+\bar\varepsilon (N+5)\leq -\bar\varepsilon$ (implying $N\leq N_0$),
\begin{align*}
&\triplenorm{(S_{p+1}-S_p)f(Y_0)}_{L_t^1(\d,N)}\lesssim (\log\theta_{p+1})^\frac12\triplenorm{\langle t\rangle^\frac12f(Y_0)}_{L_t^2(\d,N_0)}\lesssim \eta^2\theta_{p+1}^{\bar\e};
\end{align*}
and if $-\gamma+\bar\varepsilon(N+5)\geq\bar\varepsilon$, one has
\begin{align*}
\triplenorm{(S_{p+1}-S_p)f(Y_0)}_{L_t^1(\d,N)}
\lesssim& (\log\theta_{p+1})^\frac12\theta_{p+1}^{\bar\varepsilon\max(N-N_0,0)} \triplenorm{\langle t\rangle^\frac12f(Y_0)}_{L_t^2(\d,N_0)}\\
 \lesssim&  \eta^2\theta_{p+1}^{-\gamma+\bar\varepsilon(N+6)},
\end{align*}
by using \eqref{S9eqSI} and the fact that $\bar\varepsilon (N_0+5)\geq\gamma$.
Along with \eqref{Veqa}, \eqref{Veqa1}, \eqref{Veqb1}, \eqref{Veq6}, \eqref{Veqb2}, \eqref{Veqc1}, \eqref{Veqc2},  we complete the proof of (V).

\noindent\underline{{\bf Step 3.}
The  proof of (VI) of Proposition \ref{S9prop2}.}

In the case when $-\gamma+\bar\varepsilon (N+5)\geq\bar\varepsilon$, we deduce from (V)(i), (V)(ii), (V)(iii)  of Proposition \ref{S9prop2} that
\begin{align*}
R_{N,\theta_{p+1}}(g_{p+1})&=\triplenorm{g_{p+1}}_{L_t^1(\d,N)}+\theta_{p+1}^{\frac12}\big\|\langle t\rangle^{\frac12} |D|^{-1}g_{p+1}\big\|_{L_t^2(H^{N+3})}+\log\langle\theta_{p+1}\rangle\||D|^{-1}g_{p+1}\|_{\frac32-\delta,N+3}\\
&\lesssim \eta^2\big(\theta_{p+1}^{-\gamma+\bar\varepsilon(N+6)}+\theta_{p+1}^{\frac12+\delta-\gamma-\beta+\bar\varepsilon(N+5)}
+\theta_{p+1}^{\frac12-\gamma-\beta+\bar\varepsilon(N+5)}\big)
\lesssim  \eta^2\theta_{p+1}^{\frac12-\gamma+\bar\varepsilon N},
\end{align*}
provided that
\beq \label{VIeqa} 6\bar\varepsilon\leq\frac12,\quad \beta\geq\delta+5\bar\varepsilon,\eeq
which are satisfied due to \eqref{S9eqP} and \eqref{S9eqebd}.

On the other hand, since $-\gamma+6\bar\varepsilon \leq-\bar\varepsilon$, we deduce from (V)(i), (V)(ii) and (V)(iv) of Proposition \ref{S9prop2} that
\begin{align*}
R_{0,\theta_{p+1}}(g_{p+1})&=\triplenorm{g_{p+1}}_{L_t^1(\d,0)}+\theta_{p+1}^{\frac12}\big\|\langle t\rangle^{\frac12} |D|^{-1}g_{p+1}\big\|_{L_t^2(H^{3})}+\log\langle\theta_{p+1}\rangle\||D|^{-1}g_{p+1}\|_{\frac32-\delta,3},\\
&\lesssim \eta^2\big(\theta_{p+1}^{\bar\e}+\theta_{p+1}^{\frac12+\delta-\gamma-\beta+5\bar\varepsilon}+\theta_{p+1}^{\frac12-\gamma-\beta+5\bar\varepsilon}\big)
\lesssim \eta^2\theta_{p+1}^{\frac12-\gamma},
\end{align*}
due to \eqref{VIeqa} and $\frac12-\gamma\geq\bar\e$.  This finishes the proof of (VI) of Proposition \ref{S9prop2} and hence the whole Proposition \ref{S9prop2}.
\end{proof}

\subsection{The proof of Proposition \ref{S9prop1} from Proposition \ref{S9col1} and Proposition \ref{S9prop2}}
Let us assume in this subsection that
\beq\label{assumS92}
\text{both Proposition \ref{S9col1} and  Proposition \ref{S9prop2} are valid}
\eeq
we are going to prove  (P1, $p+1$), (P2, $p+1$), (P3, $p+1$), that is, Proposition \ref{S9prop1} is valid for $p+1.$

\begin{proof}[Proof of Proposition \ref{S9prop1}] We shall divide its proof into the following steps:

\noindent\underline{{\bf Step 1.}
The proof of (P3, $p+1)$ of Proposition \ref{S9prop1}.}

(P3, $p+1$) is a direct consequence of \eqref{Ieq1}, \eqref{Ieq2}, \eqref{Ieq3}, \eqref{Ieq4} and the choices of
parameters (see \eqref{S9eqP} and \eqref{S9eqebd})
$$\beta\geq3\bar\varepsilon,\quad C\eta\leq\delta_1,\quad \gamma\geq\delta+\varepsilon+3\bar\varepsilon.$$

\noindent\underline{{\bf Step 2.}
The proof of (P1, $p+1$) of Proposition \ref{S9prop1}.}

Recall that $X=X_{p+1}$ solves
\begin{equation}\label{P1eq3}
X_{tt}-\Delta X_{t}-\p_3^2X=f'(S_{p+1}Y_{p+1};X)+g_{p+1}.
\end{equation}
 Due to (P3, $p+1$), the hypotheses of Theorem \ref{S5thm1} and \eqref{S5eqC4} are satisfied, so that we can apply the energy estimate \eqref{S5eqC3} to the system \eqref{P1eq3}.
When $N\geq0$ with $-\gamma+\bar\varepsilon (N+1)\geq\bar\varepsilon$ and $-\beta+\bar\e N\geq\bar\e$, we deduce from (I) (i), (ii) of Proposition \ref{S9col1} that
\begin{equation*}
\begin{split}
\widetilde\gamma_{\varepsilon,N+1}(S_{p+1}Y_{p+1})&\lesssim |S_{p+1}\p_3Y_{p+1}|_{\frac12+\varepsilon,N+1}
 +|S_{p+1} \p_tY_{p+1}|_{1+\varepsilon,N+2}+|S_{p+1}\nabla Y_{p+1}|_{0,N+1}\\
 &\quad\qquad\qquad+\|S_{p+1}\nabla Y_{p+1}\|_{0,N+1}+1\lesssim \theta_{p+1}^{-\gamma+\varepsilon+\delta+\bar\varepsilon(N+2)}+\theta_{p+1}^{-\beta+\bar\e N}.
\end{split}
\end{equation*}
Then in this case, we get, by applying the energy estimate \eqref{S5eqC3} to the system \eqref{P1eq3} and using (V) (i), (V) (ii) of Proposition \ref{S9prop2}, that
\begin{equation}\label{P1eq1}
\begin{split}
&\bigl\||D|^{-1}(\p_3X_{p+1},\p_tX_{p+1})\bigr\|_{0,N+2}+\|\nabla X_{p+1}\|_{0,N+1}+\|(\p_tX_{p+1},\p_3X_{p+1})\|_{\frac12,N+1}\\
&\quad +\|\p_tX_{p+1}\|_{L_t^2(H^{N+2})}
+\bigl\|(\p_3X_{p+1},\langle t\rangle^\frac12\nabla \p_tX_{p+1})\bigr\|_{L_t^2(H^{N+1})}+\|\nabla \p_tX_{p+1}\|_{1,N-1}\\
&\leq  C_{\varepsilon,N}\Big(\||D|^{-1}g_{p+1}\|_{1+\e,N+1}+\|\langle t\rangle^{\frac{1+\varepsilon}2}g_{p+1}\|_{L_t^2( H^{N})}\\
&\qquad\qquad+\widetilde\gamma_{\varepsilon,N+1}(S_{p+1}Y_{p+1})\big(\||D|^{-1}g\|_{1+\e,2}+\|\langle t\rangle^{\frac{1+\varepsilon}{2}}|D|^{-1}g_{p+1}\|_{L_t^2(H^1)}\big)\Big)\\
&\lesssim\eta^2 \theta_{p+1}^{\delta-\gamma-\beta+\bar\varepsilon N}\bigl(\theta_{p+1}^{\e+2\bar\varepsilon}+ \theta_{p+1}^{\frac{\varepsilon}{2}+3\bar\varepsilon}+(\theta_{p+1}^{-\gamma+\varepsilon+\delta+2\bar\varepsilon} +\theta_{p+1}^{-\beta}\big)\theta_{p+1}^{\varepsilon+3\bar\varepsilon}\bigr)\lesssim \eta\theta_{p+1}^{-\beta+\bar\varepsilon N}
\end{split}
\end{equation}
provided that $\gamma\geq\delta+\e+3\bar\e$ which is satisfied due to \eqref{S9eqP}, \eqref{S9eqebd}.
Along the same line, we have
\begin{equation}\label{P1eq2}
\begin{split}
&\bigl\||D|^{-1}(\p_tX_{p+1},\p_3X_{p+1})\bigr\|_{0,2}+\|\nabla X_{p+1}\|_{0,1}+\|(\p_tX_{p+1},\p_3X_{p+1})\|_{\frac12,1}+\|\p_tX_{p+1}\|_{L_t^2(H^{2})}\\
&\quad+\bigl\|(\p_3X_{p+1},\langle t\rangle^\frac12\nabla\p_t X_{p+1})\bigr\|_{L_t^2(H^{1})}
\leq C_\varepsilon\|\langle t\rangle^{\frac{1+\varepsilon}{2}}|D|^{-1}g_{p+1}\|_{L_t^2(H^1)}
\lesssim \eta\theta_{p+1}^{-\beta} \andf\\
 & \|\nabla \p_t X_{p+1}\|_{1,0}\leq C_\e(\||D|^{-1}g_{p+1}\|_{1+\e,2}+\|\langle t\rangle^{\frac{1+\varepsilon}{2}}|D|^{-1}g_{p+1}\|_{L_t^2(H^2)})\leq \eta\theta_{p+1}^{-\beta+\bar\e}.
\end{split}
\end{equation}
By interpolating the inequalities \eqref{P1eq1} and \eqref{P1eq2}, we achieve (P1, $p+1$) for $ N\geq 0$.

\noindent\underline{{\bf Step 3.}
The proof of (P2, $p+1$) of Proposition \ref{S9prop1}.}

Notice that by definition $S_{p+1}Y_{p+1}=0$ and $g_{p+1}=0$ for $t\geq\theta_{p+1}$.
In order to apply Proposition \ref{S2.3prop1} to the equation \eqref{P1eq3}, it remains to estimate
$$R_{N,\theta_{p+1}}\big(f'(S_{p+1}Y_{p+1};X_{p+1})\big)$$ given by \eqref{S2eq19'}.

\noindent$\bullet$\underline{\bf The  estimate of $\triplenorm{f'(S_{p+1}Y_{p+1};X_{p+1})}_{L_t^1(\d,N)}$}

It follows from  \eqref{S4.1eqf1} that
\begin{equation*}
\begin{split}
&\triplenorm{f_1'(S_{p+1}Y_{p+1};X_{p+1})}_{L_t^1(\d,N)}\lesssim \|S_{p+1}\p_3Y_{p+1}\|_{L_t^2(H^1)}\bigl( \|\p_3X_{p+1}\|_{L_t^2(H^{N+6})}\\
&\quad +|S_{p+1}\p_3Y_{p+1}|_{\frac12+\bar\varepsilon,0} \|\nabla X_{p+1}\|_{0,N+6}\bigr)+|S_{p+1}\p_3Y_{p+1}|_{\frac12+\bar\varepsilon,1}\|\nabla X_{p+1}\|_{0,1}\times\\
&\quad\times \big(\|S_{p+1}\p_3Y_{p+1}\|_{L_t^2(H^{N+6})}+\|S_{p+1}\nabla Y_{p+1}\|_{0,N+6}\|S_{p+1}\p_3Y_{p+1}\|_{L_t^2(H^3)}\big)\\
&\quad +\big(\|S_{p+1}\p_3Y_{p+1}\|_{L_t^2(H^{N+6})}+\|S_{p+1}\nabla Y_{p+1}\|_{0,N+6}|S_{p+1}\p_3Y_{p+1}|_{\frac12+\bar\varepsilon,1}\big)\|\p_3X_{p+1}\|_{L_t^2(H^1)},
\end{split}
\end{equation*}
from which, and (P1, $p+1$), (II) of Proposition \ref{S9col1} and the fact that $\beta\geq6\bar\varepsilon$, we infer
\begin{equation*}
\begin{split}
&\triplenorm{f_1'(S_{p+1}Y_{p+1};X_{p+1})}_{L_t^1(\d,0)}\lesssim\eta\theta_{p+1}^{-\beta+5\bar\varepsilon}.
\end{split}
\end{equation*}
While for $-\beta+\bar\varepsilon(N+5)\geq\bar\varepsilon$, it follows from
 (I) (II) of proposition \ref{S9col1} and (P1, $p+1$) that
\begin{equation*}
\begin{split}
&\triplenorm{f_1'(S_{p+1}Y_{p+1};X_{p+1})}_{L_t^1(\d,N)}\lesssim \eta\theta_{p+1}^{-\beta+\bar\varepsilon(N+5)}.
\end{split}
\end{equation*}
$f_2'(S_{p+1}Y_{p+1};X_{p+1})$ can be handled along the same line.

For $f_0'(S_{p+1}Y_{p+1};X_{p+1})$, we deduce from \eqref{S4.1eqf0} that
\begin{equation*}
\begin{split}
&\triplenorm{\langle t\rangle^\frac12f_0'(S_{p+1}Y_{p+1};X_{p+1})}_{L_t^2(\d,N)}\leq \|S_{p+1}\nabla Y_{p+1}\|_{0,0}\|\langle t\rangle^\frac12\nabla \p_tX_{p+1}\|_{L_t^2(H^{N+6})}\\
&\quad+\|S_{p+1}\nabla Y_{p+1}\|_{0,N+6}\|\langle t\rangle^\frac12\nabla \p_tX_{p+1}\|_{L_t^2(L^2)}
+\|\langle t\rangle^\frac12S_{p+1}\nabla \p_tY_{p+1}\|_{L_t^2(L^2)}\|\nabla X_{p+1}\|_{0,N+6}\\
&\quad+\big(\|\langle t\rangle^\frac12S_{p+1}\nabla \p_tY_{p+1}\|_{L_t^2(H^{N+6})}+\|S_{p+1}\nabla Y_{p+1}\|_{0,N+6}|S_{p+1}\nabla \p_tY_{p+1}|_{1+\bar\varepsilon,0}\big)\|\nabla X_{p+1}\|_{0,0}.
\end{split}
\end{equation*}
Notice that $f_0'(S_{p+1}Y_{p+1};X_{p+1})$ is supported in $\{0\leq t\leq \theta_{p+1}\}$ so that
$$\triplenorm{f_0'(S_{p+1}Y_{p+1};X_{p+1})}_{L_t^1(\d,N)}\lesssim (\log\theta_{p+1})^\frac12\triplenorm{\langle t\rangle^\frac12f_0'(S_{p+1}Y_{p+1};X_{p+1})}_{L_t^2(\d,N)},$$
which together with (P1, $p+1$) and (II) of Proposition \ref{S9col1} ensures that
\begin{align}
&\triplenorm{f'(S_{p+1}Y_{p+1};X_{p+1})}_{L_t^1(\d,0)}\lesssim\eta\theta_{p+1}^{-\beta+6\bar\varepsilon} \andf \label{P2eq1}
\\
&\triplenorm{f'(S_{p+1}Y_{p+1};X_{p+1})}_{L_t^1(\d,N)}\lesssim \eta\theta_{p+1}^{-\beta+\bar\varepsilon(N+6)}\quad\text{if}\ -\beta+\bar\varepsilon(N+5)\geq\bar\varepsilon. \label{P2eq2}
\end{align}

\noindent$\bullet$\underline{\bf The estimate of $\|\langle t\rangle^\frac12|D|^{-1}f'(S_{p+1}Y_{p+1};X_{p+1})\|_{L_t^2(H^{N+1})}$}

It follows from  \eqref{S4.2eqf1} that
\begin{equation*}
\begin{split}
\|&\langle t\rangle^{\frac12}|D|^{-1}f_1'(S_{p+1}Y_{p+1};X_{p+1})\|_{L_t^2(H^{N+1})}\leq |S_{p+1}\p_3Y_{p+1}|_{\frac12,0}\bigl(\|\p_3X_{p+1}\|_{L_t^2(H^{N+1})}\\
&+|S_{p+1}\nabla Y_{p+1}|_{0,N+1}\|\p_3X_{p+1}\|_{L_t^2(H^1)}\bigr)+\bigl(\|\na X_{p+1}\|_{0,N+1}+|S_{p+1}\nabla Y_{p+1}|_{0,N+1}\|\nabla X_{p+1}\|_{0,1}\bigr)\\
&\times\big(|S_{p+1}\p_3Y_{p+1}|_{\frac12+\bar\varepsilon,1}^\frac43\|S_{p+1}\p_3Y_{p+1}\|_{L_t^2(L^2)}^\frac23+|S_{p+1}\p_3Y_{p+1}|_{\frac12+\bar\varepsilon,1}^2\big)
 +|S_{p+1}\p_3Y_{p+1}|_{\frac12,N+1}\times\\
&\times \Bigl(\|\p_3X_{p+1}\|_{L_t^2(H^1)}+\big(|S_{p+1}\p_3Y_{p+1}|_{\frac12+\bar\varepsilon,0}^\frac13\|S_{p+1}\p_3Y_{p+1}\|_{L_t^2(L^2)}^\frac23
+|S_{p+1}\p_3Y_{p+1}|_{\frac12+\bar\varepsilon,0}\big)\|\nabla X_{p+1}\|_{0,1}\Bigr),
\end{split}
\end{equation*}
which together with (II) of Proposition \ref{S9col1}  and (P1, $p+1$) ensures that
\begin{equation*}
\begin{split}
&\|\langle t\rangle^{\frac12}|D|^{-1}f_1'(S_{p+1}Y_{p+1};X_{p+1})\|_{L_t^2(H^{3})}\lesssim \eta\theta_{p+1}^{-\beta+2\bar\varepsilon}.
\end{split}
\end{equation*}
For $N$ satisfying $-\gamma+\bar\varepsilon(N+1)\geq\bar\varepsilon$,
we deduce from  (I) of Proposition \ref{S9col1} and (P1, $p+1$) that
\begin{equation*}
\begin{split}
&\|\langle t\rangle^{\frac12}|D|^{-1}f_1'(S_{p+1}Y_{p+1};X_{p+1})\|_{L_t^2(H^{N+1})}\lesssim \eta\theta_{p+1}^{-\beta+\bar\varepsilon N}.
\end{split}
\end{equation*}
$f_2'(S_{p+1}Y_{p+1};X_{p+1})$ can be treated similarly.

For $f_0'(S_{p+1}Y_{p+1};X_{p+1})$, by virtue of \eqref{S4.2eqf0}, we get
\begin{equation*}
\begin{split}
\|\langle t&\rangle^{\frac12}|D|^{-1}f_0'(S_{p+1}Y_{p+1};X_{p+1})\|_{L_t^2(H^{N+1})}
\lesssim |S_{p+1}\nabla Y_{p+1}|_{0,0}\|\langle t\rangle^\frac12\nabla\p_t X_{p+1}\|_{L_t^2(H^{N+1})}\\
&\quad+|S_{p+1}\nabla Y_{p+1}|_{0,N+1}\|\langle t\rangle^\frac12\nabla\p_t X_{p+1}\|_{L_t^2(L^2)}
+ | S_{p+1}\p_tY_{p+1}|_{1+\bar\varepsilon,1}\|\nabla X_{p+1}\|_{0,N+1}\\
&\quad+\big(| S_{p+1}\p_tY_{p+1}|_{1+\bar\varepsilon,N+2}+|S_{p+1}\p_t Y_{p+1}|_{1+\bar\varepsilon,1}|S_{p+1}\nabla Y_{p+1}|_{0,N+1}\big)\|\nabla X_{p+1}\|_{0,0}.
\end{split}
\end{equation*}
As a result, it comes out
\begin{align}
&\|\langle t\rangle^{\frac12}|D|^{-1}f'(S_{p+1}Y_{p+1};X_{p+1})\|_{L_t^2(H^{3})}\lesssim \eta\theta_{p+1}^{-\beta+2\bar\varepsilon}\andf
\label{P2eq3}\\
&\|\langle t\rangle^{\frac12}|D|^{-1}f'(S_{p+1}Y_{p+1};X_{p+1})\|_{L_t^2(H^{N+1})}\lesssim \eta\theta_{p+1}^{-\beta+\bar\varepsilon N}\quad
\text{if}\ -\gamma+\bar\varepsilon(N+1)\geq\bar\varepsilon. \label{P2eq4}
\end{align}

\noindent$\bullet$\underline{\bf The Estimate of $\||D|^{-1}f'(S_{p+1}Y_{p+1};X_{p+1})\|_{\frac32-\delta,N+1}$}

By virtue of \eqref{S4.2eqf1}, we have
\begin{equation*}
\begin{split}
\|&|D|^{-1}f_1'(S_{p+1}Y_{p+1};X_{p+1})\|_{\frac32,N+1}
\leq |S_{p+1}\p_3Y_{p+1}|_{1,0}\bigl(|S_{p+1}\nabla Y_{p+1}|_{0,N+1}\|\p_3X_{p+1}\|_{\frac12,1}\\
&+\|\p_3X_{p+1}\|_{\frac12,N+1}\bigr)+\big(|S_{p+1}\p_3Y_{p+1}|_{\frac78,1}^\frac43\|S_{p+1}\p_3Y_{p+1}\|_{\frac12,0}^\frac23+|S_{p+1}\p_3Y_{p+1}|_{\frac34,1}^2\big)\times\\
&\times \bigl(\|\na X_{p+1}\|_{0,{N+1}}+|S_{p+1}\nabla Y_{p+1}|_{0,N+1}\|\nabla X_{p+1}\|_{0,1}\bigr)+|S_{p+1}\p_3Y_{p+1}|_{1,N+1}\|\p_3X_{p+1}\|_{\frac12,1}\\
& +\big(|S_{p+1}\p_3Y_{p+1}|_{\frac78,N+1}|S_{p+1}\p_3Y_{p+1}|_{\frac78,0}^\frac13\|S_{p+1}\p_3Y_{p+1}\|_{\frac12,0}^\frac23\\
&\qquad\qquad\qquad\qquad\qquad\qquad\qquad+|S_{p+1}\p_3Y_{p+1}|_{\frac34,N+1}|S_{p+1}\p_3Y_{p+1}|_{\frac34,0} \big)\|\nabla X_{p+1}\|_{0,1}.
\end{split}
\end{equation*}
Noticing from \eqref{S9eqP} that
$\frac14-\gamma\geq\bar\varepsilon,$
 so that we get, by applying (II) (i) of Proposition \ref{S9col1}, that
$$|S_{p+1}\p_3Y_{p+1}|_{\frac78,0}\leq\eta\theta_{p+1}^{\frac38-\gamma},\quad  |S_{p+1}\p_3Y_{p+1}|_{1,0}\lesssim\eta\theta_{p+1}^{\frac12-\gamma},\quad
|S_{p+1}\p_3Y_{p+1}|_{\frac34,0}\leq \eta\theta_{p+1}^{\frac14-\gamma}.$$
As a result, it comes out
\begin{equation}\label{P2eq4a}
\begin{split}
\||D|^{-1}f_1'(S_{p+1}Y_{p+1};X_{p+1})\|_{\frac32,3}\lesssim&
\eta^2\big(\theta_{p+1}^{\frac12-\gamma-\beta+3\bar\varepsilon}+\theta_{p+1}^{\frac12-\frac43\gamma-\beta+\frac{10}3\bar\varepsilon}
+\theta_{p+1}^{\frac12-2\gamma-\beta+3\bar\varepsilon}\big)\\
\lesssim & \eta^2\theta_{p+1}^{\frac12-\gamma-\beta+3\bar\e},
\end{split}
\end{equation}
provided that $\frac13\gamma+\frac23\beta\geq\frac13\bar\varepsilon$,
which is the case due to \eqref{S9eqP} and \eqref{S9eqebd}.

For $N$ with $-\gamma+\bar\varepsilon(N+1)\geq\bar\varepsilon$, we deduce from (I) (i) of Corollary \ref{S9col1} that
$$|S_{p+1}\p_3Y_{p+1}|_{\frac78,N+1}\leq\eta\theta_{p+1}^{\frac38-\gamma+\bar\varepsilon(N+1)},\quad  |S_{p+1}\p_3Y_{p+1}|_{1,N+1}\lesssim\eta\theta_{p+1}^{\frac12-\gamma+\bar\varepsilon(N+1)},$$
$$|S_{p+1}\p_3Y_{p+1}|_{\frac34,N+1}\leq \eta\theta_{p+1}^{\frac14-\gamma+\bar\varepsilon(N+1)},\quad|S_{p+1}\nabla Y_{p+1}|_{0,N+1}\lesssim\eta\theta_{p+1}^{-\gamma+\bar\varepsilon(N+1)},$$
which together with (P1, $p+1$) ensures that
\begin{equation}\label{P2eq4b}
\||D|^{-1}f_1'(S_{p+1}Y_{p+1};X_{p+1})\|_{\frac32,N+1}\lesssim \eta^2\theta_{p+1}^{\frac12-\gamma-\beta+\bar\varepsilon (N+1)}.
\end{equation}
Similar estimates as above holds for $f_2'.$

To deal with the term $f_0'(S_{p+1}Y_{p+1};X_{p+1})$, we get, by applying \eqref{S4.2eqf0}, that
\begin{equation*}
\begin{split}
\||D&|^{-1}f_0'(S_{p+1}Y_{p+1};X_{p+1})\|_{\frac32-\delta,N+1}\lesssim |S_{p+1}\nabla Y_{p+1}|_{\frac12-\delta,0}\|\nabla\p_t X_{p+1}\|_{1,N+1}\\
&\quad+|S_{p+1}\nabla Y_{p+1}|_{\frac12-\delta,N+1}\|\nabla\p_t X_{p+1}\|_{1,0}
+ |S_{p+1}\p_t Y_{p+1}|_{\frac32-\delta,1}\|\nabla X_{p+1}\|_{0,N+1}\\
&\quad+\big(|S_{p+1}\p_t Y_{p+1}|_{\frac32-\delta,N+2}+| S_{p+1}\p_tY_{p+1}|_{\frac32-\delta,1}|S_{p+1}\nabla Y_{p+1}|_{0,N+1}\big)\|\nabla X_{p+1}\|_{0,0}.
\end{split}
\end{equation*} Then along the same line to proof of \eqref{P2eq4a} and \eqref{P2eq4b}, we can show that
\begin{equation}\label{P2eq5}
\begin{split}
&\||D|^{-1}f'(S_{p+1}Y_{p+1};X_{p+1})\|_{\frac32-\delta,3}\lesssim \eta^2\theta_{p+1}^{\frac12-\gamma-\beta+4\bar\e} ,
\end{split}
\end{equation}
and for $N$ with $-\gamma+\bar\varepsilon(N+1)\geq\bar\varepsilon$, there holds
\begin{equation}\label{P2eq6}
\begin{split}
&\||D|^{-1}f'(S_{p+1}Y_{p+1};X_{p+1})\|_{\frac32-\delta,N+1}\lesssim \eta^2\theta_{p+1}^{\frac12-\gamma-\beta+\bar\e(N+2)}.
\end{split}
\end{equation}
Moreover, we can prove in the same way that
\begin{equation}\label{P2eq6'}
\begin{split}
\||D|^{-1}f'(S_pY_p;X_p)\|_{1,1}&\lesssim  \eta^2\theta_p^{-\beta+2\bar\varepsilon},\\
\||D|^{-1}f'(S_pX_p;X_p)\|_{1,N+1}&\lesssim \eta^2\theta_{p}^{-\beta+\bar\e (N+2)}\quad\text{for }-\gamma+\bar\e(N+1)\geq \bar\e.
\end{split}
\end{equation}

Recall \eqref{S2eq19'}, we get, by summarizing the estimates \eqref{P2eq1}, \eqref{P2eq3} and \eqref{P2eq5} that
\begin{equation*}
\begin{split}
&R_{0,\theta_{p+1}}\big(f'(S_{p+1}Y_{p+1};X_{p+1})\big)
\lesssim \eta^2\big(\theta_{p+1}^{-\beta+6\bar\varepsilon}+\theta_{p+1}^{\frac12-\beta+2\bar\varepsilon}
+(\log\theta_{p+1})\theta_{p+1}^{\frac12-\gamma-\beta+4\bar\e}\big)
\lesssim \eta^2\theta_{p+1}^{\frac12-\gamma},
\end{split}
\end{equation*}
provided that
\beq\label{P2eq6a}
\beta+\frac12\geq\gamma+6\bar\varepsilon,\quad\beta\geq\gamma+2\bar\varepsilon,\quad \beta\geq5\bar\e \eeq
which is  the case here due to \eqref{S9eqP} and \eqref{S9eqebd}.

While due to \eqref{P2eq6a}, $-\beta+\bar\varepsilon(N_0+5)\geq\bar\varepsilon$ and $-\gamma+\bar\varepsilon(N_0+2)\geq\bar\varepsilon,$
 by summarizing the estimates  \eqref{P2eq2}, \eqref{P2eq4} and \eqref{P2eq6}, we achieve
\begin{equation*}
\begin{split}
&R_{N_0,\theta_{p+1}}\big(f'(S_{p+1}Y_{p+1};X_{p+1})\big)
\lesssim  \eta^2\theta_{p+1}^{\bar\varepsilon(N_0+2)}\bigl(\theta_{p+1}^{-\beta+4\bar\varepsilon}+\theta_{p+1}^{\frac12-\beta}
+(\log\theta_{p+1})\theta_{p+1}^{\frac12-\gamma-\beta+2\bar\e}\bigr)
\lesssim \eta^2\theta_{p+1}^{\frac12-\gamma+\bar\varepsilon N_0}.
\end{split}
\end{equation*}

Now we apply Proposition \ref{S2.3prop1} and  (VI) of Proposition \ref{S9prop2} to \eqref{P1eq3}, to get
$$|\p_3X_{p+1}|_{1,0}+|X_{p+1,t}|_{\frac32-\d,0}+|X_{p+1}|_{\frac12,0}
\leq R_{0,\theta_{p+1}}\big(f'(S_{p+1}Y_{p+1};X_{p+1})\big)+R_{0,\theta_{p+1}}(g_{p+1})
\leq C\eta^2\theta_{p+1}^{\frac12-\gamma},
$$
and
$$\longformule{|\p_3X_{p+1}|_{1,N_0}+|X_{p+1,t}|_{\frac32-\d,N_0}+|X_{p+1}|_{\frac12,N_0}}{{}\leq R_{N_0,\theta_{p+1}}\big(f'(S_{p+1}Y_{p+1};X_{p+1})\big)+R_{N_0,\theta_{p+1}}(g_{p+1})
\leq C \eta^2\theta_{p+1}^{\frac12-\gamma+\bar\varepsilon N_0}.}
$$
Interpolating  the above two inequalities gives for all $0\leq N\leq N_0$,
\begin{equation}\label{P2eq7}
\begin{split}
&|\p_3X_{p+1}|_{1,N}+|X_{p+1,t}|_{\frac32-\d,N}+|X_{p+1}|_{\frac12,N}
\lesssim \eta\theta_{p+1}^{\frac12-\gamma+\bar\varepsilon N}.
\end{split}
\end{equation}

Whereas it follows from  Sobolev embedding Theorem and   (P1, $p+1$) that for any $0\leq N\leq N_0$,
\begin{equation}\label{P2eq8}
\begin{split}
|X_{p+1}|_{0,N}&\lesssim \|\nabla X_{p+1}\|_{0,N+1}\leq \eta\theta_{p+1}^{-\beta+\bar\varepsilon N}\leq \eta\theta_{p+1}^{-\gamma+\bar\varepsilon N},\\
|\p_3X_{p+1}|_{\frac12,N}&\lesssim \|\p_3X_{p+1}\|_{\frac12,N+2}\leq \eta\theta_{p+1}^{-\beta+\bar\varepsilon(N+1)}\leq \eta\theta_{p+1}^{-\gamma+\bar\varepsilon N},\\
|\p_tX_{p+1}|_{1-\d,N}&\lesssim \|\nabla\p_t X_{p+1}\|_{1,N+1}\leq \eta\theta_{p+1}^{-\beta+\bar\varepsilon(N+2)}\leq \eta\theta_{p+1}^{-\gamma+\bar\varepsilon N},\\
\end{split}
\end{equation}
provided that $\beta\geq\gamma+2\bar\varepsilon$, which is satisfied due to \eqref{S9eqP}.

By interpolating the inequalities \eqref{P2eq7} and \eqref{P2eq8}, we arrive at (P2, $p+1$). This completes the proof of Proposition
\ref{S9prop1} for $p+1.$ \end{proof}


\subsection{The proof of Theorem \ref{Th1}}\label{Subsect10.5}
The goal of this subsection is to prove the convergence of the approximate solutions $\{Y_p\}$ constructed via \eqref{S9.1eq5} in some appropriate
norms, which in particular ensures  Theorem \ref{Th1}.

\begin{proof}[Proof of Theorem \ref{Th1}]
We infer from \eqref{S9.1eq6}, \eqref{P2eq6'},  (P1) of Proposition \ref{S9prop1} and (V) of Proposition \ref{S9prop2}, that
\begin{equation}\label{S9.5eq1}
\begin{split}
&\|\p_{tt}X_{p}\|_{\frac12,0}\leq\|(\Delta \p_tX_{p},\p_3^2X_p)\|_{\frac12,0}+\|f'(S_pY_p;X_p)\|_{\frac12,0}+\|g_p\|_{\frac12,0}\leq C\eta\theta_p^{-\beta+2\bar\varepsilon},\\
&\|\p_{tt}X_{p}\|_{\frac12,N}\leq C\eta\theta_p^{-\beta+\bar\varepsilon(N+2)},\quad\quad \text{for $-\gamma+\bar\e(N+1)\geq\bar\e$.}
\end{split}
\end{equation}
Interpolating the above two inequalities leads to
\begin{equation}\label{S9.5eq1'}
\|\p_{tt}X_{p}\|_{\frac12,N}\leq C\eta\theta_p^{-\beta+\bar\varepsilon(N+2)},\quad\forall\ N\geq0.
\end{equation}

Due to the choices of the parameters in \eqref{S9eqP} and \eqref{S9eqebd}, it follows from (P2) of Proposition \ref{S9prop1}  that
\begin{equation*}
\begin{split}
&\sum_{p=0}^\infty|\p_3Y_{p+1}-\p_3Y_p|_{\frac34-4\bar{\e},2}=\sum_{p=0}^\infty|\p_3X_p|_{\frac34-4\bar{\e},2}\leq\eta\sum_{p=0}^\infty\theta_p^{-\bar\varepsilon}<+\infty,\\
&\sum_{p=0}^\infty|\p_tY_{p+1}-\p_tY_{p}|_{\f54-\d-4\bar{\e},2}=\sum_{p=0}^\infty |\p_tX_{p}|_{\f54-\d-4\bar{\e},2}\leq\eta\sum_{p=0}^\infty\theta_p^{-\bar\varepsilon}<+\infty,\\
&\sum_{p=0}^\infty|Y_{p+1}-Y_p|_{\f14-4\bar{\e},2}=\sum_{p=0}^\infty| X_{p}|_{\f14-4\bar{\e},2}\leq \eta\sum_{p=0}^\infty\theta_p^{-\bar\varepsilon}<+\infty.
\end{split}
\end{equation*}
Similarly, let us take $ N_0=[1/2\bar\e]+1$ and  $N_1\eqdefa [N_0/2]$, we deduce  from (P2) of Proposition \ref{S9prop1}  and \eqref{S9.5eq1} that
\begin{equation*}
\begin{split}
&\sum_{p=0}^\infty\Big(\bigl\||D|^{-1}\bigl(\p_3Y_{p+1}-\p_3Y_p,\p_tY_{p+1}-\p_tY_{p}\bigr)\bigr\|_{0,N_1+2}+\|\nabla Y_{p+1}-\nabla Y_p\|_{0,N_1+1}\\
&\qquad+\bigl\|\bigl(\p_3Y_{p+1}-\p_3Y_p,\langle t\rangle^\frac12(\nabla \p_tY_{p+1}-\nabla \p_tY_{p})\bigr)\bigr\|_{L_t^2(H^{N_1+1})}+\|\p_tY_{p+1}-\p_tY_{p}\|_{L_t^2(H^{N_1+2})} \\
&\qquad+\bigl\|\bigl(\p_tY_{p+1}-\p_tY_{p}, \p_3Y_{p+1}-\p_3Y_p\bigr)\bigr\|_{\frac12,N_1+1}+\|\nabla \p_tY_{p+1}-\nabla \p_tY_{p}\|_{1,N_1-1}\\
&\qquad\qquad\qquad\qquad\qquad\quad+\|\p_{tt}Y_{p+1}-\p_{tt}Y_{p}\|_{\frac12,N_1-2}\Big)\leq \eta\sum_{p=0}^\infty\theta_p^{-\beta+\bar \varepsilon N_1 }\leq \eta\sum_{p=0}^\infty\theta_p^{-\bar \varepsilon}<+\infty.
\end{split}
\end{equation*}
This ensures the existence of  $Y\in C^2([0,+\infty); C^{N_1-4}(\R^3))$ such that
\begin{equation}\label{S9.1eq1a}
|\p_3Y-\p_3Y_p|_{\frac34-4\bar{\e},2}+|Y_{t}-\p_tY_{p}|_{\f54-\d-4\bar{\e},2}+|Y-Y_p|_{\f14-4\bar{\e},2}\to0,
\end{equation}
and
\begin{equation}\label{S9.1eq1b}
\begin{split}
&\bigl\||D|^{-1}\bigl(\p_3Y-\p_3Y_p, Y_{t}-\p_tY_{p}\bigr)\bigr\|_{0,N_1+2}+\|\nabla Y-\nabla Y_p\|_{0,N_1+1}+\|\p_tY-\p_tY_{p}\|_{L_t^2(H^{N_1+2})}\\
&\quad+\bigl\|\bigl(\p_3Y-\p_3Y_p,\langle t\rangle^\frac12(\nabla \p_tY-\nabla \p_tY_{p})\bigr)\bigr\|_{L_t^2(H^{N_1+1})} +\|\nabla \p_tY-\nabla \p_tY_{p}\|_{1,N_1-1}\\
&\quad+\bigl\|\bigl(\p_tY-\p_tY_{p}, \p_3Y-\p_3Y_p\bigr)\bigr\|_{\frac12,N_1+1}+\|\p_{tt}Y-\p_{tt}Y_{p}\|_{\frac12,N_1-2}\to 0,\quad\text{as}\  p\to+\infty,
\end{split}
\end{equation}
which ensures  \eqref{S1eq5} and \eqref{S1eq6}.

Next we show that $Y$ is the solution to \eqref{S9.1eq1}. As a matter of fact, we first
observe  from \eqref{S9.1eq11} and \eqref{S9.1eq8} that
\begin{equation*}
\Phi(Y_{p+1})-\Phi(Y_0)=\sum_{j=0}^p e_j+\sum_{j=0}^p g_j=E_{p}+e_p-S_{p}E_p-S_p\Phi(Y_0),
\end{equation*}
which implies
\begin{equation*}
\Phi(Y_{p+1})=e_p+(1-S_{p})E_p+(1-S_p)\Phi(Y_0),
\end{equation*}
from which, \eqref{Veq5}, \eqref{Veqc2} and (IV) of Proposition \ref{S9prop2}, we infer
\begin{equation}\label{S9.5eq2}
\|\Phi(Y_{p+1})\|_{1,0}\leq\|e_p\|_{1,0}+\|(1-S_{p})E_p\|_{1,0}+\|(1-S_p)f(Y_0)\|_{1,0}
\leq C \theta_{p+1}^{\delta-\gamma-\beta+2\bar\varepsilon}.
\end{equation}

Next, we show that $\Phi(Y_{p+1})\to\Phi(Y)$ as $p\to+\infty$ in the norm $\|\cdot\|_{1,0}$. Indeed let us denote
$\wt{\Box}\eqdefa \p_{t}^2-\D\p_t-\p_3^2,$  then one has
\begin{equation}\label{S9.5eq3}
\begin{split}
\|\Phi(Y)-\Phi(Y_{p+1})\|_{1,0}\leq\|\wt{\Box}(Y-Y_{p+1})\|_{1,0}+\|f(Y)-f(Y_{p+1})\|_{1,0}.
\end{split}
\end{equation}
Using a Taylor formula, applying \eqref{S4.2eqf0}, \eqref{S4.2eqf1}, \eqref{S4.2eqf2} and using \eqref{S9.1eq1a}, \eqref{S9.1eq1b}, we have
\begin{equation*}
\begin{split}
\|f(Y)-f(Y_{p+1})\|_{1,0}&\leq\int_0^1\|f'\big((1-s)Y_{p+1}+sY;Y-Y_{p+1}\big)\|_{1,0}ds\\
&\lesssim  C\Big(\|\p_3Y-\p_3Y_{p+1}\|_{\frac12,1}+ \|Y_t-\p_tY_{p+1}\|_{\frac12,1}\\
&\quad\quad+\|\nabla Y_t-\nabla\p_t Y_{p+1}\|_{1,1}+\|\na Y-\nabla Y_{p+1}\|_{0,1}\Big)\to0,\ \text{as }p\to+\infty.
\end{split}
\end{equation*}
On the other hand, recall from \eqref{S9.1eq6} that
$$\wt{\Box} X_p=f'(S_pY_p;X_p)+g_p,$$
then we get,  by applying (P1) of Proposition \ref{S9prop1}, (II) of Proposition \ref{S9col1} and  (V)(ii) of Proposition \ref{S9prop2}, that
\begin{equation*}
\begin{split}
\|\wt{\Box} X_p\|_{1,0}&\leq \|f'(S_pY_p;X_p)\|_{1,0}+\|g_p\|_{1,0}\\
&\lesssim C\big(\|\p_3X_p\|_{\frac12,1}+\|\p_tX_{p}\|_{\frac12,1}+\|\nabla\p_t X_{p}\|_{1,1}+\|\nabla X_p\|_{0,1}\big)+\|g_p\|_{1,0}\\
&\lesssim C\theta_p^{-\beta+2\bar\varepsilon}+\theta_p^{\delta-\gamma-\beta+2\bar\varepsilon}.
\end{split}
\end{equation*}
Consequently, we achieve
\begin{equation}\label{S9.5eq4}
\begin{split}
\|\wt{\Box}(Y-Y_{p+1})\|_{1,0}\leq\sum_{j=p+1}^\infty\|\wt{\Box} X_j\|_{1,0}\leq C\sum_{j=p+1}^\infty\theta_j^{-\beta+2\bar\varepsilon}\to0,\quad\text{as }p\to\infty.
\end{split}
\end{equation}
We then deduce from \eqref{S9.5eq3} and \eqref{S9.5eq4} that
$$\|\Phi(Y)-\Phi(Y_{p+1})\|_{1,0}\to0\quad\text{as }p\to\infty,$$
which together with \eqref{S9.5eq2}  implies $\Phi(Y)=0$.
Finally, for each $p$, we have
$$Y_{p}(0,y)=Y^{(0)},\quad \p_tY_{p}(0,y)=Y^{(1)}(y),$$
therefore,
$$Y(0,y)=Y^{(0)},\quad Y_{t}(0,y)=Y^{(1)}(y),$$
and thus $Y$ is the desired classical solution to \eqref{S9.1eq1}. This ends the proof of Theorem \ref{Th1}.
\end{proof}

\renewcommand{\theequation}{\thesection.\arabic{equation}}
\setcounter{equation}{0}

\appendix

\setcounter{equation}{0}
\section{The proof of Lemmas \ref{S9lem1}, \ref{S9lem2} and \ref{S9lem3}}\label{appendixa}

The goal of this appendix is to present the proof of Lemmas \ref{S9lem1}, \ref{S9lem2} and \ref{S9lem3}.
Notice that the estimates for $e_{p,2}'$, $e_{p,2}''$ are the same as (even better than) those for $e_{p,1}'$, $e_{p,1}''$, so that we only preform the estimates for the latter in what follows.

\subsection{The proof of Lemma \ref{S9lem1}}
We divide the proof of this lemma by the following steps:

\noindent$\bullet$\underline{\it The proof of \eqref{IVieq0}}

In view of \eqref{S9.1eq10}, we get, by
applying \eqref{S7.2eqE2} (with $Y\simeq Y_p+Y_{p+1}$, $X=W=X_p$),  that for $N\geq0$,
\begin{equation*}
\begin{split}
&\|\langle t\rangle^\frac12|D|^{-1}e_{p,1}''\|_{L_t^2(H^{N+1})}
\lesssim |\p_3X_p|_{\frac12,N+1}\|\p_3X_p\|_{L_t^2(L^2)}+|\p_3X_p|_{\frac12,0}\|\p_3X_p\|_{L_t^2(H^{N+1})}\\
&\ +\sum_{j=p}^{p+1}\Bigl\{\Bigl(|\nabla Y_{j}|_{0,N+1}|\p_3X_p|_{\frac12,0}+\big(|\p_3Y_{j}|_{\frac12,N+1}+|\nabla Y_{j}|_{0,N+1}|\p_3Y_{j}|_{\frac12,1}\big)|\nabla X_p|_{0,1}\Bigr)\|\p_3X_p\|_{L_t^2(H^{1})}\\
&\ +\big(|\p_3Y_{j}|_{\frac12+\bar\varepsilon,1}+|\p_3Y_{j}|_{\frac12+\bar\varepsilon,1}^\frac13\|\p_3Y_{j}\|_{\frac12,1}^\frac23\big) \Big(|\nabla X_p|_{0,N+1}\|\p_3X_p\|_{L_t^2(L^2)}+|\nabla X_p|_{0,0}\|\p_3X_p\|_{L_t^2(H^{N+1})}\\
&\ +|\p_3X_p|_{\frac12,1}\|\nabla X_p\|_{0,N+1}+\bigl(|\p_3X_p|_{\frac12,N+1}
+|\nabla Y_{j}|_{0,N+1}|\p_3X_p|_{\frac12,1}+|\p_3Y_{j}|_{\frac12+\bar\varepsilon,0}|\nabla X_p|_{0,1}\bigr)\|\nabla X_p\|_{0,1}\Big)\\
&\ +\big(|\p_3Y_{j}|_{\frac12+\bar\varepsilon,1}^\frac43\|\p_3Y_{j}\|_{\frac12,0}^\frac23+|\p_3Y_{j}|_{\frac12+\bar\varepsilon,1}^2\big)\Big(|\nabla X_p|_{0,1}\bigl(\|\nabla X_p\|_{0,N+1}+|\nabla Y_{j}|_{0,N+1}\|\nabla X_p\|_{0,1}\bigr)\\
&\ +|\nabla X_p|_{0,N+1}\|\nabla X_p\|_{0,1}\Big)+\big(
 |\p_3Y_{j}|_{\frac12+\bar\varepsilon,N+1}+|\p_3Y_{j}|_{\frac12+\bar\varepsilon,N+1}^\frac13\|\p_3Y_{j}\|_{\frac12,N+1}^\frac23\big)|\p_3X_p|_{\frac12,1}\|\nabla X_p\|_{0,1}
\Bigr\}.
\end{split}
\end{equation*}
 Similar estimate holds for $\|\langle t\rangle|D|^{-1}e_{p,1}''\|_{L_t^2(H^{N+1})}$ with $|\p_3 X_p|_{\frac12,l}$ above being replaced by $|\p_3X_p|_{1,l}$
  and $|\nabla X_p|_{0,l}$ by $|\nabla X_p|_{\frac12,l}.$

It follows from \eqref{Ieq1}, \eqref{Ieq3} and \eqref{Ieq4} that
\begin{align}
&|\p_3Y_{p+1}|_{\frac12+\bar\varepsilon,1}\leq C\eta,\quad |\nabla Y_{p+1}|_{0,1}\leq C\eta\quad  \text{since} \ \gamma\geq3\bar\varepsilon;\label{IVieq1}\\
&\|\p_3Y_{p+1}\|_{\frac12,1}\leq C\eta \qquad\qquad\qquad\qquad\qquad  \text{since} \ \beta\geq\bar\varepsilon. \label{IVieq1'}
\end{align}
As a result, applying (P1, $p$) and (P2, $p$), it comes out
\begin{equation*}
\begin{split}
&\|\langle t\rangle^{\frac12}|D|^{-1}e_{p,1}''\|_{L_t^2(H^1)}
\lesssim \eta^2\theta_p^{-\gamma-\beta+\bar\varepsilon},\andf
\|\langle t\rangle|D|^{-1}e_{p,1}''\|_{L_t^2(H^1)}
\lesssim \eta^2\theta_p^{\frac12-\gamma-\beta+\bar\varepsilon}.
\end{split}
\end{equation*}
Interpolating between the above two inequalities gives rise to
\begin{equation}\label{IVieq2}
\begin{split}
&\|\langle t\rangle^{\frac12+k}|D|^{-1}e_{p,1}''\|_{L_t^2(H^1)}
\lesssim \eta^2\theta_p^{k-\gamma-\beta+\bar\varepsilon} \quad  \text{if} \  0\leq k\leq\frac12.
\end{split}
\end{equation}

While for $0\leq N\leq N_0-1$ such that $-\gamma+\bar\varepsilon(N+1)\geq\bar\varepsilon$ and $-\beta+\bar\varepsilon N\geq\bar\varepsilon$, we deduce from \eqref{Ieq1}, \eqref{Ieq3} and \eqref{Ieq4} that
\begin{equation}\label{IVieq3}
\begin{split}
&|\p_3Y_{p+1}|_{\frac12+\bar\varepsilon,N+1}\leq C\eta\theta_{p+1}^{-\gamma+\bar\varepsilon(N+2)},\quad |\nabla Y_{p+1}|_{0,N+1}\leq C\eta\theta_{p+1}^{-\gamma+\bar\varepsilon(N+1)}\\
&\qquad\qquad\qquad\|\p_3Y_{p+1}\|_{\frac12,N+1}\leq C\eta\theta_{p+1}^{-\beta+\bar\varepsilon N}.
\end{split}
\end{equation}
Therefore for such $N$, there hold
\begin{equation*}
\begin{split}
\|\langle t\rangle^\frac12|D|^{-1}e_{p,1}''\|_{L_t^2(H^{N+1})}\lesssim \eta^2\theta_{p}^{-\gamma-\beta+\bar\varepsilon(N+1)},\quad
 \|\langle t\rangle|D|^{-1}e_{p,1}''\|_{L_t^2(H^{N+1})}\lesssim\eta^2\theta_{p}^{\frac12-\gamma-\beta+\bar\varepsilon(N+1)}.
\end{split}
\end{equation*}
Interpolating the above two inequalities, we obtain for $0\leq k\leq\frac12$, $N\leq N_0-1$ such that $-\gamma+\bar\varepsilon(N+1)\geq\bar\varepsilon$ and $-\beta+\bar\varepsilon N\geq\bar\varepsilon$,
\begin{equation}\label{IVieq4}
\begin{split}
 \|\langle t\rangle^{\frac12+k}|D|^{-1}e_{p,1}''\|_{L_t^2(H^{N+1})}&\lesssim\eta^2\theta_{p}^{k-\gamma-\beta+\bar\varepsilon(N+1)}.
\end{split}
\end{equation}
Interpolating between \eqref{IVieq2} and \eqref{IVieq4} leads to \eqref{IVieq0}.

\noindent$\bullet$\underline{\it The proof of \eqref{IVieq-1}}

In order to do so, we get by
applying \eqref{S7.2eqE2} (with $Y\simeq Y_p$, $X=(1-S_p)Y_p$, $W=X_p$) that for $N\geq0$,
\begin{equation*}
\begin{split}
&\|\langle t\rangle^\frac12|D|^{-1}e_{p,1}'\|_{L_t^2(H^{N+1})}\lesssim \Bigl(|(1-S_p)\p_3Y_p|_{\frac12,N+1}+|\nabla Y_p|_{0,N+1}|(1-S_p)\p_3Y_p|_{\frac12,0}\\
&\ +\big(|\p_3Y_p|_{\frac12,N+1}+|\nabla Y_p|_{0,N+1}|\p_3Y_p|_{\frac12,1}\big)|(1-S_p)\nabla Y_p|_{0,1}\Bigr)\|\p_3X_p\|_{L_t^2(H^1)}\\
&\ +|(1-S_p)\p_3Y_p|_{\frac12,0}\|\p_3X_p\|_{L_t^2(H^{N+1})}+\big(|\p_3Y_p|_{\frac12+\bar\varepsilon,1}+
|\p_3Y_p|_{\frac12+\bar\varepsilon,1}^\frac13\|\p_3Y_p\|_{\frac12,1}^\frac23\big)\\
&\quad \times  \Big(|(1-S_p)\nabla Y_p|_{0,1}\bigl(\|\p_3X_p\|_{L_t^2(H^{N+1})}
+|\p_3Y_p|_{\frac12,N+1}\|\nabla X_p\|_{0,1}\bigr)\\
&\ +
|(1-S_p)\p_3Y_p|_{\frac12,N+1}\|\nabla X_p\|_{0,1}+|(1-S_p)\p_3Y_p|_{\frac12,1}\bigl(\|\nabla X_p\|_{0,N+1}+|\nabla Y_p|_{0,N+1}\|\nabla X_p\|_{0,1}\bigr)\\
&\ +|(1-S_p)\nabla Y_p|_{0,N+1}\|\p_3X_p\|_{L_t^2(L^2)}\Big)+\big(|\p_3Y_p|_{\frac12+\bar\varepsilon,1}^\frac43\|\p_3Y_p\|_{\frac12,0}^\frac23+|\p_3Y_p|_{\frac12+\bar\varepsilon,1}^2\big)\\
&\quad\times\Big(|(1-S_p)\nabla Y_p|_{0,N+1}\|\nabla X_p\|_{0,1}+|(1-S_p)\nabla Y_p|_{0,1}\bigl(\|\nabla X_p\|_{0,N+1}+|\nabla Y_p|_{0,N+1}\|\nabla X_p\|_{0,1}\bigr)\Big)\\
&\ +\big(|\p_3Y_p|_{\frac12+\bar\e,N+1}+|\p_3Y_p|_{\frac12+\bar\e,N+1}^\frac13\|\p_3Y_p\|_{\frac12,N+1}^\frac23\big)|(1-S_p)\p_3Y_p|_{\frac12,1}\|\nabla X_p\|_{0,1}\\.
\end{split}
\end{equation*}
 Similar estimate for $\|\langle t\rangle|D|^{-1}e_{p,1}'\|_{L_t^2(H^{N+1})}$ holds  with $|(1-S_p)\p_3Y_p|_{\frac12,l}$ above being replaced by $|(1-S_p)\p_3Y_p|_{1,l}$ and $|(1-S_p)\nabla Y_p|_{\frac12,l}$ by $|(1-S_p)\nabla Y_p|_{0,l}.$

So that by virtue of  \eqref{IVieq1}, \eqref{IVieq1'} and (III) of Proposition \ref{S9col1}, we infer that
\begin{equation*}
\begin{split}
&\|\langle t\rangle^\frac12|D|^{-1}e_{p,1}'\|_{L_t^2(H^1)}
\lesssim\eta^2 \theta_p^{-\gamma-\beta+\bar\varepsilon},\quad
\|\langle t\rangle|D|^{-1}e_{p,1}'\|_{L_t^2(H^1)}
\lesssim\eta^2 \theta_p^{\frac12-\gamma-\beta+\bar\varepsilon}.
\end{split}
\end{equation*}
Interpolating the above two inequalities, we obtain
\begin{equation}\label{IVieq5}
\begin{split}
\|\langle t\rangle^{\frac12+k}|D|^{-1}e_{p,1}'\|_{L_t^2(H^1)}
&\lesssim\eta^2 \theta_p^{k-\gamma-\beta+2\bar\varepsilon}\quad\text{for} \ 0\leq k\leq\frac12.
\end{split}
\end{equation}

Note that for $N\leq N_0-1$ such that $-\beta+\bar\varepsilon N\geq\bar\varepsilon$ and $-\gamma+\bar\varepsilon(N+1)\geq\bar\varepsilon$, \eqref{IVieq3} holds.
And hence we have
\begin{equation*}
\begin{split}
\|\langle t\rangle^\frac12|D|^{-1}e_{p,1}'\|_{L_t^2(H^{N+1})}\lesssim\eta^2\theta_p^{-\gamma-\beta+\bar\varepsilon(N+1)},\quad
\|\langle t\rangle|D|^{-1}e_{p,1}'\|_{L_t^2(H^{N+1})}\lesssim\eta^2\theta_p^{\frac12-\gamma-\beta+\bar\varepsilon(N+1)}.
\end{split}
\end{equation*}
Interpolating the above inequalities gives  for such $N$ and $0\leq k\leq\frac12$ that
\begin{equation}\label{IVieq6}
\|\langle t\rangle^{\frac12+k}|D|^{-1}e_{p,1}'\|_{L_t^2(H^{N+1})}\lesssim\eta^2\theta_p^{k-\gamma-\beta+\bar\varepsilon(N+1)}.
\end{equation}
Interpolating between \eqref{IVieq5} and \eqref{IVieq6} leads to \eqref{IVieq-1}.

\noindent$\bullet$\underline{\it The proof of \eqref{IVieq-2}}

It follows from  \eqref{S7.2eqE1} and \eqref{S9.1eq10} that
\begin{equation*}
\begin{split}
\|\langle t&\rangle^\frac12|D|^{-1}e_{p,0}''\|_{L_t^2(H^{N+1})}\lesssim |\nabla X_p|_{0,N+1}\|\langle t\rangle^\frac12\nabla \p_tX_{p}\|_{L_t^2(L^2)}\\
&+|\nabla X_p|_{0,0}\|\langle t\rangle^\frac12\nabla \p_tX_{p}\|_{L_t^2(H^{N+1})}+|\p_tX_{p}|_{1+\bar\varepsilon,N+2}\|\nabla X_p\|_{0,0}+|\p_tX_{p}|_{1+\bar\varepsilon,1}\|\nabla X_p\|_{0,N+1}\\
&+\sum_{j=p}^{p+1} \Bigl( |\p_tY_{j}|_{1+\bar\varepsilon,1}\big(|\nabla X_p|_{0,N+1}\|\nabla X_p\|_{0,0}+|\nabla X_p|_{0,0}\|\nabla X_p\|_{0,N+1}\big)\\
&\qquad+\big(| \p_tY_{j}|_{1+\bar\varepsilon,N+2}+|\nabla Y_{j}|_{0,N+1}| \p_tY_{j}|_{1+\bar\varepsilon,1}\big)|\nabla X_p|_{0,0}\|\nabla X_p\|_{0,0}\\
&\qquad+|\nabla Y_j|_{0,N+1}\bigl(|\nabla X_p|_{0,0}\|\langle t\rangle^\frac12\nabla\p_t X_{p}\|_{L_t^2(L^2)}+|\p_tX_{p}|_{1+\bar\varepsilon,1}\|\nabla X_p\|_{0,0}\bigr)\Bigr).
\end{split}
\end{equation*}
Recall that $\alpha=\f12-\d-\bar{\e}<\f12-\d,$  similar estimate for $\|\langle t\rangle^{\frac12+\alpha}|D|^{-1}e_{p,0}''\|_{L_t^2(H^{N+1})}$ holds with  $|\nabla X_p|_{0,l}$ above  replaced by $|\nabla X_p|_{\alpha,l}$
 and $|\p_tX_{p}|_{1+\bar\varepsilon,l}$ above  by $|\p_tX_{p}|_{\frac32-\delta,l}.$

Note from \eqref{S9eqebd} that $\gamma\geq\delta+4\bar\varepsilon$, so that
we deduce from \eqref{Ieq2} that,
\begin{equation}\label{IVieq1''}
|\p_tY_{p+1}|_{1+\bar\varepsilon,2}\leq C\eta,
\end{equation}
which implies
\begin{equation*}
\begin{split}
&\|\langle t\rangle^\frac12|D|^{-1}e_{p,0}''\|_{L_t^2(H^1)}
\lesssim \eta^2\theta_p^{\delta-\gamma-\beta+3\bar\varepsilon},\quad
\|\langle t\rangle^{\alpha+\frac12}|D|^{-1}e_{p,0}''\|_{L_t^2(H^1)}
\lesssim \eta^2\theta_p^{\frac12-\gamma-\beta+2\bar\varepsilon}.
\end{split}
\end{equation*}
Interpolating between the above two inequalities yields
\begin{equation}\label{IVieq7}
\begin{split}
&\|\langle t\rangle^{\frac12+k}|D|^{-1}e_{p,0}''\|_{L_t^2(H^1)}
\lesssim \eta^2\theta_p^{k+\delta-\gamma-\beta+3\bar\varepsilon} \quad \text{for}\ 0\leq k\leq\alpha.
\end{split}
\end{equation}

For $N\leq N_0-2$ satisfying $-\gamma+\bar\varepsilon(N+1)\geq\bar\varepsilon$, we deduce from \eqref{Ieq2} and \eqref{Ieq3} that
\begin{equation}\label{IVieq7a}
|\p_tY_{p+1}|_{1+\bar\varepsilon,N+2}\leq C\eta\theta_{p+1}^{\delta-\gamma+\bar\varepsilon(N+3)},\quad |\nabla Y_{p+1}|_{0,N+1}\leq C\eta\theta_{p+1}^{-\gamma+\bar\varepsilon(N+1)},\end{equation}
so that for such $N,$ we have
\begin{equation*}
\begin{split}
&\|\langle t\rangle^\frac12|D|^{-1}e_{p,0}''\|_{L_t^2(H^{N+1})}
\lesssim \eta^2\theta_p^{\delta-\gamma-\beta+\bar\varepsilon(N+3)},\quad
\|\langle t\rangle^{\alpha+\frac12}|D|^{-1}e_{p,0}''\|_{L_t^2(H^{N+1})}
\lesssim \eta^2\theta_p^{\frac12-\gamma-\beta+\bar\varepsilon(N+2)}.
\end{split}
\end{equation*}
By interpolating between the above inequalities, we achieve for such $N$ and $0\leq k\leq\alpha$ that
\begin{equation}\label{IVieq8}
\begin{split}
&\|\langle t\rangle^{\frac12+k}|D|^{-1}e_{p,0}''\|_{L_t^2(H^{N+1})}\lesssim \eta^2\theta_p^{k+\delta-\gamma-\beta+\bar\varepsilon(N+3)}.
\end{split}
\end{equation}
Interpolating between \eqref{IVieq7} and \eqref{IVieq8} leads to \eqref{IVieq-2}.

\noindent$\bullet$\underline{\it The proof of \eqref{IVieq-3}}

Again in view of \eqref{S9.1eq10},
we deduce from \eqref{S7.2eqE1} that
\begin{equation*}
\begin{split}
\|\langle t&\rangle^\frac12|D|^{-1}e_{p,0}'\|_{L_t^2(H^{N+1})}\lesssim
|(1-S_p)\nabla Y_p|_{0,N+1}\bigl(\|\langle t\rangle^\frac12\nabla \p_tX_{p}\|_{L_t^2(L^2)}+| \p_tY_{p}|_{1+\bar\varepsilon,1}\|\nabla X_p\|_{0,0}\bigr)\\
&+|(1-S_p)\p_tY_{p}|_{1+\bar\varepsilon,N+2}\|\nabla X_p\|_{0,0}+|(1-S_p)\p_tY_{p}|_{1+\bar\varepsilon,1}\bigl(|\nabla Y_p|_{0,N+1}\|\nabla X_p\|_{0,0}+\|\nabla X_p\|_{0,N+1}\bigr)\\
&+|(1-S_p)\nabla Y_p|_{0,0}\Bigl(\big(|\p_t Y_{p}|_{1+\bar\varepsilon,N+2}+|\nabla Y_p|_{0,N+1}| \p_tY_{p}|_{1+\bar\varepsilon,1}\big)\|\nabla X_p\|_{0,0}\\
&+|\nabla Y_p|_{0,N+1}\|\langle t\rangle^\frac12\nabla \p_tX_{p}\|_{L_t^2(L^2)}+\|\langle t\rangle^\frac12\nabla \p_tX_{p}\|_{L_t^2(H^{N+1})}+\|\nabla X_p\|_{0,N+1}| \p_tY_{p}|_{1+\bar\varepsilon,1}\Bigr).
\end{split}
\end{equation*}
 A similar estimate holds  for $\|\langle t\rangle^{\frac12+\alpha}|D|^{-1}e_{p,0}'\|_{L_t^2(H^{N+1})}$ with $|(1-S_p)\nabla Y_p|_{0,l}$ above being
  replaced by $|(1-S_p)\nabla Y_p|_{\alpha,l}$ and  $|(1-S_p)\p_tY_{p}|_{1+\bar\varepsilon,l}$ above by $|(1-S_p)\p_tY_p|_{\frac32-\delta,l}.$

Hence by virtue of \eqref{IVieq1''}, we have
\begin{equation*}
\begin{split}
&\|\langle t\rangle^\frac12|D|^{-1}e_{p,0}'\|_{L_t^2(H^1)}
\lesssim \eta^2\theta_p^{\delta-\gamma-\beta+3\bar\varepsilon},\quad
\|\langle t\rangle^{\frac12+\alpha}|D|^{-1}e_{p,0}'\|_{L_t^2(H^1)}\lesssim \eta^2\theta_p^{\frac12-\gamma-\beta+2\bar\varepsilon}.
\end{split}
\end{equation*}
Interpolating between the above inequalities gives
\begin{equation}\label{IVieq9}
\|\langle t\rangle^{\frac12+k}|D|^{-1}e_{p,0}'\|_{L_t^2(H^1)}
\lesssim \eta^2\theta_p^{k+\delta-\gamma-\beta+3\bar\varepsilon}\quad \text{for}\ 0\leq k\leq\alpha.
\end{equation}

Whereas for $N\leq N_0-2$ such that $-\gamma+\bar\varepsilon(N+1)\geq\bar\varepsilon$, we infer from \eqref{IVieq7a} that
\begin{equation*}
\begin{split}
&\|\langle t\rangle^\frac12|D|^{-1}e_{p,0}'\|_{L_t^2(H^{N+1})}\lesssim \eta^2\theta_p^{\delta-\gamma-\beta+\bar\varepsilon(N+3)},\quad
\|\langle t\rangle^{\frac12+\alpha}|D|^{-1}e_{p,0}'\|_{L_t^2(H^{N+1})}\lesssim \eta^2\theta_p^{\frac12-\gamma-\beta+\bar\varepsilon(N+2)}.
\end{split}
\end{equation*}
Interpolating between the above inequalities, we achieve  for such $N$ and $0\leq k\leq \alpha$ that
\begin{equation}\label{IVieq10}
\begin{split}
&\|\langle t\rangle^{\frac12+k}|D|^{-1}e_{p,0}'\|_{L_t^2(H^{N+1})}\lesssim \eta^2\theta_p^{k+\delta-\gamma-\beta+\bar\varepsilon(N+3)}.
\end{split}
\end{equation}
Interpolating between \eqref{IVieq9} and \eqref{IVieq10} leads to \eqref{IVieq-3}. The proof of Lemma \ref{S9lem1} is complete.

\subsection{The proof of Lemma \ref{S9lem2}}

As in the previous lemma, we shall divide the proof of this lemma into the following steps:

\noindent$\bullet$\underline{\it The proof of \eqref{IVieq-4}}

Thanks to \eqref{S9.1eq10}, we get, by
applying \eqref{S7.2eqE2} that for $N\geq0$
\begin{equation*}
\begin{split}
\|&|D|^{-1}e_{p,1}''\|_{\frac32,N+1}
\lesssim |\p_3X_p|_{1,N+1}+|\p_3X_p|_{1,0}\|\p_3X_p\|_{\frac12,N+1}+\sum_{j=p}^{p+1}\Bigl\{\|\p_3X_p\|_{\frac12,1}\times\\
&\ \times\Bigl(|\nabla Y_{j}|_{0,N+1}|\p_3X_p|_{1,0}+\big(|\p_3Y_{j}|_{\frac12,N+1}\|\p_3X_p\|_{\frac12,0}+|\nabla Y_{j}|_{0,N+1}|\p_3Y_{j}|_{\frac12,1}\big)|\nabla X_p|_{\frac12,1}\Bigr)\\
&\ +  \big(|\p_3Y_{j}|_{\frac12,1}+|\p_3Y_{j}|_{\frac12,1}^\frac13\|\p_3Y_{j}\|_{\frac12,1}^\frac23\big)\Big(|\nabla X_p|_{\frac12,N+1}\|\p_3X_p\|_{\frac12,0}+|\nabla X_p|_{\frac12,0}\|\p_3X_p\|_{\frac12,N+1}\\
&\ +\bigl(|\p_3Y_{j}|_{\frac12,N+1}|\nabla X_p|_{\frac12,1}
+|\p_3X_p|_{1,N+1}\bigr)\|\nabla X_p\|_{0,1}+|\p_3X_p|_{1,1}\bigl(\|\nabla X_p\|_{0,N+1}\\
&\ +|\nabla Y_{j}|_{0,N+1}\|\nabla X_p\|_{0,1}\bigr)\Big)+\big(|\p_3Y_{j}|_{\frac12,1}^\frac43\|\p_3Y_{j}\|_{\frac12,0}^\frac23+|\p_3Y_{j}|_1^2\big)\times
\\
&\ \times \Big(|\nabla X_p|_{\frac12,1}\|\nabla X_p\|_{0,N+1}+\bigl(|\nabla X_p|_{\frac12,N+1}+|\nabla Y_{j}|_{0,N+1}|\nabla X_p|_{\frac12,1}\bigr)\|\nabla X_p\|_{0,1}\Big)\\
&\qquad\qquad\qquad\qquad +\big(|\p_3Y_{j}|_{\frac12,N+1}+|\p_3Y_{j}|_{\frac12,N+1}^\frac13\|\p_3Y_{j}\|_{\frac12,N+1}^\frac23\big)|\p_3X_p|_{1,1}\|\nabla X_p\|_{0,1}
\Bigr\}.
\end{split}
\end{equation*}
 A similar estimate holds for $\||D|^{-1}e_{p,1}''\|_{1,N+1}$ with $|\p_3X_p|_{1,l}$ and  $|\nabla X_p|_{\frac12,l}$  above being
  replaced by $|\p_3X|_{\frac12,l}$ and  $|\nabla X_p|_{0,l}$ respectively.

Hence it follows from  \eqref{IVieq1}, (P1, $p$) and (P2, $p$) that
\begin{equation*}
\begin{split}
&\||D|^{-1}e_{p,1}''\|_{1,1}\lesssim\eta^2\theta_p^{-\gamma-\beta+\bar\varepsilon},\qquad
\||D|^{-1}e_{p,1}''\|_{\frac32,1}\lesssim\eta^2\theta_p^{\frac12-\gamma-\beta+\bar\varepsilon}.
\end{split}
\end{equation*}
Interpolating the above two inequalities yields
\begin{equation}\label{IViieq1}
\||D|^{-1}e_{p,1}''\|_{1+k,1}\lesssim\eta^2\theta_p^{k-\gamma-\beta+\bar\varepsilon}\quad \text{for}\ 0\leq k\leq\frac12.
\end{equation}

Whereas for $N\leq N_0-1$ with $-\gamma+\bar\varepsilon(N+1)\geq\bar\varepsilon$ and $-\beta+\bar\varepsilon N\geq\bar\varepsilon$, \eqref{IVieq3} holds,
we infer that
\begin{equation*}
\||D|^{-1}e_{p,1}''\|_{\frac32,N+1}\lesssim \eta^2\theta_p^{\frac12-\gamma-\beta+\bar\varepsilon(N+1)},\quad
\||D|^{-1}e_{p,1}''\|_{1,N+1}\lesssim \eta^2\theta_p^{-\gamma-\beta+\bar\varepsilon(N+1)}.
\end{equation*}
Interpolating the above two inequalities leads to
\begin{equation}\label{IViieq2}
\begin{split}
\||D|^{-1}e_{p,1}''\|_{1+k,N+1}\lesssim \eta^2\theta_p^{k-\gamma-\beta+\bar\varepsilon(N+1)}
\end{split}
\end{equation}   for $0\leq k\leq \frac12$, and $N\leq N_0-1$ with $-\gamma+\bar\varepsilon(N+1)\geq\bar\varepsilon$ and $-\beta+\bar\varepsilon N\geq\bar\varepsilon$.

Interpolating between \eqref{IViieq1} and \eqref{IViieq2} gives rise to \eqref{IVieq-4}.

\noindent$\bullet$\underline{\it The proof of \eqref{IVieq-5}}

Applying \eqref{S7.2eqE2} to $e_{p,1}'$ determined by \eqref{S9.1eq10} gives that for $N\geq0$,
\begin{equation*}
\begin{split}
&\||D|^{-1}e_{p,1}'\|_{\frac32,N+1}\lesssim |(1-S_p)\p_3Y_p|_{1,0}\bigl(\|\p_3X_p\|_{\frac12,N+1}+|\nabla Y_p|_{0,N+1}\|\p_3X_p\|_{\frac12,0}\bigr)\\
&+|(1-S_p)\p_3Y_p|_{1,N+1}\|\p_3X_p\|_{\frac12,0}+  \big(|\p_3Y_p|_{\frac12,1}+|\p_3Y_p|_{\frac12,1}^\frac13\|\p_3Y_p\|_{\frac12,1}^\frac23\big)\times\\
& \times\Big(|(1-S_p)\nabla Y_p|_{\frac12,N+1}\|\p_3X_p\|_{\frac12,0}+|(1-S_p)\nabla Y_p|_{\frac12,1}\bigl(\|\p_3X_p\|_{\frac12,N+1}
+|\p_3Y_p|_{\frac12,N+1}\|\nabla X_p\|_{0,1}\bigr)\\
&+|(1-S_p)\p_3Y_p|_{1,N+1}\|\nabla X_p\|_{0,1}+|(1-S_p)\p_3Y_p|_{1,1}\bigl(\|\nabla X_p\|_{0,N+1}+|\nabla Y_p|_{0,N+1}\|\nabla X_p\|_{0,1}\bigr)\Bigr)\\
&+\big(|\p_3Y_p|_{\frac12,N+1}+|\nabla Y_p|_{0,N+1}|\p_3Y_p|_{\frac12,1}\big)|(1-S_p)\nabla Y_p|_{\frac12,1}\|\p_3X_p\|_{\frac12,1}\\
&+\big(|\p_3Y_p|_{\frac12,N+1}+|\p_3Y_p|_{\frac12,N+1}^\frac13\|\p_3Y_p\|_{\frac12,N+1}^\frac23\big)|(1-S_p)\p_3Y_p|_{1,1}\|\nabla X_p\|_{0,1}\\
&+ \big(|\p_3Y_p|_{\frac12,1}^\frac43\|\p_3Y_p\|_{\frac12,0}^\frac23+|\p_3Y_p|_{\frac12,1}^2\big)\Big(|(1-S_p)\nabla Y_p|_{\frac12,N+1}
\|\nabla X_p\|_{0,1}\\
&\qquad\qquad\qquad\qquad\qquad+|(1-S_p)\nabla Y_p|_{\frac12,1}\bigl(\|\nabla X_p\|_{0,N+1}+|\nabla Y_p|_{0,N+1}
\|\nabla X_p\|_{0,1}\bigr)\Bigr).
\end{split}
\end{equation*}
 A similar estimate holds for $\||D|^{-1}e_{p,1}'\|_{1,N+1}$ with $|(1-S_p)\p_3X_p|_{1,l}$ and $|(1-S_p)\nabla X_p|_{\frac12,l}$  above
 being replaced by $|(1-S_p)\p_3Y_p|_{\frac12,l}$ and  $|(1-S_{p})\nabla X_p|_{0,l}$ respectively.

Hence we deduce from  \eqref{IVieq1} that
\begin{equation*}
\begin{split}
&\||D|^{-1}e_{p,1}'\|_{1,1}\lesssim\eta^2\theta_p^{-\gamma-\beta+\bar\varepsilon},\quad\||D|^{-1}e_{p,1}'\|_{\frac32,1}\lesssim\eta^2\theta_p^{\frac12-\gamma-\beta+\bar\varepsilon}.
\end{split}
\end{equation*}
Interpolating the above two inequalities yields
\begin{equation}\label{IViieq3}
\begin{split}
&\||D|^{-1}e_{p,1}'\|_{1+k,1}\lesssim\eta^2\theta_p^{k-\gamma-\beta+\bar\varepsilon}\quad \text{for}  \ 0\leq k\leq\frac12.
\end{split}
\end{equation}

For $N\leq N_0-1$ satisfying $-\beta+\bar\varepsilon N\geq\bar\varepsilon$ and $-\gamma+\bar\varepsilon(N+1)\geq\bar\varepsilon$, \eqref{IVieq3} holds,  so that
we infer that
\begin{equation*}
\begin{split}
&\||D|^{-1}e_{p,1}'\|_{1,N+1}\lesssim \eta^2\theta_p^{-\gamma-\beta+\bar\varepsilon(N+1)},\quad \||D|^{-1}e_{p,1}'\|_{\frac32,N+1}\lesssim \eta^2\theta_p^{\frac12-\gamma-\beta+\bar\varepsilon(N+1)}.
\end{split}
\end{equation*}
Interpolating the above inequalities leads to
\begin{equation}\label{IViieq4}
\begin{split}
&\||D|^{-1}e_{p,1}'\|_{1+k,N+1}\lesssim \eta^2\theta_p^{k-\gamma-\beta+\bar\varepsilon(N+1)},
\end{split}
\end{equation}  for $0\leq k\leq\frac12$, $N\leq N_0-1$ such that $-\beta+\bar\varepsilon N\geq\bar\varepsilon$ and $-\gamma+\bar\varepsilon(N+1)\geq\bar\varepsilon$.

We then conclude the proof of \eqref{IVieq-5}  by interpolating between \eqref{IViieq3} and \eqref{IViieq4}.

\noindent$\bullet$\underline{\it The proof of \eqref{IVieq-6}}

Applying \eqref{S7.2eqE1} $e_{p,0}'',$ which is determined by \eqref{S9.1eq10} gives that for $N\geq0$,
\begin{equation*}
\begin{split}
\|&|D|^{-1}e_{p,0}''\|_{1,N+1}\lesssim  |\nabla X_p|_{0,N+1}\|\nabla \p_tX_{p}\|_{1,0}+|\nabla X_p|_{0,0}\|\nabla \p_tX_{p}\|_{1,N+1} +|\p_tX_{p}|_{1,N+2}\|\nabla X_p\|_{0,0}\\
&\ +|\p_tX_{p}|_{1,1}\|\nabla X_p\|_{0,N+1}+\sum_{j=p}^{p+1}\Bigl( |\p_tY_{j}|_{1,1}\big(|\nabla X_p|_{0,N+1}\|\nabla X_p\|_{0,0}+|\nabla X_p|_{0,0}\|\nabla X_p\|_{0,N+1}\big)\\
&\ \ +\big(| \p_tY_{j}|_{1,N+2}+|\nabla Y_{j}|_{0,N+1}| \p_tY_{j}|_{1,1}\big)|\nabla X_p|_{0,0}\|\nabla X_p\|_{0,0}\\
&\ \qquad\qquad \qquad\qquad\qquad\qquad\quad+|\nabla Y_j|_{0,N+1}\bigl(|\nabla X_p|_{0,0}\|\nabla \p_tX_{p}\|_{1,0}+|\p_tX_{p}|_{1,1}\|\nabla X_p\|_{0,0}\bigr)\Bigr).
\end{split}
\end{equation*}
An similar estimate holds for $\||D|^{-1}e_{p,0}''\|_{\frac32-\delta,N+1}$ with $|\nabla X_p|_{0,l}$ and $|\p_tX_{p}|_{1,l}$
being replaced by $|\nabla X_p|_{\frac12-\delta,l}$ and $|\p_tX_{p}|_{\frac32-\delta,l}$ respectively.

Therefore,
\begin{equation*}
\begin{split}
&\||D|^{-1}e_{p,0}''\|_{1,1}\lesssim \eta^2\theta_p^{\delta-\gamma-\beta+2\bar\varepsilon},\quad
\||D|^{-1}e_{p,0}''\|_{\frac32-\delta,1}\lesssim \eta^2\theta_p^{-\gamma-\beta+2\bar\varepsilon}.
\end{split}
\end{equation*}
By interpolating between the above two inequalities, we obtain
\begin{equation}\label{IViieq5}
\begin{split}
&\||D|^{-1}e_{p,0}''\|_{1+k,1}
\lesssim \eta^2\theta_p^{k+\delta-\gamma-\beta+2\bar\varepsilon}\quad\text{for}\ \ 0\leq k\leq\frac12-\delta.
\end{split}
\end{equation}

For $N\leq N_0-2$ with $-\gamma+\bar\varepsilon(N+1)\geq\bar\varepsilon$, we get,  by applying \eqref{IVieq7a}, that
\begin{equation*}
\begin{split}
&\||D|^{-1}e_{p,0}''\|_{1,N+1}\lesssim \eta^2\theta_{p}^{\delta-\gamma-\beta+\bar\varepsilon(N+2)},\quad
\||D|^{-1}e_{p,0}''\|_{\frac32-\delta,N+1}\lesssim \eta^2\theta_{p}^{\frac12-\gamma-\beta+\bar\varepsilon(N+2)}.
\end{split}
\end{equation*}
Interpolating the above two inequalities, we obtain
\begin{equation}\label{IViieq6}
\begin{split}
&\||D|^{-1}e_{p,0}''\|_{1+k,N+1}\lesssim \eta^2\theta_{p}^{k+\delta-\gamma-\beta+\bar\varepsilon(N+2)}
\end{split}
\end{equation} for $0\leq k\leq\frac12-\delta$, $N\leq N_0-2$ satisfying $-\gamma+\bar\varepsilon(N+1)\geq\bar\varepsilon.$

By interpolating the inequalities  \eqref{IViieq5} and \eqref{IViieq6}, we conclude the proof of \eqref{IVieq-6}.

\noindent$\bullet$\underline{\it The proof of \eqref{IVieq-7}}

Again by applying \eqref{S7.2eqE1} $e_{p,0}'$ determined by \eqref{S9.1eq10}, we obtain for $N\geq0$
\begin{equation*}
\begin{split}
\||&D|^{-1}e_{p,0}'\|_{1,N+1}
\lesssim |(1-S_p)\nabla Y_p|_{0,0}\Bigl(\big(| \p_tY_{p}|_{1,N+2}+|\nabla Y_p|_{0,N+1}|\p_t Y_{p}|_{1,1}\big)\|\nabla X_p\|_{0,0}\\
&+| \p_tY_{p}|_{1,1}\|\nabla X_p\|_{0,N+1}+\|\nabla \p_tX_{p}\|_{1,N+1}+|\nabla Y_p|_{0,N+1}\|\nabla \p_tX_{p}\|_{1,0}\Bigr) \\
&+|(1-S_p)\nabla Y_p|_{0,N+1}\bigl(| \p_tY_{p}|_{1,1}\|\nabla X_p\|_{0,0}+ \|\nabla \p_tX_{p}\|_{1,0}\bigr)\\
&+|(1-S_p)\p_tY_{p}|_{1,N+2}\|\nabla X_p\|_{0,0}+|(1-S_p)\p_tY_{p}|_{1,1}\bigl(\|\nabla X_p\|_{0,N+1}+|\nabla Y_p|_{0,N+1}\|\nabla X_p\|_{0,0}\bigr).
\end{split}
\end{equation*}
 A similar estimate holds  for $\||D|^{-1}e_{p,0}'\|_{\frac32-\delta,N+1}$ with $|(1-S_p)\nabla Y_p|_{0,l}$
  and $|(1-S_p)\p_tY_{p}|_{1,l}$ being replaced by $|(1-S_p)\nabla Y_p|_{\frac12-\delta,l}$ and
   $|(1-S_p)\p_tY_{p}|_{\frac32-\delta,l}$ respectively.

So that it follows from \eqref{IVieq1''} that
\begin{equation*}
\||D|^{-1}e_{p,0}'\|_{1,1}
\lesssim \eta^2\theta_p^{\delta-\gamma-\beta+2\bar\varepsilon},\quad
\||D|^{-1}e_{p,0}'\|_{\frac32-\delta,1}
\lesssim\eta^2\theta_p^{\frac12-\gamma-\beta+2\bar\varepsilon}.
\end{equation*}
Interpolating the above two inequalities yields
\begin{equation}\label{IViieq7}
\||D|^{-1}e_{p,0}'\|_{1+k,1}
\lesssim \eta^2\theta_p^{k+\delta-\gamma-\beta+2\bar\varepsilon}\quad\text{for}\ \ 0\leq k\leq\frac12-\delta.
\end{equation}
For $N\leq N_0-2$ with $-\gamma+\bar\varepsilon(N+1)\geq\bar\varepsilon$, we deduce from \eqref{IVieq7a} that
\begin{equation*}
\begin{split}
&\||D|^{-1}e_{p,0}'\|_{1,N+1}\lesssim \eta^2\theta_p^{\delta-\gamma-\beta+\bar\varepsilon(N+2)},\quad
\||D|^{-1}e_{p,0}'\|_{\frac32-\delta,N+1}\lesssim \eta^2\theta_p^{\frac12-\gamma-\beta+\bar\varepsilon(N+2)}.
\end{split}
\end{equation*}
Interpolating the above two inequalities yields
\begin{equation}\label{IViieq8}
\begin{split}
&\||D|^{-1}e_{p,0}'\|_{1+k,N+1}\lesssim \eta^2\theta_p^{k+\delta-\gamma-\beta+\bar\varepsilon(N+2)}
\end{split}
\end{equation}
for $0\leq k\leq\frac12-\delta$, $N\leq N_0-2$ such that $-\gamma+\bar\varepsilon(N+1)\geq\bar\varepsilon.$
By interpolating the inequalities \eqref{IViieq7}, \eqref{IViieq8}, we obtain \eqref{IVieq-7}.

This completes the proof of Lemma \ref{S9lem2}.

\subsection{The proof of Lemma \ref{S9lem3}}

We divide the proof of this lemma by the following steps:

\noindent$\bullet$\underline{\it The proof of \eqref{IVieq-a}}

Applying \eqref{S7.3eqE2}  to $_{p,1}''$ (with $Y\simeq Y_p+Y_{p+1}$, $X=W=X_p$), which is determined by \eqref{S9.1eq10}, yields that for any $N\geq0$,
\begin{equation*}
\begin{split}
&\triplenorm{e_{p,1}''}_{L_t^1(\d,N)}
\lesssim\|\p_3X_p\|_{L_t^2(H^{N+6})}\|\p_3X_p\|_{L_t^2(L^2)}+\sum_{j=p}^{p+1}\|\p_3X_p\|_{L_t^2(H^1)}\Big\{\Bigl(\|\nabla Y_j\|_{0,N+6}|\p_3X_p|_{\frac12+\bar\varepsilon,0}\\
&\ +|\p_{3}Y_{j}|_{\frac12+\bar\varepsilon,0}\|\nabla X_p\|_{0,N+6} +\big(\|\p_3Y_{j}\|_{L_t^2(H^{N+6})}+\|\nabla Y_{j}\|_{0,N+6}|\p_3Y_{j}|_{\frac12+\bar\varepsilon,1}\big)|\nabla X_p|_{0,1}
\Bigr)\\
&\ +\Bigl(|\p_{3}Y_{j}|_{\frac12+\bar\varepsilon,0}\|\nabla X_p\|_{0,0}+\|\p_3Y_{j}\|_{L_t^2(L^2)}|\nabla X_p|_{0,0}+|\p_3Y_{j}|_{\frac12+\bar\varepsilon,0}\|\nabla X_p\|_{0,0}\Bigr)\|\p_3X_p\|_{L_t^2(H^{N+6})}\\
&\ +\Big(\|\p_3Y_{j}\|_{L_t^2(H^{N+6})}+\|\nabla Y_{j}\|_{0,N+6}\bigl(|\p_3Y_{j}|_{\frac12+\bar\varepsilon,0}+\|\p_3Y_{j}\|_{L_t^2(H^3)}\big)\Big)\|\nabla X_p\|_{0,1}|\p_3X_p|_{\frac12+\bar\varepsilon,1}\\
&\ +|\p_3Y_{j}|_{\frac12+\bar\varepsilon,1}\|\p_3Y_{p}\|_{L_t^2(H^3)}\|\nabla X_p\|_{0,N+6}\|\nabla X_p\|_{0,0}\\
&\ +\big(|\p_3Y_{j}|_{\frac12+\bar\varepsilon,1}\|\p_3Y_{j}\|_{L_t^2(H^{N+6})} +\|\nabla Y_{j}\|_{0,N+6}|\p_3Y_{p}|_{\frac12+\bar\varepsilon,1}\|\p_3Y_{j}\|_{L_t^2(H^3)}\big)|\nabla X_p|_{0,0}\|\nabla X_p\|_{0,1}\Bigr\}.
\end{split}
\end{equation*}

Note from \eqref{S9eqP} and \eqref{S9eqebd} that   $\beta\geq6\bar\varepsilon,$ so that we deduce from
\eqref{Ieq4} that
\begin{equation}\label{IViiieq1}
\|\p_3Y_{p+1}\|_{0,6}+\|\nabla Y_{p+1}\|_{0,6}+\|\p_3Y_{p+1}\|_{L_t^2(H^6)}\leq C\eta.
\end{equation}
As a result, it comes out
\begin{equation}\label{IViiieq1a}
\triplenorm{e_{p,1}''}_{L_t^1(\d,0)}\lesssim \eta^2\big(\theta_p^{-2\beta+5\bar\varepsilon}+\theta_p^{-\gamma-\beta+5\bar\varepsilon}\big)
\lesssim \eta^2\theta_p^{-\beta-\gamma+5\bar\varepsilon}.
\end{equation}
The case when $N\leq N_0-6$ with $-\beta+\bar\varepsilon(N+5)\geq\bar\varepsilon,$ it follows from \eqref{Ieq4} that
\begin{equation}\label{IViiieq2}
\|\nabla Y_{p+1}\|_{0,N+6}+\|\p_3Y_{p+1}\|_{L_t^2(H^{N+6})}\leq C\eta\theta_p^{-\beta+\bar\varepsilon(N+5)},
\end{equation}
hence, we achieve
\begin{equation}\label{IViiieq2a}
\triplenorm{e_{p,1}''}_{L_t^1(\d,N)}\lesssim \eta^2\big(\theta_p^{-2\beta+\bar\varepsilon(N+5)}+\theta_p^{-\gamma-\beta+\bar\varepsilon(N+5)}\big)
\lesssim \eta^2\theta_p^{-\beta-\gamma+\bar\varepsilon(N+5)}.
\end{equation}
\eqref{IVieq-a} then follows by interpolating \eqref{IViiieq1a} and \eqref{IViiieq2a}.

\noindent$\bullet$\underline{\it The proof of \eqref{IVieq-b}}

By applying \eqref{S7.3eqE2} to $e_{p,1}'$ ( with $Y\simeq Y_p$, $X=(1-S_p)Y_p$ and $W=X_p$) and noticing that
$$\|(1-S_p)\p_3Y\|_{N}\leq \|\p_3Y\|_{N},\quad \|(1-S_p)\nabla Y_p\|_{N}\leq\|\nabla Y\|_{N}$$
we get
\begin{equation*}
\begin{split}
&\triplenorm{e_{p,1}'}_{L_t^1(\d,N)}
\lesssim\Bigl(\|\p_3 Y_p\|_{L_t^2(H^{N+6})}\bigl(1+|(1-S_p)\nabla Y_p|_{0,1}\bigr)+\|\nabla Y_p\|_{0,N+6}|(1-S_p)\p_3 Y_p|_{\frac12+\bar\varepsilon,0}\\
&+
|\p_{3}Y_p|_{\frac12+\bar\varepsilon,0}\|\nabla Y_p\|_{0,N+6}\Bigr)\|\p_3X_p\|_{L_t^2(H^1)}+\Bigl(\|\p_3 Y_p\|_{L_t^2(L^2)}\bigl(1+|(1-S_p)\nabla Y_p|_{0,0}\bigr)\\
&+|\p_{3}Y_p|_{\frac12+\bar\varepsilon,0}\|\nabla Y_p\|_{0,0}\Bigr)\|\p_3X_p\|_{L_t^2(H^{N+6})}+\Bigl(|\p_3Y_p|_{\frac12+\bar\varepsilon,1}\bigl(\|\nabla X_p\|_{0,N+6}(1+ \|\nabla Y_p\|_{0,0}\\
&+|(1-S_p)\nabla Y_p|_{0,0})+\|\nabla Y_p\|_{0,N+6}(|\nabla X_p|_{0,0}+\|\nabla X_p\|_{0,0})\bigr)\\
&+\|\nabla Y_p\|_{0,N+6}|(1-S_p)\p_3 Y_p|_{\frac12+\bar\varepsilon,1}\|\nabla X_p\|_{0,1}\Bigr)+\|\p_3Y_p\|_{L_t^2(H^3)}+\|\p_3Y_p\|_{L_t^2(H^{N+6})}\times\\
&\times\Bigl(\|\nabla Y_p\|_{0,0}|\p_3X_p|_{\frac12+\bar\varepsilon,0}+\bigl(|\p_3Y_p|_{\frac12+\bar\varepsilon,0}
+|(1-S_p)\p_3 Y_p|_{\frac12+\bar\varepsilon,1}\bigr)\|\nabla X_p\|_{0,1}\\
&+\|\p_3 Y_p\|_{L_t^2(H^1)}|\nabla X_p|_{0,0}\Bigr)+\|\nabla Y_p\|_{0,N+6}|\p_3Y_p|_{\frac12+\bar\varepsilon,0}\|\nabla Y_p\|_{0,0}|\p_3X_p|_{\frac12+\bar\varepsilon,0}
\\
&+\big(|(1-S_p)\nabla Y_p|_{0,0}\|\nabla X_p\|_{0,1}+\|\nabla Y_p\|_{0,1}|\nabla X_p|_{0,0}\big)\\
&\qquad\qquad\qquad\qquad\times\big(|\p_3Y_p|_{\frac12+\bar\varepsilon,1}\|\p_3Y_p\|_{L_t^2(H^{N+6})}+\|\nabla Y_p\|_{0,N+6}|\p_3Y_p|_{\frac12+\bar\varepsilon,1}\|\p_3Y_p\|_{L_t^2(H^3)}\big).
\end{split}
\end{equation*}
Inserting \eqref{IViiieq1} into the above inequality for $N=0$ gives
\begin{equation*}
\begin{split}
&\triplenorm{e_{p,1}'}_{L_t^1(0)}\lesssim\eta^2\big( \theta_p^{-\beta+5\bar\varepsilon}+\theta_p^{-\gamma-\beta+5\bar\varepsilon}+\theta_p^{-\gamma+\bar\varepsilon}\big)
\lesssim \eta^2\theta_p^{-\gamma+5\bar\varepsilon}.
\end{split}
\end{equation*}
Whereas
for $N\leq N_0-6$ such that $-\beta+\bar\varepsilon(N+5)\geq\bar\varepsilon$, by substituting  \eqref{IViiieq2} into the above inequality, we achieve
\begin{equation*}
\begin{split}
&\triplenorm{e_{p,1}'}_{L_t^1(\d,N)}\lesssim\eta^2\big( \theta_p^{-\beta+\bar\varepsilon(N+5)}+\theta_p^{-\gamma-2\beta+\bar\varepsilon(N+7)}+\theta_p^{-\gamma-\beta+\bar\varepsilon(N+6)}\big)
\lesssim \eta^2\theta_p^{-\beta+\bar\varepsilon(N+5)}.
\end{split}
\end{equation*}
Then \eqref{IVieq-b} follows by interpolating the above two inequalities.

\noindent$\bullet$\underline{\it The proof of \eqref{IVieq-c}}

Applying \eqref{S7.3eqE1} to $e_{p,0}''$ gives,
\begin{equation*}
\begin{split}
&\triplenorm{\langle t\rangle^\frac12e_{p,0}''}_{L_t^2(\d,N)} \lesssim \|\nabla X_p\|_{0,0}\|\langle t\rangle^\frac12\nabla\p_t X_{p}\|_{L_t^2(H^{N+6})}+\|\nabla X_p\|_{0,N+6}\|\langle t\rangle^\frac12\nabla\p_t X_{p}\|_{L_t^2(L^2)}\\
&\quad+\sum_{j=p}^{p+1}\Bigl(
| \p_tY_{j}|_{1+\bar\varepsilon,1}\|\nabla X_p\|_{0,N+6}\|\nabla X_p\|_{0,0}+\|\nabla Y_{j}\|_{0,N+6}|\nabla X_p|_{0,0}\|\langle t\rangle^\frac12\nabla\p_t X_{p}\|_{L_t^2(L^2)}\\
&\qquad\qquad +\big(\|\langle t\rangle^\frac12\nabla \p_tY_{j}\|_{L_t^2(H^{N+6})}+\|\nabla Y_{j}\|_{0,N+6}| \p_tY_{j}|_{1+\bar\varepsilon,1}\big)|\nabla X_p|_{0,0}\|\nabla X_p\|_{0,0}\Bigr).
\end{split}
\end{equation*}
Again due to   $\beta\geq7\bar\varepsilon,$  we deduce from
\eqref{Ieq4} that
\begin{equation}\label{IViiieq1g}
\|\nabla Y_{p+1}\|_{0,6}+\|\langle t\rangle^\frac12\nabla\p_tY_{p+1}\|_{L_t^2(H^6)}\leq C\eta.
\end{equation}
As a result, it comes out
\begin{equation*}
\triplenorm{\langle t\rangle^\frac12 e_{p,0}''}_{L_t^2(\d,0)} 
\lesssim \eta^2\theta_p^{-\beta-\gamma+5\bar\varepsilon}.
\end{equation*}

The case when $N\leq N_0-6$ with $-\beta+\bar\varepsilon(N+5)\geq\bar\varepsilon,$ it follows from \eqref{Ieq4} that
\begin{equation}\label{IViiieq2h}
\|\nabla Y_{p+1}\|_{0,N+6}+\|\langle t\rangle^\frac12 \nabla\p_tY_{p+1}\|_{L_t^2(H^{N+6})}\leq C\eta\theta_p^{-\beta+\bar\varepsilon(N+5)},
\end{equation}
so that in this case, we have
\begin{equation*}
\begin{split}
&\triplenorm{\langle t\rangle^\frac12 e_{p,0}''}_{L_t^2(\d,N)}
\lesssim \eta^2\theta_p^{-\beta-\gamma+\bar\varepsilon(N+5)}.
\end{split}
\end{equation*}
\eqref{IVieq-c} follows by
interpolating the above inequalities.

\noindent$\bullet$\underline{\it The proof of \eqref{IVieq-d}}

Applying \eqref{S7.3eqE1} to $e_{p,0}'$ gives
\begin{equation*}
\begin{split}
&\triplenorm{\langle t\rangle^\frac12e_{p,0}'}_{L_t^2(\delta,N)} \lesssim\Big(\|\nabla Y_p\|_{0,N+6}\bigl(
(1+|(1-S_p)\nabla Y_p|_{0,0})\|\nabla X_p\|_{0,0}+\|\nabla Y_p\|_{0,0}|\nabla X_p|_{0,0}\big)\\
&\quad+\|\nabla Y_p\|_{0,0}\|\nabla X_p\|_{0,N+6}\Big)|\p_t Y_{p}|_{1+\bar\varepsilon,1}+ \|\nabla Y_p\|_{0,0}\|\langle t\rangle^\frac12\nabla\p_t X_{p}\|_{L_t^2(H^{N+6})}\\
&\quad +\bigl(1+|(1-S_p)\nabla Y_p|_{0,0}\bigr)\|\nabla Y_p\|_{0,N+6}\|\langle t\rangle^\frac12\nabla\p_t X_{p}\|_{L_t^2(L^2)}\\
&\quad +\bigl(\|\nabla X_p\|_{0,N+6}
+\|\nabla Y_p\|_{0,N+6}|\nabla X_p|_{0,0}\bigr)\|\langle t\rangle^\frac12\nabla\p_t Y_{p}\|_{L_t^2(L^2)}\\
&\quad+\bigl((1+|(1-S_p)\nabla Y_p|_{0,0})\|\nabla X_p\|_{0,0}+\|\nabla Y_p\|_{0,0}|\nabla X_p|_{0,0}\bigr)\|\langle t\rangle^\frac12\nabla\p_t Y_{p}\|_{L_t^2(H^{N+6})}.
\end{split}
\end{equation*}
Then as in the proof of \eqref{IVieq-c}, we deduce that
\begin{equation*}
\triplenorm{\langle t\rangle^\frac12e_{p,0}'}_{L_t^2(\d,0)}
\lesssim\eta^2 \theta_p^{-\gamma+5\bar\varepsilon},
\end{equation*}
and
\begin{equation*}
\triplenorm{\langle t\rangle^\frac12e_{p,0}'}_{L_t^2(\d,N)}
\lesssim \eta^2\theta_p^{-\beta+\bar\varepsilon(N+5)}.
\end{equation*}
for $N\leq N_0-6$  with $-\beta+\bar\varepsilon(N+5)\geq\bar\varepsilon.$
Then \eqref{IVieq-d} follows by  interpolating the above inequalities.

This ends the proof of Lemma \ref{S9lem3}.

\medbreak \noindent {\bf Acknowledgments.} P. Zhang would like to thank Professor Fanghua Lin and Professor Jalal Shatah for profitable discussions.
 P. Zhang is partially
supported by NSF of China under Grant 11371347 and innovation grant from National
Center for Mathematics and Interdisciplinary Sciences.

\end{document}